\documentclass[10pt, leqno]{amsart} 
\usepackage{amssymb,amscd,amsfonts,amsbsy}
\usepackage{latexsym}
\usepackage{exscale}
\usepackage{amsmath,amsthm,amsfonts}
\usepackage{mathrsfs}
\usepackage{xcolor}
\usepackage[colorlinks=true,linkcolor=blue,citecolor=red,urlcolor=red, 
]{hyperref}
\usepackage{esint}
\usepackage{stmaryrd}
\usepackage{pifont}

\usepackage[utf8]{inputenc}

\calclayout

\usepackage{autobreak}
\allowdisplaybreaks

\makeatletter
\@namedef{subjclassname@2020}{%
  \textup{2020} Mathematics Subject Classification}
\makeatother

\theoremstyle{plain}
\newtheorem{theorem}[equation]{Theorem}
\newtheorem{lemma}[equation]{Lemma}

\newtheorem{proposition}[equation]{Proposition}

\theoremstyle{definition}
\newtheorem{definition}[equation]{Definition}

\numberwithin{equation}{section} 

\def\B{\mathcal{B}}
\def\C{\mathbb{C}}
\def\N{\mathbb{N}}
\def\D{\mathcal{D}}

\def\F{\mathscr{F}}
\def\S{\mathcal{S}}
\def\G{\mathcal{G}}
\def\M{\mathcal{M}}
\def\K{\mathcal{K}}
\def\O{\mathcal{O}} 
\def\Q{\mathcal{Q}} 

\def\I{\mathbb{I}}
\def\R{\mathbb{R}}
\def\Rn{\mathbb{R}^n}

\def\Z{\mathbb{Z}}
\def\W{\mathcal{W}}
\def\w{\omega}
\def\b{\mathbf{b}}

\def\dist{\operatorname{dist}}
\def\diam{\operatorname{diam}}

\def\d{\operatorname{d}}
\def\rd{\operatorname{rd}}
\def\supp{\operatorname{supp}}
\def\BMO{\operatorname{BMO}}

\def\CMO{\operatorname{CMO}}

\def\loc{\operatorname{loc}}
\def\sgn{\operatorname{sgn}}

\def\ch{\operatorname{ch}}

\DeclareMathOperator*{\esssup}{ess\,sup}
\DeclareMathOperator*{\essinf}{ess\,inf}

\begin{document}

\author{Mingming Cao}
\address{Mingming Cao\\
Institute for Advanced Study in Mathematics\\ 
Harbin Institute of Technology\\ 
Harbin 150001\\
People's Republic of China}
\email{mingming.cao@hit.edu.cn}

\author{Honghai Liu}
\address{Honghai Liu\\
School of Mathematics and Information Science\\
Henan Polytechnic University\\
Jiaozuo 454000\\
People's Republic of China} \email{hhliu@hpu.edu.cn}

\author{Zengyan Si}
\address{Zengyan Si\\
School of Mathematics and Information Science\\
Henan Polytechnic University\\
Jiaozuo 454000\\
People's Republic of China} \email{zengyan@hpu.edu.cn}

\author{K\^{o}z\^{o} Yabuta}
\address{K\^{o}z\^{o} Yabuta\\
Research Center for Mathematics and Data Science\\
Kwansei Gakuin University\\
Gakuen 2-1, Sanda 669-1337\\
Japan}\email{kyabuta3@kwansei.ac.jp}


\year=2026 \month=11 \day=20

\date{\today}




\begin{abstract}
In this paper we develop the bilinear $T1$ theorem to characterize the (weighted) compactness of bilinear Calder\'{o}n--Zygmund operators. Let $T$ be a bilinear operator associated with a standard bilinear Calder\'{o}n--Zygmund kernel. We prove that $T$ can be extended to a compact bilinear operator from $L^{p_1}(w_1^{p_1}) \times L^{p_2}(w_2^{p_2})$ to $L^p(w^p)$ for all exponents $\frac1p = \frac{1}{p_1} + \frac{1}{p_2}>0$ with $p_1, p_2 \in (1, \infty]$ and for all weights $(w_1, w_2) \in A_{(p_1, p_2)}$ if and only if the following hypotheses are satisfied: (H1) $T$ is associated with a compact bilinear Calder\'{o}n--Zygmund kernel, (H2) $T$ satisfies the weak compactness property, and (H3) $T(1,1), T^{*1}(1,1), T^{*2}(1,1) \in \mathrm{CMO}(\mathbb{R}^n)$. This is also equivalent to the endpoint compactness: (1) $T$ is compact from $L^1(w_1) \times L^1(w_2)$ to $L^{\frac12, \infty}(w^{\frac12})$ for all $(w_1, w_2) \in A_{(1, 1)}$, or (2) $T$ is compact from $L^{\infty}(w_1^{\infty}) \times L^{\infty}(w_2^{\infty})$ to $\mathrm{CMO}_{\lambda}(w^{\infty})$ for all $(w_1, w_2) \in A_{(\infty, \infty)}$, where $0<\lambda<1/2$. Besides, any of these properties is equivalent to the fact that $T$ admits a compact bilinear dyadic representation, which is the first time to provide the exact dyadic structure behind compact bilinear Calder\'{o}n--Zygmund operators. 

Our main approaches consist of the following new ingredients: (i) a dyadic representation of a compact bilinear Calder\'{o}n--Zygmund operator as an average of some compact bilinear dyadic shifts and paraproducts; (ii) extrapolation of endpoint compactness for bilinear operators; and (iii) compactness criterion in weighted Lorentz spaces. Finally, to illustrate the applicability of our result, we demonstrate the hypotheses (H1)--(H3) through examples including bilinear continuous/dyadic paraproducts, bilinear pseudo-differential operators, and bilinear commutators.  
\end{abstract}

\title{A characterization of compactness via bilinear $T1$ theorem} 

\maketitle

\tableofcontents

\section{Introduction}
The classical Calder\'{o}n--Zygmund theory originated in the initial work of Calder\'{o}n and Zygmund in the 1950s \cite{CZ52}, which was motivated by the connections with potential theory and elliptic partial differential equations. Over the past seven decades, this theory and its branches have proved to be a powerful tool in many fields, such as real and complex analysis, operator theory, approximation theory, partial differential equations, and geometric measure theory, see \cite{DS91, DS93, HMM, HMU, HMMM, M5, M4, Stein70, Stein93}. For additional related research, see \cite{GR, J83, M90, MC}.

The multilinear Calder\'{o}n--Zygmund theory, pioneered by Coifman and Meyer \cite{CM75, CM78}, was inspired by the natural appearance in the study of certain singular integral operators, for example, the Calder\'{o}n commutator, paraproducts, and pseudodifferential operators. Subsequently, this topic was further developed by several authors, including Christ and Journ\'{e} \cite{CJ}, Lacey and Thiele \cite{LT97, LT99}, Kenig and Stein \cite{KS}, Grafakos and Torres \cite{GT1, GT2}, and so on. It is worth mentioning that the remarkable work \cite{GT1} provides a systematic treatment to general multilinear Calder\'{o}n--Zygmund operators. In the weighted theory, genuinely multilinear weighted estimates were first established by Lerner et al. using a natural class of multiple weights in their groundbreaking paper \cite{LOPTT}. Beyond that, it has witnessed the fundamental importance of $L^p$ estimates for multilinear singular integrals in pure and applied analysis ---for example, the fractional Leibniz rules used in dispersive equations \cite{GO, KP}, 
paraproducts employed to Navier--Stokes equations \cite{M96}, and certain multilinear transforms applied to the stability of the absolutely continuous spectrum of Schr\"{o}dinger operators \cite{CK}.

The boundedness of operators on various function spaces is a central theme of the (multilinear) Calder\'{o}n--Zygmund theory. Specially, the $L^2$ boundedness is essential to the further analysis, such as achieving $L^p$ boundedness, the weak type $(1, 1)$ estimate, and weighted norm inequalities. It is therefore natural to ask what conditions to guarantee the boundedness of $T$ on $L^2$. A necessary and sufficient condition for this to happen is given by the celebrated $T1$  theorem due to David and Journ\'{e} \cite{DJ}. More precisely, 
\begin{align*}
&\text{a singular integral $T$ associated with a Calder\'{o}n--Zygmund}
\\ 
&\text{kernel is bounded on $L^2(\Rn)$ if and only if it satisfies the}
\\
&\text{weak boundedness property and $T1, T^*1 \in  \BMO(\Rn)$}.  
\end{align*}
This formulation of the $T1$ theorem was extended to the multilinear setting by Christ and Journ\'{e} \cite{CJ}, while an alternative expression of the multilinear $T1$ theorem was provided by Grafakos and Torres \cite{GT1}. The latter is very effective in obtaining some continuity results for multilinear translation invariant operators and multilinear pseudodifferential operators. Although further developments were made in \cite{Hart14, Hart15}, more modern bilinear dyadic-probabilistic methods were exploited in \cite{LMOV} on Euclidean spaces and \cite{MV} on nonhomogeneous spaces.  

Recently, the study of compactness has attracted a lot of attention. In terms of commutators $[b, T]$, one would like to use the compactness of $[b, T]$ to characterize the fact $b \in \CMO(\Rn)$. This is the case of bilinear Riesz potential \cite{CT}. The weighted compactness was investigated in \cite{BDMT, TYY} for bilinear Calder\'{o}n--Zygmund operators, where the latter weakened the $\CMO$ condition by a new vanishing Lipschitz-type condition. Besides, for non-degenerate Calder\'{o}n--Zygmund operators, a  characterization in the off-diagonal case was obtained in \cite{HLTY, HOS}. All compactness results aforementioned are traced back to the work of Uchiyama \cite{Uch}. Furthermore, extrapolation of compactness introduced in \cite{HL23} provides a new approach to establish weighted compactness, and shortly after, the multilinear extension was presented in \cite{CIRXY, COY, HL22}, which have a wide range of applications. On the contrary, the existing literature exploring the compactness of singular integrals is exceedingly scarce. Based on new $T1$ type assumptions, Villarroya \cite{Vil} first characterized the compactness for a large class of singular integral operators, which later was extended to the endpoint case in \cite{OV, PPV}.


Establishing the $T1$ theorem for the compactness of multilinear singular integrals is a natural and interesting problem. 
In view of the complexity of the multilinear structure, the first and last authors and Yang \cite{CYY} began with the bi-parameter setting to gain some insight and observations. Indeed, the authors\cite{CYY}  developed a general theory of the compactness for a class of bi-parameter singular integrals. The purpose of the current work is to solve this problem by establishing a characterization of (weighted) compactness of bilinear singular integrals via bilinear $T1$ theorem. Before stating the precise result, let us illustrate some difficulties brought by the bilinearity with some examples and facts.
 

First, in general, bilinear Calder\'{o}n--Zygmund operators (cf. Definition \ref{def:CZO}) are not compact. A   typical example is the bilinear Riesz transform $\mathcal{R}_j$ defined by 
\begin{align*}
\mathcal{R}_j (f, g)(x) 
:= \mathrm{p.v. } \iint_{\R^{2n}} 
\frac{(x_j-y_j) + (x_j-z_j)}{(|x-y|^2 + |x-z|^2)^{\frac{2n+1}{2}}} f(y) g(z)\, dy \, dz,  
\end{align*}
where $x_j$ is the $j$-th coordinate of $x$, $j=1, \ldots, n$. Obviously, $\mathcal{R}_j$ is associated with a standard bilinear Calder\'{o}n--Zygmund kernel. Given a bilinear operator $T$, if we define its adjoints via 
\begin{align*}
\langle T(f_1, f_2), g \rangle 
= \langle T^{*1}(g, f_2), f_1 \rangle 
= \langle T^{*2}(f_1, g), f_2 \rangle, 
\end{align*} 
then by the translation invariance property of $\mathcal{R}_j$ and the antisymmetry of its kernel,  
\begin{align*}
\langle \mathcal{R}_j(\mathbf{1}_Q, \mathbf{1}_Q), \mathbf{1}_Q \rangle
= 0 = \mathcal{R}_j(1, 1) = \mathcal{R}_j^{*1}(1, 1) = \mathcal{R}_j^{*2}(1, 1), 
\end{align*}
for all cubes $Q \subset \Rn$. By \cite{GT1} or \cite{LOPTT}, $\mathcal{R}_j$ is a bilinear Calder\'{o}n--Zygmund operator, i.e., 
\begin{align*}
\textstyle\text{$\mathcal{R}_j$ is bounded from $L^{p_1} \times L^{p_2}$ to $L^p$ 
for all $\frac1p = \frac{1}{p_1} + \frac{1}{p_2}$ with $p_1, p_2 \in (1, \infty)$}, 
\end{align*}
but Lemma \ref{lem:trans}, along with the translation invariance property of $\mathcal{R}_j$, implies that 
\begin{align*}
\textstyle\text{$\mathcal{R}_j$ is not compact from $L^{p_1} \times L^{p_2}$ to $L^p$ 
for all $\frac1p = \frac{1}{p_1} + \frac{1}{p_2}$ with $p_1, p_2 \in (1, \infty)$}. 
\end{align*}
This means that the assumption of bilinear Calder\'{o}n--Zygmund kernels is not sufficient to achieve  compactness, which is not only the main reason why there is very limited literature concerning the compactness of singular integrals, but also indicates that one has to strengthen the assumption on kernels in order to obtain compactness of bilinear Calder\'{o}n--Zygmund operators. 

Second, in the bilinear setting, the natural range for compactness is $L^{p_1} \times L^{p_2} \to L^p$ for $\frac1p = \frac{1}{p_1} + \frac{1}{p_2}$ with $p_1, p_2 \in (1, \infty)$, and reaching the quasi-Banach range $p < 1$ (even for boundedness) can often turn out to be a crucial challenge. From the technical perspective, there is no existing approach to  this problem: even in the Banach range $p \ge 1$, the convoluted geometric structure among cubes $I_1, I_2, I_3$ makes it impossible to establish good estimates for the pair $\langle T(\psi_{I_1}, \psi_{I_2}), \psi_{I_3} \rangle$ in terms of the space and frequency localization of bump functions using wavelet analysis as in \cite{Vil}, and for the pair $\langle T(h_{I_1}, h_{I_2}), h_{I_3} \rangle$ in terms of the relative size and distance of cubes using dyadic analysis as in \cite{CYY}. 

Third, in terms of the weighted $L^{p_1}(w_1^{p_1}) \times L^{p_2}(w_2^{p_2}) \to L^p(w^p)$ compactness, this problem involves a genuinely bilinear nature because the right class of weights is $A_{(p_1, p_2)}$ introduced in \cite{LOPTT}, which is strictly richer than $(w_1^{p_1}, w_2^{p_2}) \in A_{p_1} \times A_{p_2}$. In the context of the natural multiple weights class $A_{(p_1, p_2)}$, one has to work with component weights that are linked another, and each individual weight may be not locally integrable but their product should behave well. This has happened to weighted compactness of commutators in \cite{BDMT, CT}, which only treats product weights $(w_1^{p_1}, w_2^{p_2}) \in A_{p_1} \times A_{p_2}$ with $p, p_1, p_2 \in (1, \infty)$, and thus does not quite embody  the bilinear nature of the problem.

In this paper, we will overcome these difficulties and give a bilinear compact $T1$ theorem, which is not only valid for the natural multiple weights class $A_{(p_1, p_2)}$ but also enables us to work with the endpoint exponents $p_i=1$ or $p_i=\infty$.

\subsection{The main theorem}
Now let us proceed to formulate some basic definitions and present our theorem.  
Let $\mathscr{X}$, $\mathscr{Y}$, $\mathscr{Z}$ be quasi-normed spaces and $T: \mathscr{X} \times \mathscr{Y} \to \mathscr{Z}$ be a bilinear operator. We say that $T$ is \emph{bounded} from $\mathscr{X} \times \mathscr{Y}$ to $\mathscr{Z}$ if   
\begin{align*}
\|T\|_{\mathscr{X} \times \mathscr{Y} \to \mathscr{Z}} 
:= \inf\big\{C_0>0: \|T(f, g)\|_{\mathscr{Z}} 
\le C_0 \|f\|_{\mathscr{X}} \|g\|_{\mathscr{Y}},  \, \, 
f \in \mathscr{X}, g \in \mathscr{Y} \big\} < \infty. 
\end{align*}
We say that $T$ is \emph{compact} from $\mathscr{X} \times \mathscr{Y}$ to $\mathscr{Z}$ if for all bounded sets $A \times B \subset \mathscr{X} \times \mathscr{Y}$, the set $T(A \times B)$ is \emph{precompact} in $\mathscr{Z}$, i.e., $\overline{T(A \times B)}$ is a compact subset of $\mathscr{Z}$. Equivalently, $T$ is compact if for all bounded sequences $\{(f_k, g_k)\} \subset \mathscr{X} \times \mathscr{Y}$, the sequence $\{T(f_k, g_k)\}$ has a convergent subsequence in $\mathscr{Z}$. See \cite{BT} for more properties about the compactness of bilinear operators. More definitions and notation are given in Section \ref{sec:pre}. To state the main theorem conveniently, we make the following hypotheses: 
\begin{list}{(\theenumi)}{\usecounter{enumi}\leftmargin=1cm \labelwidth=1cm \itemsep=0.1cm \topsep=0.2cm \renewcommand{\theenumi}{H\arabic{enumi}}}
 
\item\label{H1} $T$ is associated with a compact bilinear Calder\'{o}n--Zygmund kernel (cf. Definition \ref{def:CCZK}), 

\item\label{H2} $T$ satisfies the weak compactness property (cf. Definition \ref{def:WCP}), 

\item\label{H3} $T(1,1), T^{*1}(1,1), T^{*2}(1,1) \in \CMO(\Rn)$ (cf. Definition \ref{def:BMO}). 

\end{list}

Our main result is a compact bilinear $T1$ theorem as follows. 

\begin{theorem}\label{thm:cpt} 
Let $0<\lambda<1/2$. Let $T: L_c^{\infty}(\Rn) \times L_c^{\infty}(\Rn) \to L_{\loc}^1(\Rn)$ be a bilinear operator associated with a standard bilinear Calder\'{o}n--Zygmund kernel. Then the following are equivalent: 
\begin{list}{(\theenumi)}{\usecounter{enumi}\leftmargin=1.2cm \labelwidth=1cm \itemsep=0.1cm \topsep=0.2cm \renewcommand{\theenumi}{\alph{enumi}}}
			
\item[\textup{(a)\phantom{$'$}}] $T$ satisfies the hypotheses \eqref{H1}, \eqref{H2}, and \eqref{H3}. 

\item[\textup{(b)\phantom{$'$}}] $T$ admits a compact bilinear dyadic representation (cf. Definition $\ref{def:repre}$). 

\item[\textup{(c)\phantom{$'$}}] $T$ extends to a compact operator from $L^{p_1}(\Rn) \times L^{p_2}(\Rn)$ to $L^p(\Rn)$ for all $p_1, p_2 \in (1, \infty]$, where $\frac1p = \frac{1}{p_1} + \frac{1}{p_2}>0$. 

\item[\textup{(c)$'$}] $T$ extends to a compact operator from $L^{p_1}(w_1^{p_1}) \times L^{p_2}(w_2^{p_2})$ to $L^p(w^p)$ for all $p_1, p_2 \in (1, \infty]$ and for all $(w_1, w_2) \in A_{(p_1, p_2)}$, where $\frac1p = \frac{1}{p_1} + \frac{1}{p_2}>0$ and $w=w_1 w_2$. 

\item[\textup{(d)\phantom{$'$}}] $T$ extends to a compact operator from $L^1(\Rn) \times L^1(\Rn)$ to $L^{\frac12, \infty}(\Rn)$. 

\item[\textup{(d)$'$}] $T$ extends to a compact operator from $L^1(w_1) \times L^1(w_2)$ to $L^{\frac12, \infty}(w^{\frac12})$ for all $(w_1, w_2) \in A_{(1, 1)}$, where $w=w_1 w_2$. 

\item[\textup{(e)\phantom{$'$}}] $T$ extends to a compact operator from $L^{\infty}(\Rn) \times L^{\infty}(\Rn)$ to $\CMO(\Rn)$.  

\item[\textup{(e)$'$}] $T$ extends to a compact operator from $L^{\infty}(w_1^{\infty}) \times L^{\infty}(w_2^{\infty})$ to $\CMO_{\lambda}(w^{\infty})$ for all $(w_1, w_2) \in A_{(\infty, \infty)}$, where $w=w_1 w_2$. 
\end{list}
\end{theorem}

To prove Theorem \ref{thm:cpt}, we will adopt the following strategy: 
\begin{align*}
& \text{(a)} 
\Longrightarrow \text{(b)} 
\Longrightarrow \text{(c)}' 
\Longrightarrow \text{(c)} 
\Longrightarrow \text{(a)}, 
\\
& \text{(a)} 
\Longrightarrow \text{(d)}' 
\Longrightarrow \text{(d)} 
\Longrightarrow \text{(c)}', 
\, \text{ and}
\\
& \text{(a)} 
\Longrightarrow \text{(e)}' 
\Longrightarrow \text{(e)} 
\Longrightarrow \text{(c)}'.
\end{align*}
Since the implications $\text{(c)}' \Longrightarrow \text{(c)}$, $\text{(d)}' \Longrightarrow \text{(d)}$, and $\text{(e)}' \Longrightarrow \text{(e)}$ are trivial, it suffices to show other implications. 

The remainder of this section elaborates on the novel ideas, the obstacles, and the methods in our proof, which are conducive to understand this lengthy article. Moreover, at the end of this section, as applications, we present some examples of bilinear operators implicitly in the literature, which satisfy the hypotheses \eqref{H1}, \eqref{H2}, and \eqref{H3}.

\subsection{Compact bilinear dyadic representation}
We first focus on the proof of ${\rm (a)} \Longrightarrow {\rm (b)}$. The definition of dyadic grids $\D$ and Haar functions $h_I$ is postponed until Sections \ref{sec:dyadic} and \ref{sec:Haar}. Given $i, j, k \in \N$ and a dyadic grid $\D$, we define the \emph{compact bilinear dyadic shift} as 
\begin{align*}
\mathbf{S}_{\D}^{i, j, k} (f_1, f_2)
:= \sum_{Q \in \D} A_Q^{i, j, k} (f_1, f_2)
\end{align*}
with 
\begin{align*}
A_Q^{i, j, k} (f_1, f_2) 
:= \sum_{\substack{I \in \D_i(Q) \\ J \in \D_j(Q) \\ K \in \D_k(Q)}}  
a_{I, J, K, Q} \, \langle f_1, \widetilde{h}_I \rangle 
\, \langle f_2, \widetilde{h}_J \rangle \, h_K, 
\end{align*}
where $\D_i(Q) := \{I \in \D: I \subset Q, \ell(I) = 2^{-i} \ell(Q)\}$, $(\widetilde{h}_I, \widetilde{h}_J) \in \big\{(h_I, h_J), (h_I, h^0_J), (h^0_I, h_J) \big\}$, and the coefficients $a_{I, J, K, Q}$ satisfy  
\begin{align*}
|a_{I, J, K, Q}| 
\le \mathbf{F}(Q) \frac{|I|^{\frac12} |J|^{\frac12} |K|^{\frac12}}{|Q|^2} 
\end{align*}
with  
\begin{align*}
\mathbf{F}(Q) \le 1 
\quad\text{and}\quad 
\lim_{N \to \infty} \mathbf{F}_N 
:= \lim_{N \to \infty} \sup_{\D} \sup_{Q \in \D \setminus \D(N)} 
\mathbf{F}(Q) = 0. 
\end{align*}
The notation $\D(N)$ is given in Section \ref{sec:notation}.

Define the \emph{bilinear dyadic paraproduct}   
\begin{align*}
\Pi_{\b}(f_1, f_2) 
&:= \sum_{I \in \D} b_I \, \langle f_1 \rangle_I \langle f_2 \rangle_I \,  h_I, 
\end{align*}
where $\b := \{b_I\}_{I \in \D}$ is a sequence of real numbers such that 
\begin{align*}
\|\b\|_{\BMO_{\D}}
:= \sup_{Q \in \D} \bigg(\frac{1}{|Q|} \sum_{I \in \D: \, I \subset Q} |b_I|^2 \bigg)^{\frac12}
< \infty.  
\end{align*}

\begin{definition}\label{def:repre}
Given a bilinear operator $T$, we say that $T$ admits a \emph{compact bilinear dyadic representation} if there exists a constant $C_0 \in (0,  \infty)$ so that for all functions $f_1, f_2, f_3 \in L_c^{\infty}(\Rn)$,  
\begin{align*}
\big\langle T(f_1, f_2), f_3 \big\rangle 
&= C_0 \, \mathbb{E}_{\w} \sum_{k=0}^{\infty} \sum_{i=0}^k 2^{-k\delta'/2} 
\big\langle \mathbb{S}_{\D_{\w}}^{i, k} (f_1, f_2), f_3 \big\rangle 
+ C_0 \, \mathbb{E}_{\w} \big\langle \Pi_{\b^0_{\w}} (f_1, f_2), f_3 \big\rangle 
\\
&\quad+ C_0 \, \mathbb{E}_{\w} \big\langle \Pi_{\b^1_{\w}}^{*1} (f_1, f_2), f_3 \big\rangle 
 + C_0 \, \mathbb{E}_{\w} \big\langle \Pi_{\b^2_{\w}}^{*2} (f_1, f_2), f_3 \big\rangle,   
\end{align*}
where $\delta' \in (0, 1)$, $\mathbb{S}_{\D_{\w}}^{i, k}$ is a finite sum of cancellative shifts $\mathbf{S}_{\D_{\w}}^{i, i, k}$, $\mathbf{S}_{\D_{\w}}^{i, i+1, k}$, and adjoints of such operators. The number of shifts and their adjoints in this finite sum is bounded uniformly in $i$, $k$, and $\w$. In addition, $\Pi_{\b_{\w}}$ is the bilinear dyadic paraproduct with $\b_{\w} := \{b_I\}_{I \in \D_{\w}}$ satisfying
\begin{align*}
\sup_{\w \in \Omega} \|\b_{\w}\|_{\BMO_{\D_{\w}}} \le 1 
\quad \text{ and } \quad 
\lim_{N \to \infty} \sup_{\w \in \Omega} \sup_{Q \in \D_{\w}} 
\frac{1}{|Q|} \sum_{\substack{I \in \D_{\w} \setminus \D_{\w}(N): \\ I \subset Q}} |b_I|^2 
= 0. 
\end{align*}
\end{definition}

We, for the first time, introduce a compact dyadic representation in order to investigate the compactness of singular integrals, even in the linear case. It gives the exact dyadic structure behind compact bilinear Calder\'{o}n--Zygmund operators by representing them as an average of some compact bilinear dyadic shifts and paraproducts, which allows us to reduce initial problems into corresponding dyadic problems. More importantly, the dyadic-probabilistic method indeed overcomes the drawbacks of the techniques of \cite{CYY, Vil} in the bilinear setting, which explains why we here develop a compact dyadic representation. Such a general representation theorem was proved by Hyt\"{o}nen in \cite{Hyt12, Hyt17} to settle the well-known $A_2$ conjecture for general Calder\'{o}n--Zygmund operators. Its usefulness and powerfulness have been shown in the last decade. For example, a representation theorem also holds in the multi-parameter setting \cite{HLMV, Mar12} and the bilinear case \cite{LMOV, LMV}, and can be applied to obtain weighted estimates for maximal truncations of Calder\'{o}n--Zygmund operators \cite{HLMORSU} and for commutators \cite{HLW, HPW, OPS}. Furthermore, we will establish a compact dyadic representation in the multi-parameter situation to study weighted compactness of multi-parameter singular integrals in our forthcoming papers \cite{CCLLYZ, CY}.

\subsection{Compactness criterions in Lorentz spaces}  
Before moving on to the proof of weighted compactness, we present some compactness criterions. To establish the endpoint compactness of $L^1 \times L^1 \to L^{\frac12, \infty}$, it is natural to study compactness criterions in Lorentz spaces $L^{p, q}(\Rn)$, where $0<p<\infty$ and $0<q \le \infty$. The characterization of precompactness in Lebesgue spaces  was first discovered by Kolmogorov \cite{Kolm}, and then extended by Riesz \cite{Riesz} and Tsuji \cite{Tsu}. Such result has shown its own utility by application to compactness of commutators (cf. \cite{BDMT, CIRXY, COY, CY, CYY, CT}). Inspired by these work, we will establish Kolmogorov--Riesz theorems in (weighted) Lorentz spaces. 

Let $\tau_h$ be the translation operator, i.e., $\tau_h f(x) := f(x-h)$ for all $x, h \in \Rn$.

\begin{theorem}\label{thm:KRLpq}
Let $0<p, q<\infty$ and $\K \subset L^{p, q}(\Rn)$. Then $\K$ is precompact in $L^{p, q}(\Rn)$ if and only if the following are satisfied:
\begin{list}{\rm (\theenumi)}{\usecounter{enumi}\leftmargin=1cm \labelwidth=1cm \itemsep=0.2cm \topsep=0.2cm \renewcommand{\theenumi}{\alph{enumi}}}
 
\item\label{RKLpq-1} ${\displaystyle \sup_{f \in \K} \|f\|_{L^{p, q}} < \infty}$, 

\item\label{RKLpq-2} ${\displaystyle \lim_{A \to \infty} \sup_{f \in \K} 
\|f \mathbf{1}_{B(0, A)^c}\|_{L^{p, q}}=0}$, 

\item\label{RKLpq-3} $\displaystyle \lim_{|h| \to 0} \sup_{f \in \K} 
\|\tau_h f - f \|_{L^{p, q}}=0$. 
\end{list}  
Moreover, in the case $q=\infty$, the conditions \eqref{RKLpq-1}, \eqref{RKLpq-2}, and \eqref{RKLpq-3} are sufficient, but \eqref{RKLpq-2} and \eqref{RKLpq-3} are not necessary in general.  
\end{theorem}

Unfortunately, Theorem \ref{thm:KRLpq} does not hold in the weighted case. In fact, it just gives a sufficient condition of precompactness in $L^{p, q}(w)$. To be precise, we obtain the following result.

\begin{theorem}\label{thm:KRW}
Let $w$ be a weight so that $w, w^{-\lambda} \in L^1_{\loc}(\Rn)$ for some $\lambda>0$. Let $0<p, q<\infty$ and $\K \subset L^{p, q}(w)$. Then $\K$ is precompact in $L^{p, q}(w)$ if the following are satisfied:
\begin{list}{\rm (\theenumi)}{\usecounter{enumi}\leftmargin=1cm \labelwidth=1cm \itemsep=0.2cm \topsep=0.2cm \renewcommand{\theenumi}{\alph{enumi}}}
 
\item\label{RKW-1} ${\displaystyle \sup_{f \in \K} \|f\|_{L^{p, q}(w)} < \infty}$, 

\item\label{RKW-2} ${\displaystyle \lim_{A \to \infty} \sup_{f \in \K} 
\|f \mathbf{1}_{B(0, A)^c}\|_{L^{p, q}(w)}=0}$, 

\item\label{RKW-3} ${\displaystyle \lim_{|h| \to 0} \sup_{f \in \K} 
\|\tau_h f - f\|_{L^{p, q}(w)}=0}$. 
\end{list}
Moreover, \eqref{RKW-1} and \eqref{RKW-2} are necessary, but \eqref{RKW-3} is not necessary in general. In the case $q=\infty$, the conditions \eqref{RKW-1}, \eqref{RKW-2}, and \eqref{RKW-3} are sufficient, but \eqref{RKW-2} and \eqref{RKW-3} are not necessary in general.  
\end{theorem}

As seen from Theorem \ref{thm:KRW}, to give a characterization of precompactness in $L^{p, q}(w)$ with $q \in (0, \infty)$, one has to replace the condition \eqref{RKW-3} by a weaker condition. To achieve this, we show the equivalence as follows. It is worth mentioning that the exponents $p, q$ are allowed to be quasi-Banach.

\begin{theorem}\label{thm:KRWA}
Let $w \in A_{p_0}$ for some $p_0 \in (1, \infty)$. Let $0<p, q<\infty$ and $\K \subseteq L^{p, q}(w)$. Let $0<a<   \min\{p/p_0, q, 1\}$. Then $\K$ is precompact in $L^{p, q}(w)$ if and only if the following are satisfied:
\begin{list}{\rm (\theenumi)}{\usecounter{enumi}\leftmargin=1cm \labelwidth=1cm \itemsep=0.2cm \topsep=0.2cm \renewcommand{\theenumi}{\alph{enumi}}}
 
\item\label{RKWA-1} ${\displaystyle \sup_{f \in \K} \|f\|_{L^{p, q}(w)} < \infty}$, 

\item\label{RKWA-2} ${\displaystyle \lim_{A \to \infty} \sup_{f \in \K} 
\|f \mathbf{1}_{B(0, A)^c}\|_{L^{p, q}(w)}=0}$, 

\item\label{RKWA-3} ${\displaystyle \lim_{r \to 0} \sup_{f \in \K} 
\bigg\| \bigg( \fint_{B(0, r)} |\tau_y f - f|^a \, dy \bigg)^{\frac1a} \bigg\|_{L^{p, q}(w)} = 0}$. 
\end{list}
Moreover, in the case $q=\infty$, the conditions \eqref{RKWA-1}, \eqref{RKWA-2}, and \eqref{RKWA-3} are sufficient, but \eqref{RKWA-2} and \eqref{RKWA-3} are not necessary in general.  
\end{theorem}

From Theorems \ref{thm:KRLpq}--\ref{thm:KRWA}, we see that the precompactness in $L^{p, \infty}(\Rn)$ becomes much more delicate. A main reason is that the collection of continuous functions and functions with compact support is not dense in $L^{p, \infty}(w)$ for all exponents $p \in (0, \infty)$ and all weights $w \in L^1_{\loc}(\Rn)$ (cf. Lemma \ref{lem:dense-Lpq}). Besides, the space $L^{p, \infty}(\Rn)$ with $p \in (0, \infty)$ does not have absolutely continuous quasi-norm (cf. Lemma \ref{lem:ACN}), which leads to the invalidity of some classical precompactness criterion (for example, Lemma \ref{lem:UAC}). Beyond that, in terms of  $\BMO(\Rn)$ and $\CMO(\Rn)$, they are not lattices, which means that $|f| \le |g|$ does not imply $\|f\|_{\BMO} \le \|g\|_{\BMO}$. This indicates that it is impossible to characterize precompactness in $\CMO$ as Theorems \ref{thm:KRLpq}--\ref{thm:KRWA}. Based on these facts, we have to seek new precompactness criterion in $L^{p, \infty}(\Rn)$ and $\CMO(\Rn)$. For this purpose, we make use of the projection operator in \eqref{def:PN} and establish the following characterizations of precompactness.

\begin{theorem}\label{thm:PNT-cpt}
Let $T$ be a bilinear operator and $\frac1p = \frac{1}{p_1} + \frac{1}{p_2} \in (0, 1)$ with $p_1, p_2 \in (1, \infty]$. 
\begin{list}{\rm (\theenumi)}{\usecounter{enumi}\leftmargin=1.2cm \labelwidth=1cm \itemsep=0.2cm \topsep=.2cm \renewcommand{\theenumi}{\arabic{enumi}}}

\item\label{list:PN1} Assume that $T$ is bounded from $L^{p_1}(\Rn) \times L^{p_2}(\Rn)$ to $L^p(\Rn)$. Then $T$ is compact from $L^{p_1}(\Rn) \times L^{p_2}(\Rn)$ to $L^p(\Rn)$ if and only if $\lim_{N \to \infty} \sup_{\D} \|P_{\D(N)}^{\perp} T\|_{L^{p_1} \times L^{p_2} \to L^p} = 0$, which is also equivalent to $\lim_{N \to \infty} \|P_{\D(N)}^{\perp} T\|_{L^{p_1} \times L^{p_2} \to L^p} = 0$ for some dyadic grid $\D$.

\item\label{list:PN2} Assume that $T$ is bounded from $L^1(\Rn) \times L^1(\Rn)$ to $L^{\frac12, \infty}(\Rn)$. Then $T$ is compact from $L^1(\Rn) \times L^1(\Rn)$ to $L^{\frac12, \infty}(\Rn)$ if $\lim_{N \to \infty} \|P_{\D(N)}^{\perp} T\|_{L^1 \times L^1 \to L^{\frac12, \infty}} = 0$ for some dyadic grid $\D$. But the converse fails.

\item\label{list:PN3} Assume that $T$ is bounded from $L^{\infty}(\Rn) \times L^{\infty}(\Rn)$ to $\BMO(\Rn)$. Then $T$ is compact from $L^{\infty}(\Rn) \times L^{\infty}(\Rn)$ to $\CMO(\Rn)$ if and only if $\lim_{N \to \infty} \sup_{\D} \|P_{\D(N)}^{\perp} T\|_{L^{\infty} \times L^{\infty} \to \BMO_{\D}} = 0$. 

\end{list} 
\end{theorem}

\subsection{Weighted compactness of dyadic operators} 
To demonstrate $\text{(b)} \Longrightarrow \text{(c)}'$, in light of Definition \ref{def:repre}, we will use the duality and prove the weighted compactness of bilinear dyadic shifts and paraproducts (cf. Theorems \ref{thm:SD-cpt} and \ref{thm:Pi-cpt}). The duality brings about the restriction $p \in (1, \infty)$. To obtain the quasi-Banach range, we apply Rubio de Francia extrapolation of compactness (cf. Theorem \ref{thm:EP-Lp}), which asserts that to obtain the weighted compactness in the full range of exponents, it just needs the unweighted compactness for some Banach exponent, if weighted boundedness is already known. This provides a great convenience for most applications. Here we also consider the situations where some of the exponents of the Lebesgue spaces appearing in the hypotheses and/or in the conclusion can be possibly infinity, which extends the results in \cite{COY, HL23}. Meanwhile, these advantages of extrapolation and Theorem \ref{thm:KRLpq} are used in the proof of Theorems \ref{thm:SD-cpt} and \ref{thm:Pi-cpt}, while both the sufficiency and the necessity of Theorem \ref{thm:KRWA} are needed to show Theorem \ref{thm:EP-Lp}.

\begin{theorem}\label{thm:SD-cpt}
For all $i, j, k \in \N$, $\mathbb{E}_{\w} \mathbf{S}_{\D_{\w}}^{i, j, k}$ is compact from $L^{p_1}(w_1^{p_1})  \times L^{p_2}(w_2^{p_2})$ to $L^p(w^p)$ for all $p_1, p_2 \in (1, \infty]$ and for all $(w_1, w_2) \in A_{(p_1, p_2)}$, where $\frac1p = \frac{1}{p_1} + \frac{1}{p_2}>0$ and $w=w_1 w_2$.
\end{theorem}

\begin{theorem}\label{thm:Pi-cpt}
Let $\Pi_{\b_{\w}}$ be the bilinear dyadic paraproduct defined in Definition $\ref{def:repre}$. Then $\mathbb{E}_{\w} \Pi_{\b_{\w}}$, $\mathbb{E}_{\w} \Pi_{\b_{\w}}^{*1}$, and $\mathbb{E}_{\w} \Pi_{\b_{\w}}^{*2}$ are compact from $L^{p_1}(w_1^{p_1})  \times L^{p_2}(w_2^{p_2})$ to $L^p(w^p)$ for all $p_1, p_2 \in (1, \infty]$ and for all $(w_1, w_2) \in A_{(p_1, p_2)}$, where $\frac1p = \frac{1}{p_1} + \frac{1}{p_2}>0$ and $w=w_1 w_2$.
\end{theorem}

\begin{theorem}\label{thm:EP-Lp}
Assume that $T$ is a bilinear operator such that  
\begin{align*}
\text{$T$ is compact from $L^{p_1}(\Rn) \times L^{p_2}(\Rn)$ to $L^p(\Rn)$}   
\end{align*}
for some $\frac1p = \frac{1}{p_1} + \frac{1}{p_2}>0$ with $p_1, p_2 \in [1, \infty]$, and 
\begin{align*}
\|T\|_{L^{q_1}(v_1^{q_1}) \times L^{q_2}(v_2^{q_2}) \to L^q(v^q)} 
\le \Xi \big([\vec{v}]_{A_{\vec{q}}} \big)
\end{align*}
for some $\frac1q = \frac{1}{q_1} + \frac{1}{q_2}$ with $q_1, q_2 \in [1, \infty]$ and for all $(v_1, v_2) \in A_{(q_1, q_2)}$, where $v=v_1 v_2$ and $\Xi: [1, \infty) \to [1, \infty)$ is a non-decreasing function. Then 
\begin{align*}
\text{$T$ is compact from $L^{r_1}(w_1^{r_1}) \times L^{r_2}(w_2^{r_2})$ to $L^r(w^r)$} 
\end{align*}
for all $r_1, r_2 \in (1, \infty]$ and for all $(w_1, w_2) \in A_{(r_1, r_2)}$, where $\frac1r = \frac{1}{r_1} + \frac{1}{r_2}>0$ and $w=w_1 w_2$.  
\end{theorem}

We mention that the proof of Theorem \ref{thm:EP-Lp} is quite different from that of \cite[Theorem 1.1]{COY}, which does not consider the endpoint case $p_i=\infty$, or/and  $q_i=\infty$, or/and $r_i=\infty$. Additionally, our proof does not rely on \cite[Theorem 1.1]{COY};  instead, it invokes Theorem \ref{thm:KRWA} and multilinear extrapolation of boundedness with exponents being infinity (cf. Theorem \ref{thm:RdF}).

\subsection{Necessary hypotheses}
The implication ${\rm (c)} \Longrightarrow {\rm (a)}$ follows from Theorems \ref{thm:CCZK}--\ref{thm:T1CMO}. Indeed, to justify the hypotheses \eqref{H1}--\eqref{H3}, it just requires the compactness holding for one triple of exponents $(p, p_1, p_2)  \in (1, \infty)^3$.

\begin{theorem}\label{thm:CCZK} 
Let $T$ be a bilinear operator associated with a standard bilinear Calder\'{o}n--Zygmund kernel $K$. Let $\frac{1}{p}=\frac{1}{p_1}+\frac{1}{p_2}$ with $p, p_1, p_2 \in (1, \infty)$. If $T$ is compact from $L^{p_1}(\Rn) \times L^{p_2}(\Rn)$ to $L^p(\Rn)$, then $K$ is a compact bilinear Calder\'{o}n--Zygmund kernel. 
\end{theorem}

\begin{theorem}\label{thm:WCP} 
Let $\frac1p = \frac{1}{p_1} + \frac{1}{p_2}$ with $p, p_1, p_2 \in (1, \infty)$. If a bilinear operator $T$ is bounded from $L^{p_1}(\Rn) \times L^{p_2}(\Rn)$ to $L^p(\Rn)$, then
\begin{align*}
|\langle T(\mathbf{1}_I, \mathbf{1}_I), \mathbf{1}_I \rangle|  |I|^{-1} 
\lesssim \|P_{\D(N)}^{\perp} T\|_{L^{p_1} \times L^{p_2} \to L^p} 
+ \|T\|_{L^{p_1} \times L^{p_2} \to L^p} F(I; 2N), 
\end{align*} 
for all $I \in \D$ and $N \in \N_+$, where 
\begin{align*}
& F(I; N) := F_1(2^N \ell(I)) F_2(2^{-N} \ell(I)) F_3(N^{-1} \rd(I, 2^N \I)), 
\\
& F_1(t) := (1+t^{-1})^{-\frac{n}{p}}, 
\quad\text{ and }\quad 
F_2(t)=F_3(t) := \mathbf{1}_{[0, 1]}(t).
\end{align*} 
In particular, if $T$ is compact from $L^{p_1}(\Rn) \times L^{p_2}(\Rn)$ to $L^p(\Rn)$, then it satisfies the weak compactness property. 
\end{theorem}

\begin{theorem}\label{thm:T1CMO} 
Let $T$ be a bilinear operator associated with a standard bilinear Calder\'{o}n--Zygmund kernel $K$. Let $\frac{1}{p}=\frac{1}{p_1}+\frac{1}{p_2}$ with $p, p_1, p_2 \in (1, \infty)$. If $T$ is compact from $L^{p_1}(\Rn) \times L^{p_2}(\Rn)$ to $L^p(\Rn)$, then $T(1,1), T^{*1}(1,1), T^{*2}(1,1) \in \CMO(\Rn)$. 
\end{theorem}

\subsection{Extrapolation of endpoint compactness} 
To show ${\rm (d) \Longrightarrow (c)'}$, we establish extrapolation of $L^{p, \infty}$ compactness, see Theorem \ref{thm:EP-Lpinfty}. By means of Theorem \ref{thm:EP-Lp}, Theorem \ref{thm:EP-Lpinfty} can be reduced to showing the unweighted $L^{r_1} \times L^{r_2} \to L^r$ compactness, where $r_1, r_2 \in (1, \infty)$. As seen above, Theorems \ref{thm:KRLpq}--\ref{thm:KRWA} do not give the necessary conditions of the precompactness in $L^{p, \infty}(\Rn)$, which leads they can not be directly applied to show Theorem \ref{thm:EP-Lpinfty}. To circumvent this obstacle, we first establish relative compactness in a suitable Lebesgue space of the normalized images of pairs of characteristic functions. Theorem \ref{thm:KRLpq} then provides uniform tail and translation estimates for these images, which are upgraded by the bilinear Marcinkiewicz interpolation with initial restricted weak type conditions (cf. Theorem \ref{thm:BM}). Moreover, such interpolation needs three initial points that form a triangle in $\R^2$, which is provided by extrapolation of boundedness (cf. Theorem \ref{thm:RdF}).

\begin{theorem}\label{thm:EP-Lpinfty}
Assume that $T$ is a bilinear operator such that 
\begin{align*}
\text{$T$ is compact from $L^{p_1}(\Rn) \times L^{p_2}(\Rn)$ to $L^{p, \infty}(\Rn)$}  
\end{align*}
for some $\frac1p = \frac{1}{p_1} + \frac{1}{p_2}$ with $p_1, p_2 \in [1, \infty]$, and 
\begin{align*}
\|T\|_{L^{q_1}(v_1^{q_1}) \times L^{q_2}(v_2^{q_2}) \to L^q(v^q)} 
\le \Xi \big([\vec{v}]_{A_{\vec{q}}} \big)
\end{align*}
for some $\frac1q = \frac{1}{q_1} + \frac{1}{q_2}$ with $q_1, q_2 \in [1, \infty]$ and for all $(v_1, v_2) \in A_{(q_1, q_2)}$, where $v=v_1 v_2$ and $\Xi: [1, \infty) \to [1, \infty)$ is a non-decreasing function. Then 
\begin{align*}
\text{$T$ is compact from $L^{r_1}(w_1^{r_1}) \times L^{r_2}(w_2^{r_2})$ to $L^r(w^r)$} 
\end{align*}
for all $r_1, r_2 \in (1, \infty]$ and for all $(w_1, w_2) \in A_{(r_1, r_2)}$, where $\frac1r = \frac{1}{r_1} + \frac{1}{r_2}>0$ and $w=w_1 w_2$. 
\end{theorem}

To prove ${\rm (e) \Longrightarrow (c)'}$, we establish extrapolation of $\CMO$ compactness, see Theorem \ref{thm:EP-Linfty}. We utilize Theorem \ref{thm:EP-Lp} to reduce Theorem \ref{thm:EP-Linfty} to the $L^{r_1} \times L^{r_2} \to L^r$ compactness, whose proof relies on a linear interpolation between $L^2 \to L^2$ and $L^{\infty} \to \BMO_{\D}$. Since there is extremely limited information  about the necessary conditions of $\CMO$ compactness, we will use the projection operator which characterizes the $L^p$ and $\CMO$ compactness, see Theorem \ref{thm:PNT-cpt}.

\begin{theorem}\label{thm:EP-Linfty}
Assume that $T$ is a bilinear operator such that  
\begin{align*}
\text{$T$ is compact from $L^{\infty}(\Rn) \times L^{\infty}(\Rn)$ to $\CMO(\Rn)$},  
\end{align*}
and 
\begin{align*}
\|T\|_{L^{q_1}(v_1^{q_1}) \times L^{q_2}(v_2^{q_2}) \to L^q(v^q)} 
\le \Xi \big([\vec{v}]_{A_{\vec{q}}} \big)
\end{align*}
for some $\frac1q = \frac{1}{q_1} + \frac{1}{q_2}$ with $q_1, q_2 \in [1, \infty]$ and for all $(v_1, v_2) \in A_{(q_1, q_2)}$, where $v=v_1 v_2$ and $\Xi: [1, \infty) \to [1, \infty)$ is a non-decreasing function. Then 
\begin{align*}
\text{$T$ is compact from $L^{r_1}(w_1^{r_1}) \times L^{r_2}(w_2^{r_2})$ to $L^r(w^r)$} 
\end{align*}
for all $r_1, r_2 \in (1, \infty]$ and for all $(w_1, w_2) \in A_{(r_1, r_2)}$, where $\frac1r = \frac{1}{r_1} + \frac{1}{r_2}>0$ and $w=w_1 w_2$.  
\end{theorem}

\subsection{Endpoint weighted compactness} 
To conclude the proof of Theorem \ref{thm:cpt}, it remains to show ${\rm (a) \Longrightarrow (d)'}$ and ${\rm (a) \Longrightarrow (e)'}$, for which it needs the implicit ${\rm (a) \Longrightarrow (c)}$ just proved above. Besides, the proof of ${\rm (a) \Longrightarrow (d)'}$ is based on approximating the original operator by finite-dimensional operators
and a stronger result than the unweighted $L^1 \times L^1 \to L^{\frac12, \infty}$ compactness (see \eqref{L1Li-1}--\eqref{L1Li-3}). The latter is obtained by the Calder\'{o}n--Zygmund decomposition. It is the first time to use Kolmogorov--Riesz theorems to investigate endpoint compactness. Moreover, to justifyt ${\rm (a) \Longrightarrow (e)'}$, we establish a compactness criterion in $\CMO_{\lambda}(w^{\infty})$ (see Theorem \ref{thm:cmob}) and some oscillation estimates. It was shown that for any Calder\'{o}n--Zygmund operator $T$, 
\begin{align*}
M^{\#} (Tf)(x) 
\lesssim M_r f(x), \quad x \in \Rn, \, \, r>1.
\end{align*}
But this inequality does not hold for $r=1$ and $T$ being the Hilbert transform, see \cite{CM78}. 
The same argument holds in the bilinear case, but it can not yield the desired weighted estimate if $(w_1, w_2) \in A_{(\infty, \infty)}$. This is the main reason why we use $\CMO_{\lambda}$ instead of $\CMO$.

\subsection{Applications} 
To illustrate the applicability of Theorem \ref{thm:cpt}, we present several kinds of bilinear operators, which satisfy the hypotheses \eqref{H1}--\eqref{H3}.

The first example is bilinear paraproducts. Let $\Phi, \Psi \in \mathscr{C}_c^{\infty}(\Rn)$ be radial functions such that $\supp (\Phi) \subset B(0, 1)$, $\int_{\Rn} \Phi (x) \, dx = 1$, $\int_{\Rn} \Psi(x) \, dx = 0$, and $\int_0^{\infty} |\widehat{\Psi}(t e_1)|^2 \frac{dt}{t} =1$, where $e_1 := (1, 0, \ldots, 0)$. For any $t>0$, define convolution operators 
\begin{align*}
P_t f := \Phi_t *f 
\quad\text{ and }\quad 
Q_t f := \Psi_t * f, 
\end{align*} 
where $\Phi_t(x) := t^{-n} \Phi(t^{-1} x)$ and $\Psi_t(x) := t^{-n} \Psi(t^{-1} x)$. Given $b \in \BMO(\Rn)$, the \emph{bilinear continuous paraproduct} $\pi_b$ is defined by 
\begin{align*}
\pi_b (f, g) 
:= \int_0^{\infty} Q_t \big( (Q_t b) (P_t f) (P_t g) \big) \, \frac{dt}{t}. 
\end{align*}
The paraproduct goes back to the seminal work of Bony \cite{Bony} and Coifman and Meyer \cite{CM78}, and one of its significant applications is the $T1$ theorem due to David and Journ\'{e} \cite{DJ}. The bilinear extension first appeared explicitly in the work of Yabuta \cite{Yab}, and was used to prove bilinear $T1$ theorem \cite{Hart14}.

\begin{theorem}\label{thm:Pibc}
For any $b \in \CMO(\Rn)$, $\pi_b$ satisfies the hypotheses \eqref{H1}, \eqref{H2}, and \eqref{H3}. 
\end{theorem}

The second example is bilinear dyadic paraproducts, which have developed into a fundamental tool in bilinear dyadic analysis, see \cite{CY, LMOV, MV}. Fix a dyadic grid $\D$ and let $\{\psi_I\}_{I \in \D}$ be a system of wavelets (cf. Definition \ref{def:wavelet}). Given a function $b \in \BMO(\Rn)$, we define the \emph{bilinear dyadic paraproduct}  as 
\begin{align*}
\Pi_b(f_1, f_2) 
:= \sum_{I \in \D} \langle b, \psi_I \rangle 
\langle f_1, \varphi_I \rangle \langle f_2, \varphi_I \rangle \,  \psi_I,
\end{align*}
where $\varphi \in \mathscr{C}_c^{\infty}(\Rn)$, $\supp \varphi \subset (-\frac12, \frac12)^n$, $\int_{\Rn} \varphi \, dx = 1$, and $\varphi_I(x) = |I|^{-1} \varphi((x-c_I)/\ell(I))$.

\begin{theorem}\label{thm:Pib}
For any $b \in \CMO(\Rn)$, $\Pi_b$ satisfies the hypotheses \eqref{H1}, \eqref{H2}, and \eqref{H3}. 
\end{theorem}

Let us proceed to the third example. Given a symbol $\sigma$, the \emph{bilinear pseudo-differential operator} $T_{\sigma}$ is defined by
\begin{align*}
T_{\sigma}(f_1, f_2)(x) := \iint_{\R^{2n}} \sigma(x, \xi_1, \xi_2)
e^{2\pi i x \cdot (\xi_1+\xi_2)} \widehat{f}_1(\xi_1) \widehat{f}_2(\xi_2) \, d\xi_1 \, d\xi_2,
\end{align*}
for all Schwartz functions $f_1, f_2$, where $\widehat{f}(\xi) := \int_{\Rn} e^{-2 \pi i x \cdot \xi} f(x) \,  dx$. Given $m \in \R$ and $\rho,\delta \in [0, 1]$, we say that $\sigma \in \mathcal{K}_{\rho,\delta}^m$ if it satisfies 
\begin{align*}
\big|\partial_{x}^{\alpha} \partial_{\xi}^{\beta} \partial_{\eta}^{\gamma} \sigma(x, \xi, \eta) \big|
\leq C_{\alpha, \beta, \gamma}(x, \xi, \eta) (1+ |\xi| + |\eta|)^{m + \delta|\alpha| - \rho (|\beta| + |\gamma|)}, 
\end{align*}
for all multi-indices $\alpha$, $\beta$ and $\gamma$, where $C_{\alpha, \beta, \gamma}$ is a bounded function satisfying 
\begin{align*}
\lim_{|x| + |\xi| + |\eta| \to \infty} C_{\alpha,\beta, \gamma}(x, \xi, \eta)
=0.
\end{align*}
If the functions $C_{\alpha,\beta,\gamma}$ are replaced by nonnegative constants, without the vanishing requirement, we obtain the usual symbol class $\mathcal{S}_{\rho,\delta}^m$. 

In the linear case, the condition $\sigma \in \mathcal{K}_{\rho,\delta}^m$ implicitly originated in \cite{Cor} to establish $L^2$ compactness of $T_{\sigma}$, and was further adapted in \cite{CST} to obtain weighted $L^p$ compactness of $T_{\sigma}$.

\begin{theorem}\label{thm:Tsig}
Let $\sigma \in \mathcal{K}_{1, \delta}^0$ with $\delta \in [0, 1)$. Then the bilinear pseudo-differential operator $T_{\sigma}$ satisfies the hypotheses \eqref{H1}, \eqref{H2}, and \eqref{H3}. 
\end{theorem}

Our last example is bilinear commutators. A function $b$ on $\Rn$ is called \emph{Lipschitz} if there exists a constant $C_0 \in (0, \infty)$ such that 
\begin{align*}
|b(x) - b(y)| \le C_0 \, |x-y| 
\quad\text{ for all } x, y \in \Rn. 
\end{align*}
Given a bilinear operator $T$ and a measurable function $b$, we define, whenever they make sense, the \emph{bilinear commutators} of $T$ as 
\begin{align*}
[b, T]_1(f_1, f_2) &:= b \, T(f_1, f_2) - T(b f_1, f_2), 
\\
[b, T]_2(f_1, f_2) &:= b \, T(f_1, f_2) - T(f_1, b f_2). 
\end{align*}
As mentioned before, the study of compactness of commutators has been very extensive, see \cite{BDMT, CIRXY, COY, CT, HL22, HL23, HLTY, TYY, Uch}. Unfortunately, in general, even if it is assumed that $b \in \CMO(\Rn)$, the corresponding kernel of commutators is not a (bilinear) Calder\'{o}n--Zygmund kernel. Considering this fact and the bilinear Calder\'{o}n--Zygmund kernel shown in \cite{BO}, we use 
the condition $\sigma \in \mathcal{K}_{\rho,\delta}^m$ as above to obtain the following result.

\begin{theorem}\label{thm:bT}
Let $\sigma \in \mathcal{K}_{1,0}^1$ and $b$ be a Lipschitz function with $\nabla b \in L^{\infty}(\Rn)$. Then the bilinear commutators $[b, T_{\sigma}]_1$ and $[b, T_{\sigma}]_2$ satisfy the hypotheses \eqref{H1}, \eqref{H2}, and \eqref{H3}. 
\end{theorem}

\subsection{Structure of the paper}
The rest of the paper is organized as follows. Section \ref{sec:pre} contains some preliminaries including notation, bilinear Calder\'{o}n--Zygmund operators, dyadic grids, Haar functions, $\BMO$ and $\CMO$ spaces, Muckenhoupt weights, and some technical lemmas. In Section \ref{sec:repre}, we show the compact bilinear dyadic representation, namely, ${\rm (a) \Longrightarrow (b)}$. Section \ref{sec:KR} is devoted to establishing compactness criterion in Lorentz spaces, Theorems \ref{thm:KRLpq}--\ref{thm:PNT-cpt}. 
After that, we present the proof of Theorems \ref{thm:SD-cpt}--\ref{thm:Pi-cpt} in Section \ref{sec:Lp}, 
and the proof of Theorems \ref{thm:CCZK}--\ref{thm:T1CMO} in Section \ref{sec:nec}. Then in Section \ref{sec:EP}, we focus on showing Rubio de Francia extrapolation of compactness, Theorems \ref{thm:EP-Lp}, \ref{thm:EP-Lpinfty}, and \ref{thm:EP-Linfty}. Besides, the implications ${\rm (a) \Longrightarrow (d)'}$ and ${\rm (a) \Longrightarrow (e)'}$ are proved in Section \ref{sec:endpoint} and Section \ref{sec:Linfty} respectively. Finally, we show Theorems \ref{thm:Pibc}--\ref{thm:bT} in Section \ref{sec:app}. For completeness, we give three appendices: Appendix \ref{sec:A-CMO} presents a dyadic characterization of $\CMO$ space; Appendix \ref{sec:A-WBP} includes the equivalence of the weak boundedness properties between our dyadic and Hart's smooth versions \cite{Hart14}; and Appendix \ref{sec:A-Lorentz} covers some properties in Lorentz spaces.

\medskip
\noindent{\bf Remark.} 
We should mention that some interesting related work has been done while this manuscript was in preparation due to Fragkos, Green, and Wick \cite{FGW}. Although our work has some overlap with \cite{FGW}, it was carried out independently and exhibits essential differences in two respects: (i) how to obtain decay in the representation coefficients. Fragkos, Green, and Wick \cite{FGW} deduced decay of the representation coefficients from vanishing testing conditions on the operator, whereas we start from explicit vanishing conditions on the kernels and transfer this decay to the representation coefficients via a random dyadic decomposition; (ii) the weighted endpoint compactness. These results, including weighted endpoint compactness and extrapolation of endpoint compactness, are developed here for the first time and have not appeared previously in the literature. Apart from the compactness criteria in weighted Lorentz spaces, the projection criteria provides an additional tool for treating weighted endpoint compactness, where standard arguments in Lebesgue spaces are invalid. Finally, we note that our current work was completed in April 2024 (arXiv:2404.14013), but our investigation on this topic began earlier with the bi-parameter case (arXiv:2303.10965), and some of the techniques therein have already been employed in this paper.

\medskip
\noindent{\bf Acknowledgments.} 
The first author acknowledges financial support from Spanish Ministry of Science and Innovation through the Ram\'{o}n y Cajal  2021 (RYC2021-032600-I), through the ``Severo Ochoa Programme for Centres of Excellence in R\&D'' (CEX2023-001347-S), and through PID2022-141354NB-I00. The third author is supported by Natural Science Foundation of Henan (No. 202300410184). Part of this work was carried out while the first three authors were participating in the \emph{International Summer School on Harmonic Analysis at Harbin Institute of Technology}. The authors would like to express their gratitude to the Institute for Advanced Study in Mathematics for the support and nice working environment. The authors are grateful to Professor Tuomas Hyt\"{o}nen for his valuable comments to refine some results.

\section{Preliminaries}\label{sec:pre}

\subsection{Notation}\label{sec:notation}
Let us introduce some notation used throughout this article. 
\begin{itemize}

\item For convenience, we use $\ell^{\infty}$ metric on $\Rn$. 

\item Let $\Z := \{0, \pm 1, \pm 2, \ldots\}$ be the set of all integers.  

\item Let $\N := \{0, 1, 2, \ldots\}$ be the set of natural numbers. 

\item Let $\N_+ := \{1, 2, \ldots\}$ be the set of positive integers.  

\item Write $\I := [-\frac12, \frac12)^n$ and $\lambda \I := [-\frac{\lambda}{2}, \frac{\lambda}{2})^n$ for any $\lambda>0$. 

\item Given $p \in (1, \infty)$, let $p'$ denote the H\"{o}lder conjugate exponent of $p$, i.e., $\frac{1}{p} + \frac{1}{p'} =1$. 

\item The translation operator is defined by $\tau_h f(x) := f(x-h)$ for all $x, h \in \Rn$. 

\item For any $p \in (0, \infty]$ and a weight $w$ on $\Rn$, we simply write $\|f\|_{L^p(w^p)} := \|fw\|_{L^p}$. 

\item For a measurable set $A \subset \Rn$ with $0<|A|<\infty$, write $\langle f \rangle_A = \fint_{A} f\,dx := \frac{1}{|A|} \int_A f\,dx$. 

\item By a cube $I$ in $\Rn$ we mean $I := \prod_{i=1}^n [a_i, a_i+\ell)$, where $a_i \in \R$ and $\ell>0$.  

\item Let $\Q$ denote the family of all cubes in $\Rn$.

\item Let $\D$ denote a generic dyadic grid on $\Rn$ (see Section \ref{sec:dyadic}). 

\item Given $k \in \N$ and $I \in \D$, write $\D_k(I) := \{I' \in \D: I' \subset I, \ell(I') = 2^{-k} \ell(I)\}$. 

\item For every $k \in \N$ and $I \in \D$, set $\ch(I) := \{I' \in \D: I' \subset I, \ell(I') = \ell(I)/2\}$, and let $I^{(k)}$ denote the unique dyadic cube $J \in \D$ such that $I \subset J$ and $\ell(J) = 2^k \ell(I)$. 

\item For any cube $I \subset \Rn$, we denote its center by $c_I$ and its sidelength by $\ell(I)$. For any $\lambda>0$, we denote by $\lambda I$ the cube with the center $c_I$ and sidelength $\lambda \ell(I)$.  

\item The distance between cubes $I$ and $J$ is given by $\d(I, J) := \inf \{|x-y|: x \in I, y \in J\}$.  

\item The relative distance between cubes $I$ and $J$ is defined by 
$\rd(I, J) := \frac{\d(I, J)}{\max\{\ell(I), \, \ell(J)\}}$. 

\item For every $N \in \N$, set $\mathcal{Q}(N) :=\{I \in \Q: 2^{-N} \leq  \ell(I) \leq 2^N, \rd(I, 2^N \I) \leq N\}$ and $\D(N) := \D \cap \mathcal{Q}(N)$. 

\item Let $\mathscr{C}(\Rn)$ be the space of continuous functions on $\Rn$. 

\item Let $\mathscr{C}_c^{\infty}(\Rn)$ be the collection of smooth functions with compact support on $\Rn$. 

\item Let $L_c^{\infty}(\Rn)$ be the set of all compactly supported and bounded functions on $\Rn$.

\item Given $\lambda>0$, let $L_{\loc}^{\lambda}(\Rn)$ consists of functions $f$ on $\Rn$ such that $|f|^{\lambda}$ is locally integrable.

\item We shall use $A \lesssim B$ and $A \simeq B$ to mean, respectively, that $A \leq C B$ and $0<c \leq A/B\leq C$, where the constants $c$ and $C$ are harmless positive constants, not necessarily the same at each occurrence, which depend only on dimension and the constants appearing in the hypotheses of theorems.  
\end{itemize}

\subsection{Bilinear Calder\'{o}n--Zygmund operators}

\begin{definition}\label{def:FF}
Let $\F$ consist of all triples $(F_1, F_2, F_3)$ of bounded functions $F_1, F_2, F_3: [0, \infty) \to [0, \infty)$ satisfying
\begin{align*}
\lim_{t \to 0} F_1(t)
=\lim_{t \to \infty} F_2(t)
= \lim_{t \to \infty} F_3(t)
= 0 = F_1(0).
\end{align*}
Let $\F_0$ be the collection of all bounded functions $F: \Q \to [0, \infty)$ satisfying
\begin{align*}
\lim_{\ell(Q) \to 0} F(Q)
= \lim_{\ell(Q) \to \infty} F(Q)
= \lim_{|c_Q| \to \infty} F(Q)
=0. 
\end{align*}
\end{definition}

\begin{definition}\label{def:CCZK}
Let $\Delta$ be the diagonal of $\Rn \times \Rn \times \Rn$. A function $K: (\R^{3n} \setminus \Delta) \to \C$ is called \emph{a compact bilinear Calder\'{o}n--Zygmund kernel} if there exists $\delta \in (0, 1]$ such that 
\begin{align}
\label{eq:Size} 
|K(x, y, z)| 
&\leq  \frac{F(x, y, z)}{(|x-y| + |x-z|)^{2n}},    
\\ 
\label{eq:Holder-1} 
|K(x, y, z) - K(x', y, z)| 
&\leq F(x, y, z) \frac{|x-x'|^{\delta}}{(|x-y| + |x-z|)^{2n+\delta}}  
\end{align}
whenever $|x-x'| \leq \max\{|x-y|, |x-z| \}/2$, 
\begin{align}\label{eq:Holder-2} 
|K(x, y, z) - K(x, y', z)|  
\leq F(x, y, z) \frac{|y-y'|^{\delta}}{(|x-y| + |x-z|)^{2n+\delta}} 
\end{align}
whenever $|y - y'| \leq \max\{|x-y|, |x-z|\}/2$,  
\begin{align}\label{eq:Holder-3} 
|K(x, y, z) - K(x, y, z')|  
\leq F(x, y, z) \frac{|z-z'|^{\delta}}{(|x-y| + |x-z|)^{2n+\delta}}
\end{align}
whenever $|z-z'| \leq \max\{|x-y|, |x-z| \}/2$, where 
\begin{align*}
F(x, y, z) := F_1(|x-y| + |x-z|) F_2(|x-y| + |x-z|) F_3(|x+y| + |x+z|) 
\end{align*}
with the triple $(F_1, F_2, F_3) \in \F$. 

We say that $K$ is \emph{a standard bilinear Calder\'{o}n--Zygmund kernel} if the function $F$ above is replaced by a uniform constant $C \ge 0$. The smallest constant $C$ is denoted by $\|K\|_{\mathrm{CZ}(\delta)}$.  
\end{definition}

\begin{definition}\label{def:CZO} 
We say that a bilinear operator $T: L_c^{\infty}(\Rn) \times L_c^{\infty}(\Rn) \to L_{\loc}^1(\Rn)$ is \emph{associated with a compact (or standard) bilinear Calder\'{o}n--Zygmund kernel} if there exists a compact (or standard) bilinear Calder\'{o}n--Zygmund kernel $K$ such that  
\begin{align*}
\langle T(f_1, f_2), f_3 \rangle 
:= \int_{\Rn} T(f_1, f_2)(x) f_3 \, dx
= \int_{\R^{3n}} K(x, y, z) f_1(y) f_2(z) f_3(x) \, dx \, dy \, dz, 
\end{align*}
for any $f_1, f_2, f_3 \in L_c^{\infty}(\Rn)$ with $\supp(f_1) \cap \supp(f_2) \cap \supp(f_3) = \varnothing$. 

We assume that $T$ admits formal transposes $T^{*1},T^{*2}: L_c^{\infty}(\Rn) \times L_c^{\infty}(\Rn)
\to L^1_{\loc}(\Rn)$, which are defined by
\begin{align*}
\langle T(f_1, f_2), f_3 \rangle
= \langle T^{*1}(f_3,f_2), f_1 \rangle
= \langle T^{*2}(f_1, f_3), f_2 \rangle
\end{align*}
for all $f_1, f_2, f_3 \in L_c^{\infty}(\Rn)$.

An operator $T$ is called a \emph{bilinear Calder\'{o}n--Zygmund operator} if it is associated with a standard bilinear Calder\'{o}n--Zygmund kernel and can be boundedly extended from $L^{q_1}(\Rn) \times L^{q_2}(\Rn)$ to $L^q(\Rn)$ for some $\frac1q = \frac{1}{q_1} + \frac{1}{q_2}$ with $q_1, q_2 \in (1, \infty)$. 
\end{definition}

For $h \in L_c^{\infty}(\Rn)$ with $\int_{\Rn}  h \, dx= 0$, we define
\begin{align*}
\langle T(1,1), h \rangle
:= \lim_{\lambda \to \infty}
\langle T(\mathbf{1}_{\lambda \I}, \mathbf{1}_{\lambda \I}), h \rangle.
\end{align*}
The limit exists by the kernel regularity and the cancellation of $h$. Thus, $T(1, 1)$ is understood modulo constants, and $T(1, 1) \in \CMO(\Rn)$ means that this functional is represented by a $\CMO$ function. The same convention applies to $T^{*1}(1, 1)$ and $T^{*2}(1, 1)$.

\begin{definition}\label{def:WCP}
Let $T: L_c^{\infty}(\Rn) \times L_c^{\infty}(\Rn) \to L_{\loc}^1(\Rn)$ be a bilinear operator. We say that $T$ satisfies the \emph{weak compactness property} if there exists $F \in \mathscr{F}_0$ such that 
\begin{align*}
|\langle T (\mathbf{1}_Q, \mathbf{1}_Q), \mathbf{1}_Q \rangle| 
\le F(Q) |Q|, \quad \forall \, Q \in \D, \, \forall \, \D. 
\end{align*}
We say that $T$ satisfies the \emph{weak boundedness property} if the function $F$ above is replaced by a uniform constant $C \ge 0$. 
\end{definition}

\subsection{Dyadic grids}\label{sec:dyadic}
Let $\D_0$ be the standard dyadic grid on $\Rn$:
\begin{align*}
\D_0 := \big\{2^{-k}([0, 1)^n + m): k \in \Z, \, m \in \Z^n \big\}.
\end{align*}
Let $\Omega :=(\{0, 1\}^n)^{\Z}$ and let $\mathbb{P}_{\w}$ be the natural probability measure on $\Omega$: each component $\w_j$ has an equal probability $2^{-n}$ of taking any of the $2^n$ values in $\{0, 1\}^n$, and all components are independent of each other. Let $\mathbb{E}_{\w}$ denote the expectation over the random variables $\w_j$, $j \in \Z$. 

Given $\w=(\w_j)_{j \in \Z} \in \Omega$, the \emph{random dyadic grid} $\D_{\w}$ on $\Rn$ is defined by
\begin{align*}
\D_{\w} := \bigg\{Q+\w := Q + \sum_{j: 2^{-j}< \ell(Q)} 2^{-j} \w_j : Q \in \D_0\bigg\}.
\end{align*}
By a \emph{dyadic grid} $\D$ we mean that $\D=\D_{\w}$ for some $\w \in \Omega$. A cube $Q \in \D=\D_{\w}$ is called \emph{bad} if there exists a cube $Q' \in \D$ such that 
\begin{align*}
\ell(Q') \ge 2^r \ell(Q)
\quad\text{ and }\quad 
\d(Q, \partial Q') \le \ell(Q)^{\gamma} \ell(Q')^{1-\gamma}, 
\end{align*}
where $r \in \N_+$, $\gamma = \frac{\delta}{2(2n+\delta)}$, and $\delta > 0$. Otherwise a cube is called \emph{good}. Let $\D_{\w, \text{good}}$ denote the family of good cubes in $\D_{\w}$. Note that $\pi_{\text{good}} := \mathbb{P}_{\w}(Q + \w \text{ is good}) = \mathbb{E}_{\w} \mathbf{1}_{\text{good}}(Q + \w)$ is independent of the choice of $Q \in \D_0$. The appearing parameter $r \in \N_+$ is a large enough fixed integer so that $\pi_{\text{good}} > 0$.

\subsection{Haar functions}\label{sec:Haar}
Let $h_I$ be an $L^2$ normalized Haar function related to $I \in \D$, where $\D$ is a dyadic grid on $\Rn$. With this we mean that $h_I$, $I = I_1 \times \cdots \times I_n$, is one of the $2^n$ functions $h_I^\eta$, $\eta=(\eta_1, \ldots, \eta_n) \in \{0, 1\}^n$, defined by
\begin{align*}
h_I^\eta := h_{I_1}^{\eta_1} \otimes \cdots \otimes h_{I_n}^{\eta_n},
\end{align*}
where $h_{I_i}^0 = |I_i|^{-\frac12} \mathbf{1}_{I_i}$ and $h_{I_i}^1 = |I_i|^{-\frac12}(\mathbf{1}_{I_i^-}-\mathbf{1}_{I_i^+})$ for every $i = 1, \ldots, n$. Here $I_i^-$ and $I_i^+$ are the left and right halves of the interval $I_i$ respectively. If $\eta \neq 0$, the Haar function is cancellative : $\int_{\Rn} h_I^{\eta} \, dx= 0$. All the cancellative Haar functions form an orthonormal basis of $L^2(\Rn)$. If $f \in L^2(\Rn)$, we may write
\begin{align*} 
f= \sum_{I \in \D} \sum_{\eta \in \{0, 1 \}^n \setminus \{0\}^n} \langle f, h_I^\eta \rangle h_I^\eta.
\end{align*}
However, we suppress the finite $\eta$ summation and just write $f = \sum_{I \in \D} \langle f, h_I \rangle h_I.$

For every $k \in \N$ and $Q \in \D$, let $\D_k(Q) := \{I \in \D: I \subset Q, \ell(I) = 2^{-k} \ell(Q)\}$.
Define 
\begin{align*}
\Delta_Q f := \sum_{Q' \in \ch(Q)} 
\big(\langle f \rangle_{Q'} - \langle f \rangle_Q \big) \mathbf{1}_{Q'} 
\quad\text{ and }\quad 
\Delta_Q^k f := \sum_{I \in \D_k(Q)} \Delta_I f. 
\end{align*}
Then for every $p \in (1, \infty)$ and $f \in L^p(\Rn)$, there holds  
\begin{align}\label{f-decom}
f = \sum_{Q \in \D} \Delta_Q f 
= \sum_{Q \in \D} \langle f, h_Q \rangle \, h_Q, 
\end{align}
where the convergence takes place unconditionally (that is, independently of the order) in $L^p(\Rn)$. 

Given $k \in \N$ and a dyadic grid $\D$, define the dyadic square function as 
\begin{align*}
S_{\D}^k f 
:= \bigg(\sum_{Q \in \D} |\Delta_Q^k f|^2 \bigg)^{\frac12}. 
\end{align*}
Denote $S_{\D} := S_{\D}^0$. It follows from \cite[Theorem 2.1]{Wil} that 
\begin{align}\label{ddf-1}
\|S_{\D} f\|_{L^r} \simeq \|f\|_{L^r}, \quad\text{for all } r  \in (1, \infty). 
\end{align}
Moreover, it can be shown by \cite[Theorem 2.5]{Wil} and extrapolation that for any $k \in \N$,  
\begin{align}\label{JSD-3}
\|f\|_{L^r(v)} 
\lesssim \|S_{\D}^k f\|_{L^r(v)},  \quad \text{for all } r \in (0, \infty) \text{ and } v \in A_{\infty}.  
\end{align}

For each $k \in \Z$, we define the operators 
\begin{align}\label{def:Ek}
E_{2^k} f := \sum_{Q \in \D : \, \ell(Q)=2^k} \langle f \rangle_Q \, \mathbf{1}_Q 
\quad\text{ and }\quad 
D_{2^k} f:= E_{2^{k-1}} f - E_{2^k} f. 
\end{align}
Then it is not hard to check that  
\begin{align}\label{ddf-2}
E_{2^k} f = \sum_{Q \in \D: \, \ell(Q)>2^k} \Delta_Q f, \qquad 
D_{2^k} f = \sum_{Q \in \D: \, \ell(Q)=2^k} \Delta_Q f, 
\end{align}
and 
\begin{align}\label{ddf-3}
\bigg\|\bigg(\sum_{k \in \Z} |D_{2^k} f|^2 \bigg)^{\frac12} \bigg\|_{L^r} 
\simeq \|f\|_{L^r}, \quad\text{ for all } r \in (1, \infty). 
\end{align}

Given $N \in \N$, we define the \emph{projection operator} and its \emph{orthogonal operator} by 
\begin{align}\label{def:PN}
P_{\D(N)} f  := \sum_{Q \in \D(N)} \langle f, h_Q \rangle \, h_Q 
\quad\text{ and }\quad 
P_{\D(N)}^{\perp} f := f - P_{\D(N)} f,
\end{align}
for all $f \in L_{\loc}^1(\Rn)$. For a bounded bilinear operator $T: L^1(\Rn) \times L^1(\Rn) \longrightarrow L^{1/2,\infty}(\Rn)$, assume there exists a dense subspace $E \subset L^1(\Rn)$ such that $T(f_1, f_2) \in L^1_{\text{loc}}(\Rn)$, $f_1, f_2 \in E$. We first define $P_{\D(N)} T$ and $P_{\D(N)}^{\perp} T$ on $E \times E$ according to \eqref{def:PN}, and adopt the convention that
\begin{align*}
\| P_{\D(N)}^{\perp} T \|_{L^1 \times L^1 \to L^{1/2, \infty}} 
:= \sup_{\substack{f_1, f_2 \in E \\ \|f_1\|_1 \le 1, \|f_2\|_1 \le 1}} 
\| T(f_1, f_2) - P_{\D(N)}T(f_1, f_2) \|_{L^{1/2, \infty}}.
\end{align*}
When this quantity is finite, \(P_{\D(N)}^{\perp} T\) uniquely extends continuously to $L^1 \times L^1$, and we denote this extension by the same symbol; correspondingly, we define $P_{\D(N)}T := T - P_{\D(N)}^{\perp} T$ on $L^1 \times L^1$.

There holds
\begin{align}\label{ddf-4}
\sup_{N \in \N} \|P_{\D(N)}^{\perp} f\|_{L^r} 
\lesssim \|f\|_{L^r},\quad\text{ for all } r \in (1, \infty). 
\end{align}
Indeed, by \eqref{ddf-1}, for all $N \in \N$, 
\begin{align*}
\|P_{\D(N)} f\|_{L^r} 
& \simeq \|S_{\D}(P_{\D(N)} f)\|_{L^r} 
= \bigg\| \bigg(\sum_{Q \in \D} |\Delta_Q(P_{\D(N)} f)|^2 \bigg)^{\frac12} \bigg\|_{L^r}
\\
&= \bigg\| \bigg(\sum_{Q \in \D(N)} |\Delta_Q f|^2 \bigg)^{\frac12} \bigg\|_{L^r}
\le \|S_{\D} f\|_{L^r} 
\simeq \|f\|_{L^r}, 
\end{align*}
where the implicit constants are independent of $N$. We should mention that the inequality \eqref{ddf-4} does not hold for any $r \in (0, 1)$ because $P_{\D_0(N)}^{\perp} \mathbf{1}_{[0, 1)} = 2^{-N} \mathbf{1}_{[0, 2^N)}$ on the standard dyadic grid $\D_0$.

\subsection{$\BMO$ and $\CMO$ spaces}
\begin{definition}\label{def:BMO}
A locally integrable function $f: \Rn \to \R$ belongs to $\BMO(\Rn)$ if
\begin{align*}
\|f\|_{\BMO} := \sup_{Q \in \Q} \fint_{Q} |f(x) - \langle f \rangle_Q| \, dx < \infty.
\end{align*}
A dyadic analogue $\BMO_{\D}$ is defined on the dyadic grid $\D$ in place of $\Q$. Let $\CMO(\Rn)$ denote the closure of $\mathscr{C}_c^{\infty}(\Rn)$ in $\BMO(\Rn)$. Additionally, the space $\CMO(\Rn)$ is endowed with the norm of $\BMO(\Rn)$. 
\end{definition}

By \cite{LTW, Uch} and the John--Neirenberg inequality, we see that for any $p \in [1, \infty)$, 
\begin{align}
\label{ffbb-1}
f \in \CMO(\Rn) 
& \iff f \in \BMO(\Rn) 
\, \, \text{and} \, 
\lim_{N \to \infty} \sup_{Q \notin \Q(N)} 
\bigg[\fint_Q |f - \langle f \rangle_Q|^p \bigg]^{\frac1p} = 0
\\
\label{ffbb-2}
& \iff f \in \BMO(\Rn) 
\, \, \text{and} \, 
\lim_{N \to \infty} \sup_{\D} \|P_{\D(N)}^{\perp} f\|_{\BMO_{\D}} = 0.
\end{align}
An alternative proof of \eqref{ffbb-2} is give in Theorem \ref{thm:PDN}.

Let $\mathrm{H}^1(\Rn)$ denote the Hardy space (cf. \cite[Chapter VII, Section 3]{Stein70}). It was shown by Fefferman and Stein \cite{FS} that $\BMO(\Rn)$ is the dual space of $\mathrm{H}^1(\Rn)$.

Given $\lambda>0$ and $f \in L_{\loc}^{\lambda}(\Rn)$, the sharp maximal function $M_{\lambda}^{\#} f$ is defined by 
\begin{align*}
M_{\lambda}^{\#} f(x) 
:= \sup_{Q \ni x: Q \in \Q} \O_{\lambda}(f; Q) 
\quad \text{ and } \quad 
\O_{\lambda}(f; Q)
: = \inf_{c \in \R}
\bigg(\fint_Q |f(y) - c|^{\lambda} \, dy \bigg)^{\frac{1}{\lambda}}. 
\end{align*}
Similarly, we define its dyadic version $M_{\D, \lambda}^{\#}$ if $\Q$ is replaced by a dyadic grid $\D$. 
Denote $M^{\#} := M_1^{\#}$ and $M_{\D}^{\#} := M_{\D, 1}^{\#}$. The operator $M_{\lambda}^{\#}$ was introduced by Str\"{o}mberg \cite{Str} and is a modification of $M^{\#}$ of Fefferman and Stein \cite{FS}. Additionally, it was shown in \cite[p. 518]{Str} that 
\begin{align}\label{ffbb-3}
\|f\|_{\BMO} 
\simeq \|M_{\lambda}^{\#} f\|_{L^{\infty}}, \quad \lambda > 0. 
\end{align}
Inspired by \eqref{ffbb-1}--\eqref{ffbb-3}, we define the weighted $\BMO$ and $\CMO$ spaces as follows.

\begin{definition}\label{def:BMOw}
Let $\lambda > 0$ and $w$ be a weight on $\Rn$. The space $\BMO_{\lambda}(w^{\infty})$ consists of all functions $f \in L_{\loc}^{\lambda}(\Rn)$ satisfying
\begin{align*}
\|f\|_{\BMO_{\lambda}(w^{\infty})} 
:= \|(M_{\lambda}^{\#}f) w\|_{L^{\infty}} < \infty. 
\end{align*}
The $\CMO_{\lambda}(w^{\infty})$ consists of all functions $f \in L_{\loc}^{\lambda}(\Rn)$ satisfying
\begin{align*}
f \in \BMO_{\lambda}(w^{\infty})
\quad \text{ and } \quad
\lim_{N \to \infty} \|(M_{\lambda, N}^{\#}f) w\|_{L^{\infty}} = 0,
\end{align*}
where 
\begin{align*}
M_{\lambda, N}^{\#} f(x) 
:= \sup_{Q \ni x: Q \notin \Q(N)} \O_{\lambda}(f; Q). 
\end{align*}
When $w \equiv 1$, we simply write $\BMO_{\lambda}(\Rn) = \BMO_{\lambda}(w^{\infty})$ and $\CMO_{\lambda}(\Rn) = \CMO_{\lambda}(w^{\infty})$.
\end{definition}

\begin{lemma}\label{lem:cmob}
For every $\lambda > 0$, there holds 
\begin{align*}
\BMO_{\lambda}(\Rn) = \BMO(\Rn) 
\quad \text{ and } \quad 
\CMO_{\lambda}(\Rn) = \CMO(\Rn).
\end{align*}
\end{lemma}

\begin{proof}
The fact $\BMO_{\lambda}(\Rn) = \BMO(\Rn)$ follows from \eqref{ffbb-3}. Let $f \in \BMO(\Rn)$. If $\lambda \ge 1$, for each $c \in \R$, 
\begin{align*}
\bigg(\fint_Q |f - \langle f \rangle_Q|^{\lambda} \bigg)^{\frac{1}{\lambda}}
\le \bigg(\fint_Q |f - c|^{\lambda} \bigg)^{\frac{1}{\lambda}}
+ |\langle f \rangle_Q - c|
\le 2 \bigg(\fint_Q |f - c|^{\lambda} \bigg)^{\frac{1}{\lambda}},
\end{align*}
which gives
\begin{align}\label{eq:oqq-1}
\sup_{Q \notin \Q(N)} \O_{\lambda}(f; Q)
\simeq \sup_{Q \notin \Q(N)} \bigg(\fint_Q |f - \langle f \rangle_Q|^{\lambda} \bigg)^{\frac{1}{\lambda}}. 
\end{align}
Let $\lambda \in (0, 1)$. Let $m_Q(f)$ be a median of $f$ on $Q$, namely, a real number satisfying
\begin{align*}
\big|\{x \in Q : f(x) > m_Q(f) \}\big| \le |Q|/2
\quad \text{ and } \quad
\big|\{x \in Q : f(x) < m_Q(f) \}\big| \le |Q|/2.
\end{align*}
Then for all $c \in \R$, 
\begin{align*}
|m_Q(f) - c|^{\lambda} \le 2 \fint_Q |f - c|^{\lambda}
\quad \text{ and } \quad
\fint_Q |f - m_Q(f)|^{\lambda}
\le 3 \O_{\lambda}(f; Q)^{\lambda}.
\end{align*}
By the first estimate above, 
\begin{align*}
|m_Q(f) - \langle f \rangle_Q| 
\le 2 \fint_Q |f - \langle f \rangle_Q| 
\le 2 \|f\|_{\BMO},
\end{align*}
which along with the John--Nirenberg inequality implies 
\begin{align*}
\bigg(\fint_Q |f - m_Q(f)|^2 \, \bigg)^{\frac12} 
\le |m_Q(f) - \langle f \rangle_Q| + \bigg(\fint_Q |f - \langle f \rangle_Q|^2 \, \bigg)^{\frac12} 
\lesssim \|f\|_{\BMO}.
\end{align*}
By Jensen's and H\"{o}lder's inequalities, we have for $\theta := \frac{\lambda}{2 - \lambda} \in (0, 1)$,
\begin{align*}
\O_{\lambda}(f; Q)
& \le \fint_Q |f - \langle f \rangle_Q|
\le 2 \fint_Q |f - m_Q(f)|^{\theta} |f - m_Q(f)|^{1-\theta}
\\
& \le 2 \bigg(\fint_Q |f - m_Q(f)|^{\lambda} \, \bigg)^{\frac{\theta}{\lambda}}
\bigg(\fint_Q |f - m_Q(f)|^2 \, \bigg)^{\frac{1-\theta}{2}} 
\\
& \lesssim \O_{\lambda}(f; Q)^{\theta} \|f\|_{\BMO}^{1-\theta},
\end{align*}
which leads to
\begin{align}\label{eq:oqq-2}
\sup_{Q \notin \Q(N)} \O_{\lambda}(f; Q)
\le \sup_{Q \notin \Q(N)} \fint_Q |f - \langle f \rangle_Q|
\lesssim \sup_{Q \notin \Q(N)} \O_{\lambda}(f; Q)^{\theta} 
\|f\|_{\BMO}^{1-\theta}
\end{align}
Thus, the assertion $\CMO_{\lambda}(\Rn) = \CMO(\Rn)$ is a consequence of \eqref{ffbb-1} and \eqref{eq:oqq-1}--\eqref{eq:oqq-2}.
\end{proof}

\subsection{Muckenhoupt weights}
A measurable function $w$ on $\Rn$ is called a \emph{weight} if $0<w(x)<\infty$ for a.e.~$x \in \Rn$. Given $p \in (1, \infty)$, we define the Muckenhoupt class $A_p$ as the collection of all weights $w$ on $\Rn$ satisfying 
\begin{equation*}
[w]_{A_p} 
:= \sup_{Q \in \Q} \langle w \rangle_Q \langle w^{1-p'} \rangle_Q^{p-1}
< \infty. 
\end{equation*} 
For the endpoint cases, we say that $w \in A_1$ if  
\begin{equation*}
[w]_{A_1} 
:= \sup_{Q \in \Q} \langle w \rangle_Q \big(\esssup_Q w^{-1} \big) 
< \infty;
\end{equation*}
we say that $w \in A_{\infty}$ if  
\begin{equation*}
[w]_{A_{\infty}} 
:= \sup_{Q \in \Q} \langle w \rangle_Q \exp \big(\langle \log w^{-1} \rangle_Q \big) 
< \infty.
\end{equation*}
Then, $A_{\infty} = \bigcup_{p \ge 1}A_p$, see \cite[Corollary 7.3.4]{Gra1}.

Let us recall some properties of Muckenhoupt classes. The first one is the reverse H\"{o}lder inequality (cf. \cite{CGPSZ, HP, LOP}): for every $w \in A_p$ with $p \in [1, \infty]$, there exists $r_w = 1 + \frac{1}{c_n [w]_{A_p}} \in (1, 2)$ such that 
\begin{align}\label{eq:RH}
\langle w^{r_w} \rangle_Q^{\frac{1}{r_w}} 
\le 2 \langle w \rangle_Q,  \quad\forall Q \in \Q.
\end{align}
The second one is the openness of the class $A_p$: for any $p \in (1, \infty)$, 
\begin{align}\label{eq:open}
w \in A_p \quad \Longrightarrow \quad
w \in A_{p-\varepsilon} \, \, \text{ for some } \varepsilon \in (0, p-1). 
\end{align}
Additionally, it was show in \cite[p. 253]{CR} that
\begin{align}\label{eq:AiBMO}
w \in A_{\infty} \quad \Longrightarrow \quad 
\log w \in \BMO.
\end{align}
A consequence of the John--Nirenberg inequality is the following (cf. \cite[Lemma 3.5]{BMMST}):
\begin{equation}\label{eq:JN}
\begin{aligned}
& \forall b \in \BMO \text{ and } p \in (1, \infty),   
\exists \lambda_0 \in (0, 1) \text{ such that } 
\\
& \langle e^{|\lambda| |b - \langle b \rangle_Q|} \rangle_Q \lesssim 1 
\quad \text{ and } \quad
e^{\lambda b} \in A_p \text{ for all } |\lambda| \le \lambda_0.
\end{aligned}
\end{equation}
More properties about Muckenhoupt classes can be found in \cite[Section 7]{Gra1}.

Next, we turn to the multilinear version of Muckenhoupt weights introduced in \cite{LOPTT}. 

\begin{definition}\label{def:Ap}
Given $\vec{p} = (p_1, \ldots, p_m)$ with $1 \le p_1, \ldots, p_m \le \infty$ and $\vec{w} = (w_1, \ldots, w_m)$ with $0<w_1, \ldots, w_m<\infty$ a.e. on $\Rn$, we say that $\vec{w} \in A_{\vec{p}}$ if 
\begin{align*}
[\vec{w}]_{A_{\vec{p}}} 
:= \sup_{Q \in \Q} \langle w^p \rangle_Q^{\frac1p} 
\prod_{i=1}^m \langle w_i^{-p_i'} \rangle_Q^{\frac{1}{p_i'}} < \infty,
\end{align*}
where $\frac1p = \sum_{i=1}^m \frac{1}{p_i}$ and $w= \prod_{i=1}^m w_i$. If $p_i = 1$, $\langle w_i^{-p_i'} \rangle_Q^{\frac{1}{p_i'}}$ is understood as $(\essinf_Q w_i)^{-1}$; and if $p=\infty$, $\langle w^p \rangle_Q^{\frac1p}$ is understood as $\esssup_Q w$. 
\end{definition}

The following characterization of the class $A_{\vec{p}}$ was essentially contained in \cite[Theorem 3.6]{LOPTT}. Although \cite[Theorem 3.6]{LOPTT} only concerns the setting $1 \le p_1, \ldots, p_m<\infty$, some minor modifications of its proof will give the corresponding result in the case $p_i=\infty$ for some/all $i \in \{1, \ldots, m\}$.

\begin{lemma}\label{lem:weight}
Let $\vec{p}=(p_1, \ldots, p_m)$ with $1 \leq p_1, \ldots, p_m \le \infty$ and let $\vec{w}=(w_1, \ldots, w_m)$ be a vector of weights.  Then 
\begin{align*}
\vec{w} \in A_{\vec{p}} 
\quad \iff \quad 
w^p \in A_{mp} \quad \text{and} \quad w_i^{-p'_i} \in A_{mp'_i}, \, \, i=1, \ldots, m, 
\end{align*}
where $\frac1p = \sum_{i=1}^m \frac{1}{p_i}$ and $w= \prod_{i=1}^m w_i$. If $p_i = 1$, $[w_i^{-p_i'}]_{A_{mp'_i}}^{1/p'_i}$ is understood as $[w_i^{1/m}]_{A_1}^m$; and if $p=\infty$, $[w^p]_{A_{mp}}^{1/p}$ is understood as $[w^{-1/m}]_{A_1}^m$. Moreover, 
\begin{align*}
\max \bigg\{[w^p]_{A_{mp}}^{\frac1p}, \, \max_{1 \le i \le m} 
[w^{-p'_i}]_{A_{mp'_i}}^{\frac{1}{p'_i}} \bigg\}
\le [\vec{w}]_{A_{\vec{p}}}
\le [w^p]_{A_{mp}}^{\frac1p} \prod_{i=1}^m [w^{-p'_i}]_{A_{mp'_i}}^{\frac{1}{p'_i}}. 
\end{align*}
\end{lemma}

It follows from \eqref{eq:RH} and Lemma \ref{lem:weight} that 
\begin{align}\label{eq:wopen}
\vec{w} \in A_{\vec{p}}  \quad \Longrightarrow \quad
\exists \lambda = \lambda([w]_{A_{\vec{p}}}) \in (1, 2) \text{ so that } 
[\vec{w}^{\lambda}]_{A_{\vec{p}}} \le 2^{2m} [\vec{w}]_{A_{\vec{p}}},
\end{align}
where $\lambda: [1, \infty) \to (1, 2)$ is a decreasing function.

The result below concerns the relationship between $A_{(\infty, \infty)}$ and $A_{(p_1, p_2)}$, which asserts that 
$A_{(\infty, \infty)}$ is a proper subset of the union of $A_{(p_1, p_2)}$, and every $A_{(p_1, p_2)}$ is not a subset of $A_{(\infty, \infty)}$. This is quite different from the linear case.

\begin{lemma}
$A_{(\infty, \infty)} \subsetneq \bigcup_{1<p_1, p_2<\infty} A_{(p_1, p_2)}$ and 
$\bigcap_{1 \le p_1, p_2<\infty} A_{(p_1, p_2)} \not\subset A_{(\infty, \infty)}$. 
\end{lemma}

\begin{proof}
Let $(w_1, w_2) \in A_{(\infty, \infty)}$ and $w=w_1 w_2$. Then Lemma \ref{lem:weight} implies $w_1^{-1}, w_2^{-1} \in A_2$ and $w^{-\frac12} \in A_1$. By \eqref{eq:RH}, there exist $\widetilde{r}_1, \widetilde{r}_2, r \in (1, \infty)$ such that 
\begin{align}\label{wra-1}
w_1^{-\widetilde{r}_1}, w_2^{-\widetilde{r}_2} \in A_2 
\quad\text{ and }\quad 
w^{-\frac{r}{2}} \in A_1. 
\end{align}
Note that 
\begin{align}\label{wra-2}
v \in A_s \quad \Longrightarrow \quad v^{\theta} \in A_s 
\quad\text{ for all } \theta \in (0, 1). 
\end{align}
Take $1<r_i \le \min\{r, \widetilde{r}_i\}$ and $p_i=r'_i$, $i=1, 2$, then $\frac1p := \frac{1}{p_1} + \frac{1}{p_2} = 2-\frac{1}{r_1} - \frac{1}{r_2} \le \frac{2(r-1)}{r}$, which is equivalent to $\frac{p}{2p-1} \le \frac{r}{2}$. Using these and \eqref{wra-1}--\eqref{wra-2}, we obtain 
\begin{align}\label{wra-3}
w_i^{-p'_i} = w_i^{-r_i} \in A_2 \subset A_{2p'_i}, \quad i=1, 2, 
\quad\text{ and }\quad 
w^{-\frac{p}{2p-1}} \in A_1 \subset A_{\frac{2p}{2p-1}}. 
\end{align}
Since $w^{-\frac{p}{2p-1}} \in A_{\frac{2p}{2p-1}}$ is equivalent to $w^p \in A_{2p}$, the estimate \eqref{wra-3} and Lemma \ref{lem:weight} give $(w_1, w_2) \in A_{(p_1, p_2)}$. This shows $A_{(\infty, \infty)} \subset \bigcup_{1<p_1, p_2<\infty} A_{(p_1, p_2)}$.

To justify $\bigcap_{1 \le p_1, p_2<\infty} A_{(p_1, p_2)} \not\subset A_{(\infty, \infty)}$, it suffices to take $w_1(x) := 1 + \log^+ \frac{1}{|x|}$ and $w_2 \equiv 1$. We claim that
\begin{align}\label{eq:wa1}
w_1^{\alpha} \in A_1, \quad \forall \alpha > 0.
\end{align}
Indeed, given a cube $Q$, if $r_Q := \sup_{x \in Q} |x| \le 2 \inf_{x \in Q} |x|$, then 
\begin{align*}
\langle w_1^{\alpha} \rangle_Q
\le \big(\sup_Q w_1 \big)^{\alpha} 
\le (1 + \log 2)^{\alpha} \inf_Q w_1^{\alpha}.
\end{align*}
If $r_Q > 2 \inf_{x \in Q} |x|$, then for all $x, y \in Q$, $|x| \le |x-y| + |y| \le \diam(Q) + |y|$, which gives
\begin{align*}
r_Q = \sup_{x \in Q} |x|
\le \diam(Q) + \inf_{y \in Q} |y|
\le \diam(Q) + r_Q/2.
\end{align*}
Thus, $r_Q \le 2 \diam(Q)$, $Q \subset B(0, r_Q)$, $|B(0, r_Q)| \simeq |Q|$, and hence,
\begin{align*}
\langle w_1^{\alpha} \rangle_Q
\lesssim \langle w_1^{\alpha} \rangle_{B(0, r_Q)}
& = \frac{n}{r_Q^n} \int_0^{r_Q} \bigg(1 + \log^+ \frac1t \bigg)^{\alpha} t^{n-1} \, dt
\\
& = n \int_0^1 \bigg(1 + \log^+ \frac{1}{s r_Q} \bigg)^{\alpha} s^{n-1} \, ds
\\
& \le n \bigg(1 + \log^+ \frac{1}{r_Q} \bigg)^{\alpha} 
\int_0^1 \bigg(1 + \log \frac1s \bigg)^{\alpha} s^{n-1} \, ds
\\
& \lesssim \bigg(1 + \log^+ \frac{1}{r_Q} \bigg)^{\alpha} 
= \essinf_Q w_1^{\alpha}.
\end{align*}
This shows \eqref{eq:wa1}.

Now let $p_1, p_2 \in [1, \infty)$ and $\frac1p = \frac{1}{p_1} + \frac{1}{p_2}$. Then by \eqref{eq:wa1},
\begin{align*}
\langle w_1^p \rangle_Q^{\frac1p} \lesssim \essinf_Q w_1 
\quad \text{ and } \quad
\langle w_1^{-p'_1} \rangle_Q^{\frac{1}{p'_1}} \le (\essinf_Q w_1)^{-1}, 
\end{align*}
which means $(w_1, w_2) \in A_{(p_1, p_2)}$. But for every cube $Q \ni 0$, 
\begin{align*}
0 < \langle w_1^{-1} \rangle_Q < \infty = \esssup_Q w_1,
\end{align*}
which gives
\begin{align*}
[(w_1, w_2)]_{A_{(\infty, \infty)}} 
= \sup_{Q \in \Q} (\esssup_Q w_1) \langle w_1^{-1} \rangle_Q 
= \infty.
\end{align*}
This proves $(w_1, w_2) \in \bigcap_{1 \le p_1, p_2<\infty} A_{(p_1, p_2)}$ and $(w_1, w_2) \notin A_{(\infty, \infty)}$.

Finally, if $A_{(\infty, \infty)} = \bigcup_{1<p_1, p_2<\infty} A_{(p_1, p_2)}$, then $\bigcap_{1 \le p_1, p_2<\infty} A_{(p_1, p_2)} \subset \bigcup_{1<p_1, p_2<\infty} A_{(p_1, p_2)} = A_{(\infty, \infty)}$, which contradicts the preceding conclusion. Hence, $A_{(\infty, \infty)} \subsetneq \bigcup_{1<p_1, p_2<\infty} A_{(p_1, p_2)}$. This completes the proof. 
\end{proof}

Let us present the multivariable Rubio de Francia extrapolation theorem from \cite{Nie}. 
\begin{theorem}\label{thm:RdF}
Let $\mathcal{F}$ be a collection of $(m+1)$-tuples of non-negative functions. Assume that there exist 
some $q_1, \ldots, q_m \in [1, \infty]$ such that 
\begin{align*}
\|f\|_{L^q(v^q)} \big(\text{resp. } \|f\|_{L^{q, \infty}(v^q)} \big)
\le \Xi([\vec{v}]_{A_{\vec{q}}}) \prod_{i=1}^m \|f_i\|_{L^{q_i}(v_i^{q_i})}, 
\quad \forall \, (f, \vec{f}) \in \mathcal{F},  
\end{align*}
for all $\vec{v} \in A_{\vec{q}}$, where $\frac1q = \sum_{i=1}^m \frac{1}{q_i}$, $v=\prod_{i=1}^m v_i$, and $\Xi: [1, \infty) \to [1, \infty)$ is a non-decreasing function. Then 
\begin{align*}
\|f\|_{L^p(w^p)} \big(\text{resp. } \|f\|_{L^{p, \infty}(w^p)} \big)
\le 2^{m^3} \Xi \big(C_0 [\vec{w}]_{A_{\vec{p}}}^{\frac{\alpha}{m}} \big)^m 
\prod_{i=1}^m \|f_i\|_{L^{p_i}(w_i^{p_i})}, 
\quad \forall \, (f, \vec{f}) \in \mathcal{F},  
\end{align*}
for all $p_1, \ldots, p_m \in (1, \infty]$ and for all $\vec{w} \in A_{\vec{p}}$, where $\frac1p = \sum_{i=1}^m \frac{1}{p_i}>0$, $w=\prod_{i=1}^m w_i$, $C_0 = C_0(\vec{p}, \vec{q})$, and $\alpha = \alpha(\vec{p}, \vec{q})$. 
\end{theorem}

To give a characterization of $A_{\vec{p}}$, let us define the multilinear maximal operator
\begin{equation*}
\mathcal{M}(\vec{f})(x) 
:= \sup_{Q \in \Q} \prod_{i=1}^m \langle |f_i| \rangle_Q \, \mathbf{1}_Q(x). 
\end{equation*}
In the case $m=1$, we write $M = \mathcal{M}$ and $M_{\lambda} f = M(|f|^{\lambda})^{\frac{1}{\lambda}}$ for any $\lambda>0$.

\begin{theorem}\label{thm:Mbdd}
Let $p_1, \ldots, p_m \in (1, \infty]$. Then $\vec{w} \in A_{\vec{p}}$ if and only if $\mathcal{M}$ is bounded from $L^{p_1}(w_1^{p_1}) \times \cdots \times L^{p_m}(w_m^{p_m})$ to $L^p(w^p)$, where $\frac1p = \sum_{i=1}^m \frac{1}{p_i}$ and $w=\prod_{i=1}^m w_i$. 
\end{theorem}

\begin{proof} 
In the situation $p_1, \ldots, p_m \in (1, \infty)$, this result was proved in \cite[Theorem 3.7]{LOPTT}. But the endpoint case $p_i=\infty$ will occur in the current scenario, we here present a new proof.  

To show the sufficiency, assume that $\vec{w} \in A_{\vec{\infty}}$ and $w=\prod_{i=1}^m w_i$. By \cite[p. 792]{HP}, each cube $Q$ is contained in a shifted dyadic cube $Q_{\alpha} \in \mathscr{D}_{\alpha}$ with $\ell(Q_{\alpha}) \le 6 \ell(Q)$ for some $\alpha \in \Lambda := \{0, 1/3\}^n$, where 
\begin{align*}
\mathscr{D}_{\alpha} 
:= \big\{2^{-k} \big([0, 1)^n + m + (-1)^k \alpha \big): k \in \Z, m \in \Z^n\big\}. 
\end{align*}
Thus, 
\begin{equation}\label{MMD-1}
\mathcal{M}(\vec{f})(x) 
\lesssim \sup_{\alpha \in \Lambda} \mathcal{M}_{\mathscr{D}_{\alpha}}(\vec{f})(x) 
:= \sup_{\alpha \in \Lambda} \sup_{x \in Q \in \mathscr{D}_{\alpha}}  
\prod_{i=1}^m \langle |f_i| \rangle_Q, \quad x \in \Rn. 
\end{equation}
For each cube $Q \in \bigcup_{\alpha \in \Lambda} \mathscr{D}_{\alpha}$, there exists a subset $E_Q \subset Q$ such that $|E_Q|=0$ and $w(x) \le \esssup_{Q} w$ for all $x \in Q \setminus E_Q$. Denote $E := \bigcup_{\alpha \in \Lambda} \bigcup_{Q \in \mathscr{D}_{\alpha}} E_Q$. Then $|E|=0$ and for all $x \in \Rn \setminus E$,  
\begin{equation}\label{MMD-2}
\sup_{\alpha \in \Lambda} \sup_{x \in Q \in \mathscr{D}_{\alpha}}  
\frac{w(x)}{\esssup\limits_Q w} \le 1. 
\end{equation}
Gathering \eqref{MMD-1} and \eqref{MMD-2}, we obtain that for any $x \in \Rn \setminus E$, 
\begin{align*}
\mathcal{M}(\vec{f})(x) w(x) 
&\lesssim w(x) \sup_{\alpha \in \Lambda} \sup_{x \in Q \in \mathscr{D}_{\alpha}}  
\prod_{i=1}^m \|f_i w_i\|_{L^{\infty}} \langle w_i^{-1} \rangle_Q
\le [\vec{w}]_{A_{\vec{\infty}}} \prod_{i=1}^m \|f_i w_i\|_{L^{\infty}},  
\end{align*}
which asserts that $\|\mathcal{M}(\vec{f}) w\|_{L^{\infty}} \lesssim [\vec{w}]_{A_{\vec{\infty}}} \prod_{i=1}^m \|f_i w_i\|_{L^{\infty}}$ for all $\vec{w} \in A_{\vec{\infty}}$. By Theorem \ref{thm:RdF}, this implies that $\mathcal{M}$ is bounded from $L^{p_1}(w_1^{p_1}) \times \cdots \times L^{p_m}(w_m^{p_m})$ to $L^p(w^p)$ for all $p_1, \ldots, p_m \in (1, \infty]$ and $\vec{w} \in A_{\vec{p}}$. 

To demonstrate the necessity, assume that $\M$ is bounded from $L^{p_1}(w_1^{p_1}) \times \cdots \times L^{p_m}(w_m^{p_m})$ to $L^p(w^p)$. Let $Q \subset \Rn$ be a cube. Then
\begin{align*}
\M(\vec{f})(x)
\ge \mathbf{1}_Q(x) \prod_{i=1}^m \langle |f_i| \rangle_Q, \quad x \in \Rn,
\end{align*}
which gives 
\begin{align*}
\|w \mathbf{1}_Q\|_{L^p} \prod_{i=1}^m \langle |f_i| \rangle_Q 
\le \|\M(\vec{f}) w\|_{L^p}
\le C_0 \prod_{i=1}^m \|f_i w_i\|_{L^{p_i}},
\end{align*}
where $C_0 \in (0, \infty)$. Taking $f_i = (1 + w_i)^{-1} \mathbf{1}_Q$ for each $i=1, \ldots, m$, we have $0 < \|w \mathbf{1}_Q\|_{L^p} < \infty$. Then using the fact $\sup_{\|f_i w_i\|_{L^{p_i}} \le 1} \int_Q |f_i| = \|w_i^{-1} \mathbf{1}_Q\|_{L^{p'_i}}$, we obtain 
\begin{align*}
\langle w^p \rangle_Q^{\frac1p} \prod_{i=1}^m \langle w_i^{-p'_i} \rangle_Q^{\frac{1}{p'_i}} 
= \|w \mathbf{1}_Q\|_{L^p} \prod_{i=1}^m \frac{\|w_i^{-1} \mathbf{1}_Q\|_{L^{p'_i}}}{|Q|} 
\le C_0,
\end{align*}
which implies $\vec{w} \in A_{\vec{p}}$ as desired.  
\end{proof}

\subsection{Technical lemmas} 
We end up this section with some useful technical results. 

\begin{lemma}[\cite{CYY}]\label{lem:diag}
Let $r>1$ and $1<s<1+\frac1n$. Assume that a bounded and decreasing function $F_2: [0, \infty) \to [0, \infty)$ satisfies  $\lim_{t \to \infty} F_2(t) = 0$. Then there exists a bounded and decreasing function $\widetilde{F}_2: [0, \infty) \to [0, \infty)$ satisfying $\lim_{t \to \infty} \widetilde{F}_2(t) = 0$ such that for all cubes $I, J \subset \Rn$ with $\ell(I) = \ell(J)$, $I \cap J = \varnothing$, and $\d(I, J)=0$,
\begin{align*}
\mathscr{I}_1
&:= \bigg(\fint_I \fint_J F_2(|x - y|)^r \, dx \, dy \bigg)^{\frac1r}
\lesssim \widetilde{F}_2(\ell(I)),
\end{align*}
and
\begin{align*}
\mathscr{I}_2
&:= \bigg(\fint_I \fint_J \frac{1}{|x - y|^{s n}} \, dx \, dy  \bigg)^{\frac1s}
\lesssim |I|^{-1}.
\end{align*}
\end{lemma}

\begin{lemma}\label{lem:improve}
Given $(F_1, F_2, F_3) \in \F$, let
\begin{align*}
F(x, y, z) := F_1(|x-y| + |x-z|) F_2(|x-y| + |x-z|) F_3(|x+y| + |x+z|).
\end{align*}
\begin{list}{\rm (\theenumi)}{\usecounter{enumi}\leftmargin=1.2cm \labelwidth=1cm \itemsep=0.2cm \topsep=.2cm \renewcommand{\theenumi}{\roman{enumi}}}

\item\label{size-imp} There exists a triple $(F'_1, F'_2, F'_3) \in \F$ such that $F'_1$ is monotone increasing, $F'_2$ and $F'_3$ are monotone decreasing, and 
\begin{align*}
F(x, y, z)
\le F'(x, y, z)
&:= F'_1(|x-y| + |x-z|) F'_2(|x-y| + |x-z|) 
\\  
&\qquad\times F'_3 \bigg(1 + \frac{|x+y| + |x+z|}{1+|x-y|+|x-z|}\bigg).
\end{align*}

\item\label{Holder-imp}
Assume that $F_1$ is monotone increasing. Then there exist $\delta' \in (0, \delta)$ and $(F'_1, F'_2, F'_3) \in \F$ such that $F'_1$ is monotone increasing, $F'_2$ and $F'_3$ are monotone decreasing, and
\begin{align*}
\frac{F(x, y, z) |x-x'|^{\delta}}{(|x-y| + |x-z|)^{2n + \delta}}
\le \frac{F'(x, y, z) |x-x'|^{\delta'}}{(|x-y| + |x-z|)^{2n + \delta'}},
\end{align*}
whenever $|x-x'| \leq \max\{|x-y|, |x-z\}/2$, where
\begin{align*}
F'(x, y, z)
:= F'_1(|x-x'|) F'_2(|x-y| + |x-z|) F'_3 \bigg(1 + \frac{|x+y| + |x+z|}{1+ |x-y| + |x-z|}\bigg).
\end{align*}
\end{list}
\end{lemma}

\begin{proof}
Note that $|x+y| + |x+z| \ge \frac{|x+y| + |x+z|}{1+|x-y| + |x-z|}$. Then by definition, 
\begin{align}\label{imp-2}
\lim_{|x-y| + |x-z| \to 0} F(x, y, z)
&= \lim_{|x-y| + |x-z| \to \infty} F(x, y, z)
\\ \nonumber 
&= \lim_{1 + \frac{|x+y| + |x+z|}{1+|x-y| + |x-z|} \to \infty} F(x, y, z)
=0.
\end{align}
Define
\begin{align*}
F'_1(t) &:= \sup_{|x-y| + |x-z| \le t} F(x, y, z)^{\frac13}, 
\\
F'_2(t) &:= \sup_{|x-y| + |x-z| \ge t} F(x, y, z)^{\frac13},
\\
F'_3(t) &:= \sup_{1 + \frac{|x+y| + |x+z|}{1+|x-y| + |x-z|} \ge t} F(x, y, z)^{\frac13}.
\end{align*}
It is obvious that $F'_1$ is bounded and monotone increasing, both $F'_2$ and $F'_3$ are bounded and monotone decreasing. Moreover, \eqref{imp-2} implies that $(F'_1, F'_2, F'_3) \in \F$,
\begin{align}\label{FX-1}
F(x, y, z)^{\frac13} 
\le F'_i(|x-y| + |x-z|), \quad i=1,2, 
\end{align}
and 
\begin{align}\label{FX-2}
F(x, y, z)^{\frac13} 
\le F'_3 \bigg(1 + \frac{|x+y| + |x+z|}{1+|x-y| + |x-z|}\bigg). 
\end{align}
Thus, part \eqref{size-imp} is a consequence of \eqref{FX-1} and \eqref{FX-2}.

To prove part \eqref{Holder-imp}, pick $\delta' \in (0, \delta)$ and denote 
\begin{align}\label{def:Fxyy}
F(x, y, z; x')
:= F(x, y, z) \frac{|x-x'|^{\delta - \delta'}}{(|x-y| + |x-z|)^{\delta - \delta'}} 
\end{align}
for all $|x-x'| \leq \max\{|x-y|, |x-z|\}/2$. 
We claim that
\begin{align}\label{imp-3}
\lim_{|x-x'| \to 0} F(x, y, z; x')
&= \lim_{|x-y| + |x-z| \to \infty} F(x, y, z; x')
\\ \nonumber 
&= \lim_{1+\frac{|x+y| + |x+z|}{1+|x-y| + |x-z|} \to \infty} F(x, y, z; x')
=0.
\end{align}
Indeed, the last two limits in \eqref{imp-3} follow from \eqref{imp-2}. To show the first one, assume that there exists a sequence $\{(x_k, y_k, z_k; x'_k)\}_{k \in \N}$ such that $|x_k-x'_k| \le \max\{|x_k-y_k|, |x_k-z_k|\}/2$ with $\lim_{k \to \infty} |x_k-x'_k|=0$, but
\begin{align}\label{infF}
\inf_k F(x_k, y_k, z_k; x'_k)>0.
\end{align}
If $\liminf_{k \to \infty} \frac{|x_k-x'_k|}{|x_k-y_k| + |x_k-z_k|}=0$, then the boundedness of $F$ gives $\liminf_{k \to \infty} F(x_k, y_k, z_k; x'_k)=0$, which contradicts \eqref{infF}. If $\liminf_{k \to \infty} \frac{|x_k-x'_k|}{|x_k-y_k| + |x_k-z_k| } > 0$, then there exist a constant $C_0>0$ and a subsequence (which we relabel) $\{(x_k, y_k, z_k, x'_k)\}_{k \in \N}$ such that $|x_k-y_k| + |x_k-z_k| \le C_0 |x_k-x'_k|$. By the monotonicity of  $F_1$ and that $\lim_{k \to \infty} |x_k-x'_k|=0$, we conclude 
\begin{align*}
0 &\le F(x_k, y_k, z_k; x'_k)
\le F(x_k, y_k, z_k)
\\
&\lesssim F_1(|x_k-y_k| + |x_k-z_k|) 
\le F_1(C_0 |x_k-x'_k|) \to 0,
\end{align*}
as $k \to \infty$. Thus, $\lim_{k \to \infty} F(x_k, y_k, z_k; x'_k) =0$, which contradicts \eqref{infF}. Consequently, the first limit in \eqref{imp-3} holds.

Define
\begin{align*}
F'_1(t) &:= \sup_{|x-x'| \le t} F(x, y, z; x')^{\frac13},
\\ 
F'_2(t) &:= \sup_{|x-y| + |x-z| \ge t} F(x, y, z; x')^{\frac13},
\\
F'_3(t) &:= \sup_{1+\frac{|x+y| + |x+z|}{1+|x-y| + |x-z|} \ge t} F(x, y, z; x')^{\frac13}.
\end{align*}
It is easy to check that $F'_1$ is bounded and monotone increasing, both $F'_2$ and $F'_3$ are bounded and monotone decreasing. Besides, \eqref{imp-3} gives that $(F'_1, F'_2, F'_3) \in \F$, 
\begin{align}\label{FX-3}
F(x, y, z; x')^{\frac13} \le F'_1(|x-x'|), \qquad 
F(x, y, z; x')^{\frac13} \le F'_2(|x-y| + |x-z|), 
\end{align}
and 
\begin{align}\label{FX-4}
F(x, y, z; x')^{\frac13}
\le F'_3 \bigg(1 + \frac{|x+y| + |x+z|}{1+|x-y| + |x-z|}\bigg). 
\end{align}
Therefore, \eqref{FX-3} and \eqref{FX-4} imply 
\begin{align*}
\frac{F(x, y, z) \, |x-x'|^{\delta}}{(|x-y| + |x-z|)^{2n + \delta}}
&= \frac{F(x, y, z; x') \, |x-x'|^{\delta'}}{(|x-y| + |x-z|)^{2n + \delta'}}
\le \frac{F'(x, y, z) \, |x-x'|^{\delta'}}{(|x-y| + |x-z|)^{2n + \delta'}},
\end{align*}
which coincides with part \eqref{Holder-imp}. The proof is complete.
\end{proof}

\begin{lemma}\label{lem:PP}
The following statements hold: 
\begin{list}{\rm (\theenumi)}{\usecounter{enumi}\leftmargin=1.2cm \labelwidth=1cm \itemsep=0.2cm \topsep=.2cm \renewcommand{\theenumi}{\arabic{enumi}}}

\item\label{list:P1} If the estimates in Definition $\ref{def:CCZK}$ are satisfied for different triples $(F_1^1, F_2^1, F_3^1) \in \F$ and $(F_1^2, F_2^2, F_3^2) \in \F$ respectively, then they also hold for a new triple $(F_1, F_2, F_3) \in \F$.

\item\label{list:P2} If the weak compactness property in Definition $\ref{def:WCP}$ is satisfied for different functions $F_1 \in \F_0$ and $F_2 \in \F_0$ respectively, then it also holds for a new function $F \in \F_0$. 

\item\label{list:P3} For each triple $(F_1, F_2, F_3) \in \F$ in Definitions $\ref{def:CCZK}$, we may assume that $F_1$ is monotone increasing while $F_2$ and $F_3$ are monotone decreasing. 

\item\label{list:P4} Since any dilation of functions in $\F_0$ and $\F$ still belongs to the original space, we may omit all universal constants appearing in the argument involving these functions.

\item\label{list:P5} In light of Lemma $\ref{lem:improve}$, we may use alternative estimates for kernels in Definitions $\ref{def:CCZK}$. 
\end{list}

\end{lemma}

\begin{proof}
Let us first show item \eqref{list:P1}. Given $(F_1^1, F_2^1, F_3^1) \in \F$ and $(F_1^2, F_2^2, F_3^2) \in \F$ in Definition \ref{def:CCZK}, if we define
\begin{align*}
F_j := \max\{F_j^1, F_j^2\}, \quad j=1, 2, 3, 
\end{align*}
then the size and H\"{o}lder conditions hold for $(F_1, F_2, F_3) \in \F$. Item \eqref{list:P2} can be shown in the same way. To show item \eqref{list:P3}, given $(F_1, F_2, F_3) \in \F$ in Definitions \ref{def:CCZK}, we define
\begin{align*}
F_1^*(t) := \sup_{0 < s \le t} F_1(s), \quad
F_2^*(t) := \sup_{s \ge t} F_2(s), \quad\text{and}\quad
F_3^*(t) := \sup_{s \ge t} F_3(s).
\end{align*}
Then it is easy to check that all estimates in Definitions \ref{def:CCZK} are satisfies for $(F^*_1, F^*_2, F^*_3) \in \F$, $F^*_1$ is monotone increasing, $F^*_2$ and $F^*_3$ are monotone decreasing. 
Items \eqref{list:P4} and \eqref{list:P5} are direct. 
\end{proof}

\section{Compact bilinear dyadic representation}\label{sec:repre} 
The goal of this section is to prove ${\rm (a) \Longrightarrow (b)}$ in Theorem \ref{thm:cpt}, for which we will combine the techniques in \cite{CYY} and \cite{LMOV}. To proceed to the proof, let us keep Lemma \ref{lem:PP} in mind, namely, we always assume that all functions in $\mathscr{F}_0$ and $\mathscr{F}$ appearing in the hypotheses \eqref{H1}--\eqref{H3} satisfy those properties there. In view of Lemma \ref{lem:improve}, we will use the alternative smoothness estimates for the bilinear Calder\'{o}n--Zygmund kernels. This leads to the appearance of a new exponent $\delta' \in (0, \delta)$, where $\delta$ is the exponent in Definition \ref{def:CCZK}. Then the goodness parameter is chosen as 
\begin{align*}
\gamma = \frac{\delta'}{2(2n + \delta')} 
\quad \text{so that } \quad
\delta' - (2n + \delta') \gamma = \delta'/2.
\end{align*}
Additionally, by Theorem \ref{thm:WBP},
\begin{align}\label{eq:TL44}
T: L^4(\Rn) \times L^4(\Rn) \to L^2(\Rn) \quad \text{is bounded}.
\end{align}

\subsection{Initial reductions}  
Fix $f_1, f_2, f_3 \in L_c^{\infty}(\Rn)$. Take a large integer $N \ge 2$ so that $\supp f_j \subset B(0, 2^{N-2})$, $j=1, 2, 3$. For every $M \in \N$, there holds
\begin{align}\label{eq:emw}
E_{2^{-M}}^{\w} f
:= \sum_{\substack{I \in \D_{\w} \\ \ell(I) = 2^{-M}}} 
\langle f \rangle_I \, \mathbf{1}_I
= \sum_{\substack{I \in \D_{\w} \\ 2^{-M} < \ell(I) \le 2^N}} \Delta_I^N f,
\end{align}
where $\Delta_I^N f := \Delta_I f$ if $\ell(I) < 2^N$, and $\Delta_I^N f := \Delta_I f + \langle f \rangle_I \mathbf{1}_I$ if $\ell(I) = 2^N$. By \eqref{eq:TL44} and Lebesgue dominated convergence theorem,
\begin{align*}
\langle T(f_1, f_2), f_3 \rangle 
= \lim_{M \to \infty} \mathbb{E}_{\w}
\langle T(E_{2^{-M}}^{\w} f_1, E_{2^{-M}}^{\w} f_2), E_{2^{-M}}^{\w} f_3 \rangle
=: \lim_{M \to \infty} \mathscr{S}_M.
\end{align*}
We use \eqref{eq:emw} to split 
\begin{align*}
\mathscr{S}_M
= \mathbb{E}_{\w} \sum_{\substack{I_1, I_2, I_3 \in \D_{\w} \\ 2^{-M} < \ell(I_i) \le 2^N}} 
\langle T(\Delta_{I_1}^N f_1, \Delta_{I_2}^N f_2), \Delta_{I_3}^N f_3 \rangle
= \sum_{j=1}^3 \mathscr{S}_{M, j},
\end{align*}
where
\begin{align*}
\mathscr{S}_{M, 1}
& := \mathbb{E}_{\w} 
\sum_{\substack{I_3 \in \D_{\w} \\ 2^{-M} < \ell(I_3) \le 2^N}} 
\sum_{I_1 \in \D_{\w} \atop \ell(I_3) \le \ell(I_1) \le 2^N} 
\sum_{I_2 \in \D_{\w} \atop \ell(I_3) \le \ell(I_2) \le 2^N} 
\langle T(\Delta_{I_1}^N f_1, \Delta_{I_2}^N f_2), \Delta_{I_3}^N f_3 \rangle,
\\
\mathscr{S}_{M, 2}
& := \mathbb{E}_{\w} 
\sum_{\substack{I_2 \in \D_{\w} \\ 2^{-M} < \ell(I_2) < 2^N}} 
\sum_{I_1 \in \D_{\w} \atop \ell(I_2) \le \ell(I_1) \le 2^N} 
\sum_{I_3 \in \D_{\w} \atop \ell(I_2) < \ell(I_3) \le 2^N}
\langle T(\Delta_{I_1}^N f_1, \Delta_{I_2}^N f_2), \Delta_{I_3}^N f_3 \rangle,
\\
\mathscr{S}_{M, 3}
& := \mathbb{E}_{\w} 
\sum_{\substack{I_1 \in \D_{\w} \\ 2^{-M} < \ell(I_1) < 2^N}} 
\sum_{I_2 \in \D_{\w} \atop \ell(I_1) < \ell(I_2) \le 2^N} 
\sum_{I_3 \in \D_{\w} \atop \ell(I_1) < \ell(I_3) \le 2^N} 
\langle T(\Delta_{I_1}^N f_1, \Delta_{I_2}^N f_2), \Delta_{I_3}^N f_3 \rangle.
\end{align*}
Throughout this paper, we only focus on $\mathscr{S}_{M, 1}$ since other two terms can be handled with minor modifications. We rewrite 
\begin{align*}
\mathscr{S}_{M, 1} 
= \mathbb{E}_{\w} \sum_{\substack{I_3 \in \D_0 \\ 2^{-M} < \ell(I_3) \le 2^N}}
\big\langle T(E_{\ell(I_3)/2}^{\w} f_1, E_{\ell(I_3)/2}^{\w} f_2), \Delta_{I_3+\w}^N f_3 \big\rangle. 
\end{align*}
Note that $\mathbf{1}_{\text{good}}(I_3 + \w)$ depends on $\w_j$ for $2^{-j} \ge \ell(I_3)$, while both $E_{\ell(I_3)/2}^{\w} f_1$ and $E_{\ell(I_3)/2}^{\w} f_2$ depend on $\w_j$ for $2^{-j} <  \ell(I_3)/2 < \ell(I_3)$, and $\Delta_{I_3+\w} f_3$ depends on $\w_j$ for $2^{-j} < \ell(I_3)$. Then using independence, we arrive at 
\begin{align*}
\mathscr{S}_{M, 1} 
&= \frac{1}{\pi_{\text{good}}} 
\sum_{\substack{I_3 \in \D_0 \\ 2^{-M} < \ell(I_3) \le 2^N}}
\mathbb{E}_{\w} \big[ \mathbf{1}_{\text{good}}(I_3+\w)\big] 
\mathbb{E}_{\w} \big[\big\langle T(E_{\ell(I_3)/2}^{\w} f_1, E_{\ell(I_3)/2}^{\w} f_2), 
\Delta_{I_3+\w}^N f_3 \big\rangle \big]
\\
&= \frac{1}{\pi_{\text{good}}} 
\sum_{\substack{I_3 \in \D_0 \\ 2^{-M} < \ell(I_3) \le 2^N}} 
\mathbb{E}_{\w} \big[ \mathbf{1}_{\text{good}}(I_3+\w)  
\big\langle T(E_{\ell(I_3)/2}^{\w} f_1, E_{\ell(I_3)/2}^{\w} f_2), 
\Delta_{I_3+\w}^N f_3 \big\rangle \big]
\\
&= \frac{1}{\pi_{\text{good}}} \mathbb{E}_{\w} 
\sum_{\substack{I_3 \in \D_{\w, \text{good}} \\ 2^{-M} < \ell(I_3) \le 2^N}} 
\big\langle T(E_{\ell(I_3)/2}^{\w} f_1, E_{\ell(I_3)/2}^{\w} f_2), \Delta_{I_3}^N f_3 \big\rangle
\\
&= \frac{1}{\pi_{\text{good}}} \mathbb{E}_{\w} 
\sum_{\substack{I_3 \in \D_{\w, \text{good}} \\ 2^{-M} < \ell(I_3) \le 2^N}} 
\sum_{I_1 \in \D_{\w} \atop \ell(I_3) \le \ell(I_1) \le 2^N} 
\sum_{I_2 \in \D_{\w} \atop \ell(I_3) \le \ell(I_2) \le 2^N} 
\langle T(\Delta_{I_1}^N f_1, \Delta_{I_2}^N f_2), \Delta_{I_3}^N f_3 \rangle 
\\
&=: \frac{1}{\pi_{\text{good}}} \mathbb{E}_{\w} \mathscr{S}_{M, 1}(\w). 
\end{align*}

From now on, fix $\w$ and $M$, let $\D = \D_{\w}$, $E_{2^k} = E_{2^k}^{\w}$, and $\mathscr{S} = \mathscr{S}_{M, 1}(\w)$. By \eqref{ddf-2}, we have 
\begin{align*}
\mathscr{S}_1
& = \sum_{\substack{I_3 \in \D_{\text{good}} \\ 2^{-M} < \ell(I_3) \le 2^N}} 
\sum_{I_1 \in \D \atop \ell(I_3) \le \ell(I_1) \le 2^N} 
\sum_{I_2 \in \D \atop \ell(I_1) \le \ell(I_2) \le 2^N} 
\langle T(\Delta_{I_1}^N f_1, \Delta_{I_2}^N f_2), \Delta_{I_3}^N f_3 \rangle 
\\
&\quad + \sum_{\substack{I_3 \in \D_{\text{good}} \\ 2^{-M} < \ell(I_3) \le 2^N}} 
\sum_{I_1 \in \D \atop \ell(I_3) \le \ell(I_2) < 2^N} 
\sum_{I_2 \in \D \atop \ell(I_2) < \ell(I_1) \le 2^N} 
\langle T(\Delta_{I_1}^N f_1, \Delta_{I_2}^N f_2), \Delta_{I_3}^N f_3 \rangle 
\\
& = \sum_{\substack{I_3 \in \D_{\text{good}} \\ 2^{-M} < \ell(I_3) \le 2^N}} 
\sum_{\substack{I_1, I_2 \in \D \\ \ell(I_3) \le \ell(I_1) = 2 \ell(I_2) \le 2^N}}  
\langle T(\Delta_{I_1}^N f_1, \langle f_2 \rangle_{I_2} \mathbf{1}_{I_2}), \Delta_{I_3}^N f_3 \rangle
\\ 
&\quad + \sum_{\substack{I_3 \in \D_{\text{good}} \\ 2^{-M} < \ell(I_3) \le 2^N}} 
\sum_{I_1, I_2 \in \D \atop \ell(I_3) \le \ell(I_1) = \ell(I_2) < 2^N} 
\langle T(\langle f_1 \rangle_{I_1} \mathbf{1}_{I_1}, \Delta_{I_2}^N f_2), \Delta_{I_3}^N f_3 \rangle
\\ 
&=: \mathscr{S}_{1, 1} + \mathscr{S}_{1, 2}. 
\end{align*}
By symmetry, we will be mainly concerned with the first term $\mathscr{S}_{1, 1}$:  
\begin{align}\label{S11} 
\mathscr{S}_{1, 1}
= \mathscr{S}_{1, 1}^1 + \mathscr{S}_{1, 1}^2 + \mathscr{S}_{1, 1}^3, 
\end{align}
where 
\begin{align*} 
\mathscr{S}_{1, 1}^1  
& := \sum_{\substack{I_3 \in \D_{\text{good}} \\ 2^{-M} < \ell(I_3) \le 2^N}} 
\sum_{\substack{I_1, I_2 \in \D \\ \ell(I_3) \le \ell(I_1) = 2 \ell(I_2) \le 2^N \\ 
\max\{\d(I_3, I_1), \, \d(I_3, I_2)\} > \ell(I_3)^{\gamma} \ell(I_2)^{1-\gamma}}} 
\langle T(\Delta_{I_1}^N f_1, \langle f_2 \rangle_{I_2} \mathbf{1}_{I_2}), \Delta_{I_3}^N f_3 \rangle, 
\\
\mathscr{S}_{1, 1}^2  
& := \sum_{\substack{I_3 \in \D_{\text{good}} \\ 2^{-M} < \ell(I_3) \le 2^N}} 
\sum_{\substack{I_1, I_2 \in \D \\ \ell(I_3) \le \ell(I_1) = 2 \ell(I_2) \le 2^N \\ 
\max\{\d(I_3, I_1), \, \d(I_3, I_2)\} \le \ell(I_3)^{\gamma} \ell(I_2)^{1-\gamma} \\ 
I_3 \cap I_1 = \varnothing \text{ or } I_3 \cap I_2 = \varnothing \text{ or } I_3 = I_1}} 
\langle T(\Delta_{I_1}^N f_1, \langle f_2 \rangle_{I_2} \mathbf{1}_{I_2}), \Delta_{I_3}^N f_3 \rangle, 
\\
\mathscr{S}_{1, 1}^3  
& := \sum_{\substack{I_3 \in \D_{\text{good}} \\ 2^{-M} < \ell(I_3) \le 2^N}} 
\sum_{\substack{I_1, I_2 \in \D \\ \ell(I_1) = 2 \ell(I_2) \le 2^N \\ I_3 \subset I_2 \subset I_1}} 
\langle T(\Delta_{I_1}^N f_1, \langle f_2 \rangle_{I_2} \mathbf{1}_{I_2}), \Delta_{I_3}^N f_3 \rangle.  
\end{align*}
These three parts are called separated, adjacent, and nested respectively. Recall that we use the $\ell^{\infty}$ metric.

Henceforth, given $(F_1, F_2, F_3) \in \mathscr{F}$, we define  
\begin{align}\label{eq:FF23}
\widetilde{F}_2(t) 
:= \sum_{k=0}^{\infty} 2^{-k \theta} F_2(2^{-k}t)
\quad\text{ and }\quad 
\widetilde{F}_3(I) 
:= \sum_{k=0}^{\infty} 2^{-k \theta} F_3 (\rd(2^k I, \I)),
\end{align} 
where the parameter $\theta \in (0, 1)$ is harmless and small enough. By homogeneity, we may assume that 
\begin{align*}
\|F\|_{L^{\infty}} \le 2^{-(r+1)n}, \quad 
\|F_i\|_{L^{\infty}} \le 1, 
\quad\text{and}\quad 
\|\widetilde{F}_j\|_{L^{\infty}} \le 1, 
\end{align*}
for all $i=1, 2, 3$ and $j=2, 3$, where $F \in \mathscr{F}_0$ comes from Definition \ref{def:WCP}. Denote $\Lambda := \{0, 1\}^n \setminus \{0\}^n$ and
\begin{align*}
\Lambda_I^N := 
\begin{cases}
\{0, 1\}^n \setminus \{0\}^n, & \ell(I) < 2^N, 
\\
\{0, 1\}^n, & \ell(I) = 2^N,
\end{cases}
\quad \text{so that} \quad
\Delta_I^N f = \sum_{\eta \in \Lambda_I^N} 
\langle f, h_I^{\eta} \rangle h_I^{\eta}. 
\end{align*}

\subsection{Separated part}
We begin with the term $\mathscr{S}_{1, 1}^1$. 
Note that $\supp f_j \subset B(0, 2^{N-2})$ and $\supp f_j \cap I_j \ne \varnothing$, $j=1, 2, 3$, which implies $\max\{\d(I_1, I_2), \d(I_1, I_3), \d(I_2, I_3)\} \le 2^{N-1}$. If $\ell(I_3) = 2^N$, then $\ell(I_1) = 2^N$, $\ell(I_2) = 2^{N-1}$, and 
\begin{align*}
\max\{\d(I_3, I_1), \d(I_3, I_2)\} \le 2^{N-1}
< 2^{\gamma-1} 2^N
= \ell(I_3)^{\gamma} \ell(I_2)^{1-\gamma}.
\end{align*}
This means $\ell(I_3) < 2^N$, $\Delta_{I_3}^N = \Delta_{I_3}$, and
\begin{align*}
\langle T(\Delta_{I_1}^N f_1, \langle f_2 \rangle_{I_2} \mathbf{1}_{I_2}), \Delta_{I_3}^N f_3 \rangle
= \sum_{\eta_1 \in \Lambda_{I_1}^N, \eta_3 \in \Lambda} 
\langle T(h_{I_1}^{\eta_1}, h_{I_2}^0), h_{I_3}^{\eta_3} \rangle
\langle f_1, h_{I_1}^{\eta_1} \rangle
\langle f_2, h_{I_2}^0 \rangle
\langle f_3, h_{I_3}^{\eta_3} \rangle.
\end{align*}

In the current scenario, by \cite[Lemma 4.1]{LMOV}, there exists a minimal cube $Q := I_1 \vee I_2 \vee I_3 \in \D$ such that 
\begin{align}\label{QSEP}
I_1 \cup I_2 \cup I_3 \subset Q
\quad\text{ and }\quad 
\max\{\d(I_3, I_1), \, \d(I_3, I_2)\} \gtrsim \ell(I_3)^{\gamma} \ell(Q)^{1-\gamma}.
\end{align} 
Observe that 
\begin{align}\label{PP-1}
\max\{\d(I_3, I_1), \, \d(I_3, I_2)\} \ge \ell(I_3).
\end{align}
Indeed, if $\ell(I_3) < \ell(I_1)$, then $\ell(I_3) \le \ell(I_1)/2 = \ell(I_2)$ and $\max\{\d(I_3, I_1), \, \d(I_3, I_2)\} 
> \ell(I_3)^{\gamma} \ell(I_2)^{1-\gamma} \ge \ell(I_3)$. If $\ell(I_3) = \ell(I_1)$, then $\ell(I_3) = \ell(I_1) = 2 \ell(I_2)$ and $\max\{\d(I_3, I_1), \, \d(I_3, I_2)\} > \ell(I_3)^{\gamma} \ell(I_2)^{1-\gamma} = 2^{\gamma} \ell(I_2) > \ell(I_2)$, which indicates $\max\{\d(I_3, I_1), \, \d(I_3, I_2)\} \ge 2 \ell(I_2) = \ell(I_3)$. 

In addition, for all $x \in I_3$, $y \in I_1$, and $z \in I_2$, 
\begin{align}\label{PP-2}
&1 + \frac{|x+y| + |x+z|}{1+ |x-y| + |x-z|}
\ge \frac{2 (|y| + |z|)}{1+ |x-y| + |x-z|}
\\ \nonumber 
&\ge \frac{4 \d(Q, \I)}{1 + 2 \ell(Q)}
\ge \frac{\d(Q, \I)}{\max\{\ell(Q), 1\}}
= \rd(Q, \I).
\end{align}
If we let 
\begin{align*}
F(x, y, z)
:= F_1(|x-c_{I_3}|) F_2(|x-y| + |x-z|) F_3 \bigg(1 + \frac{|x+y| + |x+z|}{1+ |x-y| + |x-z|}\bigg), 
\end{align*}
then, by the monotoncity of $F_1$, $F_2$, and $F_3$, \eqref{PP-1}, and \eqref{PP-2}, 
\begin{align}\label{PP-3}
F(x, y, z)
\le F_1(\ell(I_3)) F_2(\ell(I_3)) F_3(\rd(Q, \I)) 
=: F(I_3, Q). 
\end{align}
In light of \eqref{QSEP}, \eqref{PP-1}, and \eqref{PP-3}, we use the cancellation of $h_{I_3}$ and the H\"{o}lder condition of $K$ to deduce 
\begin{align}\label{TH-1}
|\langle T(h_{I_1}^{\eta_1}, h^0_{I_2}), h_{I_3}^{\eta_3} \rangle|
&\le \int_{\R^{3n}} \frac{F(x, y, z) \, \ell(I_3)^{\delta'}}{(|x-y| + |x-z|)^{2n+\delta'}} 
|h_{I_1}^{\eta_1}(y)| |h^0_{I_2}(z)| |h_{I_3}^{\eta_3}(x)|  \, dx\, dy \, dz 
\\ \nonumber 
&\le \frac{F(I_3, Q) \, \ell(I_3)^{\delta'}}{\max\{\d(I_3, I_1), \, \d(I_3, I_2)\}^{2n+\delta'}} 
|I_1|^{\frac12} |I_2|^{\frac12} |I_3|^{\frac12} 
\\ \nonumber 
&\le C_1 \frac{F(I_3, Q) \, \ell(I_3)^{\delta'}}{[\ell(I_3)^{\gamma} \ell(Q)^{1-\gamma}]^{2n+\delta'}} 
|I_1|^{\frac12} |I_2|^{\frac12} |I_3|^{\frac12} 
\\ \nonumber 
&= C_1 \bigg[ \frac{\ell(I_3)}{\ell(Q)} \bigg]^{\frac{\delta'}{2}} 
F(I_3, Q) \frac{|I_1|^{\frac12} |I_2|^{\frac12} |I_3|^{\frac12}}{|Q|^2}.  
\end{align}

Let $C_0 \in (0, \infty)$ be a universal constant chosen later. Given $\eta_1 \in \{0, 1\}^n$ and $\eta_3 \in \{0, 1\}^n \setminus \{0\}^n$, set 
\begin{align*}
a_{I_1, I_2, I_3, Q}^{\eta_1, \eta_3; N, M} 
= \frac{\langle T(h_{I_1}^{\eta_1}, h^0_{I_2}), h_{I_3}^{\eta_3} \rangle}{C_0 [\ell(I_3)/\ell(Q)]^{\frac{\delta'}{2}}} 
\end{align*}
if $I_1, I_2 \in \D$ and $I_3 \in \D_{\text{good}}$ satisfy $2^{-M} < \ell(I_3) \le 2^N$, $\ell(I_3) \le \ell(I_1) = 2 \ell(I_2) \le 2^N$, $\max\{\d(I_3, I_1), \, \d(I_3, I_2)\} > \ell(I_3)^{\gamma} \ell(I_2)^{1-\gamma}$, $I_1 \vee I_2 \vee I_3 = Q$, and $\eta_1 \in \Lambda_{I_1}^N$, and otherwise set $a_{I_1, I_2, I_3, Q}^{\eta_1, \eta_3; N, M} = 0$. This enables us to rewrite 
\begin{align*}
\mathscr{S}_{1, 1}^1  
&= \sum_{\substack{0 \le i \le k \\ Q \in \D}}  
\sum_{\substack{I_1, I_2 \in \D, I_3 \in \D_{\rm{good}} \\ 2^{-M} < \ell(I_3) \le 2^N \\ 
\max\{\d(I_3, I_1), \, \d(I_3, I_2)\} > \ell(I_3)^{\gamma} \ell(I_2)^{1-\gamma} \\ 
2 \ell(I_2) = \ell(I_1) = 2^{-i} \ell(Q) \le 2^N \\ \ell(I_3) = 2^{-k} \ell(Q), \, I_1 \vee I_2 \vee I_3 = Q}} 
\sum_{\substack{\eta_1 \in \Lambda_{I_1}^N \\ \eta_3 \in \Lambda}}
\langle T(h_{I_1}^{\eta_1}, h^0_{I_2}), h_{I_3}^{\eta_3} \rangle 
\langle f_1, h_{I_1}^{\eta_1} \rangle \langle f_2, h^0_{I_2} \rangle \langle f_3, h_{I_3}^{\eta_3} \rangle
\\
&= C_0 \sum_{0 \le i \le k} 2^{-k \delta'/2}
\sum_{Q \in \D} \sum_{\substack{I_1 \in \D_i(Q) \\ 
I_2 \in \D_{i+1}(Q) \\ I_3 \in \D_k(Q)}} 
\sum_{\substack{\eta_1 \in \{0, 1\}^n \\ \eta_3 \in \Lambda}}
a_{I_1, I_2, I_3, Q}^{\eta_1, \eta_3, N, M}  
\langle f_1, h_{I_1}^{\eta_1} \rangle \langle f_2, h^0_{I_2} \rangle \langle f_3, h_{I_3}^{\eta_3} \rangle
\\
&= C_0 \sum_{k=0}^{\infty} \sum_{i=0}^k 2^{-k \delta'/2} 
\big\langle \mathbf{S}_{\D; N, M}^{i, i+1, k}(f_1, f_2), f_3 \big\rangle,  
\end{align*}
where we have used that for any $C_0 \ge C_1$, the estimate \eqref{TH-1} and the condition $(F_1, F_2, F_3) \in \F$ imply 
\begin{align}\label{aijkq}
\sup_{N, M}
|a_{I_1, I_2, I_3, Q}^{\eta_1, \eta_3; N, M}| 
\le \mathbf{F}(Q) \frac{|I_1|^{\frac12} |I_2|^{\frac12} |I_3|^{\frac12}}{|Q|^2} 
\end{align}
with  
\begin{align}\label{FIJK}
\mathbf{F}(Q) \le 1 
\quad\text{and}\quad 
\lim_{N \to \infty} \mathbf{F}_N 
:= \lim_{N \to \infty} \sup_{\D} \sup_{Q \in \D \setminus \D(N)} \mathbf{F}(Q) = 0. 
\end{align}
In the current scenario, the exact form of $\mathbf{F}(Q)$ is given by 
\begin{align*}
\mathbf{F}(Q) 
:= F_1(\ell(Q)) F_2(2^{-k} \ell(Q)) F_3(\rd(Q, \I)). 
\end{align*}
In fact we do not care about the exact form of $\mathbf{F}(Q)$, but just require \eqref{aijkq} and \eqref{FIJK}.

\subsection{Adjacent part}\label{sec:adj}
In this case, by \cite[Lemma 4.3]{LMOV}, there exists a minimal cube $Q := I_1 \vee I_2 \vee I_3 \in \D$ such that 
\begin{align}\label{QDIA}
I_1 \cup I_2 \cup I_3 \subset Q
\quad\text{ and }\quad \ell(Q) \le 2^r \ell(I_3).
\end{align} 
We would like to prove 
\begin{align}\label{TH-2}
|\langle T(h_{I_1}^{\eta_1}, h^0_{I_2}), h_{I_3}^{\eta_3} \rangle| 
\le C_2 \bigg[\frac{\ell(I_3)}{\ell(Q)}\bigg]^{\frac{\delta'}{2}} F(I_3, Q) 
\frac{|I_1|^{\frac12} |I_2|^{\frac12} |I_3|^{\frac12}}{|Q|^2}, 
\end{align}
where 
\begin{align*}
F(I_3, Q)
:= F_1(\ell(I_3)) \widetilde{F}_2(\ell(I_3)) F_3(\rd(Q, \I)) 
+ \sum_{\substack{0 \le k \le r \\ I_3 \in \D_k(Q)}} \sum_{I_2 \in \ch(I_3)} F(I_2). 
\end{align*}
Recall that $\widetilde{F}_2$ is defined in \eqref{eq:FF23}. 

First, consider the case $I_3 \cap I_1 = \varnothing$. In this case, $\d(I_3, I_1) \le \ell(I_3)^{\gamma} \ell(I_2)^{1-\gamma} \le \ell(I_1)$, and hence, $I_3 \subset 5I_1 \setminus I_1$. Then the size condition gives 
\begin{align}\label{TDH-1}
|\langle T(h_{I_1}^{\eta_1}, h^0_{I_2}), h_{I_3}^{\eta_3} \rangle| 
\le |I_1|^{-\frac12} |I_2|^{-\frac12} |I_3|^{-\frac12}  
\int_{I_2} \int_{I_1} \int_{I_3} 
\frac{F(x, y, z) \, dx \, dy \, dz}{(|x-y| + |x-z|)^{2n}},  
\end{align}
where 
\begin{align*}
F(x, y, z)
:= F_1(|x-y| + |x-z|) F_2(|x-y| + |x-z|) F_3 \bigg(1 + \frac{|x+y| + |x+z|}{1+ |x-y| + |x-z|}\bigg). 
\end{align*}
Let $\beta \in (0, 1)$ be an auxiliary parameter. Note that for any $x \in I_3$, we have $I_2-x \subset \{|z| \le \ell(Q)\}$ and 
\begin{align}\label{TDH-2}
\int_{I_2} \frac{dz}{|x-z|^{n-\beta}} 
= \int_{I_2-x} \frac{dz}{|z|^{n-\beta}} 
\lesssim \int_0^{\ell(Q)} t^{\beta-1} \, dt 
\lesssim \ell(Q)^{\beta}. 
\end{align}
Denote 
\begin{align*}
D_0 & :=  \{(x, y) \in I_1 \times (5I_1 \setminus I_1): |x-y| > \ell(I_1)\}, 
\\
D_k & :=  \{(x, y) \in I_1 \times (5I_1 \setminus I_1): 2^{-k} \ell(I_1) < |x-y| \le 2^{1-k} \ell(I_1)\}, \quad  k \ge 1,
\end{align*}
which satisfies $|D_k| \lesssim \ell(I_1)^{n-1} (2^{1-k} \ell(I_1))^{n+1} \lesssim 2^{-k(n+1)} \ell(I_1)^{2n}$. Thus, for every $0 < \theta < 1- \beta$, 
\begin{align}\label{TDH-3}
&\int_{I_1} \int_{5I_1 \setminus I_1} 
\frac{F_2(|x-y|)}{|x-y|^{n+\beta}} dx \, dy 
\le \sum_{k \ge 0} \iint_{D_k} \frac{F_2(|x-y|)}{|x-y|^{n+\beta}} dx \, dy
\\ \nonumber 
& \lesssim F_2(\ell(I_1)) \ell(I_1)^{-n-\beta} |I_1|^2 
+ \sum_{k \ge 1} 2^{-k(1-\beta)} \ell(I_1)^{n-\beta} F_2(2^{-k} \ell(I_1))
\\ \nonumber
& \le \widetilde{F}_2(\ell(I_1)) |I_1| \, \ell(I_1)^{-\beta}
\le \widetilde{F}_2(\ell(I_3)) |I_1| \, \ell(I_1)^{-\beta},
\end{align}
and for all $(x, y, z) \in I_3 \times I_1 \times I_2$, $|x-y| + |x-z| \le 2 \ell(Q) \le 2^{r+1} \ell(I_3)$ and 
\begin{align}\label{TDH-4}
1 &+ \frac{|x+y| + |x+z|}{1+ |x-y| + |x-z|}
\ge \frac{2 (|y| + |z|)}{1+ |x-y| + |x-z|}
\\ \nonumber 
&\ge \frac{4 \d(Q, \I)}{1 + 2 \ell(Q)}
\ge \frac{\d(Q, \I)}{\max\{\ell(Q), 1\}}
= \rd(Q, \I).
\end{align}
Hence, invoking \eqref{QDIA}--\eqref{TDH-4}, we arrive at 
\begin{align}\label{TD}
|\langle T(h_{I_1}, h^0_{I_2}), h_{I_3}\rangle| 
&\lesssim |I_1|^{-\frac12} |I_2|^{-\frac12} |I_3|^{-\frac12} 
F_1(\ell(I_3)) F_3(\rd(Q, \I))
\\ \nonumber 
&\quad\times  \int_{I_1} \int_{7I_1 \setminus I_1} 
\frac{F_2(|x-y|)}{|x-y|^{n+\beta}} 
\bigg(\int_{I_2} \frac{dz}{|x-z|^{n-\beta}}\bigg) \, dx \, dy 
\\ \nonumber
&\lesssim F_1(\ell(I_3)) \widetilde{F}_2(\ell(I_3)) F_3(\rd(Q, \I)) 
\frac{|I_1|^{\frac12} |I_2|^{\frac12} |I_3|^{\frac12}}{|Q|^2} 
\bigg[\frac{\ell(I_3)}{\ell(Q)}\bigg]^{\frac{\delta'}{2}}. 
\end{align}
In the same way, the above bound holds in the case $I_3 \cap I_2 = \varnothing$. 

Next, we treat the case $I_3 \cap I_2 \neq \varnothing$ and $I_3=I_1$. Obviously, $I_2 \in \ch(I_3)$. Then we rewrite 
\begin{align}\label{TD-1}
\langle T(h_{I_3}^{\eta_1}, h^0_{I_2}), h_{I_3}^{\eta_3} \rangle 
= \sum_{I'_3, I''_3 \in \ch(I_3)} \langle h_{I_3}^{\eta_1} \rangle_{I'_3} 
\langle h_{I_3}^{\eta_3} \rangle_{I''_3} |I_2|^{-\frac12} 
\langle T(\mathbf{1}_{I'_3}, \mathbf{1}_{I_2}), \mathbf{1}_{I''_3}\rangle.  
\end{align}
If $I'_3 = I_2 = I''_3 $, the weak compactness property implies  
\begin{align}\label{TD-2}
|\langle T(\mathbf{1}_{I_2}, \mathbf{1}_{I_2}), \mathbf{1}_{I_2}\rangle|
\le F(I_2) \, |I_2|.  
\end{align}
If $I''_3 \neq I_2$, then $I''_3 \subset 3I_2 \setminus I_2$, which is similar to the case $I_3 \cap I_1 = \varnothing$.  As in \eqref{TD}, there holds 
\begin{align}\label{TD-3}
|\langle T(\mathbf{1}_{I'_3}, \mathbf{1}_{I_2}), \mathbf{1}_{I''_3}\rangle|
&\lesssim F_1(\ell(I''_3)) \widetilde{F}_2(\ell(I'_3)) F_3(\rd(Q, \I)) 
\frac{|I'_3| |I_2| |I''_3|}{|Q|^2} \bigg[\frac{\ell(I''_3)}{\ell(Q)}\bigg]^{\frac{\delta'}{2}}
\\ \nonumber 
&\lesssim F_1(\ell(I_3)) \widetilde{F}_2(\ell(I_3)) F_3(\rd(Q, \I)) 
\frac{|I_1| |I_2| |I_3|}{|Q|^2} \bigg[\frac{\ell(I_3)}{\ell(Q)}\bigg]^{\frac{\delta'}{2}}. 
\end{align}
Analogously, in the case $I'_3 \neq I_2$, one has 
\begin{align}\label{TD-4}
&|\langle T(\mathbf{1}_{I'_3}, \mathbf{1}_{I_2}), \mathbf{1}_{I''_3}\rangle|
= |\langle T^{*1}(\mathbf{1}_{I''_3}, \mathbf{1}_{I_2}), \mathbf{1}_{I'_3}\rangle|
\\ \nonumber 
&\lesssim F_1(\ell(I_3)) \widetilde{F}_2(\ell(I_3)) F_3(\rd(Q, \I)) 
\frac{|I_1| |I_2| |I_3|}{|Q|^2} \bigg[\frac{\ell(I_3)}{\ell(Q)}\bigg]^{\frac{\delta'}{2}}. 
\end{align}
Gathering \eqref{TD-1}--\eqref{TD-4}, we obtain \eqref{TH-2}.

In this situation, one has $\ell(I_1) \simeq \ell(I_2) \simeq \ell(I_3) \simeq \ell(Q)$. Letting 
\begin{align*}
\mathbf{F}(Q) 
:= F_1(\ell(Q)) \widetilde{F}_2(\ell(Q)) F_3(\rd(Q, \I)) 
+ \sum_{\substack{0 \le k \le r \\ I_2 \in \D_{k+1}(Q)}} F(I_2), 
\end{align*}
we invoke \eqref{TH-2} to obtain 
\begin{align}\label{TH-22}
|\langle T(h_{I_1}^{\eta_1}, h^0_{I_2}), h_{I_3}^{\eta_3} \rangle| 
\le C_2 \bigg[\frac{\ell(I_3)}{\ell(Q)}\bigg]^{\frac{\delta'}{2}} 
\mathbf{F}(Q) 
\frac{|I_1|^{\frac12} |I_2|^{\frac12} |I_3|^{\frac12}}{|Q|^2}. 
\end{align}
Given $\eta_1, \eta_3 \in \{0, 1\}^n$, set 
\begin{align*}
a_{I_1, I_2, I_3, Q}^{\eta_1, \eta_3; N, M}  
= \frac{\langle T(h_{I_1}^{\eta_1}, h^0_{I_2}), h_{I_3}^{\eta_3} \rangle}{C_0 [\ell(I_3)/\ell(Q)]^{\frac{\delta'}{2}}} 
\end{align*}
if $I_1, I_2 \in \D$ and $I_3 \in \D_{\text{good}}$ satisfy $2^{-M} < \ell(I_3) \le 2^N$, $\ell(I_3) \le \ell(I_1) = 2 \ell(I_2) \le 2^N$, $\max\{\d(I_3, I_1)\, \d(I_3, I_2)\} \le \ell(I_3)^{\gamma} \ell(I_2)^{1-\gamma}$, either $I_3 \cap I_1 = \varnothing$ or $I_3 \cap I_2 = \varnothing$ or $I_3=I_1$, $I_1 \vee I_2 \vee I_3 = Q$, $\eta_1 \in \Lambda_{I_1}^N$, and $\eta_3 \in \Lambda_{I_3}^N$, and otherwise set $a_{I_1, I_2, I_3, Q}^{\eta_1, \eta_3; N, M} = 0$. Then by \eqref{TH-22} and that $F \in \F_0$ and $(F_1, F_2, F_3) \in \F$, we see that \eqref{aijkq} and \eqref{FIJK} hold whenever $C_0 \ge \max\{C_1, C_2\}$. Thus, in light of  \eqref{TH-22}, we have 
\begin{align*}
\mathscr{S}_{1, 1}^2  
&= \sum_{\substack{0 \le i \le k \le r \\ Q \in \D}} 
\sum_{\substack{I_1, I_2 \in \D, \, I_3 \in \D_{\text{good}} \\ 
2^{-M} < \ell(I_3) \le \ell(I_1) \le 2^N \\ 
\max\{\d(I_3, I_1), \, \d(I_3, I_2)\} \le \ell(I_3)^{\gamma} \ell(I_2)^{1-\gamma} \\ 
I_3 \cap I_1 = \varnothing \text{ or } I_3 \cap I_2 = \varnothing \text{ or } I_3 = I_1\\ 
2 \ell(I_2) = \ell(I_1) = 2^{-i} \ell(Q) \\ \ell(I_3) = 2^{-k} \ell(Q), \, I_1 \vee I_2 \vee I_3 = Q}} 
\sum_{\substack{\eta_1 \in \Lambda_{I_1}^N \\ \eta_3 \in \Lambda_{I_3}^N}}
\langle T(h_{I_1}^{\eta_1}, h^0_{I_2}), h_{I_3}^{\eta_3} \rangle 
\langle f_1, h_{I_1}^{\eta_1} \rangle \langle f_2, h^0_{I_2} \rangle \langle f_3, h_{I_3}^{\eta_3} \rangle
\\
&= C_0 \sum_{0 \le i \le k \le r}  2^{-k \delta'/2} 
\sum_{Q \in \D} \sum_{\substack{I_1 \in \D_i(Q) \\ 
I_2 \in \D_{i+1}(Q) \\ I_3 \in \D_k(Q)}} 
\sum_{\eta_1, \eta_3 \in \{0, 1\}^n} 
a_{I_1, I_2, I_3, Q}^{\eta_1, \eta_3, N, M} 
\langle f_1, h_{I_1}^{\eta_1} \rangle \langle f_2, h^0_{I_2} \rangle 
\langle f_3, h_{I_3}^{\eta_3} \rangle
\\
&= C_0 \sum_{k=0}^r \sum_{i=0}^k 2^{-k \delta'/2} 
\big\langle \mathbf{S}_{\D; N, M}^{i, i+1, k}(f_1, f_2), f_3 \big\rangle. 
\end{align*}

\subsection{Nested part} 
In this scenario, $\ell(I_3) \le \ell(I_2) \le 2^{N-1}$ and $\Delta_{I_3}^N f_3 = \Delta_{I_3} f_3$. Thus, 
\begin{align*}
\mathscr{S}_{1, 1}^3  
= \sum_{\substack{I_3 \in \D_{\text{good}}, I_2 \supset I_3 \\ 2^{-M} < \ell(I_3), \ell(I_2) < 2^N}} 
\sum_{\substack{\eta_1 \in \Lambda_{I_2^{(1)}}^N \\ \eta_3 \in \Lambda}}
\big\langle T(h_{I_2^{(1)}}^{\eta_1}, \mathbf{1}_{I_2}), h_{I_3}^{\eta_3} \big\rangle 
\langle f_1, h_{I_2^{(1)}} \rangle \langle f_2 \rangle_{I_2} \langle f_3, h_{I_3}^{\eta_3} \rangle.  
\end{align*}
We begin with the decomposition 
\begin{align*}
\langle T(h_{I_2^{(1)}}^{\eta_1}, \mathbf{1}_{I_2}), h_{I_3}^{\eta_3} \rangle 
= \langle h_{I_2^{(1)}}^{\eta_1} \rangle_{I_2} \langle T(1, 1), h_{I_3}^{\eta_3} \rangle 
- \langle h_{I_2^{(1)}}^{\eta_1} \rangle_{I_2} \langle T(1, \mathbf{1}_{I_2^c}), h_{I_3}^{\eta_3} \rangle 
+ \langle T(\phi_{I_2}, \mathbf{1}_{I_2}), h_{I_3}^{\eta_3} \rangle, 
\end{align*}
where $\phi_{I_2}^{\eta_1} := \mathbf{1}_{I_2^c} \big(h_{I_2^{(1)}}^{\eta_1} - \langle h_{I_2^{(1)}}^{\eta_1} \rangle_{I_2} \big)$. This allows us to write 
\begin{align}\label{SS-1}
\mathscr{S}_{1, 1}^3  
= \mathscr{S}_{1, 1}^{3, 1} - \mathscr{S}_{1, 1}^{3, 2}  + \mathscr{S}_{1, 1}^{3, 3}, 
\end{align}
where 
\begin{align*}
\mathscr{S}_{1, 1}^{3, 1}  
& := \sum_{\substack{I_3 \in \D_{\text{good}}, I_2 \supset I_3 \\ 2^{-M} < \ell(I_3), \ell(I_2) < 2^N}} 
\sum_{\substack{\eta_1 \in \Lambda_{I_2^{(1)}}^N \\ \eta_3 \in \Lambda}} 
\langle h_{I_2^{(1)}}^{\eta_1} \rangle_{I_2} \langle T(1, 1), h_{I_3}^{\eta_3} \rangle 
\langle f_1, h_{I_2^{(1)}}^{\eta_1} \rangle \langle f_2 \rangle_{I_2} \langle f_3, h_{I_3}^{\eta_3} \rangle, 
\\
\mathscr{S}_{1, 1}^{3, 2}  
& := \sum_{\substack{I_3 \in \D_{\text{good}}, I_2 \supset I_3 \\ 2^{-M} < \ell(I_3), \ell(I_2) < 2^N}} 
\sum_{\substack{\eta_1 \in \Lambda_{I_2^{(1)}}^N \\ \eta_3 \in \Lambda}} 
\langle h_{I_2^{(1)}}^{\eta_1} \rangle_{I_2} \langle T(1, \mathbf{1}_{I_2^c}), h_{I_3}^{\eta_3} \rangle 
\langle f_1, h_{I_2^{(1)}}^{\eta_1} \rangle \langle f_2 \rangle_{I_2} \langle f_3, h_{I_3}^{\eta_3} \rangle, 
\\ 
\mathscr{S}_{1, 1}^{3, 3}  
& := \sum_{\substack{I_3 \in \D_{\text{good}}, I_2 \supset I_3 \\ 2^{-M} < \ell(I_3), \ell(I_2) < 2^N}} 
\sum_{\substack{\eta_1 \in \Lambda_{I_2^{(1)}}^N \\ \eta_3 \in \Lambda}} 
\langle T(\phi_{I_2}^{\eta_1}, \mathbf{1}_{I_2}), h_{I_3} \rangle  
\langle f_1, h_{I_2^{(1)}}^{\eta_1} \rangle \langle f_2 \rangle_{I_2} \langle f_3, h_{I_3}^{\eta_3} \rangle.   
\end{align*}
Note that
\begin{align*}
\mathscr{S}_{1, 1}^{3, 1}  
= \sum_{\substack{I_3 \in \D_{\text{good}}, \\ 2^{-M} < \ell(I_3), \ell(I_2) < 2^N}} 
\sum_{\eta_3 \in \Lambda}
\langle T(1, 1), h_{I_3}^{\eta_3} \rangle \langle f_3, h_{I_3}^{\eta_3} \rangle
\sum_{\substack{I_2 \supset I_3 \\ \ell(I_2) < 2^N}}
\langle \Delta_{I_2^{(1)}}^N f_1 \rangle_{I_2} \langle f_2 \rangle_{I_2}, 
\end{align*}
where
\begin{align*}
\langle \Delta_{I_2^{(1)}}^N f_1 \rangle_{I_2}
= \begin{cases}
\langle f_1 \rangle_{I_2} - \langle f_1 \rangle_{I_2^{(1)}}, & \ell(I_2) < 2^{N-1},
\\
\langle f_1 \rangle_{I_2}, & \ell(I_2) = 2^{N-1}.
\end{cases}
\end{align*}
Likewise, for the term $\mathscr{S}_{1, 2}$, performing the decomposition as in \eqref{S11}, we obtain 
\begin{align*}
\mathscr{S}_{1, 2}^3  
&:= \sum_{\substack{I_3 \in \D_{\text{good}} \\ 2^{-M} < \ell(I_3) < 2^N}} 
\sum_{\substack{I_2 \in \D, I_2 \supsetneq I_3 \\ \ell(I_2) < 2^N }} 
\langle T(\langle f_1 \rangle_{I_2} \mathbf{1}_{I_2}, \Delta_{I_2} f_2), \Delta_{I_3} f_3\rangle 
\\
&= \sum_{\substack{I_3 \in \D_{\text{good}} \\ 2^{-M} < \ell(I_3) < 2^N}} 
\sum_{\substack{I_2 \in \D, I_2 \supset I_3 \\ \ell(I_2) < 2^{N-1} }} 
\langle T(\langle f_1 \rangle_{I_2^{(1)}} \mathbf{1}_{I_2^{(1)}}, \Delta_{I_2^{(1)}} f_2), \Delta_{I_3} f_3\rangle .
\end{align*}
Then splitting $\mathscr{S}_{1, 2}^3$ as in \eqref{SS-1} produces the corresponding paraproduct term 
\begin{align*}
\mathscr{S}_{1, 2}^{3, 1}  
&:= \sum_{\substack{I_3 \in \D_{\text{good}} \\ 2^{-M} < \ell(I_3) < 2^N }} 
\sum_{\eta_3 \in \Lambda}
\langle T(1, 1), h_{I_3}^{\eta_3} \rangle \langle f_3, h_{I_3}^{\eta_3} \rangle
\sum_{\substack{I_2 \in \D, I_2 \supset I_3 \\ \ell(I_2) < 2^{N-1}}}  
\langle f_1 \rangle_{I_2^{(1)}} (\langle f_2 \rangle_{I_2} - \langle f_2 \rangle_{I_2^{(1)}}).  
\end{align*}
Hence, 
\begin{align}\label{PSS}
\mathscr{S}_{1, 1}^{3, 1} + \mathscr{S}_{1, 2}^{3, 1} 
= \sum_{\substack{I_3 \in \D_{\text{good}}, \eta_3 \in \Lambda \\ 2^{-M} < \ell(I_3) < 2^N }} 
\langle T(1, 1), h_{I_3}^{\eta_3} \rangle 
\langle f_1 \rangle_{I_3} \langle f_2 \rangle_{I_3} \langle f_3, h_{I_3}^{\eta_3} \rangle
= C_0 \big\langle \Pi_{\b; N, M}(f_1, f_2), f_3 \big\rangle, 
\end{align}
where we set 
\begin{align*}
b_I^{\eta_3} = C_0^{-1} \langle T(1, 1), h_I^{\eta_3} \rangle \, 
\mathbf{1}_{\{I \in \D_{\text{good}}\}}. 
\end{align*}
Choose $C_0 \ge \max\{C_1, C_2, \|T(1, 1)\|_{\BMO}\}$. Then the fact $T(1, 1) \in \CMO(\Rn)$ implies 
\begin{align*}
\sup_{\D} \|\{b_I^{\eta}\}_{\D}\|_{\BMO_{\D}} \le 1 
\quad \text{ and } \quad 
\lim_{N \to \infty} \sup_{\D} \sup_{Q \in \D} 
\frac{1}{|Q|} \sum_{I \notin \D(N): \, I \subset Q} |b_I^{\eta}|^2 
= 0. 
\end{align*}

To analyze $\mathscr{S}_{1, 1}^{3, 2}$, we claim that 
\begin{align}\label{TH-3}
|\langle h_{I_2^{(1)}} \rangle_{I_2}| 
|\langle T(1, \mathbf{1}_{I_2^c}), h_{I_3}^{\eta_3} \rangle| 
\le C_3 \bigg[\frac{\ell(I_3)}{\ell(I_2)}\bigg]^{\frac{\delta'}{2}} 
F(I_2, I_3) |I_2|^{-\frac12} |I_3|^{\frac12}, 
\end{align}
where 
\begin{align}\label{eq:FFF}
F(I_2, I_3) :=  F_1(\ell(I_3)) \widetilde{F}_2(\ell(I_3)) \widetilde{F}_3(I_2).
\end{align}
Recall that $\widetilde{F}_2$ and $\widetilde{F}_3$ are defined in \eqref{eq:FF23}. Indeed, if $\ell(I_2) > 2^r \ell(I_3)$, then for all $x \in I_3$ and $z \in I_2^c$, the goodness of $I_3$ yields
\begin{align}\label{FZX-1}
|x-z| \ge \d(I_3, I_2^c) 
> \ell(I_3)^{\gamma} \ell(I_2)^{1-\gamma} 
\ge \ell(I_3)^{\frac12} \ell(I_2)^{\frac12} 
\ge \ell(I_3) \ge 2|x-c_{I_3}|, 
\end{align}
which together with the H\"{o}lder condition of $K$ implies  
\begin{align}\label{FZX-2}
|\langle T(\mathbf{1}_{3I_3}, \mathbf{1}_{I_2^c}), h_{I_3}^{\eta_3} \rangle|
\le  |I_3|^{-\frac12} \int_{I_2^c} \int_{3I_3} \int_{I_3} 
\frac{F(x, y, z) \ell(I_3)^{\delta'}}{(|x-y| + |x-z|)^{2n+\delta'}} dx \, dy \, dz, 
\end{align}
where 
\begin{align}\label{FXC}
F(x, y, z)
:= F_1(|x-c_{I_3}|) F_2(|x-y| + |x-z|) F_3 \bigg(1 + \frac{|x+y| + |x+z|}{1+ |x-y| + |x-z|}\bigg). 
\end{align}
The inequality \eqref{TDH-2} gives 
\begin{align}\label{FZX-3}
\int_{3I_3} \frac{dy}{|x-y|^{n-\beta}} 
\lesssim \ell(I_3)^{\beta}, \qquad x \in I_3, 
\end{align}
and for any $\alpha>0$ and $\ell := \ell(I_3)^{\frac12} \ell(I_2)^{\frac12}$, 
\begin{align}\label{FZX-4}
\psi(\alpha) 
& := \int_{I_3} \int_{|x-z| \ge \ell} F_3\bigg( \frac{4|x|}{1+|x-z|}\bigg) \frac{dz \, dx}{|x-z|^{n+\alpha}} 
\\ \nonumber 
&\le \sum_{k \ge 0} \int_{I_3} \int_{2^k \ell \le |x-z| < 2^{k+1} \ell} 
(2^k \ell)^{-n-\alpha} F_3(\rd(2^k I_2, \I)) \, dz \, dx 
\\ \nonumber 
&\lesssim |I_3| \big[\ell(I_3)^{\frac12} \ell(I_2)^{\frac12}\big]^{-\alpha} 
\sum_{k \ge 0} 2^{-k \alpha} F_3(\rd(2^k I_2, \I))
\end{align}
provided that for all $x \in I_3$ and $2^k \ell \le |x-z| < 2^{k+1} \ell$, 
\begin{align*}
\frac{4|x|}{1+|x-z|}
\ge \frac{4 \d(2^k I_2, 0)}{1+2^{k+1} \ell(I_2)} 
\ge \frac{\d(2^k I_2, 0)}{\max\{2^k \ell(I_2), 1\}}
\ge \rd(2^k I_2, \I). 
\end{align*}
Then it follows from \eqref{FZX-1}--\eqref{FZX-4} that 
\begin{align}\label{FZX-5}
|\langle T(\mathbf{1}_{3I_3}, \mathbf{1}_{I_2^c}), h_{I_3} \rangle|
&\lesssim  F_1(\ell(I_3)) F_2(\ell(I_3)) |I_3|^{-\frac12} 
\ell(I_3)^{\delta' + \beta} \psi(\delta' + \beta)
\\ \nonumber 
&\lesssim F_1(\ell(I_3)) F_2(\ell(I_3)) \widetilde{F}_3(I_2) 
|I_3|^{\frac12} \bigg[\frac{\ell(I_3)}{\ell(I_2)}\bigg]^{\frac{\delta'}{2}}. 
\end{align}
On the other hand, by \eqref{FZX-1}, we see that for any $x \in I_3$, $y \in (3I_3)^c$, and $z \in I_2^c$, 
\begin{align*}
|x-y| + |x-z| 
&\ge \max\{|x-y|, |x-z|\} 
\ge |x-z| 
\\ 
&\ge \ell(I_3)^{\frac12} \ell(I_2)^{\frac12} =: \ell 
\ge \ell(I_3)  
\ge 2 |x-c_{I_3}|,
\end{align*}
which along with the H\"{o}lder condition of $K$ leads to  
\begin{align*}
|\langle T(\mathbf{1}_{(3I_3)^c}, \mathbf{1}_{I_2^c}), h_{I_3}^{\eta_3} \rangle| 
& \le |I_3|^{-\frac12} \int_{I_3} \bigg[ \int_{I_2^c} \int_{(3I_3)^c} 
\frac{\ell(I_3)^{\delta'} F(x, y, z)}{(|x-y| + |x-z|)^{2n+\delta'}} \, dy \, dz \bigg] \, dx,  
\end{align*}
where $F(x, y, z)$ is defined in \eqref{FXC}. Given $k \ge 0$ and $x \in I_3 \subset I_2$, letting 
\begin{align*}
R_k(x) := \{(y, z) \in \R^{2n}: 2^k \ell \le |x-y| + |x-z| < 2^{k+1} \ell\},
\end{align*} 
we have for all $(y, z) \in R_k(x)$, 
\begin{align*}
1 + \frac{|x+y| + |x+z|}{1+ |x-y| + |x-z|}
\ge \frac{4|x|}{1+ 2^{k+1} \ell(I_2)}
\ge \frac{\d(2^k I_2, 0)}{\max\{2^k \ell(I_2), 1\}}
\ge \rd(2^k I_2, \I).  
\end{align*}
Thus, using the estimates above and the monotonicity of $F_1, F_2, F_3$, we conclude   
\begin{align}\label{FZX-6}
|\langle T(\mathbf{1}_{(3I_3)^c}, \mathbf{1}_{I_2^c}), h_{I_3}^{\eta_3} \rangle|
&\lesssim  F_1(\ell(I_3)) F_2(\ell(I_3)) |I_3|^{-\frac12} \ell(I_3)^{\delta'} 
\\ \nonumber 
&\quad \times \int_{I_3} \bigg[\sum_{k \ge 0} \iint_{R_k(x)} 
\frac{F_3(\rd(2^k I_2, \I))}{(2^k \ell)^{2n+\delta'}} dy\, dz\bigg] \, dx
\\ \nonumber 
&\lesssim F_1(\ell(I_3)) F_2(\ell(I_3)) |I_3|^{\frac12} \ell(I_3)^{\delta'} 
\ell^{-\delta'} \sum_{k \ge 0} 2^{-k \delta'} F_3(\rd(2^k I_2, \I)) 
\\ \nonumber 
&\le F_1(\ell(I_3)) F_2(\ell(I_3)) \widetilde{F}_3(I_2) 
|I_3|^{\frac12} \bigg[\frac{\ell(I_3)}{\ell(I_2)}\bigg]^{\frac{\delta'}{2}}. 
\end{align}
Consequently, \eqref{FZX-5} and \eqref{FZX-6} imply \eqref{TH-3} in the case $\ell(I_2) > 2^r \ell(I_3)$. 

To treat the case $\ell(I_2) \le 2^r \ell(I_3)$, we split 
\begin{align*}
\langle T(1, \mathbf{1}_{I_2^c}), h_{I_3}^{\eta_3} \rangle 
&= \langle T(\mathbf{1}_{3I_3}, \mathbf{1}_{3I_2 \setminus I_2}), h_{I_3}^{\eta_3} \rangle 
+ \langle T(\mathbf{1}_{3I_3}, \mathbf{1}_{(3I_2)^c}), h_{I_3}^{\eta_3} \rangle
\\
&\quad+ \langle T(\mathbf{1}_{(3I_3)^c}, \mathbf{1}_{3I_2 \setminus I_2}), h_{I_3}^{\eta_3} \rangle 
+ \langle T(\mathbf{1}_{(3I_3)^c}, \mathbf{1}_{(3I_2)^c}), h_{I_3}^{\eta_3} \rangle.
\end{align*}
The first term is similar to $\langle T(h_{I_1}, h_{I_2}^0), h_{I_3}^{\eta_3} \rangle$ in the case $I_3 \cap I_2 = \varnothing$ in Section \ref{sec:adj}, which along with \eqref{TH-2} yields the desired bound. The second and third terms are symmetric, while the estimates for the second and last terms are analogous to \eqref{FZX-5} and \eqref{FZX-6} respectively, but now $\ell$ is replaced by $\ell(I_3)$. Eventually, both are dominated by $F_1(\ell(I_3)) F_2(\ell(I_3)) \widetilde{F}_3(I_2)  |I_3|^{\frac12}$, which together with $\ell(I_2) \simeq \ell(I_3)$ gives the estimate as desired. This shows \eqref{TH-3} in the case $\ell(I_2) \le 2^r \ell(I_3)$.

For $\mathscr{S}_{1, 1}^{3, 3}$, we have 
\begin{align}\label{TH-4}
|\langle T(\phi_{I_2}^{\eta_1}, \mathbf{1}_{I_2}), h_{I_3}^{\eta_3} \rangle| 
\le C_4 \bigg[\frac{\ell(I_3)}{\ell(I_2)}\bigg]^{\frac{\delta'}{2}} 
F(I_2, I_3) |I_2|^{-\frac12} |I_3|^{\frac12}, 
\end{align}
where $F(I_2, I_3)$ is defined in \eqref{eq:FFF}. Indeed, in the case $\ell(I_2) > 2^r \ell(I_3)$, it is similar to $\langle T(\mathbf{1}_{(3I_3)^c}, \mathbf{1}_{I_2^c}), h_{I_3} \rangle$. If $\ell(I_2) \le 2^r \ell(I_3)$, we split 
\begin{align*}
\langle T(\phi_{I_2}^{\eta_1}, \mathbf{1}_{I_2}), h_{I_3}^{\eta_3} \rangle 
= \langle T(\phi_{I_2}^{\eta_1} \mathbf{1}_{3I_2 \setminus I_2}, \mathbf{1}_{I_2}), h_{I_3}^{\eta_3} \rangle 
+ \langle T(\phi_{I_2}^{\eta_1} \mathbf{1}_{(3I_2)^c}, \mathbf{1}_{I_2}), h_{I_3}^{\eta_3} \rangle, 
\end{align*}
where the last two terms are similar to $\langle T(\mathbf{1}_{3I_3}, \mathbf{1}_{3I_2 \setminus I_2}), h_{I_3}^{\eta_3} \rangle$ and $\langle T(\mathbf{1}_{3I_3}, \mathbf{1}_{(3I_2)^c}), h_{I_3}^{\eta_3} \rangle$ respectively.

Now set 
\begin{align*}
a_{I_1, I_2, I_3, I_1}^{\eta_1, \eta_3; N, M} 
= \frac{- \langle h_{I_1}^{\eta_1} \rangle_{I_2} \langle T(1, \mathbf{1}_{I_2^c}), h_{I_3}^{\eta_3} \rangle 
+ \langle T(\phi_{I_2}^{\eta_1}, \mathbf{1}_{I_2}), h_{I_3}^{\eta_3} \rangle}{C_0 |I_2|^{\frac12} [\ell(I_3)/\ell(I_2^{(1)})]^{\frac{\delta'}{2}}} 
\end{align*}
if $I_2 \in \D$ and $I_3 \in \D_{\text{good}}$ satisfy $I_1 = I_2^{(1)}$, $I_3 \subset I_2$, $2^{-M} < \ell(I_1), \ell(I_3) \le 2^N$, $\eta_1 \in \Lambda_{I_1}^N$, and $\eta_3 \in \Lambda$, and otherwise set $a_{I_1, I_2, I_3, I_1}^{\eta_1, \eta_3; N, M} = 0$. For any $k \ge 0$ and $I_3 \in \D_k(I_2)$, let 
\begin{align*}
\mathbf{F}(I_2) 
:= F_1(\ell(I_2)) \widetilde{F}_2(2^{-k} \ell(I_2)) \widetilde{F}_3(I_2). 
\end{align*}
Then the estimates \eqref{TH-3} and \eqref{TH-4} and that $(F_1, F_2, F_3) \in \F$ imply both \eqref{aijkq} and \eqref{FIJK} hold whenever 
\begin{align*}
C_0 \ge \max\{C_1, C_2, 2^{2n+1}(C_3 + C_4), \|T(1, 1)\|_{\BMO}\}.
\end{align*} 
Consequently, by \eqref{TH-3} and \eqref{TH-4}, we deduce 
\begin{align*}
- \mathscr{S}_{1, 1}^{3, 2} + \mathscr{S}_{1, 1}^{3, 3} 
&= C_0 \sum_{k=1}^{\infty} 2^{-k \delta'/2} 
\sum_{I_1 \in \D} \sum_{\substack{I_2 \in \ch(I_1) \\ I_3 \in \D_k(I_1)}} 
\sum_{\substack{\eta_1 \in \{0, 1\}^n \\ \eta_3 \in \Lambda}}
a_{I_1, I_2, I_3, I_1}^{\eta_1, \eta_3; N, M} \langle f_1, h_{I_1}^{\eta_1} \rangle 
\langle f_2, h^0_{I_2} \rangle \langle f_3, h_{I_3}^{\eta_3} \rangle
\\
&= C_0 \sum_{k=1}^{\infty} 2^{-k \delta'/2} 
\big\langle \mathbf{S}_{\D; N, M}^{0, 1, k}(f_1, f_2), f_3 \big\rangle. 
\end{align*}

For the remaining terms, one can obtain the corresponding dyadic shifts symmetrically and similarly. To guarantee \eqref{aijkq} holds, choose $C_0$ to be the largest one of finite bounds $\|T(1, 1)\|_{\BMO}$, $\|T^{*1}(1, 1)\|_{\BMO}$, $\|T^{*2}(1, 1)\|_{\BMO}$, and $2C_j$, $j=1, 2, \ldots$. Collecting all parts together, we conclude the proof of the compact bilinear dyadic representation. 

Finally, observe that up to the additional top-scale terms, the coefficients $a_{I, J, K, Q; N, M}$, dyadic shifts $\mathbf{S}_{\D; N, M}$, and paraproducts $\Pi_{b; N, M}$ are precisely $a_{I, J, K, Q}, \mathbf{S}_{\D}$, and $\Pi_{b}$ with the scale restrictions specified above. The additional terms vanish on finite cancellative Haar expansions for sufficiently large $N$. By density and the summable complexity decay, the uniform $L^4 \times L^4 \to L^2$ bounds  allow us to remove the truncations. Moreover, the compactness conditions on the coefficients remain unchanged. The details of this limiting argument are left to the reader.

\section{Characterizations of compactness}\label{sec:KR}
In this section, we will demonstrate Theorems \ref{thm:KRLpq}--\ref{thm:PNT-cpt}, which are the crucial foundation to study weighted compactness throughout this paper.

\subsection{Unweighted Kolmogorov--Riesz theorems}
We present the following criterion of precompactness in the Banach range, which will be very helpful to prove Theorems \ref{thm:KRLpq}--\ref{thm:PNT-cpt}. 

\begin{theorem}\label{thm:RKB}
Let $1<p, q<\infty$, $w \in A_p$,  and $\K \subset L^{p, q}(w)$. Then $\K$ is precompact in $L^{p, q}(w)$ if and only if the following are satisfied:
\begin{list}{\rm (\theenumi)}{\usecounter{enumi}\leftmargin=1cm \labelwidth=1cm \itemsep=0.2cm \topsep=0.2cm \renewcommand{\theenumi}{\alph{enumi}}}
 
\item\label{RKB-1} ${\displaystyle \sup_{f \in \K} \|f\|_{L^{p, q}(w)} < \infty}$, 

\item\label{RKB-2} ${\displaystyle \lim_{A \to \infty} \sup_{f \in \K} 
\|f \mathbf{1}_{B(0, A)^c}\|_{L^{p, q}(w)}=0}$, 

\item\label{RKB-3} ${\displaystyle \lim_{r \to 0} \sup_{f \in \K} 
\|f - \langle f \rangle_{B(\cdot, r)}\|_{L^{p, q}(w)}=0}$. 
\end{list}
Moreover, in the case $q=\infty$, the conditions \eqref{RKB-1}, \eqref{RKB-2}, and \eqref{RKB-3} are sufficient, but \eqref{RKB-2} and \eqref{RKB-3} are not necessary in general.  
\end{theorem}

\begin{proof}
First, let us prove the necessity. Assume that $\K$ is precompact in $L^{p, q}(w)$. In particular, $\K$ is totally bounded. Given $\varepsilon>0$,   there exists a finite number of functions $\{f_k\}_{k=1}^N \subset \K$ such that $\K \subseteq \bigcup_{k=1}^N B(f_k,\varepsilon)$, where 
\begin{align*}
B(f_k, \varepsilon) := \big\{f \in L^{p, q}(w): \|f-f_k\|_{L^{p, q}(w)} < \varepsilon \big\}, 
\quad k=1, \ldots, N.  
\end{align*}
Let $f \in \K$ be an arbitrary function. Then there exists some $k \in \{1,\ldots,N\}$ so that 
\begin{align}\label{eq:fk-f}
\|f - f_k\|_{L^{p, q}(w)} \le \varepsilon,
\end{align}
which gives 
\begin{align*}
\|f\|_{L^{p, q}(w)} 
\lesssim \|f - f_k\|_{L^{p, q}(w)} + \|f_k\|_{L^{p, q}(w)} 
\le \varepsilon + \max_{1 \leq k \leq N} \|f_k\|_{L^{p, q}(w)}. 
\end{align*}
This shows the condition \eqref{RKB-1}. Note that 
\begin{align}\label{wloc}
w \in A_p \quad \Longrightarrow \quad 
w, w^{-\frac{1}{p-1}} \in L_{\loc}^1(\Rn). 
\end{align}
Since $\mathscr{C}_c^{\infty}(\Rn)$ is dense in $L^{p, q}(w)$ (cf. Lemma \ref{lem:dense-Lpq}), there exists $g_k \in \mathscr{C}_c^{\infty}(\Rn)$ such that 
\begin{align*}
\|f_k - g_k\|_{L^{p, q}(w)} 
\le \varepsilon, 
\end{align*}
which together with \eqref{eq:fk-f} implies  
\begin{align}\label{wlp-1}
\|f - g_k\|_{L^{p, q}(w)} 
\lesssim \|f - f_k\|_{L^{p, q}(w)} + \|f_k - g_k\|_{L^{p, q}(w)} 
\lesssim \varepsilon. 
\end{align}
Set $\supp g_k \subset B(0, A_k)$ for each $k=1, \ldots, N$, and $A_0 := \max\{A_1, \ldots, A_N\}$. Then for any $A \ge A_0$, by \eqref{eq:fk-f} and \eqref{wlp-1},  
\begin{align*}
\|f \mathbf{1}_{B(0, A)^c}\|_{L^{p, q}(w)} 
\lesssim \|f - g_k\|_{L^{p, q}(w)} 
+ \|g_k \mathbf{1}_{B(0, A)^c}\|_{L^{p, q}(w)} 
\lesssim \varepsilon. 
\end{align*}
which justifies the condition \eqref{RKB-2}. To proceed, we split 
\begin{align}\label{wlp-2}
\|f - \langle f \rangle_{B(\cdot, r)}\|_{L^{p, q}(w)} 
&\lesssim \|f - g_k\|_{L^{p, q}(w)} + \|g_k - \langle g_k \rangle_{B(\cdot, r)}\|_{L^{p, q}(w)}
\\ \nonumber 
&\quad+ \| \langle g_k \rangle_{B(\cdot, r)} - \langle f \rangle_{B(\cdot, r)}\|_{L^{p, q}(w)}. 
\end{align}
Since $g_k \in \mathscr{C}_c^{\infty}(\Rn)$, there exists some $r_{0, k}>0$ so that 
\begin{align*}
\sup_{|x-y| < r_0} |g_k(x) - g_k(y)| 
\le \varepsilon \, w(B(0, 2A_0))^{-\frac1p}.
\end{align*}
Hence, for any $0<r<\min\{A_0, r_{0, 1}, \ldots, r_{0, N}\}$, this in turn gives 
\begin{align}\label{wlp-3}
\|g_k - \langle g_k \rangle_{B(\cdot, r)}\|_{L^{p, q}(w)} 
\le \bigg\|\mathbf{1}_{B(0, 2A_0)} \fint_{B(0, r)} 
|g_k(\cdot) - g_k(\cdot+y)| \, dy \bigg\|_{L^{p, q}(w)} 
\lesssim \varepsilon. 
\end{align} 
To treat the last term, observe that 
\begin{align*}
|\langle g_k \rangle_{B(x, r)} - \langle f \rangle_{B(x, r)}| 
\leq \fint_{B(x, r)} |f- g_k| \, dy 
\leq M(f-g_k)(x), 
\end{align*}
and 
\begin{align}\label{Mhst}
\|M h\|_{L^{s, t}(v)}
\lesssim \|h\|_{L^{s, t}(v)}, 
\quad 1<s<\infty, \quad 1 \le t \le \infty,  \quad v \in A_s, 
\end{align} 
where the latter was given in \cite[eq. (2.55)]{CMM}. In view of \eqref{wlp-1}, these yield 
\begin{align}\label{wlp-4}
\| \langle g_k \rangle_{B(\cdot, r)} - \langle f \rangle_{B(\cdot, r)}\|_{L^{p, q}(w)} 
\le \|M(f - g_k)\|_{L^{p, q}(w)}  
\lesssim \|f - g_k\|_{L^{p, q}(w)} 
\lesssim \varepsilon. 
\end{align}
Gathering \eqref{eq:fk-f}--\eqref{wlp-4}, we deduce that for any $r \in (0, r_0)$, 
\begin{align*}
\|f - \langle f \rangle_{B(\cdot, r)}\|_{L^{p, q}(w)} 
\lesssim \varepsilon,\quad\text{uniformly in } f \in \K.  
\end{align*}
This shows the condition \eqref{RKB-3}. 

Next, we turn to the sufficiency. Assume that the conditions \eqref{RKB-1}, \eqref{RKB-2}, and \eqref{RKB-3} hold. Given $\varepsilon>0$, the conditions \eqref{RKB-2} and \eqref{RKB-3} imply that there exist $A>0$ and $r>0$ such that 
\begin{align}\label{eq:WLP-1} 
\|f \mathbf{1}_{B(0, A)^c}\|_{L^{p, q}(w)} \le \varepsilon \quad\text{ and }\quad 
\|f - \langle f \rangle_{B(\cdot, r)}\|_{L^{p, q}(w)} \le \varepsilon, \,\, \forall \, f \in \K.
\end{align}
By \eqref{eq:open}, one can find some $\kappa \in (1, p)$ so that $w \in A_{\kappa}$. Thus, $w, w^{1-\kappa'} \in L_{\loc}^1(\Rn)$. Pick $p_0 \in (1, p/\kappa)$, $p_1 \in (\kappa p_0, p)$, and $s = p_1/p_0$. Then, 
\begin{align}\label{kpqs} 
\frac{1}{s-1} < \frac{1}{\kappa-1} 
\quad\text{ and }\quad 
\frac{\kappa-1}{s-1} \frac{1}{p_0 s'} = \frac{\kappa-1}{p_1}. 
\end{align}
Recall the well-known Kolmogorov's inequality (see \cite[Exercise 1.1.11]{Gra1}):   
\begin{align}\label{eq:Kolm}
\bigg(\fint_E |f|^s \, d\mu \bigg)^{\frac1s} 
\le \bigg(\frac{t}{t-s}\bigg)^{\frac1s} \mu(E)^{-\frac1t} 
\|f\|_{L^{t, \infty}(E)}, \quad 0 < s<t<\infty, 
\end{align}
where $E$ is a measurable set with $0<\mu(E)<\infty$.

Fix $x \in \overline{B(0, A)}$, and denote $B := B(x, r) \subset B(0, 2A+r) =: B_0$. Then by H\"{o}lder's inequality, \eqref{kpqs}, and \eqref{eq:Kolm}, 
\begin{align}\label{eq:WLP-12}
&\bigg(\fint_B |f|^{p_0} \, dz \bigg)^{\frac{1}{p_0}} 
= \bigg(\fint_B |f|^{p_0} w^{\frac1s} w^{-\frac1s}\, dz \bigg)^{\frac{1}{p_0}} 
\\ \nonumber
& \le \bigg(\fint_B |f|^{p_1} w \, dz \bigg)^{\frac{1}{p_1}} 
\bigg(\fint_B w^{-\frac{1}{s-1}} \, dz \bigg)^{\frac{1}{p_0 s'}} 
\\ \nonumber
& \le \bigg(\frac{w(B)}{|B|} \bigg)^{\frac{1}{p_1}}
\bigg(\fint_B |f|^{p_1} \, dw \bigg)^{\frac{1}{p_1}} 
\bigg(\fint_B w^{-\frac{1}{\kappa-1}} \, dz \bigg)^{\frac{\kappa-1}{s-1} \frac{1}{p_0 s'}} 
\\ \nonumber
& \lesssim \bigg(\frac{w(B)}{|B|} \bigg)^{\frac{1}{p_1}} 
w(B)^{-\frac1p} \|f\|_{L^{p, \infty}(w)}
\bigg(\fint_B w^{1-\kappa'} \, dz \bigg)^{\frac{\kappa-1}{p_1}} 
\\ \nonumber
&\lesssim |B|^{-\frac{\kappa}{p_1}} w(B)^{\frac{1}{p_1} - \frac1p} 
w^{1-\kappa'}(B)^{\frac{\kappa-1}{p_1}} \|f\|_{L^{p, \infty}(w)} 
\\ \nonumber 
&\lesssim r^{-\frac{n\kappa}{p_1}} w(B_0)^{\frac{1}{p_1} - \frac1p} 
w^{1-\kappa'}(B_0)^{\frac{\kappa-1}{p_1}} \|f\|_{L^{p, q}(w)}. 
\end{align} 
The fact $w, w^{1-\kappa'} \in L^1_{\loc}(\Rn)$ implies $w(B_0), w^{1-\kappa'}(B_0) < \infty$. Now invoking \eqref{eq:WLP-12}, we have for any $x, y \in \overline{B(0, A)}$,  
\begin{equation}\label{eq:WLP-2}
| \langle f \rangle_{B(x, r)}| 
\lesssim r^{-\frac{n \kappa}{p_1}} w(B_0)^{\frac{1}{p_1} - \frac1p} 
w^{1-\kappa'}(B_0)^{\frac{\kappa-1}{p_1}} \|f\|_{L^{p, q}(w)}, 
\end{equation}
and  
\begin{align}\label{eq:WLP-3}
& |\langle f \rangle_{B(x, r)} - \langle f \rangle_{B(y, r)}| 
\le \frac{1}{|B(0, r)|} \int_{\Rn} |f(z)| |\mathbf{1}_{B(x, r)} - \mathbf{1}_{B(y, r)}| \, dz 
\\ \nonumber 
& \lesssim r^{-n} \bigg(\int_{B(x, r) \cup B(y, r)} |f|^{p_0} \, dz \bigg)^{\frac{1}{p_0}} 
\bigg(\int_{\Rn} |\mathbf{1}_{B(x, r)} - \mathbf{1}_{B(y, r)}|^{p'_0} \, dz \bigg)^{\frac1{p'_0}} 
\\ \nonumber 
& \lesssim r^{-\frac{n}{p'_0} -\frac{n \kappa}{p_1}} w(B_0)^{\frac{1}{p_1} - \frac1p} 
w^{1-\kappa'}(B_0)^{\frac{\kappa-1}{p_1}} \|f\|_{L^{p, q}(w)} 
\bigg(\int_{\Rn} |\mathbf{1}_{B(x, r)} - \mathbf{1}_{B(y, r)}| \, dz \bigg)^{\frac1{p'_0}}. 
\end{align} 
It follows from the condition \eqref{RKB-1} and \eqref{eq:WLP-2}--\eqref{eq:WLP-3} that $\{\langle f \rangle_{B(\cdot, r)}\}_{f \in \K}$ is equi-bounded and equi-continuous on the closed ball $\overline{B(0, A)}$. By the classical Ascoli--Arzel$\grave{\rm a}$ theorem (cf. \cite[p. 85]{Yos}), it is precompact in $\mathscr{C}(\overline{B(0, A)})$, and hence, totally bounded in $\mathscr{C}(\overline{B(0, A)})$. Then, one can find a finite number of functions $\{f_j\}_{j=1}^N \subset \K$ such that 
\begin{equation*}
\inf_{1 \le j \le N} \sup_{|x| \le A} |\langle f \rangle_{B(x, r)} - \langle f_j \rangle_{B(x, r)}| 
\le \varepsilon \, w(B(0, A))^{-\frac1p} \ \text{ for all }f \in \K, 
\end{equation*}
which gives that for each $f \in \K$ there exists $j \in \{1,\ldots, N\}$ such that 
\begin{equation}\label{eq:WLP-4}
\sup_{|x| \le A} |\langle f \rangle_{B(x, r)} - \langle f_j \rangle_{B(x, r)}| 
\le \varepsilon \, w(B(0, A))^{-\frac1p}.
\end{equation}
Consequently, we utilize \eqref{eq:WLP-1} and \eqref{eq:WLP-4} to arrive at 
\begin{align*}
&\|f - f_j\|_{L^{p, q}(w)}
\lesssim \|(f - f_j) \mathbf{1}_{B(0, A)}\|_{L^{p, q}(w)} 
+ \|(f - f_j) \mathbf{1}_{B(0, A)^c}\|_{L^{p, q}(w)}
\\
&\lesssim \|(f - \langle f \rangle_{B(\cdot, r)}) \mathbf{1}_{B(0, A)}\|_{L^{p, q}(w)}
+ \|(\langle f \rangle_{B(\cdot, r)} - \langle f_j \rangle_{B(\cdot, r)}) 
\mathbf{1}_{B(0, A)}\|_{L^{p, q}(w)}
\\
&\quad+ \| (\langle f_j \rangle_{B(\cdot, r)} - f_j) \mathbf{1}_{B(0, A)}\|_{L^{p, q}(w)}
+ \|f \mathbf{1}_{B(0, A)^c}\|_{L^{p, q}(w)} 
+ \|f_j \mathbf{1}_{B(0, A)^c}\|_{L^{p, q}(w)} 
\\
&\le \|f- \langle f \rangle_{B(\cdot, r)}\|_{L^{p, q}(w)}
+ \varepsilon \, w(B(0, A))^{-\frac1p} \|\mathbf{1}_{B(0, A)}\|_{L^{p, q}(w)}
\\
&\quad+\|f_j - \langle f_j \rangle_{B(\cdot, r)}\|_{L^{p, q}(w)}
+ \|f \mathbf{1}_{B(0, A)^c}\|_{L^{p, q}(w)} 
+ \|f_j \mathbf{1}_{B(0, A)^c}\|_{L^{p, q}(w)}
\\
&\lesssim 5\varepsilon,
\end{align*}
which shows that $\K$ is totally bounded. Therefore, $\K$ is precompact in $L^{p, q}(w)$. 

Finally, let us treat the case $q=\infty$. One can easily verify that the proof of the sufficiency above still holds. The necessity of the condition \eqref{RKB-1} is trivial. To show that the conditions \eqref{RKB-2} and \eqref{RKB-3} are not necessary, we construct a counterexample in $\R$. Given $1<p<\infty$, denote 
\begin{align*}
w \equiv 1 \in A_p \quad\text{ and }\quad 
f (x) := |x|^{-\frac1p} \mathbf{1}_{\{x>0\}}. 
\end{align*}
Then $\|f\|_{L^{p, \infty}} = 1$. Let $r>0$. For any $0<x<r$, 
\begin{align*}
\langle f \rangle_{B(x, r)} 
= \frac{1}{2r} \int_0^{x+r} y^{-\frac1p} \, dy 
= \frac{p}{2r(p-1)} (x+r)^{1-\frac1p}, 
\end{align*}
and hence, for all $0<x<2r \big(\frac{p-1}{2p} \big)^p$, 
\begin{align*}
f(x) - \langle f \rangle_{B(x, r)}  
\ge x^{-\frac1p} - \frac{p}{2r(p-1)} (2r)^{1-\frac1p} 
> \frac{x^{-\frac1p}}{2}. 
\end{align*}
Then for every $\lambda > p' (2r)^{-\frac1p}$ (equivalently, $(2\lambda)^{-p} < 2r \big( \frac{p-1}{2p}\big)^p$), the above inequality gives 
\begin{align*}
|\{x \in \R: |f(x) - \langle f \rangle_{B(x, r)}| > \lambda\}|
\ge \bigg|\bigg\{0<x<2r \Big( \frac{p-1}{2p} \Big)^p: \frac{x^{-\frac1p}}{2} > \lambda \bigg\}\bigg|
= (2\lambda)^{-p}, 
\end{align*}
which immediately implies 
\begin{align}\label{exp-1}
\|f - \langle f \rangle_{B(\cdot, r)}\|_{L^{p, \infty}} 
\ge 1/2, \quad\text{ for all } r>0. 
\end{align}
Additionally, for all $A>0$ and $0<\lambda<A^{-\frac1p}$, 
\begin{align*}
\lambda |\{x \in B(0, A)^c: |f(x)| > \lambda\}|^{\frac1p}
= \lambda (\lambda^{-p} - A)^{\frac1p}
= (1-A \lambda^p)^{\frac1p}, 
\end{align*}
which leads to 
\begin{align}\label{exp-2}
\|f \mathbf{1}_{B(0, A)^c}\|_{L^{p, \infty}} 
= 1, \quad\text{ for all } A>0. 
\end{align}
Since $\K := \{f\}$ is a compact set in $L^{p, \infty}(\R)$, the estimates \eqref{exp-1} and \eqref{exp-2} show that the conditions \eqref{RKB-2} and \eqref{RKB-3} are not necessary. 
\end{proof}

\begin{proof}[\bf Proof of Theorem \ref{thm:KRLpq}]
Before starting the proof, let us give some reductions. Denote 
\begin{align*}
\K^+ := \{f^+: f \in \K\} \quad\text{and}\quad 
\K^- := \{f^-: f \in \K\}, 
\end{align*}
where 
\begin{align*}
f^+ := (|f| + f)/2 \quad\text{and}\quad 
f^- := (|f| - f)/2. 
\end{align*}
Then for all $f, g \in \K$ and $x \in \Rn$, 
\begin{align*}
& 0 \le f^+(x) \le |f(x)|, \qquad |f^+(x) - g^+(x)| \le |f(x) - g(x)|, 
\\
& 0 \le f^-(x) \le |f(x)|, \qquad |f^-(x) - g^-(x)| \le |f(x) - g(x)|, 
\\
&\text{and}\quad 
|f(x)-g(x)| \le |f^+(x) - g^+(x)| + |f^-(x) - g^-(x)|,  
\end{align*}
which along with $(\tau_h f)^{\pm} = \tau_h f^{\pm}$ implies that 
\begin{align*}
\text{\eqref{RKLpq-1}--\eqref{RKLpq-3} hold for $\K$
$\iff$ \eqref{RKLpq-1}--\eqref{RKLpq-3} hold for $\K^+$ and $\K^-$}, 
\end{align*}
and 
\begin{align*}
\text{$\K$ is precompact in $L^{p, q}(\Rn)$ $\iff$ 
$\K^+$ and $\K^-$ are precompact in $L^{p, q}(\Rn)$}.
\end{align*}
Considering these, we may assume that $\K$ is a family of non-negative functions.

Pick a positive number $a$ such that $a<\min\{1, p, q\}$. Then $1<\frac{p}{a}, \frac{q}{a} < \infty$. Setting 
$\K^a := \{f^a: f \in \K\}$, we claim that 
\begin{align}\label{GP-1}
\text{\eqref{RKLpq-1}--\eqref{RKLpq-3} hold for $\K \subset L^{p, q}(\Rn)$
$\iff$ \eqref{RKLpq-1}--\eqref{RKLpq-3} hold for $\K^a \subset L^{\frac{p}{a}, \frac{q}{a}}(\Rn)$}, 
\end{align}
and
\begin{align}\label{GP-2}
\text{$\K$ is precompact in $L^{p, q}(\Rn) \iff $ 
$\K^a$ is precompact in $L^{\frac{p}{a}, \frac{q}{a}}(\Rn)$}.
\end{align}
To justify \eqref{GP-1} and \eqref{GP-2}, we present an elementary calculation (cf. \cite{Tsu}): for any $\alpha \in (0,1)$, 
\begin{equation}\label{eq:sata}
|s^{\alpha} - t^{\alpha}| 
\le |s-t|^{\alpha} 
\le \frac{1}{\alpha} \bigg(\frac{s+t}{|s-t|}\bigg)^{1-\alpha} |s^{\alpha} - t^{\alpha}|, 
\quad \text{ for all }s, t>0.
\end{equation}
Then for all $f, g \in \K$,  
\begin{align}\label{faga-1}
\|f^a - g^a\|_{L^{\frac{p}{a}, \frac{q}{a}}} 
\le \||f - g|^a\|_{L^{\frac{p}{a}, \frac{q}{a}}} 
= \|f - g\|_{L^{p, q}}^a.  
\end{align}
Fix $\varepsilon>0$ and $f, g \in \K$, then denote 
\begin{equation*}
E_{\varepsilon} 
:= \bigg\{x \in \Rn: \frac{f(x) + g(x)}{|f(x) - g(x)|}\le \frac{1}{\varepsilon} \bigg\}. 
\end{equation*}
By \eqref{eq:sata}, we have 
\begin{align}\label{faga-2}
\|f - g\|_{L^{p, q}}
&\le \||f - g|^a \mathbf{1}_{E_{\varepsilon}}\|_{L^{\frac{p}{a}, \frac{q}{a}}}^{\frac1a}  
+ \|(f - g) \mathbf{1}_{E_{\varepsilon}^c}\|_{L^{p, q}}
\\ \nonumber 
&\lesssim \varepsilon^{1-\frac1a} \||f^a - g^a| \mathbf{1}_{E_{\varepsilon}}\|_{L^{\frac{p}{a}, \frac{q}{a}}}^{\frac1a} 
+ \varepsilon \|(f + g) \mathbf{1}_{E_{\varepsilon}^c}\|_{L^{p, q}}
\\ \nonumber
&\lesssim \varepsilon^{1-\frac1a} \|f^a - g^a\|_{L^{\frac{p}{a}, \frac{q}{a}}}^{\frac1a} 
+ \varepsilon \|f^a\|_{L^{\frac{p}{a}, \frac{q}{a}}}^{\frac1a}  
+ \varepsilon \|g^a\|_{L^{\frac{p}{a}, \frac{q}{a}}}^{\frac1a}
\\ \nonumber
&\le \varepsilon^{1-\frac1a} \|f^a - g^a\|_{L^{\frac{p}{a}, \frac{q}{a}}}^{\frac1a} 
+ 2 \varepsilon K_0, 
\end{align}
where the implicit constants are independent of $\varepsilon$, $f$, and $g$. Here, 
\begin{align*}
K_0 := \sup_{f \in \K} \|f\|_{L^{p, q}} 
= \sup_{f \in \K} \|f^a\|_{L^{\frac{p}{a}, \frac{q}{a}}}^{\frac1a}.
\end{align*}  
Thus, \eqref{faga-1} and \eqref{faga-2} imply \eqref{GP-1} holds. 

Furthermore, \eqref{faga-1} gives the left-to-right implication in \eqref{GP-2}. To show the reverse, assume that $\K^a$ is precompact in $L^{\frac{p}{a}, \frac{q}{a}}(\Rn)$, and let $\{f_j\}$ be an arbitrary sequence of functions in $\K$. By the  precompactness of $\K^a$, one can find a Cauchy subsequence of $\{f_j^a\}$ (which we relabel). Then $K_0 := \sup_j \|f_j^a\|_{L^{\frac{p}{a}, \frac{q}{a}}}^{\frac1a} < \infty$, and given $\varepsilon>0$, there exists an integer $N=N(\varepsilon)$ such that for all $i, j \ge N$, 
\begin{equation}\label{eq:faij}
\|f_i^a - f_j^a\|_{L^{\frac{p}{a}, \frac{q}{a}}} 
\le \varepsilon.
\end{equation}
Fix $i, j \ge N$. By \eqref{eq:faij}, the estimate \eqref{faga-2} applied to $(f, g) = (f_i, f_j)$ implies  
\begin{align*}
\|f_i - f_j\|_{L^{p, q}}
\lesssim \varepsilon^{1-\frac1a} \|f_i^a - f_j^a\|_{L^{\frac{p}{a}, \frac{q}{a}}}^{\frac1a}  
+ 2 \varepsilon K_0 
\le \varepsilon + 2 \varepsilon K_0, 
\end{align*}
where the implicit constants are independent of $i$, $j$, and $\varepsilon$. This asserts that $\{f_j\}$ is a Cauchy sequence in $\K \subset L^{p, q}(\Rn)$. Thus, by the completeness of $L^{p, q}(\Rn)$, $\K$ is precompact in $L^{p, q}(\Rn)$.

By \eqref{GP-1} and \eqref{GP-2}, to conclude the proof, it is enough to show the case $1<p, q<\infty$. The sufficiency follows from Theorem \ref{thm:RKB} and the following estimates 
\begin{align}\label{eq:hr}
&\|f - \langle f \rangle_{B(\cdot, r)}\|_{L^{p, q}} 
\le \bigg\|\fint_{B(0, r)} |f - \tau_y f| \, dy \bigg\|_{L^{p, q}}
\\ \nonumber 
&\lesssim \fint_{B(0, r)} \|f - \tau_y f\|_{L^{p, q}}  \, dy
\le \sup_{|h|<r} \|\tau_h f -f\|_{L^{p, q}}, 
\end{align}
where to achieve the second inequality, we used the duality between $L^{p, q}(\Rn)$ and $L^{p', q'}(\Rn)$, and Fubini's theorem. To show the necessity, we assume that $\K$ is precompact in $L^{p, q}(\Rn)$. Let $\varepsilon>0$ be an arbitrary number. Then, there exists a finite number of functions $\{f_j\}_{j=1}^N \subset L^{p, q}(\Rn)$ such that for each $f \in \K$, there exists some $j \in \{1, \ldots, N\}$ such that $\|f-f_j\|_{L^{p, q}(\Rn)} \le \varepsilon$. Hence, 
\begin{align*}
\sup_{f \in \K} \|f\|_{L^{p, q}} 
\le \sup_{f \in \K} \big[\inf_{1 \le j \le N} \|f-f_j\|_{L^{p, q}} 
+ \sup_{1 \le j \le N} \|f_j\|_{L^{p, q}} \big] 
\le \varepsilon + \sup_{1 \le j \le N} \|f_j\|_{L^{p, q}}, 
\end{align*}
which proves the condition \eqref{RKLpq-1} holds. 

Let $f \in \K$. Then 
\begin{align}\label{fgpq-1}
\|f-f_j\|_{L^{p, q}} \le \varepsilon, \quad\text{for some } j \in \{1, \ldots, N\}.
\end{align} 
Since $\mathscr{C}_c^{\infty}(\Rn)$ is dense in $L^{p, q}(\Rn)$ (cf. Lemma \ref{lem:dense-Lpq}), there exists $g_j \in \mathscr{C}_c^{\infty}(\Rn)$ such that $\|f_j-g_j\|_{L^{p, q}} \le \varepsilon$. This together with \eqref{fgpq-1} gives 
\begin{align}\label{fgpq-2}
\|f-g_j\|_{L^{p, q}} 
\le \|f-f_j\|_{L^{p, q}} + \|f_j - g_j\|_{L^{p, q}}
\le 2\varepsilon.
\end{align} 
If we let $A_0>0$ satisfy $\bigcup_{j=1}^N \supp g_j \subset B(0, A_0)$, then for any $A \ge A_0$, \eqref{fgpq-2} implies 
\begin{align*}
\|f \mathbf{1}_{B(0, A)^c}\|_{L^{p, q}} 
\lesssim \|f-g_j\|_{L^{p, q}} 
+ \|g_j \mathbf{1}_{B(0, A)^c}\|_{L^{p, q}}
\le 2 \varepsilon, 
\end{align*}
which coincides with the condition \eqref{RKLpq-2}. Moreover, by the fact $g_j \in \mathscr{C}_c^{\infty}(\Rn)$, there exists $\delta_0=\delta_0(\varepsilon)>0$ such that 
\begin{align}\label{fgpq-3}
\sup_{|x-y| < \delta_0} |g_j(x) - g_j(y)| 
\le \varepsilon |B(0, 2A_0)|^{-\frac1p}.
\end{align}
Consequently, for every $|h| < \min\{\delta_0, A_0\}$, the inequality \eqref{fgpq-3} yields 
\begin{equation}\label{fgpq-4}
\|\tau_h g_j - g_j\|_{L^{p, q}} 
\le \varepsilon |B(0, 2A_0)|^{-\frac1p} \|\mathbf{1}_{B(0, 2A_0)}\|_{L^{p, q}}
\lesssim \varepsilon, 
\end{equation}
provided $\supp(\tau_h g_j-g_j) \subset B(0, 2A_0)$. Invoking \eqref{fgpq-2} and \eqref{fgpq-4}, we deduce that for all $|h|<\min\{\delta_0, A_0\}$, 
\begin{align*}
\|\tau_h f - f\|_{L^{p, q}} 
& \le \|\tau_h f - \tau_h g_j\|_{L^{p, q}} 
+ \|\tau_h g_j - g_j\|_{L^{p, q}} 
+ \|g_j - f\|_{L^{p, q}}
\\
& = 2 \|f - g_j\|_{L^{p, q}} 
+ \|\tau_h g_j - g_j\|_{L^{p, q}} 
\lesssim 5 \varepsilon,  
\end{align*}
where the implicit constant is independent of $\varepsilon$, $h$, and $f$. This immediately implies the condition \eqref{RKLpq-3} as desired.

It remains to deal with the case $0<p<\infty$ and $q=\infty$. As above, we are reduced to treating the case $1<p<\infty$. The sufficiency of \eqref{RKLpq-1}, \eqref{RKLpq-2}, and \eqref{RKLpq-3} is a consequence of Theorem \ref{thm:RKB} and \eqref{eq:hr} with $q=\infty$. To obtain the latter, we use the duality between $L^{p, \infty}(\Rn)$ and $L^{p', 1}(\Rn)$, see  \cite[Exercise 1.4.12]{Gra1}. In addition, the necessity of \eqref{RKLpq-1} is trivial, while by Theorem \ref{thm:RKB},  \eqref{RKLpq-2} is not necessary. To prove \eqref{RKLpq-3} is not necessary, we denote 
\begin{align*}
f (x) := |x|^{-\frac1p} \mathbf{1}_{\{x>0\}}, \quad x \in \R. 
\end{align*}
Then $\|f\|_{L^{p, \infty}} = 1$, and for any $h>0$, $\tau_h f(x) = |x-h|^{-\frac1p}$ if $x>h$, $\tau_h f(x)=0$ if $x \le h$. Thus, for all $h>0$ and $\lambda > h^{-\frac1p}$, 
\begin{align*}
|\{x \in \R: |\tau_h f(x) - f(x)| > \lambda\}|
& \ge |\{0<x<h: |\tau_h f(x) - f(x)| > \lambda\}|
\\
&= |\{0<x<h: |x|^{-\frac1p} > \lambda\}|
= \lambda^{-p}, 
\end{align*}
which gives  
\begin{align*}
\|\tau_h f - f\|_{L^{p, \infty}} 
\ge 1, \quad\text{ for all } h>0. 
\end{align*}
This asserts that $\K := \{f\}$ is a compact set in $L^{p, \infty}(\R)$, but does not satisfy \eqref{RKB-3}.  
\end{proof}

\subsection{Weighted Kolmogorov--Riesz theorems}
We are ready to prove Theorem \ref{thm:KRW}--\ref{thm:KRWA}.

\begin{proof}[\bf Proof of Theorem \ref{thm:KRW}]
Write $p_0 := 1+\frac{1}{\lambda}$. By the rescaling argument as in the proof of Theorem \ref{thm:KRLpq}, it suffices to consider the case $p_0^2<p, q<\infty$. Pick $p_1 \in (p_0^2, p)$ and $s=p_1/p_0$. Analogously to \eqref{eq:hr}, there holds 
\begin{align*}
\|f - \langle f \rangle_{B(\cdot, r)}\|_{L^{p, q}(w)} 
\lesssim \fint_{B(0, r)} \|f - \tau_y f\|_{L^{p, q}(w)}  \, dy
\le \sup_{|h|<r} \|\tau_h f -f\|_{L^{p, q}(w)}. 
\end{align*}
Thus, the proof is similar to that of the sufficiency of Theorem \ref{thm:RKB}, where $\kappa$ is replaced by $p_0$. 

Since $\mathscr{C}_c^{\infty}(\Rn)$ is dense in $L^{p, q}(w)$ (cf. Lemma \ref{lem:dense-Lpq}), the necessity of \eqref{RKW-1} and \eqref{RKW-2} can be shown much as in the proof of the necessity of Theorem \ref{thm:KRLpq}. The condition \eqref{RKW-3} is not necessary because \cite[Theorem 2.8 (b)]{COY} presented a counterexample in the case $w \in A_2$ and $0<p=q<\infty$. 

When $q=\infty$, the sufficiency of \eqref{RKW-1}--\eqref{RKW-3} can be proved as above, while by Theorem \ref{thm:KRLpq}, \eqref{RKW-2} and \eqref{RKW-3} are not necessary.  This completes the proof. 
\end{proof}

\begin{proof}[\bf Proof of Theorem \ref{thm:KRWA}]
To show the sufficiency, assume that the conditions \eqref{RKWA-1}--\eqref{RKWA-3} hold. As argued in the proof of Theorem \ref{thm:KRLpq}, we may assume that $\K$ is a family of non-negative functions. Since 
\begin{align*}
\big|f	^a(x) - \langle f^a \rangle_{B(x, r)} \big| 
\le \fint_{B(0, r)} |\tau_y f (x) - f(x)|^a \, dy, 
\end{align*}
the conditions \eqref{RKWA-1}--\eqref{RKWA-3} imply 
\begin{align}
\label{fGfG-1}
& \sup_{f \in \K} \| f^a\|_{L^{\frac{p}{a}, \frac{q}{a}}(w)} < \infty, 
\\
\label{fGfG-2}
\lim_{A \to \infty} & \sup_{f \in \K} \|f^a \mathbf{1}_{B(0, A)^c}\|_{L^{\frac{p}{a}, \frac{q}{a}}(w)}=0, 
\\ 
\label{fGfG-3}
\lim_{r \to 0} & \sup_{f \in \K}\| f^a - \langle f^a \rangle_{B(\cdot, r)}\|_{L^{\frac{p}{a}, \frac{q}{a}}(w)}=0.
\end{align}
Note that $1<\frac{p}{a}, \frac{q}{a}<\infty$ and $w \in A_{p_0} \subset A_{\frac{p}{a}}$. Then by Theorem \ref{thm:RKB} and \eqref{fGfG-1}--\eqref{fGfG-3}, $\K^a := \{f^a: f \in \K\}$ is precompact in $L^{\frac{p}{a}, \frac{q}{a}}(w)$. Much as in \eqref{GP-2}, one can show that $\K$ is precompact in $L^{p, q}(w)$.

To justify the necessity, assume that $\K$ is precompact in $L^{p, q}(w)$. The condition $w \in A_{p_0}$ implies $w, w^{1-p'_0} \in L^1_{\loc}(\Rn)$. Then this, together with Theorem \ref{thm:KRW}, gives the conditions \eqref{RKWA-1} and \eqref{RKWA-2}. To check \eqref{RKWA-3}, let $\varepsilon>0$. By the precompactness of $\K$, there exists a finite number of functions $\{f_j\}_{j=1}^N \subset \K$ such that $\inf_{1 \le j \le N} \|f-f_j\|_{L^{p, q}(w)} \le \varepsilon$ for all $f \in \K$. Fix $f \in \K$. Then there is some $j \in \{1, \ldots, N\}$ such that 
\begin{align}\label{eq:ffe}
\|f-f_j\|_{L^{p, q}(w)} \le \varepsilon. 
\end{align}
By Lemma \ref{lem:dense-Lpq}, $\mathscr{C}_c^{\infty}(\Rn)$ is dense in $L^{p, q}(w)$. Then for each $j \in \{1, \ldots, N\}$, there exists $g_j \in \mathscr{C}_c^{\infty}(\Rn)$ so that $\|f_j-g_j\|_{L^{p, q}(w)} \le \varepsilon$, which gives 
\begin{align}\label{fgjj}
\|f-g_j\|_{L^{p, q}(w)}
\lesssim \|f-f_j\|_{L^{p, q}(w)}
+ \|f_j-g_j\|_{L^{p, q}(w)} 
\le 2 \varepsilon,  
\end{align}
and 
\begin{align}\label{xygj}
\sup_{|x-y| < r_0} |g_j(x) - g_j(y)| 
\le \varepsilon \, w(B(0, 2A_0))^{-\frac1p} 
\quad\text{ for some } r_0>0, 
\end{align}
where $\bigcup_{j=1}^N \supp(g_j) \subset B(0, A_0)$. We split 
\begin{align}\label{fjjj}
\mathcal{I} 
& := \bigg\|\bigg(\fint_{B(0, r)} |\tau_y f - f|^a \, dy\bigg)^{\frac1a} \bigg\|_{L^{p, q}(w)} 
\\ \nonumber 
& \lesssim \bigg\|\bigg(\fint_{B(0, r)} |\tau_y f - \tau_y g_j|^a \, dy\bigg)^{\frac1a} \bigg\|_{L^{p, q}(w)} 
\\  \nonumber
&\quad + \bigg\|\bigg(\fint_{B(0, r)} |\tau_y g_j - g_j|^a \, dy\bigg)^{\frac1a} \bigg\|_{L^{p, q}(w)} 
\\ \nonumber
&\quad + \bigg\|\bigg(\fint_{B(0, r)} |g_j - f|^a \, dy\bigg)^{\frac1a} \bigg\|_{L^{p, q}(w)} 
\\ \nonumber
&=: \mathcal{I}_1 + \mathcal{I}_2 + \mathcal{I}_3. 
\end{align}
It follows from the fact $w \in A_{p_0} \subset A_{\frac{p}{a}}$, \eqref{Mhst}, and \eqref{fgjj} that  
\begin{align}\label{fjjj-1} 
\mathcal{I}_1 
&\le \big\|M \big(|f-g_j|^a \big)^{\frac1a} \big\|_{L^{p, q}(w)}
= \big\|M \big(|f-g_j|^a \big)\big\|_{L^{\frac{p}{a}, \frac{q}{a}}(w)}^{\frac1a} 
\\ \nonumber 
&\lesssim \| |f-g_j|^a\|_{L^{\frac{p}{a}, \frac{q}{a}}(w)}^{\frac1a} 
= \|f-g_j\|_{L^{p, q}(w)} 
\lesssim 2 \varepsilon. 
\end{align}
For all $0<r<\min\{r_0, A_0\}$, using \eqref{xygj} and that $\supp(\tau_y g_j -g_j) \subset B(0, 2A_0)$ for any $|y|<r$, we have 
\begin{align}\label{fjjj-2} 
\mathcal{I}_2  
\le \|\mathbf{1}_{B(0, 2A_0)}\|_{L^{p, q}(w)} 
\sup_{|x-y| < r_0} |g_j(x) - g_j(y)| 
\lesssim \varepsilon. 
\end{align}
Additionally, \eqref{fgjj} implies 
\begin{align}\label{fjjj-3} 
\mathcal{I}_3 
= \|g_j-f\|_{L^{p, q}(w)} 
\lesssim 2 \varepsilon. 
\end{align}
Now gathering \eqref{fjjj}--\eqref{fjjj-3}, we conclude that 
\begin{align*}
\mathcal{I} 
\lesssim 5 \varepsilon,  
\quad\text{ for all } 0<r<\min\{r_0, A_0\}. 
\end{align*}
This yields \eqref{RKWA-3}.  

In the case $q=\infty$, the sufficiency of \eqref{RKWA-1}--\eqref{RKWA-3} is the same as above because of the rescaling argument and Theorem \ref{thm:RKB} with $q=\infty$. To prove the conditions \eqref{RKW-2} and \eqref{RKW-3} are not necessary in general, we pick 
\begin{align*}
w \equiv 1\in A_{p_0}(\R), \quad 
f (x) := |x|^{-\frac1p} \mathbf{1}_{\{x>0\}}, 
\quad \text{ and } \quad
\K := \{f\}.
\end{align*}
Then $\|f\|_{L^{p, \infty}} = 1$. For all $A>0$ and $\lambda>0$,
\begin{align*}
\big|\{x \in \R: |f(x) \mathbf{1}_{B(0,A)^c}(x)| > \lambda\}\big|
= \max \{\lambda^{-p}-A, 0\},
\end{align*}
which implies
\begin{align*}
\|f\mathbf{1}_{B(0,A)^c}\|_{L^{p,\infty}}
= \sup_{0<\lambda<A^{-1/p}} (1-A\lambda^p)^{\frac1p}
= 1.
\end{align*}
Let $r>0$ and $0<x<r/2$. We have  
\begin{align*}
S_r f(x) 
&:= \bigg( \fint_{B(0, r)} |\tau_y f(x) - f(x)|^a \, dy \bigg)^{\frac1a}
\\
&\ge \bigg( \frac{1}{2r} \int_{\frac{r}{2}}^r |\tau_y f(x) - f(x)|^a \, dy \bigg)^{\frac1a}
\\
&= \bigg( \frac{1}{2r} \int_{\frac{r}{2}}^r |x|^{-\frac{a}{p}} \, dy \bigg)^{\frac1a}
=4^{-\frac1a} |x|^{-\frac1p}, 
\end{align*}
which gives 
\begin{align*}
&|\{x \in \R: S_r f(x) > \lambda\}| 
\ge |\{0<x<r/2: S_r f(x) > \lambda\}| 
\\
&\ge |\{0<x<r/2: 4^{-\frac1a} x^{-\frac1p} > \lambda\}| 
= \min\big\{r/2,  4^{-\frac{p}{a}} \lambda^{-p}\big\}
= 4^{-\frac{p}{a}} \lambda^{-p}, 
\end{align*}
provided $\lambda>4^{-\frac1a} (r/2)^{-\frac1p}$. Hence, 
\begin{align*}
\|S_r f\|_{L^{p, \infty}} 
\ge 4^{-\frac1a}, \quad \text{for all } r>0. 
\end{align*}
This shows that the compact subset $\K := \{f\}$ in $L^{p, \infty}(\R)$ does not satisfy the conditions \eqref{RKWA-2} and \eqref{RKWA-3}.  
\end{proof}

\subsection{Characterizations of compactness via projection}

The following result was shown in \cite[Lemma 6.1]{GXY}, which states that the limit of a sequence of compact bilinear operators is also compact. 

\begin{lemma}\label{lem:limit}
Let $\mathscr{X}_1$ and $\mathscr{X}_2$ be Banach spaces and $\mathscr{Y}$ be a quasi-Banach space. For each $j \in \N$, let $T_j$ be a compact bilinear operator from $\mathscr{X}_1 \times \mathscr{X}_2$ to $\mathscr{Y}$. If a bilinear operator $T: \mathscr{X}_1 \times \mathscr{X}_2 \to \mathscr{Y}$ satisfies 
\begin{align*}
\lim_{j \to \infty} \|T_j - T\|_{\mathscr{X}_1 \times \mathscr{X}_2 \to \mathscr{Y}} = 0, 
\end{align*}
then $T$ is compact from $\mathscr{X}_1 \times \mathscr{X}_2$ to $\mathscr{Y}$. 
\end{lemma}

\begin{proof}[\bf Proof of Theorem \ref{thm:PNT-cpt}]
We begin with showing item \eqref{list:PN1}. Fix a dyadic grid $\D$. Assume that $\lim_{N \to \infty} \|P_{\D(N)}^{\perp} T\|_{L^{p_1} \times L^{p_2} \to L^p} = 0$. Note that $P_{\D(N)}$ is a finite-dimensional operator for each $N \in \N$. Then, by \cite[Proposition 15.1]{FHHMZ}, $\{P_{\D(N)} T\}_{N \in \N}$ is a collection of compact operators from $L^{p_1}(\Rn) \times L^{p_2}(\Rn)$ to $L^p(\Rn)$. By Lemma \ref{lem:limit} and  that $\lim_{N \to \infty} \|P_{\D(N)} T -T\|_{L^{p_1} \times L^{p_2} \to L^p} = \lim_{N \to \infty} \|P_{\D(N)}^{\perp} T\|_{L^{p_1} \times L^{p_2} \to L^p} = 0$, we conclude that $T$ is compact from $L^{p_1}(\Rn) \times L^{p_2}(\Rn)$ to $L^p(\Rn)$. 

To conclude item \eqref{list:PN1}, let $T$ be compact from $L^{p_1}(\Rn) \times L^{p_2}(\Rn)$ to $L^p(\Rn)$, and it suffices to demonstrate that 
\begin{align}\label{NPN}
\lim_{N \to \infty} \sup_{\D} \|P_{\D(N)}^{\perp} T\|_{L^{p_1} \times L^{p_2} \to L^p} = 0. 
\end{align}
We claim that 
\begin{align}\label{eq:ddp}
\lim_{N \to \infty} \sup_{\D} \|P_{\D(N)}^{\perp} h\|_{L^p} = 0, \quad \forall h \in \mathscr{C}_c^{\infty}(\Rn). 
\end{align}
Indeed, given $h \in \mathscr{C}_c^{\infty}(\Rn)$, one can choose a large integer $N \ge 1$ so that $\supp h \subset 2^N \I$. Then $\langle h, h_I \rangle \ne 0$ implies $\dist(I, 2^N \I) = 0$, hence, all cubes $Q$ with $2^{-N} \le \ell(I) \le 2^N$ belong to $\D(N)$. Thus, 
\begin{align}\label{eq:pdnh}
P_{\D(N)} h = E_{2^{-N-1}} h - E_{2^N} h,
\end{align} 
which together with the Poincar\'{e} inequality gives
\begin{align*}
\|P_{\D(N)}^{\perp} h\|_{L^p}^p 
& \lesssim \|h - E_{2^{-N-1}} h\|_{L^p}^p + \|E_{2^N} h\|_{L^p}^p
\\
& \le \sum_{I \in \D: \ell(I)=2^{-N-1}} \int_I |h - \langle h \rangle_I|^p
+ \sum_{I \in \D: \ell(I)=2^N} |I| \langle |h| \rangle_I^p
\\
& \lesssim 2^{-Np} \|\nabla h\|_{L^p}^p 
+ 2^{-Nn(p-1)} \|h\|_{L^1}^p 
\to 0. 
\end{align*}

The compactness of $T$ implies $\K := \{T(f_1, f_2): \|f_1\|_{L^{p_i}} \le 1, i=1,2\}$ is a precompact subset in $L^p(\Rn)$. Given $\varepsilon > 0$, one can find a finite number of functions $\{h_k\}_{k=1}^{k_0} \subset \K$ so that 
\begin{align}\label{eq:dpp-1}
\sup_{f \in \K} \min_{1 \le k \le k_0} \|f - h_k\|_{L^p} 
\le \varepsilon.
\end{align}
Since $\mathscr{C}_c^{\infty}(\Rn)$ is dense in $L^p(\Rn)$, for each $k \in \{1, \ldots, k_0\}$, there exists some $\phi_k \in \mathscr{C}_c^{\infty}(\Rn)$ such that $\|h_k - \phi_k\|_{L^p} \le \varepsilon$. This and \eqref{eq:dpp-1} imply
\begin{align}\label{eq:dpp-2}
\sup_{f \in \K} \min_{1 \le k \le k_0} \|f - \phi_k\|_{L^p} 
\lesssim \varepsilon.
\end{align}
We then use \eqref{ddf-4} and \eqref{eq:dpp-2} to deduce
\begin{align*}
\sup_{\D} \sup_{f \in \K} \|P_{\D(N)}^{\perp} f\|_{L^p}
\lesssim \varepsilon + \max_{1 \le k \le k_0} \sup_{\D} \|P_{\D(N)}^{\perp} \phi_k\|_{L^p},
\end{align*}
which together with \eqref{eq:ddp} gives 
\begin{align*}
\lim_{N \to \infty} \sup_{\D} \sup_{f \in \K} \|P_{\D(N)}^{\perp} f\|_{L^p}
\lesssim \varepsilon.
\end{align*}
By the arbitrariness of $\varepsilon$, this justifies \eqref{NPN}.

The first part of item \eqref{list:PN2} follows from Lemma \ref{lem:limit} and that $P_{\D(N)}$ is a finite-dimensional operator. To show the second part, we follow \cite{OV} to construct a counterexample in $\R$. Define a bilinear operator 
\begin{align*}
T(f_1, f_2) 
:= \langle f_1, h_{I_0} \rangle \langle f_2, h_{I_0} \rangle \, \mathbf{1}_{I_0}, 
\end{align*}
where $I_0 = [0, 1)$. Since $T$ is a finite-dimensional operator, it is compact from $L^1(\Rn) \times L^1(\Rn) \to L^{\frac12, \infty}(\Rn)$. Choose $f_1 = f_2 = h_{I_0}$ such that $\|f_1\|_{L^1} = \|f_2\|_{L^1} = 1$. Let $\D$ be a dyadic grid and $N \ge 4$. As in \eqref{eq:pdnh}, there holds 
\begin{align}\label{eq:gpdn}
P_{\D(N)}^{\perp} \mathbf{1}_{I_0} 
= \mathbf{1}_{I_0} - E_{2^{-N-1}} \mathbf{1}_{I_0} + E_{2^N} \mathbf{1}_{I_0}.
\end{align}
Observe that there are at most two dyadic intervals $I_1, I_2 \in \D$ with $\ell(I_1) = \ell(I_2) = 2^N$ such that $I_1 \cap I_0 \ne \varnothing$ and $I_2 \cap I_0 \ne \varnothing$. Thus, there exists some $I \in \D$ with $\ell(I) = 2^N$ such that $|I \cap I_0| \ge |I_0|/2$, which gives
\begin{align}\label{eq:e2ng}
(E_{2^N} \mathbf{1}_{I_0}) \mathbf{1}_I 
= |I \cap I_0|/2^N \ge 2^{-N-1}.
\end{align}
Note that $\supp \mathbf{1}_{I_0} \cup \supp (E_{2^{-N-1}} \mathbf{1}_{I_0}) \subset [-1, 2]$, which together with \eqref{eq:gpdn}--\eqref{eq:e2ng} gives 
\begin{align*}
(P_{\D(N)}^{\perp} \mathbf{1}_{I_0}) \mathbf{1}_{I \setminus [-1, 2]} 
= (E_{2^N} \mathbf{1}_{I_0}) \mathbf{1}_{I \setminus [-1, 2]} 
\ge 2^{-N-1}.
\end{align*}
Consequently, for every $N \ge 4$, 
\begin{align*}
& \sup_{\lambda>0} \lambda |\{x \in \R: |P_{\D(N)}^{\perp}T(f_1, f_2)(x)| > \lambda\}|^2 
\\
& \ge 2^{-N-2} |\{x \in I \setminus [-1, 2]: 
|P_{\D(N)}^{\perp} \mathbf{1}_{I_0}(x)| > 2^{-N-2}\}|^2 
\\
&= 2^{-N-2} |I \setminus [-1, 2]|^2
\ge 2^{-N-2} (2^N -3)^2
\ge 2^{N-4}. 
\end{align*}
This shows $\lim_{N \to \infty} \inf_{\D} \|P_{\D(N)}^{\perp} T\|_{L^1 \times L^1 \to L^{\frac12, \infty}} = \infty$.

The proof of item \eqref{list:PN3} is based on Theorem \ref{thm:PDN}. Denote
\begin{align*}
\K := \big\{T(f_1, f_2): \|f_1\|_{L^{\infty}} \le 1, \|f_2\|_{L^{\infty}} \le 1\big\}
\quad \text{ and } \quad
\gamma_N := \sup_{\D} \sup_{b \in \K} 
\|P_{\D(N)}^{\perp} b\|_{\BMO_{\D}}.
\end{align*}
To prove the necessity, assume that $T: L^{\infty}(\Rn) \times L^{\infty}(\Rn) \to \CMO(\Rn)$ is compact. Given $\varepsilon > 0$, 
there exist a finite number of functions $\{b_k\|_{k=1}^{k_0} \subset \K$ such that for every $b \in \K$, there exists some $k \in \{1, \ldots, k_0\}$ with $\|b - b_k\|_{\BMO} \le \varepsilon$. By the fact $b_k \in \CMO(\Rn)$ and Theorem \ref{thm:PDN}, there exists some $N_k = N_k(\varepsilon) \ge 1$ so that 
\begin{align}\label{eq:bkbp}
\sup_{\D} \|P_{\D(N)}^{\perp} b_k\|_{\BMO_{\D}} \le \varepsilon. 
\end{align}
It follows from \eqref{pnbb} that
\begin{align}\label{pnbd}
\sup_{N \ge 1} \sup_{\D} \big( \|P_{\D(N)} b\|_{\BMO_{\D}} + \|P_{\D(N)}^{\perp} b\|_{\BMO_{\D}} \big)
\lesssim \|b\|_{\BMO},
\end{align}
which along with \eqref{eq:bkbp} yields that for all $N \ge N_0 := \max\{N_1, \ldots, N_{k_0}\}$, 
\begin{align*}
\|P_{\D(N)}^{\perp} b\|_{\BMO_{\D}}
\le \|P_{\D(N)}^{\perp} (b-b_k)\|_{\BMO_{\D}}
+ \|P_{\D(N)}^{\perp} b_k\|_{\BMO_{\D}}
\lesssim \|b-b_k\|_{\BMO} + \varepsilon
\le 2 \varepsilon.
\end{align*}

In oder to show the sufficiency, let $\lim_{N \to \infty} \gamma_N = 0$. It was shown in \cite{HP} that one can find $2^n$ dyadic grids $\{\D_j\}_{j=1}^{2^n}$ such that for every cube $Q \subset \Rn$, there exist $j \in \{1, \ldots, 2^n\}$ and $R_Q \in \D_j$ so that $Q \subset R_Q$ and $\ell(R_Q) \le 6 \ell(Q)$. This implies that for every cube $Q \subset \Rn$,
\begin{align*}
\langle |b - \langle b \rangle_Q| \rangle_Q
\le 2 \langle |b - \langle b \rangle_{R_Q}| \rangle_Q
\lesssim \langle |b - \langle b \rangle_{R_Q}| \rangle_{R_Q},
\end{align*}
which indicates 
\begin{align}\label{eq:dybmo}
\|b\|_{\BMO} 
\simeq \max_{1 \le j \le 2^n} \|b\|_{\BMO_{\D_j}}.
\end{align}
The boundedness of $T$ yields 
\begin{align}\label{eq:bkbo}
\K \subset \BMO(\Rn) \quad \text{ and } \quad 
\sup_{b \in \K} \|b\|_{\BMO} < \infty.
\end{align} 
Additionally, for every $b \in \K$, $\sup_{\D} \|P_{\D(N)}^{\perp} b\|_{\BMO_{\D}} \le \gamma_N \to 0$, which together with Theorem \ref{thm:PDN} gives $b \in \CMO(\Rn)$. That is, $\K \subset \CMO(\Rn)$. On the other hand, given $\varepsilon > 0$, there exists some $N=N(\varepsilon) \ge 1$ so that $\gamma_N \le \varepsilon$. The operator $T_N b:= (P_{\D_1(N)}b, \ldots, P_{\D_{2^n}(N)} b)$ is a finite-dimensional operator with the norm $\|T_N b\| := \max_{1 \le j \le 2^n} \|P_{\D_j(N)} b\|_{\BMO_{\D_j}}$ on its range. Then in view of \eqref{pnbd} and \eqref{eq:bkbo}, $T_N(\K)$ is totally bounded. Hence, one can find a finite number of functions $\{b_k\}_{k=1}^{k_0} \subset \K$ so that  for every $b \in \K$, there exists some $k \in \{1, \ldots, k_0\}$ with 
\begin{align*}
\max_{1 \le j \le 2^n} \|P_{\D_j(N)}(b - b_k)\|_{\BMO_{\D_j}} \le \varepsilon.
\end{align*} 
This and \eqref{eq:dybmo} give
\begin{align*}
\|b - b_k\|_{\BMO}
& \lesssim \max_{1 \le j \le 2^n} \|b - b_k\|_{\BMO_{\D_j}}
\\
& \lesssim \max_{1 \le j \le 2^n} \|P_{\D_j(N)}(b - b_k)\|_{\BMO_{\D_j}} 
\\
& \quad + \max_{1 \le j \le 2^n} \|P_{\D_j(N)}^{\perp} b\|_{\BMO_{\D_j}} 
\\
& \quad + \max_{1 \le j \le 2^n} \|P_{\D_j(N)}^{\perp} b_k\|_{\BMO_{\D_j}} 
\\
& \le \varepsilon + 2 \gamma_N 
\le 3 \varepsilon,
\end{align*} 
which means $\K$ is totally bounded. Since $\K \subset \CMO(\Rn)$ and $\CMO(\Rn)$ is complete, $\K$ is precompact. By definition, $T: L^{\infty}(\Rn) \times L^{\infty}(\Rn) \to \CMO(\Rn)$ is compact.  
\end{proof}

Finally, let us end up this section with the following result, which will provide us great convenience in practice.

\begin{lemma}\label{lem:TT}
Let $\frac1p = \frac{1}{p_1} + \frac{1}{p_2}$ with $p, p_1, p_2 \in (1, \infty)$. Let $\{\alpha_j\}_{j \ge 0}$ be a sequence of positive numbers satisfying $\sum_{j \ge 0} \alpha_j < \infty$. Assume that a bilinear operator $T$ satisfies the following: 
\begin{enumerate}
\item[(1)]  $T=\sum_{j \ge 0} \alpha_j \, T_j$, where each $T_j$ is a bilinear operator. 

\vspace{0.2cm}
\item[(2)] $\sup_{j \ge 0} \|T_j\|_{L^{p_1} \times L^{p_2} \to L^p} \le C_0 < \infty$. 

\vspace{0.2cm}
\item[(3)] For every $j \ge 0$, $T_j$ is compact from $L^{p_1}(\Rn) \times L^{p_2}(\Rn)$ to $L^p(\Rn)$. 
\end{enumerate}
Then, $T$ is compact from $L^{p_1}(\Rn) \times L^{p_2}(\Rn)$ to $L^p(\Rn)$. 
\end{lemma}

\begin{proof}
For each $N \ge 1$, let $T^N := \sum_{j=0}^N \alpha_j \, T_j$. Then $T^N$ is compact from $L^{p_1}(\Rn) \times L^{p_2}(\Rn)$ to $L^p(\Rn)$. Note that $p>1$ and 
\begin{align*}
\|T - T^N \|_{L^{p_1} \times L^{p_2} \to L^p} 
\le \sum_{j>N} \alpha_j \|T_j\|_{L^{p_1} \times L^{p_2} \to L^p} 
\le C_0 \sum_{j>N} \alpha_j,  
\end{align*} 
which tends to zero as $N \to \infty$. This and Lemma \ref{lem:limit} imply the compactness of $T$.
\end{proof}

\section{Weighted $L^{p_1} \times L^{p_2} \to L^p$ compactness}\label{sec:Lp}
Now let us justify ${\rm (b)} \Longrightarrow {\rm (c)'}$ in Theorem \ref{thm:cpt}. Assume (b) holds. Then by duality, Theorems \ref{thm:KRLpq}, \ref{thm:SD-cpt}, and \ref{thm:Pi-cpt}, and Lemma \ref{lem:TT}, we deduce that 
\begin{align}\label{eq:TP-1}
\text{$T$ is compact from $L^4(\Rn) \times L^4(\Rn)$ to $L^2(\Rn)$}.
\end{align}
On the other hand, the compact bilinear dyadic representation, along with Theorems \ref{thm:SD} and \ref{thm:Pi} below, gives that 
\begin{align}\label{eq:TP-2}
\text{$T$ is bounded from $L^4(v_1^4) \times L^4(v_2^4)$ to $L^2(v^2)$} 
\end{align}
for all $(v_1, v_2) \in A_{(4, 4)}$, where $v = v_1 v_2$. Thus, the conclusion ${\rm (c)'}$ follows from \eqref{eq:TP-1}--\eqref{eq:TP-2} and Theorem \ref{thm:EP-Lp}.

\subsection{Weighted compactness of bilinear dyadic shifts}  
In order to show Theorem \ref{thm:SD-cpt}, we first establish the weighted boundedness of bilinear dyadic shifts $\mathbf{S}_{\D}^{i, j, k}$.

\begin{theorem}\label{thm:SD}
There holds 
\begin{align*}
\sup_{i, j, k \in \N} \sup_{\D} \|\mathbf{S}_{\D}^{i, j, k}(f_1, f_2)\|_{L^p(w^p)}
\lesssim \|f_1\|_{L^{p_1}(w_1^{p_1})} \|f_2\|_{L^{p_2}(w_2^{p_2})}, 
\end{align*}
for all $p_1, p_2 \in (1, \infty]$ and for all $(w_1, w_2) \in A_{(p_1, p_2)}$, where $\frac1p = \frac{1}{p_1} + \frac{1}{p_2}>0$ and $w=w_1 w_2$.
\end{theorem}

\begin{proof}
By symmetry, it suffices to treat the case $\widetilde{h}_I = h_I$. Noting that 
\begin{align*}
\langle \Delta_Q^i f, h_I \rangle 
= \langle f, h_I \rangle \mathbf{1}_{\{I \in \D_i(Q)\}},
\end{align*} 
we have  
\begin{align*}
|\langle \mathbf{S}_{\D}^{i, j, k}(f_1, f_2), f_3 \rangle| 
\le \|\mathbf{A}_{\D}^{i, k} (f_1, f_2, f_3)\|_{L^1}, 
\end{align*}
where 
\begin{align*}
\mathbf{A}_{\D}^{i, k} (f_1, f_2, f_3) 
:= \sum_{Q \in \D} \langle |\Delta_Q^i f_1| \rangle_Q 
\langle |f_2| \rangle_Q  \langle |\Delta_Q^k f_3| \rangle_Q \, \mathbf{1}_Q. 
\end{align*}
By duality and Theorem \ref{thm:RdF}, we are deduced to showing 
\begin{align}\label{eq:AD}
\sup_{\D} \|\mathbf{A}_{\D}^{i, k} (f_1, f_2, f_3)\|_{L^p(w^p)}
\lesssim \prod_{i=1}^3 \|f_i\|_{L^{p_i}(w_i^{p_i})}, 
\end{align}
for all $p_1, p_2, p_3 \in (1, \infty)$ and for all $(w_1, w_2, w_3) \in A_{(p_1, p_2, p_3)}$, where $\frac1p = \frac{1}{p_1} + \frac{1}{p_2} + \frac{1}{p_3}$ and $w=w_1 w_2 w_3$. 

Let $\frac1p = \frac{1}{p_1} + \frac{1}{p_2} + \frac{1}{p_3}$ with $p, p_1, p_2, p_3 \in (1, \infty)$, and let $(w_1, w_2, w_3) \in A_{(p_1, p_2, p_3)}$. Write $w=w_1 w_2 w_3$. By Theorem \ref{thm:RdF} again, it is enough to prove \eqref{eq:AD} for such exponents $(p_1, p_2, p_3)$ and weights $(w_1, w_2, w_3)$. Let $p_4=p'$, $w_4=w^{-1}$, and $f_4 w_4 \in L^{p_4}$ with $\|f_4 w_4\|_{L^{p_4}} \le 1$. Observe that 
\begin{align}\label{SDK-1}
\langle |\Delta_Q^k f|, \mathbf{1}_Q \rangle 
= \langle \Delta_Q^k f, \varphi^k_{Q, f} \rangle 
= \langle f, \Delta_Q^k \varphi^k_{Q, f} \rangle,  
\end{align}
where $|\varphi^k_{Q, f}| \le \mathbf{1}_Q$. Setting 
\begin{align*}
\Phi_{i, k} := \sum_{Q \in \D}  \langle |\Delta_Q^i f_1| \rangle_Q 
\langle |f_2| \rangle_Q \langle |f_4| \rangle_Q \, \Delta_Q^k \varphi^k_{Q, f_3}, 
\end{align*}
we use the fact $\Delta_{Q'}^k \Delta_Q^k f = \Delta_Q^k f \, \mathbf{1}_{\{Q'=Q\}}$ to arrive at  
\begin{align*}
S_{\D}^k \Phi_{i, k} 
\lesssim \bigg(\sum_{Q \in \D}  
\big| \langle |\Delta_Q^i f_1| \rangle_Q \langle |f_2| \rangle_Q  
\langle |f_4| \rangle_Q \big|^2 \mathbf{1}_Q \bigg)^{\frac12}, 
\end{align*}
which, together with \eqref{JSD-3} and that $w_3^{-p'_3} \in A_{3p'_3} \subset A_{\infty}$ (see Lemma \ref{lem:weight}), gives 
\begin{align*}
&|\langle \mathbf{A}_{\D}^{i, k} (f_1, f_2, f_3), f_4 \rangle| 
\le \sum_{Q \in \D} \langle |\Delta_Q^i f_1| \rangle_Q 
\langle |f_2| \rangle_Q  \langle f_3, \Delta_Q^k \varphi^k_{Q, f_3} \rangle \langle |f_4| \rangle_Q
\\
&\le \|f_3 w_3\|_{L^{p_3}} \|\Phi_{i, k} w_3^{-1}\|_{L^{p'_3}}
\lesssim \|f_3 w_3\|_{L^{p_3}} \|(S_{\D}^k \Phi_{i, k}) w_3^{-1}\|_{L^{p'_3}}
\\ 
&\lesssim \|f_3 w_3\|_{L^{p_3}} \bigg\|\bigg(\sum_{Q \in \D}  
\big| \langle |\Delta_Q^i f_1| \rangle_Q \langle |f_2| \rangle_Q  
\langle |f_4| \rangle_Q \big|^2 \mathbf{1}_Q \bigg)^{\frac12} w_3^{-1}\bigg\|_{L^{p'_3}}. 
\end{align*}
Fix a sequence of functions $\{f_3^Q\}_{Q \in \D}$ satisfying  $\big\|\big(\sum\limits_{Q \in \D} |f_3^Q|^2 \big)^{\frac12}\big\|_{L^{p_3}} \le 1$. By duality, it suffices to prove 
\begin{align}\label{SDK-2}
\mathscr{I} := \bigg|\int_{\Rn} \sum_{Q \in \D}  
\langle |\Delta_Q^i f_1| \rangle_Q \langle |f_2| \rangle_Q  
\langle |f_4| \rangle_Q \mathbf{1}_Q f_3^Q w_3^{-1}\, dx \bigg|
\lesssim \prod_{i=1, 2, 4} \|f_i w_i\|_{L^{p_i}}. 
\end{align}

By \eqref{SDK-1}, \eqref{JSD-3}, and the fact $w_1^{-p'_1} \in A_{3p'_1} \subset A_{\infty}$ (see Lemma \ref{lem:weight}), there holds 
\begin{align}\label{SDK-3}
\mathscr{I} 
&\le \sum_{Q \in \D}  \langle f_1, \Delta_Q^i \varphi_{Q, f_1}^i \rangle
\langle |f_2| \rangle_Q  \langle |f_4| \rangle_Q \langle |f_3^Q| w_3^{-1} \rangle_Q 
\\ \nonumber 
&=: \langle f_1 w_1, \Phi_i w_1^{-1} \rangle 
\le \|f_1 w_1\|_{L^{p_1}} \|\Phi_i w_1^{-1}\|_{L^{p'_1}}
\\ \nonumber
&\lesssim \|f_1 w_1\|_{L^{p_1}} \|(S_{\D}^i \Phi_i) w_1^{-1}\|_{L^{p'_1}}, 
\end{align}
where 
\begin{align*}
\Phi_i := \sum_{Q \in \D} \langle |f_2| \rangle_Q 
\langle |f_3^Q| w_3^{-1} \rangle_Q \langle |f_4| \rangle_Q \, 
\Delta_Q^i \varphi^i_{Q, f_1}, 
\end{align*}
and 
\begin{align*}
S_{\D}^i \Phi_i 
\lesssim \bigg(\sum_{Q \in \D} \big|\langle |f_2| \rangle_Q  
\langle |f_3^Q| w_3^{-1} \rangle_Q \langle |f_4| \rangle_Q
\big|^2 \mathbf{1}_Q \bigg)^{\frac12}
\le \bigg(\sum_{Q \in \D} \mathcal{M}(f_2, f_3^Q w_3^{-1}, f_4)^2 \bigg)^{\frac12}.
\end{align*}
Note that 
\begin{align*}
\frac{1}{p'_1} = \frac{1}{p_2} + \frac{1}{p_3} + \frac{1}{p_4}, \quad 
w_1^{-1} = w_2 w_3 w_4, 
\quad\text{ and }\quad 
(w_2, w_3, w_4) \in A_{(p_2, p_3, p_4)}.
\end{align*} 
Thus, 
\begin{align}\label{SDK-4}
&\|(S_{\D}^i \Phi_i) w_1^{-1}\|_{L^{p'_1}}
\le \bigg\|\bigg(\sum_{Q \in \D} 
\mathcal{M}(f_2, f_3^Q w_3^{-1}, f_4)^2 \bigg)^{\frac12} w_1^{-1}\bigg\|_{L^{p'_1}}
\\ \nonumber 
&\lesssim \prod_{i=2, 4} \|f_i w_i\|_{L^{p_i}} \times 
\bigg\|\bigg(\sum_{Q \in \D} |f_3^Q w_3^{-1}|^2 \bigg)^{\frac12} w_3\bigg\|_{L^{p_3}} 
\le \prod_{i=2, 4} \|f_i w_i\|_{L^{p_i}}, 
\end{align}
where we have used the vector-valued inequality 
\begin{align}\label{MVV}
\bigg\| \bigg(\sum_{k_1, \ldots, k_m} 
|\mathcal{M}(f_1^{k_1}, \ldots, f_m^{k_m})|^2 \bigg)^{\frac12}\bigg\|_{L^r(v^r)}
\lesssim \prod_{i=1}^m \bigg\| \bigg(\sum_{k_i} |f_i^{k_i}|^2 \bigg)^{\frac12}\bigg\|_{L^{r_i}(v_i^{r_i})},
\end{align}
for all $\vec{r}=(r_1, \ldots, r_m)  \in (1, \infty)^m$ and for all $\vec{v} =(v_1, \ldots, v_m) \in A_{\vec{r}}$, which follows from Theorems \ref{thm:RdF} and \ref{thm:Mbdd} by iterating square summation at the exponents $q_j =2$ and $q_i = \infty$ for $i \ne j$, and multilinear extrapolation. Therefore, \eqref{SDK-2} is a consequence of \eqref{SDK-3} and \eqref{SDK-4}. 
\end{proof}

\begin{proof}[\bf Proof of Theorem \ref{thm:SD-cpt}]
Without loss of generality, we may assume that $\widetilde{h}_I = h_I$. For simplicity, denote $\mathbf{S}_{\D} := \mathbf{S}_{\D}^{i, j, k}$. Fix $\frac1p = \frac{1}{p_1} + \frac{1}{p_2}$ with $p, p_1, p_2 \in (1, \infty)$. By Theorems \ref{thm:EP-Lp} and \ref{thm:SD}, it suffices to prove that $\mathbb{E}_{\w} \mathbf{S}_{\D_{\w}}$ is compact from $L^{p_1}(\Rn)  \times L^{p_2}(\Rn)$ to $L^p(\Rn)$. In light of Theorem \ref{thm:KRLpq}, we are deduced to showing the following estimates 
\begin{align}
\label{SDFK-1}
&\sup_{\substack{\|f_1\|_{L^{p_1}} \le 1 \\ \|f_2\|_{L^{p_2}} \le 1}} 
\|\mathbb{E}_{\w} \mathbf{S}_{\D_{\w}}(f_1, f_2) \|_{L^p}
\lesssim 1, 
\\ 
\label{SDFK-2}
\lim_{A \to \infty} &\sup_{\substack{\|f_1\|_{L^{p_1}} \le 1 \\ \|f_2\|_{L^{p_2}} \le 1}}
\|\mathbb{E}_{\w} \mathbf{S}_{\D_{\w}} (f_1, f_2) \, \mathbf{1}_{B(0, A)^c}\|_{L^p} 
= 0, 
\\ 
\label{SDFK-3}
\lim_{|v| \to 0} &\sup_{\substack{\|f_1\|_{L^{p_1}} \le 1 \\ \|f_2\|_{L^{p_2}} \le 1}}
\|\tau_v \, \mathbb{E}_{\w} \mathbf{S}_{\D_{\w}} (f_1, f_2) - 
\mathbb{E}_{\w} \mathbf{S}_{\D_{\w}} (f_1, f_2)\|_{L^p} 
=0. 
\end{align}
Note that \eqref{SDFK-1} immediately follows from Theorem \ref{thm:SD}.

To justify \eqref{SDFK-2}, given $N \in \N$, we define  
\begin{align*}
\mathbf{S}_{\D}^N (f_1, f_2) 
:= \sum_{Q \notin \D(N)} A_Q^{i, j, k} (f_1, f_2).  
\end{align*}
Observe that 
\begin{align}
\label{D-1}
& \Delta_{Q'}^k \big(A_Q^{i, j, k}(f_1, f_2) \big) 
= A_Q^{i, j, k}(f_1, f_2) \mathbf{1}_{\{Q'=Q\}} 
= A_Q^{i, j, k}(\Delta_{Q'}^i f_1, f_2), 
\\
\label{D-2}
& \Delta_Q^k \big(\mathbf{S}_{\D}^N(f_1, f_2) \big) 
= A_Q^{i, j, k}(f_1, f_2) \mathbf{1}_{\{Q \not\in \D(N)\}}, 
\\
\label{D-3}
& |A_Q^{i, j, k} (f_1, f_2)| 
\le \mathbf{F}_N \langle |f_1| \rangle_Q \langle |f_2| \rangle_Q \, \mathbf{1}_Q, 
\quad\text{for all }Q \in \D \setminus \D(N).  
\end{align}
In view of \eqref{D-3}, if $\mathbf{F}_N = 0$, then $\mathbf{S}_{\D}^N \equiv 0$; if $\mathbf{F}_N \ne 0$, the coefficients of $\mathbf{F}_N^{-1} \mathbf{S}_{\D}^N$ satisfy 
\begin{align*}
\big| \mathbf{F}_N^{-1} a_{I, J, K, Q} \mathbf{1}_{\{Q \in \D \setminus \D(N)\}} \big|
\le \frac{|I|^{\frac12} |J|^{\frac12} |K|^{\frac12}}{|Q|^2},
\end{align*}
which together with Theorem \ref{thm:SD} gives
\begin{align}\label{SNL2}
\|\mathbf{S}_{\D}^N (f_1, f_2)\|_{L^p}
\lesssim \mathbf{F}_N \|f_1\|_{L^{p_1}} \|f_2\|_{L^{p_2}},  
\end{align}
where the implicit constants are independent of $\D$, $N$, $f_1$, and $f_2$. 

Let $A \ge 2^{10}$ and $N := [\frac12 \log_2 A] \ge 2$. Note that for any $\w \in \Omega$ and $K \in \D_{\w}(N)$, 
\begin{align}\label{KN}
K \subset \{|x| \le (N+2) 2^N\} 
\subset \{|x| \le 2^{2N}\} 
\subset \{|x| \le A\}. 
\end{align}
Thus, by Minkowski's inequality, \eqref{SNL2}, and \eqref{KN}, 
\begin{align*}
& \|\mathbb{E}_{\w} \mathbf{S}_{\D_{\w}} (f_1, f_2) \mathbf{1}_{B(0, A)^c}\|_{L^p} 
\le \mathbb{E}_{\w} \|\mathbf{S}_{\D_{\w}} (f_1, f_2) \mathbf{1}_{B(0, A)^c}\|_{L^p} 
\\
&= \mathbb{E}_{\w} \|\mathbf{S}_{\D_{\w}}^N (f_1, f_2) \mathbf{1}_{B(0, A)^c}\|_{L^p} 
\lesssim \mathbf{F}_N \|f_1\|_{L^{p_1}} \|f_2\|_{L^{p_2}}, 
\end{align*}
where the implicit constant is independent of $A$, $\w$, $f_1$, and $f_2$. This gives \eqref{SDFK-2}.

In order to demonstrate \eqref{SDFK-3}, observe that
\begin{align}\label{HES-2}
\|(\tau_v \mathbb{E}_{\w} \mathbf{S}_{\D_{\w}} 
- \mathbb{E}_{\w} \mathbf{S}_{\D_{\w}}) (f_1, f_2)\|_{L^p} 
\le \mathbb{E}_{\w} \|(\tau_v \mathbf{S}_{\D_{\w}}
- \mathbf{S}_{\D_{\w}}) (f_1, f_2)\|_{L^p}.  
\end{align}
Fix an integer $a>3np+1$. Let $0 < |v| < 2^{-j-1-2a}$. Then there exists an integer $N = N(v) \ge 2$ such that $2^{-a (N+1)} \le 2^{j+1} |v| < 2^{- a N}$. We split  
\begin{align}\label{HFX} 
\|(\tau_v \mathbf{S}_{\D_{\w}} - \mathbf{S}_{\D_{\w}}) (f_1, f_2)\|_{L^p} 
\le \Xi_1(v; \w, f_1, f_2) + \Xi_2(v; \w, f_1, f_2), 
\end{align}
where 
\begin{align*}
\Xi_1(v; \w, f_1, f_2) 
&:= \bigg\|\sum_{Q \in \D_{\w} \setminus \D_{\w}(N)} 
\sum_{\substack{I \in \D_{\w, i}(Q) \\ J \in \D_{\w, j}(Q) \\ K \in \D_{\w, k}(Q)}}  
a_{I, J, K, Q} \langle f_1, h_I \rangle 
\langle f_2, \widetilde{h}_J \rangle (\tau_v h_K - h_K)\bigg\|_{L^p}, 
\\ 
\Xi_2(v; \w, f_1, f_2)
&:= \bigg\|\sum_{Q \in \D_{\w}(N)} 
\sum_{\substack{I \in \D_{\w, i}(Q) \\ J \in \D_{\w, j}(Q) \\ K \in \D_{\w, k}(Q)}} 
a_{I, J, K, Q} \langle f_1, h_I \rangle 
\langle f_2, \widetilde{h}_J \rangle (\tau_v h_K - h_K)\bigg\|_{L^p}. 
\end{align*}
By changing variables and \eqref{SNL2}, we obtain 
\begin{align}\label{XVF}
\Xi_1(v; \w, f_1, f_2) 
&\le \|\tau_v \mathbf{S}_{\D_{\w}}^N(f_1, f_2)\|_{L^p} 
+ \|\mathbf{S}_{\D_{\w}}^N(f_1, f_2)\|_{L^p}
\\ \nonumber 
&= 2 \|\mathbf{S}_{\D_{\w}}^N(f_1, f_2)\|_{L^p} 
\lesssim \mathbf{F}_N \|f_1\|_{L^{p_1}} \|f_2\|_{L^{p_2}}, 
\end{align}
where the implicit constant is independent of $v$, $\w$, $f_1$, and $f_2$. Since $|v| \to 0$ implies $N \to \infty$, \eqref{FIJK} and \eqref{XVF} imply 
\begin{align}\label{HFX-1}
\lim_{|v| \to 0} \sup_{\w \in \Omega} 
\sup_{\substack{\|f_1\|_{L^{p_1}} \le 1 \\ \|f_2\|_{L^{p_2}} \le 1}} 
\Xi_1(v; \w, f_1, f_2) =0. 
\end{align}
Observe that for any $N \ge 2$, 
\begin{align}\label{car-DN}
\# \D_{\w}(N) 
&\le \# \big\{K \in \D_{\w}: 2^{-N} \le \ell(K) \le 2^N, K \subset 2^{2N+1} \I \big\} 
\\ \nonumber 
&\le \sum_{-N \le k \le N} 2^{(2N+1)n}  2^{-kn} 
\le 2^{(2N+1)n} 2^{nN+1} 
= 2^{3nN+n+1}.
\end{align}
Thus, 
\begin{align*}
\Xi_2(v; \w, f_1, f_2) 
&\leq \sum_{Q \in \D_{\w}(N)}  \langle |f_1| \rangle_Q \langle |f_2| \rangle_Q 
\sum_{K \in \D_{\w, k}(Q)} |K|^{\frac12}  \|\tau_v h_K - h_K\|_{L^p} 
\\
&\lesssim \sum_{Q \in \D_{\w}(N)} |Q|^{-\frac1p} 
\|f_1\|_{L^{p_1}} \|f_2\|_{L^{p_2}}
\sum_{K \in \D_{\w, k}(Q)} |v|^{\frac1p} \ell(K)^{\frac1p (n-1)} 
\\
&\le \sum_{Q \in \D_{\w}(N)} 2^{k n} 2^{-k \frac{n}{p}} (2^{-k} \ell(Q))^{-\frac1p}
|v|^{\frac1p} \|f_1\|_{L^{p_1}} \|f_2\|_{L^{p_2}}
\\
&\le 2^{3nN+n+1} 2^{k(n-\frac{n}{p}+\frac1p)} 2^{\frac{N}{p}} 
\big(2^{-j-1} 2^{-aN} \big)^{\frac1p} 
\|f_1\|_{L^{p_1}} \|f_2\|_{L^{p_2}}
\\
&\le 2^{kn+n+1} 2^{\frac{N}{p}(3np+1-a)} 
\|f_1\|_{L^{p_1}} \|f_2\|_{L^{p_2}}.
\end{align*}
By the choice of $a$, we see that 
\begin{align}\label{HFX-2}
\lim_{|v| \to 0} \sup_{\w \in \Omega} 
\sup_{\substack{\|f_1\|_{L^{p_1}} \le 1 \\ \|f_2\|_{L^{p_2}} \le 1}} 
\Xi_2(v; \w, f_1, f_2) =0. 
\end{align}
Therefore, \eqref{SDFK-3} is a consequence of \eqref{HES-2}, \eqref{HFX}, \eqref{HFX-1}, and \eqref{HFX-2}.  
\end{proof}

\subsection{Weighted compactness of bilinear dyadic paraproducts}  
To prove Theorem \ref{thm:Pi-cpt}, it is necessary to first show the weighted boundedness of bilinear dyadic paraproducts. 

\begin{theorem}\label{thm:Pi}
For any $\b :=\{b_I\}_{I \in \D} \in \BMO_{\D}$ and $\mathrm{T}_{\b} \in \{\Pi_{\b}, \Pi_{\b}^{*1}, \Pi_{\b}^{*2}\}$,  there holds 
\begin{align*}
\|\mathrm{T}_{\b}\|_{L^{p_1}(w_1^{p_1}) \times L^{p_2}(w_2^{p_2}) \to L^p(w^p)} 
\lesssim \|\b\|_{\BMO_{\D}}, 
\end{align*}
for all $p_1, p_2 \in (1, \infty]$ and for all $(w_1, w_2) \in A_{(p_1, p_2)}$, where $\frac1p = \frac{1}{p_1} + \frac{1}{p_2}>0$, $w=w_1 w_2$, and the implicit constant is independent of $\b$ and $\D$.
\end{theorem}

\begin{proof} 
By duality and Theorem \ref{thm:RdF}, it suffices to consider $\Pi_{\b}$. The dyadic $\mathrm{H}^1$-$\BMO$ duality gives 
\begin{align}\label{PSD-1}
|\langle \Pi_{\b}(f_1, f_2), f_3 \rangle| 
\lesssim \|\b\|_{\BMO_{\D}} \|S_{\D} \Phi\|_{L^1}, 
\end{align}
where 
\begin{align*}
\Phi := \sum_{I \in \D} \langle f_1 \rangle_I 
\langle f_2 \rangle_I \langle f_3, h_I \rangle \, h_I. 
\end{align*}
Observe that 
\begin{align}\label{PSD-2}
S_{\D} \Phi 
= \bigg(\sum_{I \in \D} |\langle f_1 \rangle_I 
\langle f_2 \rangle_I \langle f_3, h_I \rangle|^2 \frac{\mathbf{1}_I}{|I|}\bigg)^{\frac12}
=: \mathcal{S}_{\D}(f_1, f_2, f_3).
\end{align}
By duality, \eqref{PSD-1}--\eqref{PSD-2}, and Theorem \ref{thm:RdF}, we are reduced to proving 
\begin{align}\label{SSS}
\|\mathcal{S}_{\D}(f_1, f_2, f_3) \, w\|_{L^p} 
\lesssim \prod_{i=1}^3 \|f_i w_i\|_{L^{p_i}} 
\end{align}
for all $p_1, p_2, p_3 \in (1, \infty)$ and for all $(w_1, w_2, w_3) \in A_{(p_1, p_2, p_3)}$, where $\frac1p = \frac{1}{p_1} + \frac{1}{p_2} + \frac{1}{p_3}$ and $w=w_1 w_2 w_3$. 

In view of Theorem \ref{thm:RdF} again, it is enough to show \eqref{SSS} in the case $p>1$. Given a sequence of functions $\{a_I(x)\}_{I \in \D}$, we have 
\begin{align*}
\bigg\|\bigg(\sum_{I \in \D} |a_I|^2 \bigg)^{\frac12} w \bigg\|_{L^p}
= \sup \bigg\{\bigg|\int_{\Rn} \sum_{I \in \D} a_I \, b_I \, w \, dx \bigg|: 
\bigg\|\bigg(\sum_{I \in \D} |b_I|^2 \bigg)^{\frac12}\bigg\|_{L^{p'}} \le 1\bigg\}. 
\end{align*}
Fix a sequence of functions $\{f_3^I\}_{I \in \D}$ satisfying  $\big\|\big(\sum\limits_{I \in \D} |f_3^I|^2 \big)^{\frac12}\big\|_{L^{p'}} \le 1$. It suffices to demonstrate 
\begin{align*}
\mathscr{J}
:= \bigg|\int_{\Rn} \sum_{I \in \D} \langle f_1 \rangle_I 
\langle f_2 \rangle_I \langle f_3, h_I \rangle \mathbf{1}_I |I|^{-\frac12} f_3^I \, w \, dx \bigg| 
\lesssim \prod_{i=1}^3 \|f_i w_i\|_{L^{p_i}}. 
\end{align*}
By H\"{o}lder's inequality, 
\begin{align}\label{JSD-1}
\mathscr{J}
&= \Big|\Big\langle f_3 w_3, \sum_{I \in \D} \langle f_1 \rangle_I 
\langle f_2 \rangle_I \langle f_3^I \, w \rangle_I |I|^{\frac12} \, h_I \, w_3^{-1} \Big\rangle\Big|
\\ \nonumber 
&=:|\langle f_3 w_3, \Psi w_3^{-1} \rangle|
\le \|f_3 w_3\|_{L^{p_3}} \|\Psi w_3^{-1}\|_{L^{p'_3}}, 
\end{align}
and 
\begin{align*}
S_{\D} \Psi 
= \bigg(\sum_{I \in \D} |\langle f_1 \rangle_I 
\langle f_2 \rangle_I \langle f_3^I w \rangle_I|^2 \mathbf{1}_I \bigg)^{\frac12}
\le \bigg(\sum_{I \in \D} \mathcal{M}(f_1, f_2, f_3^I w)^2 \bigg)^{\frac12}.
\end{align*}
Note that 
\begin{align*}
\frac{1}{p'_3} = \frac{1}{p_1} + \frac{1}{p_2} + \frac{1}{p'}, \quad 
w_3^{-1} = w_1 w_2 w^{-1}, \quad\text{and}\quad 
(w_1, w_2, w^{-1}) \in A_{(p_1, p_2, p')}.
\end{align*}
Then \eqref{MVV} implies 
\begin{align}\label{JSD-2}
&\|(S_{\D} \Psi) w_3^{-1}\|_{L^{p'_3}} 
\le \bigg\|\bigg(\sum_{I \in \D} 
\mathcal{M}(f_1, f_2, f_3^I w)^2 \bigg)^{\frac12} w_3^{-1}\bigg\|_{L^{p'_3}}
\\ \nonumber 
&\lesssim \prod_{i=1}^2 \|f_i w_i\|_{L^{p_i}} \times 
\bigg\|\bigg(\sum_{I \in \D} |f_3^I w|^2 \bigg)^{\frac12} w^{-1}\bigg\|_{L^{p'}}
\le \prod_{i=1}^2 \|f_i w_i\|_{L^{p_i}}. 
\end{align}
Hence, it follows from \eqref{JSD-1}--\eqref{JSD-2} and \eqref{JSD-3} that  
\begin{align*}
\mathscr{J}
\le \|f_3 w_3\|_{L^{p_3}} \|\Psi w_3^{-1}\|_{L^{p'_3}} 
\lesssim \|f_3 w_3\|_{L^{p_3}} \|(S_{\D} \Psi) w_3^{-1}\|_{L^{p'_3}} 
\lesssim \prod_{i=1}^3 \|f_i w_i\|_{L^{p_i}}. 
\end{align*}
This completes the proof. 
\end{proof}

Let us recall the dyadic Carleson embedding theorem from \cite[Theorem 4.5]{HP}. 
\begin{lemma}\label{lem:Car}
Let $\D$ be a dyadic grid. Assume that the nonnegative numbers $\{\alpha_Q\}_{Q \in \D}$ satisfy 
\begin{align*}
\sum_{Q' \in \D: \, Q' \subset Q} \alpha_{Q'} 
\le A \, |Q|, \quad\forall Q \in \D. 
\end{align*}
Then for all $p \in (1, \infty)$ and $f \in L^p(\Rn)$, 
\begin{align*}
\bigg(\sum_{Q \in \D} \alpha_Q \langle |f| \rangle_Q^p \bigg)^{\frac1p} 
\le A^{\frac1p} p' \|f\|_{L^p}. 
\end{align*}
\end{lemma}

\begin{proof}[\bf Proof of Theorem \ref{thm:Pi-cpt}] 
We only present the proof for $\mathbb{E}_{\w} \Pi_{\b_{\w}}$ because other cases can be shown by the same scheme and duality. In view of Theorems \ref{thm:EP-Lp} and \ref{thm:Pi}, it is enough to show that $\mathbb{E}_{\w} \Pi_{\b_{\w}}$ is compact from $L^4(\Rn)  \times L^4(\Rn)$ to $L^2(\Rn)$. With Theorem \ref{thm:KRLpq} in hand, this is reduced to proving that 
\begin{align}
\label{PibFK-1}
&\sup_{\substack{\|f_1\|_{L^4} \le 1 \\ \|f_2\|_{L^4} \le 1}} 
\|\mathbb{E}_{\w} \Pi_{\b_{\w}}(f_1, f_2) \|_{L^2}
\lesssim 1, 
\\ 
\label{PibFK-2}
\lim_{A \to \infty} &\sup_{\substack{\|f_1\|_{L^4} \le 1 \\ \|f_2\|_{L^4} \le 1}}
\|\mathbb{E}_{\w} \Pi_{\b_{\w}} (f_1, f_2) \, \mathbf{1}_{B(0, A)^c}\|_{L^2} 
= 0, 
\\ 
\label{PibFK-3}
\lim_{|v| \to 0} &\sup_{\substack{\|f_1\|_{L^4} \le 1 \\ \|f_2\|_{L^4} \le 1}}
\|\tau_v \, \mathbb{E}_{\w} \Pi_{\b_{\w}} (f_1, f_2) - 
\mathbb{E}_{\w} \Pi_{\b_{\w}} (f_1, f_2)\|_{L^2} 
=0. 
\end{align}
Theorem \ref{thm:Pi} gives \eqref{PibFK-1} at once.

Denote $\b_{\w}^N := \big\{ b_I \mathbf{1}_{\{I \in \D_{\w} \setminus \D_{\w}(N) \}} \big\}_{I \in \D_{\w}}$. Then our assumption gives 
\begin{align}\label{ANB}
\lim_{N \to \infty} \sup_{\w \in \Omega} \|\b_{\w}^N\|_{\BMO_{\D_{\w}}} = 0. 
\end{align}
As did in \eqref{KN}, Minkowski's inequality and Theorem \ref{thm:Pi} yield  
\begin{align*}
&\|\mathbb{E}_{\w} \Pi_{\b_{\w}} (f_1, f_2) \mathbf{1}_{B(0, A)^c}\|_{L^2} 
\le \mathbb{E}_{\w} \|\Pi_{\b_{\w}} (f_1, f_2) \mathbf{1}_{B(0, A)^c}\|_{L^2} 
\\
&= \mathbb{E}_{\w} \|\Pi_{\b_{\w}^N} (f_1, f_2) \mathbf{1}_{B(0, A)^c}\|_{L^2} 
\le \mathbb{E}_{\w} \|\Pi_{\b_{\w}^N} (f_1, f_2)\|_{L^2}  
\\
&\lesssim \mathbb{E}_{\w} \|\b_{\w}^N\|_{\BMO_{\D_{\w}}} \|f_1\|_{L^4} \|f_2 \|_{L^4} 
\le \sup_{\w \in \Omega} \|\b_{\w}^N\|_{\BMO_{\D_{\w}}} \|f_1\|_{L^4} \|f_2 \|_{L^4},  
\end{align*}
which together with \eqref{ANB} implies \eqref{PibFK-2}. 

To proceed, note that 
\begin{align}\label{HPB-1}
\|(\tau_v \mathbb{E}_{\w} \Pi_{\b_{\w}} 
- \mathbb{E}_{\w} \Pi_{\b_{\w}}) (f_1, f_2)\|_{L^2} 
\le \mathbb{E}_{\w} \|(\tau_v \Pi_{\b_{\w}} - \Pi_{\b_{\w}}) (f_1, f_2)\|_{L^2}
\end{align}
Fix an integer $a > 3n+1$. Let $0 < |v| \le 2^{-2a}$. Then there exists an integer $N=N(v) \ge 2$ so that $2^{-a(N+1)} < |v| \le 2^{-aN}$. We split 
\begin{align}\label{HPB-2}
\|(\tau_v \Pi_{\b_{\w}} - \Pi_{\b_{\w}}) (f_1, f_2)\|_{L^2} 
&\le \Upsilon_1(v; \w, f_1, f_2) + \Upsilon_2(v; \w, f_1, f_2), 
\end{align}
where 
\begin{align*}
\Upsilon_1(v; \w, f_1, f_2)
&:= \bigg\|\sum_{I \in \D_{\w} \setminus \D_{\w}(N)}  b_I
\langle f_1 \rangle_I \langle f_2 \rangle_I  (\tau_v h_I - h_I) \bigg\|_{L^2}, 
\\
\Upsilon_2(v; \w, f_1, f_2) 
&:= \bigg\|\sum_{I \in \D_{\w}(N)} b_I  
\langle f_1 \rangle_I  \langle f_2 \rangle_I (\tau_v h_I - h_I) \bigg\|_{L^2}. 
\end{align*}
For the first term, we use Theorem \ref{thm:Pi} to deduce  
\begin{align*}
&\Upsilon_1(v; \w, f_1, f_2) 
\le 2 \bigg\|\sum_{I \in \D_{\w} \setminus \D_{\w}(N)} b_I  
\langle f_1 \rangle_I \langle f_2 \rangle_I  h_I \bigg\|_{L^2}
\\
&= 2 \| \Pi_{\b_{\w}^N} (f_1, f_2)\|_{L^2}
\lesssim \sup_{\w \in \Omega} \|\b_{\w}^N\|_{\BMO_{\D_{\w}}}  
\|f_1\|_{L^4} \|f_2\|_{L^4}, 
\end{align*}
where the implicit constant is independent of $\w$. This along with \eqref{ANB} gives 
\begin{align}\label{HO-1}
\lim_{|v| \to 0} \sup_{\w \in \Omega} 
\sup_{\substack{\|f_1\|_{L^4} \le 1 \\ \|f_2\|_{L^4} \le 1}} 
\Upsilon_1(v; \w, f_1, f_2) = 0.
\end{align}
To estimate $\Upsilon_2$, observe that 
\begin{align}\label{bcar}
\sum_{I \in \D_{\w}:\, I \subset Q} |b_I|^2 
\le \|\b\|_{\BMO_{\D_{\w}}}^2 \, |Q|
\le |Q|, 
\quad\forall Q \in \D_{\w}. 
\end{align}
Then, by the Cauchy-Schwarz inequality, \eqref{car-DN}, \eqref{bcar}, and Lemma \ref{lem:Car}, 
\begin{align*}
\Upsilon_2(v; \w, f_1, f_2)
& \le \sum_{I \in \D_{\w}(N)} |b_I| 
|\langle f_1 \rangle_I| |\langle f_2 \rangle_I| \|\tau_v h_I - h_I\|_{L^2}
\\
& \lesssim \sum_{I \in \D_\w(N)} |b_I| 
|\langle f_1 \rangle_I| |\langle f_2 \rangle_I| 
|v|^{\frac12} \ell(I)^{-\frac12} 
\\
& \le |v|^{\frac12} 2^{\frac{N}{2}} \big[\# \D_\w(N) \big]^{\frac12} 
\bigg[\sum_{I \in \D_{\w}} |b_I|^2 
|\langle f_1 \rangle_I|^2 |\langle f_2 \rangle_I|^2 \bigg]^{\frac12}  
\\
& \le 2^{-a\frac{N}{2}} 2^{\frac{N}{2}} 2^{\frac12 (3nN+n+1)}
\prod_{i=1}^2 \bigg[\sum_{I \in \D_{\w}} 
|b_I|^2 |\langle f_i \rangle_I |^4 \bigg]^{\frac14}
\\
& \lesssim 2^{\frac{N}{2}(3n+1-a)} \|f_1\|_{L^4} \|f_2\|_{L^4}, 
\end{align*}
where the implicit constants are independent of $v$, $\w$, $f_1$, and $f_2$. This immediately implies 
\begin{align}\label{HO-2}
\lim_{|v| \to 0} \sup_{\w \in \Omega} 
\sup_{\substack{\|f_1\|_{L^4} \le 1 \\ \|f_2\|_{L^4} \le 1}} 
\Upsilon_2(v; \w, f_1, f_2) = 0.
\end{align}
Consequently, \eqref{PibFK-3} follows from \eqref{HPB-1}--\eqref{HO-2}. The proof is complete. 
\end{proof}

\section{Necessity of hypotheses \eqref{H1}--\eqref{H3}}\label{sec:nec}
This section is devoted to showing Theorems \ref{thm:CCZK}, \ref{thm:WCP}, and \ref{thm:T1CMO}, which immediately give ${\rm (c) \Longrightarrow (a)}$ in Theorem \ref{thm:cpt}.

\subsection{Compact Calder\'{o}n--Zygmund kernel}\label{sec:CCZK}
First, let us demonstrate Theorem \ref{thm:CCZK}. To this end, we present two lemmas below. 

\begin{lemma}\label{lem:size} 
Assume that $K$ satisfies the smoothness conditions \eqref{eq:Holder-1}--\eqref{eq:Holder-3} and that 
\begin{align*}
\lim_{|x-y| + |x-z| \to \infty} K(x, y, z) = 0.
\end{align*}  
Then $K$ satisfies the size condition \eqref{eq:Size} with another triple $(F'_1, F'_2, F'_3) \in \F$.  
\end{lemma}

\begin{proof}
We claim that 
\begin{align}\label{FF-Holder} 
|K(x, y, z) - K(x', y', z')| 
\leq F(x, y, z) \frac{(|x-x'| + |y-y'| + |z-z'|)^{\delta}}{(|x-y| + |x-z|)^{2n+\delta}}  
\end{align}
whenever $|x-x'| + |y-y'| + |z-z'| \leq (|x-y| + |x-z|)/8$, where 
\begin{align*}
F(x, y, z) := F_1(|x-y| + |x-z|) F_2(|x-y| + |x-z|) F_3(|x+y| + |x+z|), 
\end{align*}
and $(F_1, F_2, F_3) \in \F$ with $F_1$ being monotone increasing, $F_2$ and $F_3$ being monotone decreasing. Indeed, for such $x$, $y$ and $z$, the condition \eqref{eq:Holder-1} and Lemma \ref{lem:PP} give 
\begin{align}\label{KKF-1}  
|K(x, y, z) - K(x', y, z)| 
\leq F_1(x, y, z) \frac{|x-x'|^{\delta}}{(|x-y| + |x-z|)^{2n+\delta}},   
\end{align}
where 
\begin{align*}
F_1(x, y, z) := F_{1, 1}(|x-y| + |x-z|) F_{1, 2}(|x-y| + |x-z|) F_{1, 3}(|x+y| + |x+z|), 
\end{align*}
and $(F_{1, 1}, F_{1, 2}, F_{1, 3}) \in \F$ with $F_{1, 1}$ being monotone increasing, $F_{1, 2}$ and $F_{1, 3}$ being monotone decreasing. It is easy to verify  that 
\begin{align}\label{XY34}  
\frac34 \le \frac{|x'-y| + |x'-z|}{|x-y| + |x-z|} \le \frac54, 
\end{align}
and 
\begin{align*} 
|y-y'| \leq (|x-y| + |x-z|)/8 \leq (|x'-y| + |x'-z|)/6, 
\end{align*}
which together with \eqref{eq:Holder-2} implies 
\begin{align}\label{KFZ-1}  
|K(x', y, z) - K(x', y', z)| 
&\leq F_1(x', y, z) \frac{|y-y'|^{\delta}}{(|x'-y| + |x'-z|)^{2n+\delta}}  
\\ \nonumber 
&\leq 2^{2n+\delta}F_1(x', y, z)  \frac{|y-y'|^{\delta}}{(|x-y| + |x-z|)^{2n+\delta}}. 
\end{align}
Define
\begin{align*}
F_{2, 1}(t) &:= \sup_{\substack{|x-y| + |x-z| \le t \\ 8|x-x'| \le |x-y| + |x-z|}} 
\big(2^{2n+\delta} F_1(x', y, z) \big)^{\frac13},
\\ 
F_{2, 2}(t) &:= \sup_{\substack{|x-y| + |x-z| \ge t \\ 8|x-x'| \le |x-y| + |x-z|}} 
\big(2^{2n+\delta} F_1(x', y, z) \big)^{\frac13},
\\
F_{2, 3}(t) &:= \sup_{\substack{|x+y| + |x+z| \ge t \\ 8|x-x'| \le |x-y| + |x-z|}} 
\big(2^{2n+\delta} F_1(x', y, z) \big)^{\frac13},
\end{align*}
and then set 
\begin{align*}
F_2(x, y, z) := F_{2, 1}(|x-y| + |x-z|) F_{2, 2}(|x-y| + |x-z|) F_{2, 3}(|x+y| + |x+z|). 
\end{align*}
Thus, $F_{2, 1}$ is bounded and monotone increasing, both $F_{2, 2}$ and $F_{2, 3}$ are bounded and monotone decreasing, and 
\begin{align}\label{KFZ-2}
2^{2n+\delta} F_1(x', y, z) \le F_2(x, y, z). 
\end{align}
Together with \eqref{XY34}, the property $\lim_{t \to 0} F_{1, 1}(t) = \lim_{t \to \infty} F_{1, 2}(t) =0$ implies 
$\lim_{t \to 0} F_{2, 1}(t) = \lim_{t \to \infty} F_{2, 2}(t) =0$. Moreover, if $|x+y| + |x+z| \le |x-y| + |x-z|$, then $\lim_{t \to \infty} F_{1, 2}(t) =0$ gives $\lim_{t \to \infty} F_{2, 3}(t) =0$. If $|x+y| + |x+z| \ge |x-y| + |x-z|$, then 
\begin{align*}
&|x'+y| + |x'+z| 
\ge |x+y| + |x+z| - 2|x-x'| 
\\
&\ge |x+y| + |x+z| - (|x-y| + |x-z|)/4
\ge 3(|x+y| + |x+z|)/4, 
\end{align*}
which along with $\lim_{t \to \infty} F_{1, 3}(t) =0$ yields $\lim_{t \to \infty} F_{2, 3}(t) = 0$. This shows $(F_{2, 1}, F_{2, 2}, F_{2, 3}) \in \F$. Collecting \eqref{XY34}--\eqref{KFZ-2}, we conclude 
\begin{align}\label{KKF-2}  
|K(x', y, z) - K(x', y', z)| 
\leq F_2(x, y, z) \frac{|y-y'|^{\delta}}{(|x-y| + |x-z|)^{2n+\delta}}.  
\end{align}
Much as above, there exists $(F_{3, 1}, F_{3, 2}, F_{3, 3}) \in \F$ such that 
\begin{align}\label{KKF-3}   
|K(x', y', z) - K(x', y', z')| 
&\leq F_3(x, y, z) \frac{|z-z'|^{\delta}}{(|x-y| + |x-z|)^{2n+\delta}},   
\end{align}
where 
\begin{align*}
F_3(x, y, z) := F_{3, 1}(|x-y| + |x-z|) F_{3, 2}(|x-y| + |x-z|) F_{3, 3}(|x+y| + |x+z|). 
\end{align*}
Hence, as did in the proof of Lemma \ref{lem:PP}, \eqref{FF-Holder} is a consequence of \eqref{KKF-1}, \eqref{KKF-2}, and \eqref{KKF-3}. 

We are going to use \eqref{FF-Holder} to show the size condition \eqref{eq:Size}. Let $x, y, z \in \Rn$ with $x \neq y$ or $x \neq z$. Let $x_0=x$, $y_0=y$, and $z_0=z$. First, in the situation $n=1$, let us consider four cases: 
\begin{align*}
{\rm I.} \, x \ge y, x \ge z; \quad 
{\rm II.} \, x < y, x \ge z; \quad 
{\rm III. } \, x \ge y, x < z; \quad 
{\rm IV. } \, x < y, x < z.
\end{align*}
In case I, let $a_1=7/6$, and for each $k \ge 1$, define 
\begin{align*}
x_k &= x_{k-1} + (|x_{k-1} - y_{k-1}| + |x_{k-1} - z_{k-1}|)/24, 
\\ 
y_k &= y_{k-1} - (|x_{k-1} - y_{k-1}| + |x_{k-1} - z_{k-1}|)/24, 
\\
z_k &= z_{k-1} - (|x_{k-1} - y_{k-1}| + |x_{k-1} - z_{k-1}|)/24. 
\end{align*}
Then for any $k \ge 0$, $x_k \ge y_k$, $x_k \ge z_k$, and 
\begin{align*}
|x_k - y_k| + |x_k - z_k| & = a_1 (|x_{k-1} - y_{k-1}| + |x_{k-1} - z_{k-1}|), 
\\ 
|x_k + y_k| + |x_k + z_k| &= |x_{k-1} + y_{k-1}| + |x_{k-1} + z_{k-1}|, 
\end{align*}
which gives 
\begin{align}
\label{XK-1} |x_k - y_k| + |x_k - z_k| & = a_1^k (|x - y| + |x - z|), 
\\ 
\label{XK-2} |x_k + y_k| + |x_k + z_k| &= |x + y| + |x + z|, 
\\ 
\label{XK-3}
|x_{k-1} - x_k| + |y_{k-1} - y_k| + |z_{k-1} - z_k| 
& = (|x_{k-1} - y_{k-1}| + |x_{k-1} - z_{k-1}|)/8. 
\end{align}
This along with \eqref{FF-Holder} implies 
\begin{align*}
|K(x_{k-1}, y_{k-1}, z_{k-1})| 
\le |K(x_k, y_k, z_k)| + 
\frac{F(x_{k-1}, y_{k-1}, z_{k-1})}{(|x_{k-1} - y_{k-1}| + |x_{k-1} - z_{k-1}|)^2}. 
\end{align*}
Iterating the inequality above, we have 
\begin{align*}
|K(x, y, z)| 
\le \lim_{k \to \infty}|K(x_k, y_k, z_k)| + 
\sum_{k=0}^{\infty} \frac{F(x_k, y_k, z_k)}{(|x_k - y_k| + |x_k - z_k|)^2}. 
\end{align*}
Choosing  
\begin{align}\label{KFA}
F'_3 (t) : = F_3(t) \quad\text{ and }\quad 
F'_i(t) := \sum_{k =0}^{\infty} a_1^{-k} F_i(a_1^k t), \quad i=1, 2, 
\end{align}
we see that $(F'_1, F'_2, F'_3) \in \F$, and \eqref{XK-1}--\eqref{XK-2} imply 
\begin{align}\label{KXF}
|K(x, y, z)| 
\lesssim \frac{F'(x, y, z)}{(|x - y| + |x - z|)^2}, 
\end{align}
where 
\begin{align*}
F'(x, y, z) := F'_1(|x-y| + |x-z|) F'_2(|x-y| + |x-z|) F'_3(|x+y| + |x+z|). 
\end{align*}
This shows \eqref{eq:Size}. In case II, let $a_2=13/12$, and we give two ways (called II-1, II-2) for later use. If $|x+z| \ge |x+y|$, we use way II-1: for each $k \ge 1$, define 
\begin{align*}
x_k &= x_{k-1} + (|x_{k-1} - y_{k-1}| + |x_{k-1} - z_{k-1}|)/24, 
\\ 
y_k &= y_{k-1} + (|x_{k-1} - y_{k-1}| + |x_{k-1} - z_{k-1}|)/24, 
\\
z_k &= z_{k-1} - (|x_{k-1} - y_{k-1}| + |x_{k-1} - z_{k-1}|)/24. 
\end{align*}
Then it is easy to check that for any $k \ge 0$, $x_k < y_k$, $x_k \ge z_k$, and 
\begin{align}
\label{XK-11}
|x_k - y_k| + |x_k - z_k| & = a_2^k (|x - y| + |x - z|), 
\\ 
\label{XK-12}
|x_k + y_k| + |x_k + z_k| \ge |x_k + z_k| 
& = |x + z| \ge (|x + y| + |x + z|)/2, 
\\ 
\label{XK-13}
|x_{k-1} - x_k| + |y_{k-1} - y_k| + |z_{k-1} - z_k| 
&= (|x_{k-1} - y_{k-1}| + |x_{k-1} - z_{k-1}|)/8. 
\end{align}
Much as in case I, one can obtain \eqref{eq:Size}. If $|x+z| < |x+y|$, we use way II-2: for each $k \ge 1$, if we define 
\begin{align*}
x_k &= x_{k-1} - (|x_{k-1} - y_{k-1}| + |x_{k-1} - z_{k-1}|)/24, 
\\ 
y_k &= y_{k-1} + (|x_{k-1} - y_{k-1}| + |x_{k-1} - z_{k-1}|)/24, 
\\
z_k &= z_{k-1} - (|x_{k-1} - y_{k-1}| + |x_{k-1} - z_{k-1}|)/24. 
\end{align*}
then for any $k \ge 0$, $x_k < y_k$, $x_k \ge z_k$, and 
\begin{align}
\label{XK-21}
|x_k - y_k| + |x_k - z_k| & = a_2^k (|x - y| + |x - z|), 
\\ 
\label{XK-22}
|x_k + y_k| + |x_k + z_k| \ge |x_k + y_k| 
& = |x + y| \ge (|x + y| + |x + z|)/2, 
\\ 
\label{XK-23}
|x_{k-1} - x_k| + |y_{k-1} - y_k| + |z_{k-1} - z_k| 
&= (|x_{k-1} - y_{k-1}| + |x_{k-1} - z_{k-1}|)/8. 
\end{align}
In either subcase, one can prove \eqref{eq:Size} by the argument in case I with $a_1$ and $F'_3(t)$ replaced by $a_2$ and $F_3(t/2)$. Furthermore, case III is analogous to case II; and replacing $x, y, z, x_k, y_k, z_k$ by $-x, -y, -z, -x_k, -y_k, -z_k$ in case I, we show \eqref{eq:Size} in case IV. 

Next, in the situation $n \ge 2$, for the sake of simplicity, we treat the most essential case: $n=2$ and $(x, y, z) \in \R^{2 \times 3}$ with $x_1 \ge y_1$, $x_1 \ge z_1$, $x_2 < y_2$, and $x_2 \ge z_2$. In the case $|x_2 + y_2| \le |x_2 + z_2|$, for each $k \ge 0$, define the sequences $\{x_1^k\}$, $\{y_1^k\}$, and $\{z_1^k\}$ as in case I, and the sequences $\{x_2^k\}$, $\{y_2^k\}$, and $\{z_2^k\}$ as in case II-1. Write $x_k = (x_1^k, x_2^k)$, $y_k = (y_1^k, y_2^k)$, and $z_k = (z_1^k, z_2^k)$. Note that 
\begin{align*}
& |x-y| + |x-z|
\le |x_k - y_k| + |x_k - z_k|
\le a_1^k (|x-y| + |x-z|),
\\
& |x_k - y_k| + |x_k - z_k|
\ge a_2^k (|x-y| + |x-z|)/2.
\end{align*}
Then by \eqref{FF-Holder}, 
\begin{align*}
|K(x_{k-1}, y_{k-1}, z_{k-1})| 
\le |K(x_k, y_k, z_k)| + 
\frac{F(x_{k-1}, y_{k-1}, z_{k-1})}{(|x_{k-1} - y_{k-1}| + |x_{k-1} - z_{k-1}|)^{2n}}, 
\end{align*}
which gives 
\begin{align}\label{KXF-1}
|K(x, y, z)| 
\le \sum_{k=0}^{\infty} \frac{F(x_k, y_k, z_k)}{(|x_k - y_k| + |x_k - z_k|)^{2n}}, 
\end{align}
where the estimates \eqref{XK-1}--\eqref{XK-13} imply 
\begin{align*}
F_3(|x_k + y_k| + |x_k + z_k|) 
\le F_3([|x + y| + |x + z|]/2)  
\end{align*}
and 
\begin{align}\label{KXF-2}
&F(x_k, y_k, z_k) 
\\ \nonumber 
&:= F_1(|x_k - y_k| + |x_k - z_k|) F_2(|x_k - y_k| + |x_k - z_k|) F_3(|x_k + y_k| + |x_k + z_k|) 
\\ \nonumber
&\le F_1(a_1^k (|x - y| + |x - z|)) F_2(|x - y| + |x - z|) F_3([|x + y| + |x + z|]/2).  
\end{align}
In the case $|x_2 + y_2| > |x_2 + z_2|$, for each $k \ge 0$, define the sequences $\{x_1^k\}$, $\{y_1^k\}$, and $\{z_1^k\}$ as in case I, and the sequences $\{x_2^k\}$, $\{y_2^k\}$, and $\{z_2^k\}$ as in case II-2. Write $x_k = (x_1^k, x_2^k)$, $y_k = (y_1^k, y_2^k)$, and $z_k = (z_1^k, z_2^k)$. As above, there holds 
\begin{align}\label{KXF-3}
|K(x, y, z)| 
\le \sum_{k=0}^{\infty} \frac{F(x_k, y_k, z_k)}{(|x_k - y_k| + |x_k - z_k|)^{2n}}, 
\end{align}
where the estimates \eqref{XK-1}--\eqref{XK-3} and \eqref{XK-21}--\eqref{XK-23} imply 
\begin{align*}
F_3(|x_k + y_k| + |x_k + z_k|) 
\le F_3( [|x + y| + |x + z|]/2 )  
\end{align*}
and 
\begin{align}\label{KXF-4}
&F(x_k, y_k, z_k) 
\\ \nonumber
&:= F_1(|x_k - y_k| + |x_k - z_k|) F_2(|x_k - y_k| + |x_k - z_k|) F_3(|x_k + y_k| + |x_k + z_k|) 
\\ \nonumber
&\le F_1(a_1^k (|x - y| + |x - z|)) F_2(|x - y| + |x - z|) F_3([|x + y| + |x + z|]/2).  
\end{align}
With \eqref{KXF-1}--\eqref{KXF-4} in hand, similarly to \eqref{KFA}, we conclude \eqref{KXF} with 
\begin{align*}
F'_1(t) := 2^{2n} \sum_{k=0}^{\infty} a_2^{-2nk} F_1(a_1^k t), \quad
F'_2(t) := F_2(t), \quad \text{ and } \quad
F'_3(t) := F_3(t/2).
\end{align*}
The proof is complete.  
\end{proof}

Given $\lambda>0$, $y \in \Rn$, and a function $\Phi$ on $\Rn$, we denote 
\begin{align*}
\tau_y \Phi(x) := \Phi(x-y) 
\quad\text{ and }\quad 
\Phi_{\lambda}(x) := \lambda^{-n} \Phi(\lambda^{-1} x), \quad x \in \Rn.
\end{align*}

\begin{lemma}\label{lem:TTK}
Let $T$ be a bilinear operator associated with a standard bilinear Calder\'{o}n--Zygmund kernel $K$ with the parameter $\delta \in (0, 1]$. Let $0 < \varepsilon < 2^{-\delta} \|K\|_{\rm{CZ}(\delta)}$, $x' \in \Rn$ and $(x, y, z) \in \R^{3n} \setminus \Delta$ such that $|x-x'| \le \frac14 \max\{|x-y|, |x-z|\}$. Then for any $0 < \lambda_i < (\varepsilon \|K\|_{\rm{CZ}(\delta)}^{-1})^{\frac{1}{\delta}} |x-x'|$, $i=1,2,3$, 
\begin{align}
\label{eq:TTT-1}
|\langle T(\tau_y \Phi_{\lambda_2}, \tau_z \Phi_{\lambda_3}), \tau_x \Phi_{\lambda_1} \rangle - K(x, y, z) | 
&\lesssim \frac{\varepsilon \, |x-x'|^{\delta}}{(|x-y|+|x-z|)^{2n+\delta}},  
\\
\label{eq:TTT-2} 
|\langle T(\tau_y \Phi_{\lambda_2}, \tau_z \Phi_{\lambda_3}), 
\tau_x \Phi_{\lambda_1} - \tau_{x'} \Phi_{\lambda_1} \rangle|
& \lesssim \frac{|x-x'|^{\delta}}{(|x-y|+|x-z|)^{2n+\delta}}, 
\end{align}
where $\Phi$ is a positive smooth function such that $\supp \Phi \subset \I$ and $\int_{\Rn} \Phi \, dx=1$. 
\end{lemma}

\begin{proof}
Fix $0 < \lambda_i < (\varepsilon \|K\|_{\rm{CZ}(\delta)}^{-1})^{\frac{1}{\delta}} |x-x'|$, $i=1,2,3$. For any $v_i \in \supp(\Phi_{\lambda_i})$, we see that 
\begin{align*}
4 \max_i |v_i| 
\le 2 \max_i \lambda_i
\le |x-x'|
\le \max\{|x-y|, |x-z|\}/4, 
\end{align*}
which implies 
\begin{align}
\label{xxvv-1} |x-x'| & \le \max\{|x +v_1 - y - v_2|, |x+v_1-z-v_3|\}/2, 
\\ 
\label{xxvv-2} |v_1| & \le \max\{|x - y - v_2|, |x-z-v_3|\}/2, 
\\
\label{xxvv-3} |v_2|  & \le \max\{|x-y|, |x-z-v_3|\}/2, 
\\ 
\label{xxvv-4} |v_3| & \le \max\{|x-y|, |x-z|\}/2.
\end{align}
Then by \eqref{eq:Holder-1}--\eqref{eq:Holder-3} and \eqref{xxvv-2}--\eqref{xxvv-4},  
\begin{align*}
&|K(x+v_1, y+v_2, z+v_3) - K(x, y, z)| 
\\ 
&\le |K(x+v_1, y+v_2, z+v_3) - K(x, y+v_2, z+v_3)| 
\\
&\quad+ |K(x, y+v_2, z+v_3) - K(x, y, z+v_3)| 
\\
&\quad + |K(x, y, z+v_3) - K(x, y, z)| 
\\
&\lesssim \frac{|v_1|^{\delta}}{(|x-y-v_2|+|x-z-v_3|)^{2n+\delta}}
\\
&\quad + \frac{|v_2|^{\delta}}{(|x-y|+|x-z-v_3|)^{2n+\delta}}
+ \frac{|v_3|^{\delta}}{(|x-y|+|x-z|)^{2n+\delta}}
\\
&\lesssim \frac{\max_i \lambda_i^{\delta}}{(|x-y|+|x-z|)^{2n+\delta}} 
\lesssim \frac{\varepsilon \, |x-x'|^{\delta}}{(|x-y|+|x-z|)^{2n+\delta}}. 
\end{align*}
Hence, we arrive at 
\begin{align*}
\text{LHS of } \eqref{eq:TTT-1} 
= \bigg|\int_{\R^{3n}} [K(x+v_1, y+v_2, z+v_3) - K(x, y, z)] 
\prod_{i=1}^3 \Phi_{\lambda_i}(v_i) \, dv_i \bigg|
\\
\lesssim \int_{\R^{3n}} \frac{\varepsilon |x-x'|^{\delta}}{(|x-y|+|x-z|)^{2n+\delta}} 
\prod_{i=1}^3 \Phi_{\lambda_i}(v_i) \, dv_i 
\leq \frac{\varepsilon |x-x'|^{\delta}}{(|x-y|+|x-z|)^{2n+\delta}}. 
\end{align*}
To prove \eqref{eq:TTT-2}, we denote
\begin{align*}
\mathbf{K}(v_1, v_2, v_3)
:= K(x+v_1, y+v_2, z+v_3) - K(x'+v_1, y+v_2, z+v_3),
\end{align*}
and 
use \eqref{eq:Holder-1} and \eqref{xxvv-1} to obtain 
\begin{align*}
|\mathbf{K}(v_1, v_2, v_3)| 
\lesssim \frac{|x-x'|^{\delta}}{(|x+v_1-y-v_2|+|x+v_1-z-v_3|)^{2n+\delta}}
\simeq \frac{|x-x'|^{\delta}}{(|x-y|+|x-z|)^{2n+\delta}}. 
\end{align*}
As a consequence, 
\begin{align*}
\text{LHS of } \eqref{eq:TTT-2} 
= \bigg|\int_{\R^{3n}} \mathbf{K}(v_1, v_2, v_3) \prod_{i=1}^3 \Phi_{\lambda_i}(v_i) \, dv_i \bigg|
\lesssim \frac{|x-x'|^{\delta}}{(|x-y|+|x-z|)^{2n+\delta}}. 
\end{align*}
The proof is complete. 
\end{proof}

\begin{proof}[\bf Proof of Theorem \ref{thm:CCZK}]
By symmetry and Lemma \ref{lem:size}, it suffices to prove that there exists $\delta' \in (0, \delta)$ such that 
\begin{align*}
\sup_{x' \in \Rn: x' \neq x \atop |x-x'| \leq \max\{|x-y|, |x-z| \}/2} \mathcal{E}(x, y, z; x') \to 0, 
\end{align*}
whenever $|x-x'| \to 0$, or $|x-y|+|x-z| \to \infty$, or $|x+y|+|x+z| \to \infty$ and $|x-y|+|x-z| \simeq 1$, where 
\begin{align*}
\mathcal{E}(x, y, z; x') 
:= |x-x'|^{-\delta'} \frac{|K(x, y, z) - K(x', y, z)| }{(|x-y| + |x-z|)^{-(2n+\delta')}}, 
\end{align*}
since the smoothness condition \eqref{eq:Holder-1} gives   
\begin{align*}
\mathcal{E}(x, y, z; x') 
\lesssim \frac{|x-x'|^{\delta-\delta'}}{(|x-y| + |x-z|)^{\delta-\delta'}} 
\lesssim 1, 
\end{align*}
whenever $|x-x'| \leq \max\{|x-y|, |x-z| \}/2$. The above immediately implies that 
\begin{equation}\label{Exk}
\begin{array}{c}
\lim\limits_{k \to \infty} \mathcal{E}(x_k, y_k, z_k; x'_k)=0 
\text{ for any sequence $\{(x_k, y_k, z_k)\}_{k \in \N}$} 
\\[4pt] 
\text{with $|x_k-x'_k| \leq \max\{|x_k-y_k|, |x_k-z_k|\}/2$} 
\end{array} 
\end{equation} 
such that 
\begin{equation}\label{eq:lim-zero}
\lim_{k \to \infty} \frac{|x_k-x'_k|}{|x_k-y_k| + |x_k-z_k|} = 0. 
\end{equation}
This means that it remains to demonstrate \eqref{Exk} when $\{(x_k, y_k, z_k)\}_{k \in \N}$ does not satisfy \eqref{eq:lim-zero} but satisfies one of the following: 
\begin{equation}\label{case123}
\begin{aligned}
&{\rm (i)} \,\, |x_k-x'_k| \to 0, \quad {\rm (ii)} \,\, |x_k-y_k|+|x_k-z_k| \to \infty, 
\\
&{\rm (iii)} \,\, |x_k+y_k|+|x_k+z_k| \to \infty \text{ and } |x_k-y_k|+|x_k-z_k| \simeq 1.
\end{aligned}
\end{equation}
In this scenario, there exists $C_0 \in (0, 1/2)$ and a subsequence (which we relabel) $\{(x_k, y_k, z_k)\}_{k \in \N}$ such that 
\begin{equation}\label{eq:lim-non-zero}
C_0 < \frac{|x_k-x'_k|}{|x_k-y_k| + |x_k-z_k|} \leq \frac12, \quad \forall k \in \N. 
\end{equation}

Let $(x_k, y_k, z_k)$ and $x'_k$ satisfy \eqref{eq:lim-non-zero}. Let $0 < \varepsilon < 2^{-\delta} \|K\|_{\rm{CZ}(\delta)}$ be fixed and arbitrarily small. We choose a sequence $\{\lambda_k\}_{k \in \N}$ with  
$\lambda_k = \big(\varepsilon \|K\|_{\rm{CZ}(\delta)}^{-1} \big)^{\frac{1}{\delta}} |x_k-x'_k|/8$. 
Using \eqref{eq:lim-non-zero} and \eqref{eq:TTT-1} applied to $(x_k, y_k, z_k; x_k + (x'_k - x_k)/4)$ and $(x'_k, y_k, z_k; x'_k + (x_k - x'_k)/4)$, we deduce that 
\begin{align}\label{Exyz}
\mathcal{E}(x_k, y_k, z_k; x'_k) 
&\leq C_0^{-\delta'}  \frac{|K(x_k, y_k, z_k) - K(x'_k, y_k, z_k)|}{(|x_k-y_k| + |x_k-z_k|)^{-2n}}
\\ \nonumber  
&\lesssim \varepsilon + (|x_k-y_k| + |x_k-z_k|)^{2n} \Delta(x_k, y_k, z_k; x'_k)
\\ \nonumber 
&\lesssim \varepsilon +  |x_k-x'_k|^{2n} \Delta(x_k, y_k, z_k; x'_k), 
\end{align} 
where $\Delta(x_k, y_k, z_k; x'_k)
:= |\langle T(\tau_{y_k} \Phi_{\lambda_k}, \tau_{z_k} \Phi_{\lambda_k}), 
\tau_{x_k} \Phi_{\lambda_k} - \tau_{x'_k} \Phi_{\lambda_k} \rangle|$. 
If we denote 
\begin{equation}\label{def:fgh}
\begin{aligned}
&f_k := |x_k-x'_k|^{\frac{n}{p'_1}} \tau_{y_k} \Phi_{\lambda_k}, \quad 
g_k := |x_k-x'_k|^{\frac{n}{p'_2}} \tau_{z_k} \Phi_{\lambda_k}, 
\\ 
&\text{and } \, h_k := |x_k-x'_k|^{\frac{n}{p}} \big(\tau_{x_k} \Phi_{\lambda_k} - \tau_{x'_k} \Phi_{\lambda_k}\big),
\end{aligned}
\end{equation}
then
\begin{align*}
\|f_k\|_{L^{p_1}} 
&= (\lambda_k^{-1} |x_k-x'_k|)^{\frac{n}{p'_1}} \|\Phi\|_{L^{p_1}}
\lesssim \varepsilon^{-\frac{n}{\delta p'_1}}, 
\\
\|g_k\|_{L^{p_2}} 
& = (\lambda_k^{-1} |x_k-x'_k|)^{\frac{n}{p'_2}} \|\Phi\|_{L^{p_2}}
\lesssim \varepsilon^{-\frac{n}{\delta p'_2}}, 
\\
\|h_k\|_{L^{p'}} 
&\lesssim (\lambda_k^{-1} |x_k-x'_k|)^{\frac{n}{p}}  \|\Phi\|_{L^{p'}}
\lesssim \varepsilon^{-\frac{n}{\delta p}}, 
\end{align*}
This says that $\{f_k\}$ is a bounded sequence in $L^{p_1}(\Rn)$ and $\{g_k\}$ is a bounded sequence in $L^{p_2}(\Rn)$. Then by the compactness of $T$, there exists a convergent subsequence (which we relabel) $\{T(f_k, g_k)\}$. This gives that for any $\eta>0$ chosen later, there is $N_{\eta}>0$ such that for all $k, j>N_{\eta}$, 
\begin{align}\label{eq:fkfj}
\|T(f_k, g_k) - T(f_j, g_j)\|_{L^p} < \eta. 
\end{align}
Thus, by \eqref{Exyz}--\eqref{eq:fkfj}, we have for all $k, j>N_{\eta}$, 
\begin{align*}
&\mathcal{E}(x_k, y_k, z_k; x'_k) 
\lesssim \varepsilon + |\langle T(f_k, g_k), h_k \rangle|
\\
&\lesssim \varepsilon + |\langle T(f_k, g_k) - T(f_j, g_j), h_k \rangle|
+ |\langle T(f_j, g_j), h_k \rangle| 
\\
&\lesssim \varepsilon + \|T(f_k, g_k) - T(f_j, g_j)\|_{L^p(\Rn)} \|h_k\|_{L^{p'}}
+ |\langle T(f_j, g_j), h_k \rangle| 
\\
&\lesssim \varepsilon + \eta \varepsilon^{-\frac{n}{\delta p}} + |\langle T(f_j, g_j), h_k \rangle| 
\lesssim \varepsilon + |\langle T(f_j, g_j), h_k \rangle|, 
\end{align*} 
provided $\eta< \varepsilon^{1+\frac{n}{\delta p}}$. To complete the proof, it suffices to show that for given $k >N_{\eta}$, there exists $j=j(k)>N_{\eta}$ sufficiently large such that 
\begin{align}\label{Tfjhk}
|\langle T(f_j, g_j), h_k \rangle| 
\lesssim \varepsilon. 
\end{align} 

To proceed, we rewrite $\langle T(f_j, g_j), h_k \rangle$ as 
\begin{align*}
\langle T(f_j, g_j), h_k \rangle
= |x_j-x'_j|^{n(2-\frac1p)} |x_k-x'_k|^{\frac{n}{p}} 
\langle T(\tau_{y_j} \Phi_{\lambda_j}, \tau_{z_j} \Phi_{\lambda_j}), 
\tau_{x_k} \Phi_{\lambda_k} - \tau_{x'_k} \Phi_{\lambda_k} \rangle. 
\end{align*} 
Recall that $T$ is compact from $L^{p_1}(\Rn) \times L^{p_2}(\Rn)$ to $L^p(\Rn)$. This and \cite[Theorem 3]{GT1} give that  
\begin{align}\label{Tfj-1}
\text{$T$ is bounded from $L^{q_1}(\Rn) \times L^{q_2}(\Rn)$ to $L^q(\Rn)$}, 
\end{align}
for all $\frac1q = \frac{1}{q_1} + \frac{1}{q_2}$ with $q_1, q_2 \in (1, \infty)$. Choose $\frac1q = \frac{1}{q_1} + \frac{1}{q_2}$ with $q, q_1, q_2 \in (1, \infty)$. Accordingly, we use \eqref{Tfj-1} to deduce that 
\begin{align}\label{Tfj-2}
|\langle T(f_j, g_j), h_k \rangle| 
&\lesssim |x_j-x'_j|^{n(2-\frac1p)} |x_k-x'_k|^{\frac{n}{p}} 
\|\tau_{y_j} \Phi_{\lambda_j}\|_{L^{q_1}} 
\\ \nonumber 
&\qquad\times \|\tau_{z_j} \Phi_{\lambda_j}\|_{L^{q_2}}  
\|\tau_{x_k} \Phi_{\lambda_k} - \tau_{x'_k} \Phi_{\lambda_k} \|_{L^{q'}} 
\\ \nonumber 
&\lesssim |x_j-x'_j|^{n(2-\frac1p)} |x_k-x'_k|^{\frac{n}{p}} 
\lambda_j^{-\frac{n}{q'_1}} \lambda_j^{-\frac{n}{q'_2}} \lambda_k^{-\frac{n}{q}} 
\\ \nonumber 
&\lesssim \varepsilon^{-\frac{2n}{\delta}} 
\bigg(\frac{|x_k-x'_k|}{|x_j-x'_j|}\bigg)^{n(\frac1p-\frac1q)}
=\varepsilon^{-\frac{2n}{\delta}} 
\bigg(\frac{|x_j-x'_j|}{|x_k-x'_k|}\bigg)^{n(\frac{1}{p'}-\frac{1}{q'})}. 
\end{align} 

If the sequence $\{(x_k, y_k, z_k)\}$ satisfies the case (i) in \eqref{case123} and \eqref{eq:lim-non-zero}, then for fixed $k >N_{\eta}$, we choose $j>k$ large enough and $q \in (1, p)$ in \eqref{Tfj-2} so that 
\begin{align*}
|\langle T(f_j, g_j), h_k \rangle| 
\lesssim \varepsilon^{-\frac{2n}{\delta}} \bigg(\frac{|x_j-x'_j|}{|x_k-x'_k|}\bigg)^{n(\frac{1}{p'}-\frac{1}{q'})}
< \varepsilon. 
\end{align*} 
If the sequence $\{(x_k, y_k, z_k)\}$ satisfies the case (ii) in \eqref{case123} and \eqref{eq:lim-non-zero}, then $|x_k-x'_k| \to \infty$. For fixed $k >N_{\eta}$, we choose $j>k$ large enough and $q \in (p, \infty)$ in \eqref{Tfj-2} so that 
\begin{align*}
|\langle T(f_j, g_j), h_k \rangle| 
\lesssim \varepsilon^{-\frac{2n}{\delta}} \bigg(\frac{|x_k-x'_k|}{|x_j-x'_j|}\bigg)^{n(\frac1p-\frac1q)} 
< \varepsilon. 
\end{align*} 
If the sequence $\{(x_k, y_k, z_k)\}$ satisfies the case (iii) in \eqref{case123} and \eqref{eq:lim-non-zero}, then 
\begin{align*}
2(|y_k| + |z_k|)
\ge (|x_k+y_k| + |x_k+z_k|) - (|x_k-y_k| + |x_k-z_k|), 
\end{align*}
which implies $\lim_{k \to \infty}(|y_k| + |z_k|)=\infty$, and furthermore, $\lim_{j \to \infty} |x_k-y_j|+|x_k-z_j|=\infty$ for any fixed $k$. Therefore, fixed $k >N_{\eta}$, we use the H\"{o}lder estimate for the kernel $K$ and \eqref{eq:lim-non-zero}, and choose $j>k$ sufficiently large to conclude that 
\begin{align*}
|\langle T(f_j, g_j), h_k \rangle|
\lesssim \frac{|x_j-x'_j|^{n(2-\frac1p)} |x_k-x'_k|^{\frac{n}{p} + \delta}}{(|x_k-y_j|+|x_k-z_j|)^{2n+\delta}}
\lesssim \frac{1}{(|x_k-y_j|+|x_k-z_j|)^{2n+\delta}}
< \varepsilon. 
\end{align*} 
This shows \eqref{Tfjhk} and completes the proof. 
\end{proof}

\subsection{Weak compactness property}\label{sec:WCP}
Next, let us give the proof of Theorem \ref{thm:WCP}. Fix $I \in \D$ and $N \in \N_+$. Let $P_N = P_{\D(N)}$. We use $P_N = P_N^*$ to split 
\begin{align}\label{TII-1}
\langle T(\mathbf{1}_I, \mathbf{1}_I), \mathbf{1}_I \rangle 
= \langle P_N^{\perp}(T(\mathbf{1}_I, \mathbf{1}_I)), \mathbf{1}_I \rangle
+ \langle  T(\mathbf{1}_I, \mathbf{1}_I), P_N \mathbf{1}_I \rangle. 
\end{align}
By H\"{o}lder's inequality, 
\begin{align}\label{TII-2}
|\langle P_N^{\perp} T(\mathbf{1}_I, \mathbf{1}_I), \mathbf{1}_I \rangle| 
\le \|P_N^{\perp} T\|_{L^{p_1} \times L^{p_2} \to L^p} |I|, 
\end{align}
and by the boundedness of $T$, 
\begin{align}\label{TII-3}
|\langle T(\mathbf{1}_I, \mathbf{1}_I), P_N \mathbf{1}_I \rangle| 
&\le \|T\|_{L^{p_1} \times L^{p_2} \to L^p} |I|^{\frac1p} \|P_N \mathbf{1}_I\|_{L^{p'}}. 
\end{align}
To conclude the desired estimate, it suffices to show 
\begin{align}\label{pnp}
\|P_N \mathbf{1}_I\|_{L^{p'}} 
\lesssim F(I; 2N) |I|^{\frac{1}{p'}}. 
\end{align}
Furthermore, if $T$ is compact from $L^{p_1}(\Rn) \times L^{p_2}(\Rn)$ to $L^p(\Rn)$, then Theorem \ref{thm:PNT-cpt} gives that 
\begin{align*}
\lim\limits_{N \to \infty} \sup_{\D} \|P_{\D(N)}^{\perp} T\|_{L^{p_1} \times L^{p_2} \to L^p} =0,  
\end{align*}
which yields the weak compactness property.

It remains to justify \eqref{pnp}. In the case $I \in \D(2N)$, the inequality \eqref{ddf-4} gives 
\begin{align}\label{NII-1}
\|P_N \mathbf{1}_I\|_{L^{p'}} 
\le \sup_{N \in \N} \|P_N\|_{L^{p'} \to L^{p'}} |I|^{\frac{1}{p'}}
\lesssim F(I; 2N) |I|^{\frac{1}{p'}}. 
\end{align}
To proceed, fix $I \notin \D(2N)$. It suffices to consider the following three cases: 
\begin{align*}
\text{(i) } \, \, \ell(I) \le 2^{-2N} \text{ and } \rd(I, 2^{2N} \I) \le 2N, \quad 
\text{(ii) } \, \, \ell(I) > 2^{2N}, \quad 
\text{(iii) } \, \, \ell(I) \le 2^{2N} \text{ and } \rd(I, 2^{2N} \I) > 2N. 
\end{align*}
By the cancellation of Haar functions,  there holds 
\begin{align}\label{NII-2}
P_N \mathbf{1}_I
=\sum_{J \in \D(N): I \subsetneq J} \langle \mathbf{1}_I, h_J \rangle h_J. 
\end{align} 
Observe that if $J \in \D(N)$ with $I \subsetneq J$, then in case (ii), there holds $\ell(J)>2^N$, and in case (iii), we have 
\begin{align*}
\d(J, 2^N \I) 
\ge \d(I, 2^N \I) - \ell(J) 
\ge 2N \cdot 2^{2N} - 2^N 
> N 2^N,  
\end{align*}
hence, $\rd(J, 2^N \I) > N$. This means that there does not exist $J \in \D(N)$ such that $I \subsetneq J$ in cases (ii) and (iii), which along with \eqref{NII-2} implies 
\begin{align}\label{NII-3}
P_N \mathbf{1}_I = 0, \quad 
\text{for any $I$ in cases (ii) and (iii)}.
\end{align}

Let us next treat the case (i). By definition, one can find a sequence of dyadic cubes $\{I_j\}_{j=-N}^N$ such that $\{J \in \D(N): I \subsetneq J\} \subset \{I_j\}_{j=-N}^N$ and $\ell(I_j)=2^j$. Then \eqref{NII-2} gives 
\begin{align}\label{NII-4}
\|P_N \mathbf{1}_I\|_{L^{p'}} 
&\le \sum_{j=-N}^N |\langle \mathbf{1}_I, h_{I_j} \rangle| \|h_{I_j}\|_{L^{p'}} 
\le \sum_{j=-N}^N |I| |I_j|^{-1+\frac{1}{p'}} 
\\ \nonumber 
&= \sum_{j=-N}^N 2^{-j \frac{n}{p}} |I| 
\lesssim 2^{N \frac{n}{p}} |I| 
\le (2^{2N} \ell(I))^{\frac{n}{p}} |I|^{\frac{1}{p'}}
\lesssim F(I; 2N) |I|^{\frac{1}{p'}}. 
\end{align}
Therefore, \eqref{pnp} follows from \eqref{NII-1}, \eqref{NII-3}, and \eqref{NII-4}. 
\qed

\subsection{$\CMO$ condition}\label{sec:T1CMO}
Finally, we turn to the proof of Theorem \ref{thm:T1CMO}. 

\begin{lemma}\label{lem:T11}
Let $T$ be a bilinear operator associated with a compact bilinear Calder\'{o}n--Zygmund kernel $K$ with parameter $\delta \in (0, 1]$. Let $\Phi \in \mathscr{C}_c^{\infty}(\Rn)$ be a cut-off function such that $\mathbf{1}_{B(0, 1)} \le \Phi \le \mathbf{1}_{B(0, 2)}$. Let $I \subset \Rn$ be a cube and let $f \in L_c^{\infty}(\Rn)$ have compact support in $I$ and mean zero. If $T$ is bounded from $L^{p_1}(\Rn) \times L^{p_2}(\Rn)$ to $L^p(\Rn)$ for some $\frac1p = \frac{1}{p_1} + \frac{1}{p_2}$ with $p, p_1, p_2 \in (1, \infty)$, then the following hold: 
\begin{list}{\rm (\theenumi)}{\usecounter{enumi}\leftmargin=1cm \labelwidth=1cm \itemsep=0.1cm \topsep=.2cm \renewcommand{\theenumi}{\arabic{enumi}}}

\item\label{Lf-1} For any $a \in \Rn$, the limit 
\begin{align*}
\mathscr{L}_a(f) := \lim_{k \to \infty} 
\bigg\langle T \bigg(\Phi \Big(\frac{\cdot-a}{2^k \ell(I)}\Big), 
\Phi \Big(\frac{\cdot-a}{2^k \ell(I)}\Big) \bigg), \, f \bigg\rangle 
\quad \text{exists}. 
\end{align*}

\item\label{Lf-2} For all $a \in \Rn$ and $k \in \N$ such that $2^k \ge 1 + 2|a-c_I|/\ell(I)$, we have 
\begin{align*}
\bigg|&\mathscr{L}_a(f) - \bigg\langle T \bigg(\Phi \Big(\frac{\cdot-a}{2^k \ell(I)}\Big), 
\Phi \Big(\frac{\cdot-a}{2^k \ell(I)}\Big) \bigg), \, f \bigg\rangle \bigg| 
\\ 
&\lesssim 2^{-k \delta} [1+|a-c_I|/\ell(I)]^{\delta}  
\sum_{k'=0}^{\infty} 2^{-k' \delta} F(2^{k'+k} \ell(I))  \|f\|_{L^1}, 
\end{align*}
where $F(t) := F_1(8t) F_2(t) F_3\big(\frac13 \big(1+\frac{|a|}{1+t} \big) \big)$. 

\item\label{Lf-3} The limit above is independent of $a$ and $\Phi$, and hence, denoted by $\mathscr{L}$. 

\item\label{Lf-4} $\mathscr{L}$ is a bounded linear functional on $\mathrm{H}^1(\Rn)$. In particular, by the duality between $\mathrm{H}^1(\Rn)$ and $\BMO(\Rn)$, we define $T(1, 1)$ as  
\begin{align*}
\mathscr{L}(f) = \langle T(1, 1), f \rangle. 
\end{align*}
\end{list} 
\end{lemma}

\begin{proof}
Let $f \in L^{\infty}_c(\Rn)$ with support in a cube $I$ and mean zero. 
Fix $a \in \Rn$ and $k \in \N$ with $2^k \geq 1 + 2|a-c_I|/\ell(I)$, and write $\Psi_k := \Phi \big(\frac{\cdot-a}{2^k \ell(I)}\big)$. It is easy to see that $\supp(\Psi_k) \subset \{x: |x-a| \leq 2^{k+1} \ell(I)\}$ and $\supp(\Psi_{k+1} - \Psi_k) \subset \{x: 2^k \ell(I) < |x-a| \le 2^{k+2} \ell(I)\}$. Let $x \in \supp(f) \subset I$, $y \in \supp(\Psi_{k+1} - \Psi_k)$, and $z \in \supp(\Psi_k)$. Then,  
\begin{align*}
|x-a| \leq |x- c_I| + |a-c_I| 
\leq \ell(I) [1 + 2|a-c_I|/\ell(I)]/2  
\leq 2^{k-1} \ell(I) 
< |y-a|/2, 
\end{align*}
and hence, $|x-y| \geq |y-a|-|x-a|>|x-a| \geq 0$. This means $\supp(f) \cap \supp(\Psi_{k+1} - \Psi_k)=\varnothing$. Moreover, if we denote 
\begin{align*}
F_a(x, y, z) := F_1(|y-a|+|z-a|) F_2(|y-a|+|z-a|) F_3 \bigg(1+\frac{|a+y| + |a+z|}{1+|y-a|+|z-a|} \bigg),  
\end{align*}
then 
\begin{align*}
F_a(x, y, z) 
\leq F_1(2^{k+3} \ell(I)) F_2(2^k \ell(I)) 
F_3 \bigg(\frac13 \Big(1 + \frac{|a|}{1+2^k \ell(I)} \Big) \bigg) 
= F(2^k \ell(I)), 
\end{align*}
where we used the monotonicity properties of $F_1$, $F_2$ and $F_3$. 

We are going to show that the sequence $\{\langle T(\Psi_k, \Psi_k), f\rangle\}_{k \geq 1}$ is Cauchy. Note that 
\begin{align}\label{eq:T-Phi-Phi}
&|\langle T(\Psi_{k+1}, \Psi_{k+1}), f\rangle - \langle T(\Psi_k, \Psi_k), f\rangle| 
\\ \nonumber 
&\leq |\langle T(\Psi_{k+1}-\Psi_k, \Psi_k), f\rangle| 
+|\langle T(\Psi_{k+1}, \Psi_{k+1} - \Psi_k), f\rangle|. 
\end{align}
It follows from the mean zero of $f$, Lemma \ref{lem:improve} \eqref{size-imp} and the fact $|x-a| \leq \max\{|y-a|, |z-a|\}/2$ that 
\begin{align}\label{PhiPhi-1}
&|\langle T(\Psi_{k+1} - \Psi_k, \Psi_k), f\rangle| 
\\ \nonumber 
&= \bigg|\int_{\R^{3n}} [K(x, y, z)-K(a, y, z)] 
(\Psi_{k+1} - \Psi_k)(y) \Psi_k(z) f(x) \, dx \, dy \, dz \bigg|
\\ \nonumber 
&\lesssim \int_{\R^{3n}} \frac{|x-a|^{\delta}F_a(x, y, z)}{(|y-a|+|z-a|)^{2n+\delta}} 
\mathbf{1}_{\supp(\Psi_{k+1} - \Psi_k)}(y) \Psi_k(z) |f(x)| \, dx \, dy \, dz
\\ \nonumber 
&\lesssim (2^k \ell(I))^n F(2^k \ell(I)) 
\int_{\R^{2n}} \frac{(\ell(I) + |a-c_I|)^{\delta}}{(2^k \ell(I))^{2n+\delta}} 
\Phi\Big(\frac{z-a}{2^k \ell(I)} \Big) |f(x)| \, dx \, dz
\\ \nonumber 
&\lesssim 2^{-k\delta}  [1+|a-c_I|/\ell(I)]^{\delta} F(2^k \ell(I))  \|f\|_{L^1}. 
\end{align}
Similarly, one has 
\begin{align*}
|\langle T(\Psi_{k+1}, \Psi_{k+1}-\Psi_k), f\rangle| 
\lesssim 2^{-k\delta}  [1+|a-c_I|/\ell(I)]^{\delta} F(2^k \ell(I))  \|f\|_{L^1}, 
\end{align*}
which together \eqref{eq:T-Phi-Phi} and \eqref{PhiPhi-1} implies  
\begin{align}\label{PhiPhi-2}
|\langle T(\Psi_{k+1}, \Psi_{k+1}), f\rangle - \langle T(\Psi_k, \Psi_k), f\rangle| 
\lesssim 2^{-k\delta}  [1+|a-c_I|/\ell(I)]^{\delta} F(2^k \ell(I))  \|f\|_{L^1}. 
\end{align}
Since $F_1, F_2, F_3$ are bounded, \eqref{PhiPhi-2} gives that $\{\langle T(\Psi_k, \Psi_k), f\rangle\}_{k \geq 1}$ is a Cauchy sequence. Therefore, the limit in item \eqref{Lf-1} exists, which we denote by $\mathscr{L}_a(f)$. This shows item \eqref{Lf-1}. 

Additionally, 
\begin{align*}
\mathscr{L}_a(f) - \langle T(\Psi_k, \Psi_k), f\rangle 
&=\mathscr{L}_a(f) - \langle T(\Psi_{k+j+1}, \Psi_{k+j+1}), f\rangle 
\\
&\quad + \sum_{k'=k}^{k+j} \big( \langle T(\Psi_{k'+1}, \Psi_{k'+1}), f\rangle 
- \langle T(\Psi_{k'}, \Psi_{k'}), f\rangle \big). 
\end{align*}
Letting $j \to \infty$ and invoking \eqref{PhiPhi-2}, we obtain 
\begin{align*}
|\mathscr{L}_a(f) - \langle T(\Psi_k, \Psi_k), f\rangle| 
&\le \sum_{k'=k}^{\infty} \big| \langle T(\Psi_{k'+1}, \Psi_{k'+1}), f\rangle 
- \langle T(\Psi_{k'}, \Psi_{k'}), f\rangle \big|
\\
& \lesssim  [1+|a-c_I|/\ell(I)]^{\delta} 
\sum_{k'=k}^{\infty} 2^{-k' \delta} F(2^{k'} \ell(I))  \|f\|_{L^1}
\\
&= 2^{-k \delta} [1+|a-c_I|/\ell(I)]^{\delta} 
\sum_{k'=0}^{\infty} 2^{-k' \delta} F(2^{k'+k} \ell(I))  \|f\|_{L^1},  
\end{align*}
which gives item \eqref{Lf-2}. 

Next, let us prove item \eqref{Lf-3}. Let $a, b \in \R$ and $k \ge 1$ so that $2^k \geq 1 + 2|a-c_I|/\ell(I)$ and $2^k \geq 1 + 2|b-c_I|/\ell(I)$. Denote $\Psi_k^a := \Phi \big(\frac{\cdot-a}{2^k \ell(I)}\big)$ and $\Psi_k^b := \Phi \big(\frac{\cdot-b}{2^k \ell(I)}\big)$. Observe that $2^k \ell(I) \ge |a-b|$ and $\supp(\Psi_k^a - \Psi_k^b) \subset \{x: 2^{k-1} \ell(I) \le |x-\frac{a+b}{2}| \le 2^{k+2} \ell(I)\}$. Then, much as in \eqref{PhiPhi-1}, there holds 
\begin{align}
\label{TPSK-1}
|\langle T(\Psi_k^a - \Psi_k^b, \Psi_k^a), f\rangle| 
&\lesssim 2^{-k\delta}  [1+|a-c_I|/\ell(I)]^{\delta} \|f\|_{L^1}, 
\\
\label{TPSK-2}
|\langle T(\Psi_k^b, \Psi_k^a - \Psi_k^b), f\rangle| 
&\lesssim 2^{-k\delta}  [1+|b-c_I|/\ell(I)]^{\delta} \|f\|_{L^1}.  
\end{align}
Consequently, it follows from item \eqref{Lf-2}, \eqref{TPSK-1}, and \eqref{TPSK-2} that  
\begin{align*}
|\mathscr{L}_a(f) - \mathscr{L}_b(f)| 
&\le |\mathscr{L}_a(f) - \langle T(\Psi_k^a, \Psi_k^a), f\rangle| 
+ |\mathscr{L}_b(f) - \langle T(\Psi_k^b, \Psi_k^b), f\rangle| 
\\
&\quad+ |\langle T(\Psi_k^a - \Psi_k^b, \Psi_k^a), f\rangle| 
+ |\langle T(\Psi_k^b, \Psi_k^a - \Psi_k^b), f\rangle| 
\\
&\lesssim 2^{-k \delta} \big\{[1+|a-c_I|/\ell(I)]^{\delta} + [1+|b-c_I|/\ell(I)]^{\delta} \big\} \|f\|_{L^1}, 
\end{align*}
which tends to zero when $k \to \infty$. This shows that $\mathscr{L}_a$ is independent of $a$. Similarly, one can prove $\mathscr{L}_a$ is independent of $\Phi$. 

Finally, to show item \eqref{Lf-4}, we modify the definition of $\Psi_k$ above into $\Psi_k := \Phi \big(\frac{\cdot-c_I}{2^k \ell(I)}\big)$. Then we use item \eqref{Lf-2}, the boundedness of $T$, \cite[Proposition 7.4.8]{Gra2}, and \eqref{PhiPhi-2} to arrive at 
\begin{align*}
|\mathscr{L}(f)| 
&\le |\mathscr{L}(f) - \langle T(\Psi_k, \Psi_k), f\rangle| 
+ \langle T(\Psi_0, \Psi_0), f\rangle| 
\\ 
&\quad+ \sum_{k'=0}^{k-1} \big( \langle T(\Psi_{k'+1}, \Psi_{k'+1}), f\rangle 
- \langle T(\Psi_{k'}, \Psi_{k'}), f\rangle \big)
\\
& \lesssim 2^{-k \delta} \|f\|_{L^1} 
+ \|\Psi_0\|_{L^{\infty}}^2 \|f\|_{\mathrm{H}^1}
+ \sum_{k'=0}^{k-1} 2^{-k' \delta} F(2^{k'} \ell(I))  \|f\|_{L^1}
\\
&\lesssim \|f\|_{\mathrm{H}^1}. 
\end{align*}
This asserts that $\mathscr{L}$ is a bounded linear functional on $\mathrm{H}^1(\Rn)$. The proof is complete. 
\end{proof}

\begin{proof}[\bf Proof of Theorem \ref{thm:T1CMO}]
By \cite[Proposition 2.7]{BLOT1}, $T^{*1} : L^{p'}(\Rn) \times L^{p_2}(\Rn) \to L^{p'_1}(\Rn)$ and $T^{*2} : L^{p_1}(\Rn) \times L^{p'}(\Rn) \to L^{p'_2}(\Rn)$ are compact. Thus, in view of Theorem \ref{thm:PDN}, it suffices to show $T(1, 1) \in \CMO(\Rn)$, which follows from 
\begin{align}\label{NPNT}
T(1, 1) \in \BMO(\Rn) \quad\text{ and }\quad
\lim_{N \to \infty} \sup_{\D} \langle P_{\D(N)}^{\perp} (T(1, 1)), f \rangle =0,  
\end{align}
uniformly for all $p'$-atoms $f$ of $\mathrm{H}^1(\Rn)$ with $f \in L_c^{\infty}(\Rn)$, $\int_{\Rn} f \, dx = 0$, $\supp(f) \subset I$, and $\|f\|_{L^{p'}} \lesssim |I|^{-\frac1p}$, for some dyadic cube $I \in \D$. Theorem \ref{thm:CCZK} implies that $T$ owns a compact bilinear Calder\'{o}n--Zygmund kernel. This, together with Lemma \ref{lem:T11} item \eqref{Lf-4}, yields that $T(1, 1) \in \BMO(\Rn)$. 

Since
\begin{align*}
P_{\D(N)} f = \sum_{J \in \D(N)} \big(|J|^{\frac12} \langle f, h_J \rangle \big) \big(|J|^{-\frac12} h_J \big) 
\end{align*}
is a finite linear combination of $\infty$-atoms $|J|^{-\frac12} h_J$, we see that $P_{\D(N)} f, P_{\D(N)}^{\perp} f \in H^1(\Rn)$. Invoking $T(1, 1) \in \BMO(\Rn)$ and $P_{\D(N)}^{\perp}f \in \mathrm{H}^1(\Rn)$, we obtain  
\begin{align}\label{PNT-1}
\langle P_{\D(N)}^{\perp} (T(1, 1)), f \rangle
= \langle T(1, 1), P_{\D(N)}^{\perp} f \rangle
=\mathscr{L}(P_{\D(N)}^{\perp} f), 
\end{align}
where the functional $\mathscr{L}$ is defined in Lemma \ref{lem:T11}. 

Let $\varepsilon>0$ be an arbitrary number. Choose $k \in \N$ so that $2^{-k \delta} < \varepsilon$. Since $T$ is compact from $L^{p_1}(\Rn) \times L^{p_2}(\Rn)$ to $L^p(\Rn)$, by Theorem \ref{thm:PNT-cpt}, there exists $N_0>0$ such that  
\begin{align}\label{PNe}
\sup_{\D} \|P_{\D(N)}^{\perp} T\|_{L^{p_1} \times L^{p_2} \to L^p} 
\le 2^{-\frac{kn}{p}} \varepsilon, \quad \text{for all } \, N \ge N_0.  
\end{align}
Considering that $\supp(f) \subset I$, $\supp(h_J) \subset J$, and $\int_{\Rn} f \, dx = \int_{\Rn} h_J \, dx=0$, we have 
\begin{align}\label{PNT-2}
P_{\D(N)}^{\perp} f 
= \sum_{J \notin \D(N): J \subset I} \langle f, h_J \rangle h_J, 
\end{align}
and by \eqref{ddf-4}, 
\begin{align}
\|P_{\D(N)}^{\perp} f\|_{L^{p'}} 
\lesssim \|f\|_{L^{p'}} 
\lesssim |I|^{-\frac1p}.
\end{align} 
To proceed, choose a smooth cut-off function $\Phi$ satisfying $\mathbf{1}_{B(0, 1)} \le \Phi \le \mathbf{1}_{B(0, 2)}$, and set $\Phi_I(x) := \Phi \big(\frac{x-c_I}{2^k \ell(I)}\big)$. By H\"{o}lder's inequality and \eqref{PNe}, 
\begin{align}\label{LL-1}
\mathscr{I}_1 
&:= |\langle P_{\D(N)}^{\perp} (T(\Phi_I, \Phi_I)), P_{\D(N)}^{\perp} f \rangle|  
\\ \nonumber
&\lesssim \|P_{\D(N)}^{\perp} T\|_{L^{p_1} \times L^{p_2} \to L^p} 
\|\Phi_I\|_{L^{p_1}} \|\Phi_I\|_{L^{p_2}} \|P_{\D(N)}^{\perp} f\|_{L^{p'}}
\\ \nonumber
&\lesssim 2^{-\frac{kn}{p}} \varepsilon  \, 
(2^k \ell(I))^{\frac{n}{p_1}} (2^k \ell(I))^{\frac{n}{p_2}} |I|^{-\frac1p}
= \varepsilon,  
\end{align}
and Lemma \ref{lem:T11} applied to $P_{\D(N)}^{\perp} f$ and $c_I$ instead of $f$ and $a$ gives  
\begin{align}\label{LL-2}
\mathscr{I}_2
&:= |\mathscr{L}(P_{\D(N)}^{\perp} f) - \langle P_{\D(N)}^{\perp} (T(\Phi_I, \Phi_I), P_{\D(N)}^{\perp} f \rangle|  
\\ \nonumber 
&\lesssim 2^{-k \delta} \sum_{k'=0}^{\infty} 2^{-k' \delta} F(2^{k'+k} \ell(I))  \|P_{\D(N)}^{\perp} f\|_{L^1}
\lesssim \varepsilon, 
\end{align}
since $\supp(P_{\D(N)}^{\perp} f) \subset I$ and $F$ is bounded. Now collecting the estimates \eqref{PNT-2}--\eqref{LL-2}, we conclude 
\begin{align*}
|\mathscr{L}(P_{\D(N)}^{\perp} f)| 
\le \mathscr{I}_1 + \mathscr{I}_2 
\lesssim \varepsilon, 
\end{align*}
which together with \eqref{PNT-1} implies \eqref{NPNT} as desired. 
\end{proof}

\section{Rubio de Francia extrapolation of compactness}\label{sec:EP}
In this section, we would like to prove Theorems \ref{thm:EP-Lp}, \ref{thm:EP-Lpinfty}, and \ref{thm:EP-Linfty}, which are the main approach to obtain weighted compactness in the quasi-Banach range.

\subsection{Extrapolation from $L^p$ compactness}
Let us first show Theorem \ref{thm:EP-Lp}. By the boundedness hypothesis and Theorem \ref{thm:RdF},
\begin{align}\label{eq:rdq}
\|T\|_{L^{q_1}(v_1^{q_1}) \times L^{q_2}(v_2^{q_2}) \to L^q(v^q)}
\le \widetilde{\Xi}([\vec{v}]_{A_{\vec{q}}})
\end{align}
for all $q_1, q_2 \in (1, \infty]$ and for all $(v_1, v_2) \in A_{(q_1, q_2)}$, where $\frac1q = \frac{1}{q_1} + \frac{1}{q_2} > 0$, $v=v_1 v_2$, and $\widetilde{\Xi}: [1, \infty) \to [1, \infty)$ is non-decreasing. Then using the compactness hypothesis, the unweighted boundedness in \eqref{eq:rdq}, and the multilinear interpolation of compactness \cite[Theorem 3.4]{COY}, we obtain 
\begin{align}\label{eq:css}
T: L^{s_1}(\Rn) \times L^{s_2}(\Rn) \to L^s(\Rn) \, \text{ is compact}  
\end{align}
for all $\frac1s = \frac{1}{s_1} + \frac{1}{s_2}$ with $s_1, s_2 \in (1, \infty)$. 

Fix $\frac1s = \frac{1}{s_1} + \frac{1}{s_2}$ with $s_1, s_2 \in (1, \infty)$ and $a \in (0, \frac12)$.  Given $\varepsilon \in (0, 1)$, by \eqref{eq:css} and Theorem \ref{thm:KRWA}, there exists $ A_0 = A_0(\varepsilon) > 1$ such that for all $A>A_0$ and $0 < \delta < A_0^{-1}$,
\begin{align}
\label{eq:css2}
& \sup_{\substack{\|f_1\|_{L^{s_1}} \le 1 \\ \|f_2\|_{L^{s_2}} \le 1}} 
\|\mathbf{1}_{B(0, A)^c} T(f_1, f_2)\|_{L^s} \le \varepsilon, 
\\
\label{eq:css3}
& \sup_{\substack{\|f_1\|_{L^{s_1}} \le 1 \\ \|f_2\|_{L^{s_2}} \le 1}}
\| S_{\delta}^a (T(f_1, f_2)) \|_{L^s} \le \varepsilon,
\end{align}
where
\begin{align*}
S_{\delta}^a f(x)
:= \bigg( \fint_{B(0, \delta)} |\tau_y f - f|^a \, dy \bigg)^{\frac1a}.
\end{align*}
Fix $A > A_0$ and $0 < \delta < A_0^{-1}$. Let $\vec{\sigma} = (\sigma_1, \sigma_2) \in A_{\vec{s}}$. In light of \eqref{eq:wopen}, there exists a decreasing function $\lambda: [1, \infty) \to (1, 2)$ such that $\lambda = \lambda([\vec{\sigma}]_{A_{\vec{s}}}) \in (1, 2)$ and  $[\vec{\sigma}^{\lambda}]_{A_{\vec{s}}} \le 16 [\vec{\sigma}]_{A_{\vec{s}}}$, which together with \eqref{eq:rdq} gives
\begin{align}\label{eq:ssv}
\sup_{\substack{\|f_1\|_{L^{s_1}(\sigma_1^{\lambda s_1})} \le 1 \\ \|f_2\|_{L^{s_2}(\sigma_2^{\lambda s_2})} \le 1}} 
\| T(f_1, f_2) \|_{L^s(\sigma^{\lambda s})} 
\le \widetilde{\Xi}([\vec{\sigma}^{\lambda}]_{A_{\vec{s}}})
\le \widetilde{\Xi}(16 [\vec{\sigma}]_{A_{\vec{s}}}). 
\end{align}
This immediately implies
\begin{align}\label{eq:sis2}
\sup_{\substack{\|f_1\|_{L^{s_1}(\sigma_1^{\lambda s_1})} \le 1 \\ \|f_2\|_{L^{s_2}(\sigma_2^{\lambda s_2})} \le 1}} 
\|\mathbf{1}_{B(0, A)^c} T(f_1, f_2)\|_{L^s(\sigma^{\lambda s})} 
\le \widetilde{\Xi}(16 [\vec{\sigma}]_{A_{\vec{s}}}),
\end{align}
and
\begin{align*}
\|S_{\delta}^a (T(f_1, f_2)) \|_{L^s(\sigma^{\lambda s})}
& \lesssim \|T(f_1, f_2) \|_{L^s(\sigma^{\lambda s})}
+ \|M_a (T(f_1, f_2)) \|_{L^s(\sigma^{\lambda s})}
\\ \nonumber
& \lesssim \big(1 + [\sigma^{\lambda s}]_{A_{s/a}}^{\frac{1}{s-a}} \big)
\|T(f_1, f_2) \|_{L^s(\sigma^{\lambda s})}
\lesssim [\sigma^{\lambda s}]_{A_{2s}}^{\frac{1}{s-a}} 
\widetilde{\Xi}(16 [\vec{\sigma}]_{A_{\vec{s}}})
\\
& \le [\vec{\sigma}^{\lambda}]_{A_{\vec{s}}}^{\frac{s}{s-a}} 
\widetilde{\Xi}(16 [\vec{\sigma}]_{A_{\vec{s}}})
\lesssim [\vec{\sigma}]_{A_{\vec{s}}}^{\frac{s}{s-a}} 
\widetilde{\Xi}(16 [\vec{\sigma}]_{A_{\vec{s}}}),
\end{align*}
provided Buckley's result \cite{Buc} and Lemma \ref{lem:weight}. Hence, 
\begin{align}\label{eq:sis3}
\sup_{\substack{\|f_1\|_{L^{s_1}(\sigma_1^{\lambda s_1})} \le 1 \\ \|f_2\|_{L^{s_2}(\sigma_2^{\lambda s_2})} \le 1}} 
\| S_{\delta}^a (T(f_1, f_2)) \|_{L^s(\sigma^{\lambda s})} 
\le C_1 [\vec{\sigma}]_{A_{\vec{s}}}^{\frac{s}{s-a}} 
\widetilde{\Xi}(16 [\vec{\sigma}]_{A_{\vec{s}}}),
\end{align}
where the constant $C_1 \in (1, \infty)$ is independent of $[\vec{\sigma}]_{A_{\vec{s}}}$. 
Note that $\sigma_i = 1^{1-\frac{1}{\lambda}} (\sigma_i^{\lambda})^{\frac{1}{\lambda}}$. Then by \eqref{eq:css2}--\eqref{eq:css3}, \eqref{eq:sis2}--\eqref{eq:sis3}, and bilinear interpolation (cf. \cite[Theorems  3.1 and 3.5]{COY}), we have 
\begin{align}
\label{eq:ss2}
& \sup_{\substack{\|f_1\|_{L^{s_1}(\sigma_1^{s_1})} \le 1 \\ \|f_2\|_{L^{s_2}(\sigma_2^{s_2})} \le 1}} 
\|\mathbf{1}_{B(0, A)^c} T(f_1, f_2)\|_{L^s(\sigma^s)} 
\le \varepsilon^{1-\frac{1}{\lambda}} \widetilde{\Xi}(16 [\vec{\sigma}]_{A_{\vec{s}}})^{\frac{1}{\lambda}}, 
\\
\label{eq:ss3}
& \sup_{\substack{\|f_1\|_{L^{s_1}(\sigma_1^{s_1})} \le 1 \\ \|f_2\|_{L^{s_2}(\sigma_2^{s_2})} \le 1}} 
\| S_{\delta}^a (T(f_1, f_2)) \|_{L^s(\sigma^s)} 
\le C_1 \varepsilon^{1-\frac{1}{\lambda}} \big( [\vec{\sigma}]_{A_{\vec{s}}}^{\frac{s}{s-a}} 
\widetilde{\Xi}(16 [\vec{\sigma}]_{A_{\vec{s}}}) \big)^{\frac{1}{\lambda}}.
\end{align}
Now let $r_1, r_2 \in (1, \infty]$ and $(w_1, w_2) \in A_{(r_1, r_2)}$. Then the estimate \eqref{eq:rdq} applied to $(w_1, w_2) \in A_{(r_1, r_2)}$ yields 
\begin{align}\label{eq:rrc-1}
\sup_{\substack{\|f_1\|_{L^{r_1}(w_1^{r_1})} \le 1 \\ \|f_2\|_{L^{r_2}(w_2^{r_2})} \le 1}} 
\| T(f_1, f_2) \|_{L^r(w^r)} < \infty. 
\end{align}
Observe that $\Phi_{\varepsilon}(t) := C_1 \varepsilon^{1-\frac{1}{\lambda(t)}} \big(t^{\frac{s}{s-a}} \widetilde{\Xi}(16t)\big)^{\frac{1}{\lambda(t)}}$ is a non-decreasing function. In order to use Theorem \ref{thm:RdF}, we define
\begin{align*}
\Psi_{\varepsilon}(t)
:= \max\bigg\{1, \, \frac{\Phi_{\varepsilon}(t)}{\Phi_{\varepsilon}(C_0 [\vec{w}]_{A_{\vec{r}}}^{\alpha/2})} \bigg\},
\end{align*}
where the constants $C_0$ and $\alpha$ are given by Theorem \ref{thm:RdF}. Then $\Psi_{\varepsilon}$ is non-decreasing, and $\Psi_{\varepsilon}(t) \in [1, \infty)$. It follows from \eqref{eq:ss2}--\eqref{eq:ss3} that
\begin{align*}
& \sup_{\substack{\|f_1\|_{L^{s_1}(\sigma_1^{s_1})} \le 1 \\ \|f_2\|_{L^{s_2}(\sigma_2^{s_2})} \le 1}} 
\bigg\| \frac{\mathbf{1}_{B(0, A)^c} T(f_1, f_2)}{\Phi_{\varepsilon}(C_0 [\vec{w}]_{A_{\vec{r}}}^{\alpha/2})} 
\bigg\|_{L^s(\sigma^s)} 
\le \Psi_{\varepsilon} ([\vec{\sigma}]_{A_{\vec{s}}}), 
\\
& \sup_{\substack{\|f_1\|_{L^{s_1}(\sigma_1^{s_1})} \le 1 \\ \|f_2\|_{L^{s_2}(\sigma_2^{s_2})} \le 1}} 
\bigg\| \frac{S_{\delta}^a (T(f_1, f_2))}{\Phi_{\varepsilon}(C_0 [\vec{w}]_{A_{\vec{r}}}^{\alpha/2})} \bigg\|_{L^s(\sigma^s)} 
\le \Psi_{\varepsilon} ([\vec{\sigma}]_{A_{\vec{s}}}),
\end{align*}
which, together with Theorem \ref{thm:RdF} applied to $\vec{w} \in A_{\vec{r}}$, implies
\begin{align}
\label{eq:rrc-2}
& \sup_{\substack{\|f_1\|_{L^{r_1}(w_1^{r_1})} \le 1 \\ \|f_2\|_{L^{r_2}(w_2^{r_2})} \le 1}} 
\|\mathbf{1}_{B(0, A)^c} T(f_1, f_2)\|_{L^r(w^r)} 
\lesssim \Phi_{\varepsilon}(C_0 [\vec{w}]_{A_{\vec{r}}}^{\alpha/2}), 
\\
\label{eq:rrc-3}
& \sup_{\substack{\|f_1\|_{L^{r_1}(w_1^{r_1})} \le 1 \\ \|f_2\|_{L^{r_2}(w_2^{r_2})} \le 1}} 
\| S_{\delta}^a (T(f_1, f_2)) \|_{L^r(w^r)} 
\lesssim \Phi_{\varepsilon}(C_0 [\vec{w}]_{A_{\vec{r}}}^{\alpha/2}),
\end{align}
where the implicit constant are independent of $\varepsilon$. Note that $\lim_{\varepsilon \to 0} \Phi_{\varepsilon}(t) = 0$ for every $t \ge 1$. Consequently, \eqref{eq:rrc-1}--\eqref{eq:rrc-3} and Theorem \ref{thm:KRWA} imply $T$ is compact from $L^{r_1}(w_1^{r_1}) \times L^{r_2}(w_2^{r_2})$ to $L^r(w^r)$. This completes the proof of Theorem \ref{thm:EP-Lp}.

\subsection{Extrapolation from $L^{p, \infty}$ compactness} 
Before showing Theorem \ref{thm:EP-Lpinfty}, let us present a bilinear version of the Marcinkiewicz interpolation theorem with initial restricted weak type conditions, which is a straightforward corollary of \cite[Theorem 7.2.2]{Gra2}.

\begin{theorem}\label{thm:BM}
Let $\frac{1}{p^k} = \frac{1}{p^k_1} + \frac{1}{p^k_2}$ with $0<p^k_1, p^k_2 \le \infty$, $k=0, 1, 2$. Assume that $T$ is a bilinear operator such that 
\begin{align*}
& \|T(\mathbf{1}_{E_1}, \mathbf{1}_{E_2})\|_{L^{p^0, \infty}} 
\le C_0 |E_1|^{\frac{1}{p^0_1}} |E_2|^{\frac{1}{p^0_2}}, 
\\
& \|T(\mathbf{1}_{E_1}, \mathbf{1}_{E_2})\|_{L^{p^1, \infty}} 
\le C_1 |E_1|^{\frac{1}{p^1_1}} |E_2|^{\frac{1}{p^1_2}}, 
\\
& \|T(\mathbf{1}_{E_1}, \mathbf{1}_{E_2})\|_{L^{p^2, \infty}} 
\le C_2 |E_1|^{\frac{1}{p^2_1}} |E_2|^{\frac{1}{p^2_2}}, 
\end{align*}
for all measurable sets $E_j \subset \Rn$ with $|E_j|<\infty$, $j=1, 2$. If the points $\big(\frac{1}{p^0_1}, \frac{1}{p^0_2} \big)$, $\big(\frac{1}{p^1_1}, \frac{1}{p^1_2} \big)$, and $\big(\frac{1}{p^2_1}, \frac{1}{p^2_2} \big)$ form a triangle in $\R^2$, then 
\begin{align*}
\|T\|_{L^{p_1} \times L^{p_2} \to L^p} 
\lesssim C_0^{\theta_0} C_1^{\theta_1} C_2^{\theta_2}, 
\end{align*}
for all $\theta_0 + \theta_1 + \theta_2 =1$ with $0<\theta_0, \theta_1, \theta_2<1$, and for all $\frac1p = \frac{1}{p_1} + \frac{1}{p_2}$ with 
\begin{align*}
\frac1p = \frac{\theta_0}{p^0} + \frac{\theta_1}{p^1} + \frac{\theta_2}{p^2} 
\quad\text{ and }\quad 
\frac{1}{p_j} = \frac{\theta_0}{p^0_j} + \frac{\theta_1}{p^1_j} + \frac{\theta_2}{p^2_j}, 
\quad j=1, 2. 
\end{align*}
\end{theorem}

\begin{proof}[\bf Proof of Theorem \ref{thm:EP-Lpinfty}]
Let $\frac1r = \frac{1}{r_1} + \frac{1}{r_2}$ with $r_1, r_2 \in (1, \infty)$. It suffices to show 
\begin{align}\label{Lrcpt-1}
\text{$T$ is compact from $L^{r_1}(\Rn) \times L^{r_2}(\Rn)$ to $L^r(\Rn)$}. 
\end{align}
Assuming \eqref{Lrcpt-1} momentarily, we conclude the proof as follows. By Theorem \ref{thm:RdF}, the $L^{q_1}(v_1^{q_1}) \times L^{q_2}(v_2^{q_2}) \to L^q(v^q)$ boundedness implies  
\begin{align}\label{Lrcpt-2}
\text{$T$ is bounded from $L^{s_1}(u_1^{s_1}) \times L^{s_2}(u_2^{s_2})$ to $L^s(u^s)$}, 
\end{align}
for all $s_1, s_2 \in (1, \infty)$ and for all $(u_1, u_2) \in A_{(s_1, s_2)}$, where $\frac1s = \frac{1}{s_1} + \frac{1}{s_2}$ and $u=u_1 u_2$. Thus, the desired conclusion immediately follows from Theorem \ref{thm:EP-Lp}, \eqref{Lrcpt-1}, and \eqref{Lrcpt-2}.

Let us next justify \eqref{Lrcpt-1}. Recall that 
\begin{align}\label{Tqq-1} 
\text{$T$ is compact from $L^{p_1}(\Rn) \times L^{p_2}(\Rn)$ to $L^{p, \infty}(\Rn)$}.  
\end{align}
Pick $\frac1q = \frac{1}{q_1} + \frac{1}{q_2}$ with $q_1, q_2 \in (4, \infty)$ and $q > p$ if $p \ne \infty$. The estimates \eqref{Lrcpt-2} and \eqref{Tqq-1} give
\begin{align}\label{Tqq-2}
\|T\|_{L^{p_1} \times L^{p_2} \to L^{p, \infty}} + 
\|T\|_{L^{q_1} \times L^{q_2} \to L^q} \le M < \infty. 
\end{align}
Given $\max \big\{\frac{3q_1}{4(q_1 - 1)}, \frac{3q_2}{4(q_2 - 1)} \big\} < \theta < 1$, define 
\begin{align}\label{rpq-1}
\frac{1}{r^0} = \frac{1-\theta}{p} + \frac{\theta}{q} 
\quad\text{ and }\quad  
\frac{1}{r^0_j} = \frac{1-\theta}{p_j} + \frac{\theta}{q_j} < \frac14, \quad j=1, 2. 
\end{align} 
Then properly choose $\theta$ in the above interval such that the point $\big(\frac{1}{r^0_1}, \frac{1}{r^0_2} \big)$ is different from $\big(\frac{1}{r_1}, \frac{1}{r_2} \big)$. In addition, for each $k=1, 2$, take $\frac{1}{r^k} = \frac{1}{r^k_1} + \frac{1}{r^k_2}$ with $1<r^k_1, r^k_2<\infty$ so that 
\begin{equation}\label{rpq-2}
\begin{aligned}
&\textstyle\text{the points $\big(\frac{1}{r^0_1}, \frac{1}{r^0_2} \big)$, $\big(\frac{1}{r^1_1}, \frac{1}{r^1_2} \big)$, and $\big(\frac{1}{r^2_1}, \frac{1}{r^2_2} \big)$ form a triangle in $\R^2$}, 
\\
&\textstyle\text{whose interior contains the target point $\big(\frac{1}{r_1}, \frac{1}{r_2} \big)$}.
\end{aligned}
\end{equation}
This enables us to write  
\begin{align}\label{rpq-3}
\frac1r = \frac{\theta_0}{r^0} + \frac{\theta_1}{r^1} + \frac{\theta_2}{r^2} 
\quad\text{ and }\quad 
\frac{1}{r_j} = \frac{\theta_0}{r^0_j} + \frac{\theta_1}{r^1_j} + \frac{\theta_2}{r^2_j}, 
\quad j=1, 2, 
\end{align}
for some $\theta_0, \theta_1, \theta_2 \in (0, 1)$ satisfying $\theta_0 + \theta_1 + \theta_2 =1$. 

Let
\begin{align*}
\G & := \bigg\{\phi_{E_1, E_2} 
:= \frac{T(\mathbf{1}_{E_1}, \mathbf{1}_{E_2})}{|E_1|^{1/p_1} |E_2|^{1/p_2}}: \, 
0 < |E_1|, |E_2| < \infty \bigg\},
\\
\K & := \bigg\{\Phi_{E_1, E_2} 
:= \frac{T(\mathbf{1}_{E_1}, \mathbf{1}_{E_2})}{|E_1|^{1/r_1^0} |E_2|^{1/r_2^0}}: \, 
0 < |E_1|, |E_2| < \infty \bigg\}.
\end{align*}
In view of \eqref{Tqq-1} and \eqref{Tqq-2}, 
\begin{align}\label{eq:gpcp}
\G \text{ is precompact in } L^{p, \infty}(\Rn)  
\quad \text{ and } \quad
\sup_{\phi \in \G} \|\phi\|_{L^{p, \infty}} \le M.
\end{align}
We claim that
\begin{align}\label{eq:Krr}
\K \text{ is precompact in } L^{r^0}(\Rn).
\end{align}
Fix measurable sets $E_1, E_2 \subset \Rn$ with $0 < |E_1|, |E_2| < \infty$. If we set $\lambda_{E_1, E_2} := \prod_{i=1}^2 |E_i|^{\frac{1}{p_i} - \frac{1}{q_i}}$, then by \eqref{Tqq-2} and \eqref{rpq-1}, 
\begin{align}\label{eq:eep}
\lambda_{E_1, E_2} \|\phi_{E_1, E_2}\|_{L^q} \le M
\quad \text{ and } \quad
\Phi_{E_1, E_2}
= \lambda_{E_1, E_2}^{\theta} \phi_{E_1, E_2}.
\end{align}
Note that 
\begin{align}\label{eq:Lorn}
\|f\|_{L^{r^0}}
\lesssim \|f\|_{L^{p, \infty}}^{1-\theta} \|f\|_{L^q}^{\theta}, 
\quad\forall f \in L^{p, \infty}(\Rn) \cap L^q(\Rn).
\end{align}
Indeed, the case $p = \infty$ is trivial, and the case $p \ne \infty$ is contained in \cite[Proposition 1.1.14]{Gra1}. As a consequence of \eqref{eq:eep} and \eqref{eq:Lorn}, we have
\begin{align}\label{eq:gpq}
\G \subset L^{p, \infty}(\Rn) \cap L^q(\Rn)
\quad \text{ and } \quad
\sup_{\Phi \in \K} \|\Phi\|_{L^{r^0}}
\lesssim M.
\end{align}
Let $\{\Phi_j\}$ be an arbitrary sequence in $\K$. In light of \eqref{eq:eep}, there exist a sequence of positive numbers $\{\lambda_j\}$ and a sequence of functions $\{\phi_j\} \subset \G$ such that 
\begin{align}\label{eq:ppjm}
\Phi_j = \lambda_j^{\theta} \phi_j 
\quad \text{ and } \quad 
\lambda_j \|\phi_j\|_{L^q} \le M.
\end{align} 
By the precompactness of $\G$ (see \eqref{eq:gpcp}), one can find a subsequence (which we relabel) $\{\phi_j\}$ such that $\phi_j \to \phi$ in $L^{p, \infty}(\Rn)$. Then choose a subsequence (which we relabel) $\{\lambda_j\}$ such that $\lambda_j \to \lambda \in [0, \infty]$. 

If $\lambda \in (0, \infty)$, there exists some $N \ge 1$ such that $\lambda_j \ge \lambda/2$ for all $j \ge N$. Together with \eqref{eq:Lorn} and \eqref{eq:ppjm}, this implies for all $j, k \ge N$,
\begin{align*}
\|\phi_j - \phi_k\|_{L^{r^0}}
\lesssim \|\phi_j - \phi_k\|_{L^{p, \infty}}^{1-\theta} \|\phi_j - \phi_k\|_{L^q}^{\theta}
\lesssim \|\phi_j - \phi_k\|_{L^{p, \infty}}^{1-\theta} 
\to 0,
\end{align*}
as $j, k \to \infty$. By the completeness of $L^{r^0}(\Rn)$, $\phi_j \to \phi$ in $L^{r^0}(\Rn)$, which further yields $\Phi_j = \lambda_j^{\theta} \phi_j \to \lambda^{\theta} \phi$ in $L^{r^0}(\Rn)$. 

If $\lambda = 0$, then by the fact $\phi_j \to \phi$ in $L^{p, \infty}(\Rn)$, given $\varepsilon > 0$, there exists some $N \ge 1$ such that $\|\phi_j - \phi_N\|_{L^{p, \infty}} \le \varepsilon$ for all $j \ge N$. Then by \eqref{eq:Lorn} and \eqref{eq:ppjm}, for all $j \ge N$, 
\begin{align*}
\|\Phi_j\|_{L^{r^0}}
& \le \lambda_j^{\theta} \|\phi_N\|_{L^{r^0}}
+ \lambda_j^{\theta} \|\phi_j - \phi_N\|_{L^{r^0}}
\\
& \lesssim \lambda_j^{\theta} \|\phi_N\|_{L^{r^0}}
+ \varepsilon^{1-\theta} \lambda_j^{\theta} \|\phi_j - \phi_N\|_{L^q}^{\theta}
\\
& \lesssim \lambda_j^{\theta} \|\phi_N\|_{L^{r^0}}
+ \varepsilon^{1-\theta} M^{\theta},
\end{align*}
which leads to $\limsup_{j \to \infty} \|\Phi_j\|_{L^{r^0}} \lesssim \varepsilon^{1-\theta} M^{\theta}$. The arbitrariness of $\varepsilon$ indicates $\lim_{j \to \infty} \|\Phi_j\|_{L^{r^0}} = 0$. 

If $\lambda = \infty$, then by \eqref{eq:ppjm}, $\|\phi_j\|_{L^q} \to 0$, hence, $\phi_j$ converges to $0$ in measure. Since $\phi_j \to \phi$ in $L^{p, \infty}(\Rn)$, $\phi_j$ converges to $\phi$ in measure. Thus, $\phi = 0$ a.e, and $\|\phi_j\|_{L^{p, \infty}} \to 0$. Using \eqref{eq:Lorn} and \eqref{eq:ppjm}, we obtain
\begin{align*}
\|\Phi_j\|_{L^{r^0}}
& \lesssim \|\Phi_j\|_{L^{p, \infty}}^{1-\theta} \|\Phi_j\|_{L^q}^{\theta}
= \|\phi_j\|_{L^{p, \infty}}^{1-\theta} (\lambda_j \|\phi_j\|_{L^q})^{\theta}
\le \|\phi_j\|_{L^{p, \infty}}^{1-\theta} M^{\theta} \to 0,
\end{align*}
as $j \to \infty$. Consequently, the above three cases imply \eqref{eq:Krr}.

Let $\varepsilon>0$ be an arbitrary number and $E_1, E_2 \subset \Rn$ be measurable sets with $|E_1|, |E_2|<\infty$. By \eqref{eq:Krr} and Theorem \ref{thm:KRLpq},  there exist $A_0 = A_0(\varepsilon)>0$ and $\delta_0=\delta_0(\varepsilon)>0$ such that 
\begin{align}
\label{TPS-1} 
\|T(\mathbf{1}_{E_1}, \mathbf{1}_{E_2})\|_{L^{r^0, \infty}} 
& \le C_1 |E_1|^{\frac{1}{r^0_1}} |E_2|^{\frac{1}{r^0_2}}, 
\\
\label{TPS-2} 
\|T(\mathbf{1}_{E_1}, \mathbf{1}_{E_2}) \mathbf{1}_{B(0, A)^c}\|_{L^{r^0, \infty}} 
& \le \varepsilon |E_1|^{\frac{1}{r^0_1}} |E_2|^{\frac{1}{r^0_2}},  
\\
\label{TPS-3}
\|(\tau_h T - T)(\mathbf{1}_{E_1}, \mathbf{1}_{E_2})\|_{L^{r^0, \infty}} 
& \le \varepsilon |E_1|^{\frac{1}{r^0_1}} |E_2|^{\frac{1}{r^0_2}}, 
\end{align}
for all $A \ge A_0$ and $0<|h| \le \delta_0$, where $C_1$ is an absolute constant. In addition, by \eqref{Lrcpt-2}, one has 
\begin{align}\label{TPS-4}
\|T(\mathbf{1}_{E_1}, \mathbf{1}_{E_2})\|_{L^{r^1, \infty}} 
\lesssim \|T(\mathbf{1}_{E_1}, \mathbf{1}_{E_2})\|_{L^{r^1}} 
\le C_2 |E_1|^{\frac{1}{r^1_1}} |E_2|^{\frac{1}{r^1_2}}, 
\end{align}
for some uniform constant $C_2$, which further implies 
\begin{align}
\label{TPS-5}
\|T(\mathbf{1}_{E_1}, \mathbf{1}_{E_2}) \mathbf{1}_{B(0, A)^c}\|_{L^{r^1, \infty}} 
&\le C_2 |E_1|^{\frac{1}{r^1_1}} |E_2|^{\frac{1}{r^1_2}}, \quad\text{ for all } A>0, 
\\ 
\label{TPS-6}
\|(\tau_h T - T)(\mathbf{1}_{E_1}, \mathbf{1}_{E_2})\|_{L^{r^1, \infty}} 
&\le C_3 |E_1|^{\frac{1}{r^1_1}} |E_2|^{\frac{1}{r^1_2}}, 
\quad\text{ for all } h \in \Rn. 
\end{align}
Likewise, 
\begin{align}
\label{TPS-7}
\|T(\mathbf{1}_{E_1}, \mathbf{1}_{E_2})\|_{L^{r^2, \infty}} 
&\le C_4 |E_1|^{\frac{1}{r^2_1}} |E_2|^{\frac{1}{r^2_2}}, 
\\ 
\label{TPS-8}
\|T(\mathbf{1}_{E_1}, \mathbf{1}_{E_2}) \mathbf{1}_{B(0, A)^c}\|_{L^{r^2, \infty}} 
&\le C_4 |E_1|^{\frac{1}{r^2_1}} |E_2|^{\frac{1}{r^2_2}}, \quad\text{ for all } A>0, 
\\ 
\label{TPS-9}
\|(\tau_h T - T)(\mathbf{1}_{E_1}, \mathbf{1}_{E_2})\|_{L^{r^2, \infty}} 
&\le C_5 |E_1|^{\frac{1}{r^2_1}} |E_2|^{\frac{1}{r^2_2}}, 
\quad\text{ for all } h \in \Rn. 
\end{align}
We mention that all constants $C_1, \ldots, C_5$ are independent of $\varepsilon$, $E_1$, $E_2$, $A$, and $h$. 

By \eqref{rpq-2}, \eqref{rpq-3}, \eqref{TPS-1}, \eqref{TPS-4}, \eqref{TPS-7}, and Theorem \ref{thm:BM},  there holds 
\begin{equation}\label{TPS-11}
\sup_{\substack{\|f_1\|_{L^{r_1}} \le 1 \\ \|f_2\|_{L^{r_2}} \le 1}} 
\|T(f_1, f_2)\|_{L^r} 
\leq C_1^{\theta_0} C_2^{\theta_1} C_4^{\theta_2}.
\end{equation}
In light of \eqref{rpq-2}, \eqref{rpq-3}, \eqref{TPS-2}, \eqref{TPS-5}, and \eqref{TPS-8}, Theorem \ref{thm:BM} applied to the bilinear operator $(f_1, f_2) \mapsto T(f_1, f_2) \mathbf{1}_{B(0, A)^c}$ gives    
\begin{equation*}
\|T(f_1, f_2) \mathbf{1}_{B(0, A)^c}\|_{L^r} 
\leq \varepsilon^{\theta_0} C_2^{\theta_1} C_4^{\theta_2} 
\|f_1\|_{L^{r_1}} \|f_2\|_{L^{r_2}}, 
\quad\text{ for all } A \ge A_0, 
\end{equation*}
which leads to    
\begin{align}\label{TPS-12}
\lim_{A \to \infty} \sup_{\substack{\|f_1\|_{L^{r_1}} \le 1 \\ \|f_2\|_{L^{r_2}} \le 1}} 
\|T(f_1, f_2) \mathbf{1}_{B(0, A)^c}\|_{L^r} 
=0.
\end{align}
Moreover, considering \eqref{rpq-2}, \eqref{rpq-3}, \eqref{TPS-3}, \eqref{TPS-6}, and \eqref{TPS-9}, we invoke Theorem \ref{thm:BM} applied to the bilinear operator $\tau_h T - T$ to arrive at 
\begin{equation*}
\|(\tau_h T - T)(f_1, f_2)\|_{L^r} 
\leq \varepsilon^{\theta_0} C_3^{\theta_1} C_5^{\theta_2} 
\|f_1\|_{L^{r_1}} \|f_2\|_{L^{r_2}}, 
\quad\text{ for all } 0<|h| \le \delta_0, 
\end{equation*}
which means    
\begin{align}\label{TPS-13}
\lim_{|h| \to 0} \sup_{\substack{\|f_1\|_{L^{r_1}} \le 1 \\ \|f_2\|_{L^{r_2}} \le 1}} 
\|(\tau_h T - T)(f_1, f_2)\|_{L^r} 
=0.
\end{align}
Therefore, \eqref{Lrcpt-1} follows at once from \eqref{TPS-11}--\eqref{TPS-13} and Theorem \ref{thm:KRLpq} with $p=q=r$. 
\end{proof}

\subsection{Extrapolation from $\CMO$ compactness}
Let us prove Theorem \ref{thm:EP-Linfty}. In view of Theorem \ref{thm:EP-Lp}, it suffices to prove
\begin{align}\label{Lrr-1}
T: L^4(\Rn) \times L^{\infty}(\Rn) \to L^4(\Rn) \, \text{ is compact}. 
\end{align}
Fix a dyadic grid $\D$. In view of Theorem \ref{thm:RdF}, the boundedness hypothesis implies   
\begin{align*}
T: L^p(\Rn) \times L^{\infty}(\Rn) \to L^p(\Rn) \, \text{ is bounded}, \quad \forall p \in (1, \infty).
\end{align*}
This immediately gives
\begin{align}\label{Lrr-2}
\sup_{N \ge 1} \|P_{\D(N)}^{\perp} T\|_{L^2 \times L^{\infty} \to L^2}
\le \|T\|_{L^2 \times L^{\infty} \to L^2} 
\le C_0 < \infty.
\end{align} 
Let $\varepsilon>0$ be an arbitrary number. In light of Theorem \ref{thm:PNT-cpt} \eqref{list:PN3}, the $L^{\infty} \times L^{\infty} \to \CMO$ compactness implies that there exists $N_0 = N_0(\varepsilon) \ge 1$ such that  
\begin{align}\label{Lrr-3}
\|P_{\D(N)}^{\perp} T\|_{L^{\infty} \times L^{\infty} \to \BMO_{\D}} 
\le \varepsilon, \quad \text{for all } N \ge N_0. 
\end{align}
A careful examination of the proof of \cite[Theorem 3.4.7]{Gra2}) reveals that it also holds for $\BMO_{\D}$ in place of $\BMO$. Hence, using interpolation between \eqref{Lrr-2} and \eqref{Lrr-3}, we obtain 
\begin{align*}
\|P_{\D(N)}^{\perp} T\|_{L^4 \times L^{\infty} \to L^4} 
\le C_0^{1/2} \varepsilon^{1/2}, \quad\text{for all } N \ge N_0.
\end{align*}
Consequently, by Theorem \ref{thm:PNT-cpt} \eqref{list:PN1}, $T$ is compact from $L^4(\Rn) \times L^{\infty}(\Rn)$ to $L^4(\Rn)$. 
\qed

\section{$L^1 \times L^1 \to L^{\frac12, \infty}$ compactness}\label{sec:endpoint}
Let us next conclude ${\rm (d) \Longrightarrow (c)'}$ from Theorem \ref{thm:EP-Lpinfty}, and ${\rm (a) \Longrightarrow (d)'}$ from Theorems \ref{thm:KRLpq} and \ref{thm:KRW}. Moreover, the Calder\'{o}n--Zygmund decomposition is an indispensable tool in the proof.

\subsection{$L^1 \times L^1 \to L^{\frac12, \infty}$ compactness implies weighted $L^{p_1} \times L^{p_2} \to L^p$ compactness}

\begin{proposition}\label{pro:WL}
Let $T$ be a bilinear operator associated with a standard bilinear Calder\'{o}n--Zygmund kernel. Then the following are equivalent: 
\begin{list}{\textup{(\theenumi)}}{\usecounter{enumi}\leftmargin=1cm \labelwidth=1cm \itemsep=0.2cm 
			\topsep=.2cm \renewcommand{\theenumi}{\roman{enumi}}}		

\item\label{WL-1} $T$ is bounded from $L^1(\Rn) \times L^1(\Rn)$ to $L^{\frac12, \infty}(\Rn)$.

\item\label{WL-2} $T$ is bounded from $L^1(w_1) \times L^1(w_2)$ to $L^{\frac12, \infty}(w^{\frac12})$ for all $(w_1, w_2) \in A_{(1, 1)}$, where $w=w_1 w_2$.

\item\label{WL-3} $T$ is bounded from $L^{p_1}(\Rn) \times L^{p_2}(\Rn)$ to $L^p(\Rn)$ for all (or for some) $p_1, p_2 \in (1, \infty]$, where $\frac1p = \frac{1}{p_1} + \frac{1}{p_2}>0$.  

\item\label{WL-4} $T$ is bounded from $L^{p_1}(w_1^{p_1}) \times L^{p_2}(w_2^{p_2})$ to $L^p(w^p)$ for all (or for some) $p_1, p_2 \in (1, \infty]$ and for all $(w_1, w_2) \in A_{(p_1, p_2)}$, where $\frac1p = \frac{1}{p_1} + \frac{1}{p_2}>0$ and $w=w_1 w_2$.

\end{list} 
\end{proposition}

\begin{proof}
It suffices to prove the implications: $\eqref{WL-1} \Longrightarrow \eqref{WL-4} \Longrightarrow \eqref{WL-3} \Longrightarrow \eqref{WL-2} \Longrightarrow \eqref{WL-1}$. The implication $\eqref{WL-2} \Longrightarrow \eqref{WL-1}$ is trivial, and  $\eqref{WL-1} \Longrightarrow \eqref{WL-4}$ is a consequence of \cite[Theorem 3.5]{Li} applied to $r=1$ and Theorem \ref{thm:RdF}. If \eqref{WL-4} holds for some exponent $\frac1p = \frac{1}{p_1} + \frac{1}{p_2} >0$ with $p_1, p_2 \in (1, \infty]$, then Theorem \ref{thm:RdF} gives that \eqref{WL-4} holds for all such exponents $(p, p_1, p_2)$. Thus, \eqref{WL-4} implies \eqref{WL-3} holds for all exponents $\frac1p = \frac{1}{p_1} + \frac{1}{p_2} >0$ with $p_1, p_2 \in (1, \infty]$. 

Now assuming \eqref{WL-3} holds for some exponent $\frac1p = \frac{1}{p_1} + \frac{1}{p_2} >0$ with $p_1, p_2 \in (1, \infty]$. By \cite[Theorem 1]{GT1}, this yields \eqref{WL-1}. As argued above, it has been shown that \eqref{WL-1} implies \eqref{WL-3} holds for all exponents $\frac1p = \frac{1}{p_1} + \frac{1}{p_2} >0$ with $p_1, p_2 \in (1, \infty]$. The latter along with \cite[Corollary 3.9]{LOPTT} gives that \eqref{WL-2} holds. Hence, $\eqref{WL-3} \Longrightarrow \eqref{WL-2}$. 
\end{proof}

Let us justify ${\rm (d) \Longrightarrow (c)'}$ in Theorem \ref{thm:cpt}. Let $T$ be a bilinear operator associated with a standard bilinear Calder\'{o}n--Zygmund kernel such that  
\begin{align}\label{Nec-1}
\text{$T$ is compact from $L^1(\Rn) \times L^1(\Rn)$ to $L^{\frac12,\infty}(\Rn)$.}
\end{align}
In particular, 
\begin{align*}
\text{$T$ is bounded from $L^1(\Rn) \times L^1(\Rn)$ to $L^{\frac12,\infty}(\Rn)$,}
\end{align*}
which together with Proposition \ref{pro:WL} yields 
\begin{align}\label{Nec-2}
\text{$T$ is bounded from $L^{q_1}(v_1^{q_1}) \times L^{q_2}(v_2^{q_2})$ to $L^q(v^q)$}  
\end{align}
for all $q_1, q_2 \in (1, \infty)$ and for all $(v_1, v_2) \in A_{(q_1, q_2)}$, where $\frac1q = \frac{1}{q_1} + \frac{1}{q_2}$ and $v=v_1 v_2$. In light of \eqref{Nec-1} and \eqref{Nec-2}, we invoke Theorem \ref{thm:EP-Lpinfty} to conclude that 
\begin{align}\label{Nec-3}
\text{$T$ is compact from $L^{p_1}(w_1^{p_1}) \times L^{p_2}(w_2^{p_2})$ to $L^p(w^p)$}
\end{align}
for all $p_1, p_2 \in (1, \infty]$ and for all $(w_1, w_2) \in A_{(p_1, p_2)}$, where $\frac1p = \frac{1}{p_1} + \frac{1}{p_2} >0$ and $w=w_1 w_2$. 
\qed

\subsection{Hypotheses \eqref{H1}--\eqref{H3} imply $L^1 \times L^1 \to L^{\frac12, \infty}$ compactness}
\begin{theorem}\label{thm:T110} 
Let $T$ be a bilinear operator associated with a standard bilinear Calder\'{o}n--Zygmund kernel. Assume that $T$ satisfies the hypotheses \eqref{H1}, \eqref{H2}, and \eqref{H3}. Then $T$ can be extended to a compact operator from $L^1(\Rn) \times L^1(\Rn)$ to $L^{\frac12, \infty}(\Rn)$.  
\end{theorem}

\begin{proof}
In view of Theorem \ref{thm:KRLpq}, it suffices to prove the following estimates: 
\begin{align}
\label{L1Li-1} & \sup_{\substack{\|f_1\|_{L^1} \le 1 \\ \|f_2\|_{L^1} \le 1}} 
\|T(f_1, f_2)\|_{L^{\frac12, \infty}} < \infty, 
\\ 
\label{L1Li-2} \lim_{A \to \infty} & \sup_{\substack{\|f_1\|_{L^1} \le 1 \\ \|f_2\|_{L^1} \le 1}} 
\|T(f_1, f_2) \mathbf{1}_{B(0, A)^c}\|_{L^{\frac12, \infty}} = 0, 
\\
\label{L1Li-3} \lim_{|h| \to 0} & \sup_{\substack{\|f_1\|_{L^1} \le 1 \\ \|f_2\|_{L^1} \le 1}}  
\|(\tau_h T - T)(f_1, f_2)\|_{L^{\frac12, \infty}}=0. 
\end{align}
By the hypotheses \eqref{H2} and \eqref{H3}, $T$ satisfies the weak boundedness property and $T(1, 1)$, $T^{*1}(1, 1)$, $T^{*2}(1, 1) \in \BMO(\Rn)$. This, along with Theorem \ref{thm:WBP}, implies that 
\begin{align}\label{Trr}
\text{$T$ is bounded from $L^{r_1}(\Rn) \times L^{r_2}(\Rn)$ to $L^r(\Rn)$}, 
\end{align}
for all $\frac1r = \frac{1}{r_1} + \frac{1}{r_2}$ with $1<r, r_1, r_2<\infty$. Then it follows from \eqref{Trr} and \cite[Theorem 1]{GT1} that 
\begin{align*}
\|T\|_{L^1 \times L^1 \to L^{\frac12, \infty}} \lesssim 1. 
\end{align*} 
Thus, \eqref{L1Li-1} holds.

To obtain \eqref{L1Li-2}, by homogeneity, it is enough to show that given $\varepsilon>0$, there exists $A_0 = A_0(\varepsilon) > 0$ such that  
\begin{align}\label{Tgb-1}
\mathcal{I} 
:= |\{x \in B(0, A)^c: |T(f_1, f_2)(x)| > 1\}|
\lesssim \varepsilon \|f_1\|_{L^1}^{\frac12} \|f_2\|_{L^1}^{\frac12}, 
\end{align}
for all $A \ge A_0$ and $f_1, f_2 \in L^1(\Rn)$ with $\|f_1\|_{L^1} = \|f_2\|_{L^1} = \mu$. Fix $\varepsilon>0$. For each $j=1, 2$, applying the Calder\'{o}n--Zygmund decomposition to the function $f_j$ at height $\alpha = \varepsilon^{-1}$ (cf. \cite[Theorem 5.3.1]{Gra1}), we obtain the following: 
\begin{list}{(\theenumi)}{\usecounter{enumi}\leftmargin=1.2cm \labelwidth=1cm \itemsep=0.2cm \topsep=.2cm \renewcommand{\theenumi}{\rm{CZ-\arabic{enumi}}}}

\item\label{CZ-1}  $f_j = g_j + b_j = g_j + \sum_{Q \in \Lambda_j} b_{j, Q}$;

\item\label{CZ-2} $\|g_j\|_{L^s} \lesssim \varepsilon^{\frac1s-1} \mu^{\frac1s}$ for all $s \in [1, \infty]$;

\item\label{CZ-3} $\supp (b_{j, Q}) \subset Q$,  $\int_{\Rn} b_{j, Q} \, dx =0$, and $\|b_{j, Q}\|_{L^1} \lesssim  \varepsilon^{-1} |Q|$; 

\item\label{CZ-4} the dyadic cubes $\{Q\}_{Q \in \Lambda_j}$ are disjoint and $\sum_{Q \in \Lambda_j} |Q| \le \varepsilon \mu$. 
\end{list}
Denote $\Omega := \bigcup_{j=1}^2 \bigcup_{Q \in \Lambda_j} 3Q$. Then, $|\Omega| \lesssim \varepsilon \mu$. By \eqref{CZ-1}, we have  
\begin{align}\label{Tgb-2} 
\mathcal{I} 
& \le |\{x \in B(0, A)^c: |T(g_1, g_2)(x)| > 1/4\}| 
\\ \nonumber 
&\quad + |\{x \in B(0, A)^c: |T(g_1, b_2)(x)| > 1/4\}|
\\ \nonumber 
&\quad + |\{x \in B(0, A)^c: |T(b_1, g_2)(x)| > 1/4\}|
\\ \nonumber 
&\quad + |\{x \in B(0, A)^c: |T(b_1, b_2)(x)| > 1/4\}|
\\ \nonumber 
&=: \mathcal{I}_{g, g}  + \mathcal{I}_{g, b} + \mathcal{I}_{b, g} + \mathcal{I}_{b, b}.   
\end{align}
Choose $p, p_1, p_2 \in (1, \infty)$ satisfying $\frac1p = \frac{1}{p_1} + \frac{1}{p_2}$. Recall that we have already  shown that $\text{(a)} \Longrightarrow \text{(b)} \Longrightarrow {\rm (c)}' \Longrightarrow {\rm (c)}$ in Theorem \ref{thm:cpt}. Thus, 
\begin{align*}
\text{$T$ is compact from $L^{p_1}(\Rn) \times L^{p_2}(\Rn)$ to $L^p(\Rn)$.} 
\end{align*}
By Theorem \ref{thm:KRLpq} applied to $p=q$, there exists $A_0 = A_0(\varepsilon) > 0$ and $\delta_0 = \delta_0(\varepsilon) > 0 $ so that 
\begin{align}
\label{TAD-1}
\|T(h_1, h_2) \mathbf{1}_{B(0, A)^c}\|_{L^p} 
&\le \varepsilon^2 \|h_1\|_{L^{p_1}} \|h_2\|_{L^{p_2}}, 
\quad\text{ for all } A \ge A_0, 
\\
\label{TAD-2}
\|(\tau_h T - T)(h_1, h_2)\|_{L^p} 
&\le \varepsilon^2 \|h_1\|_{L^{p_1}} \|h_2\|_{L^{p_2}}, 
\quad\text{ for all } 0<|h| \le \delta_0, 
\end{align}
for all $h_1 \in L^{p_1}(\Rn)$ and $h_2 \in L^{p_2}(\Rn)$. Then for all $A \ge A_0$, \eqref{CZ-2} and \eqref{TAD-1} give  
\begin{align}\label{eq:Igg}
\mathcal{I}_{g, g}  
\lesssim \|T(g_1, g_2) \mathbf{1}_{B(0, A)^c}\|_{L^p}^p 
\lesssim \varepsilon^{2p} \|g_1\|_{L^{p_1}}^p \|g_2\|_{L^{p_2}}^p 
\lesssim \varepsilon \mu. 
\end{align}
By \eqref{CZ-2}--\eqref{CZ-4} and Lemma \ref{lem:Igb} applied to $\eta=\varepsilon^2$, there holds 
\begin{align}\label{eq:Igb}
\mathcal{I}_{g, b}  
&\le |\Omega| + |\{x \in \Omega^c \cap B(0, A)^c: |T(g_1, b_2)(x)| > 1/4\}| 
\\ \nonumber 
&\lesssim \varepsilon \mu + \sum_Q \int_{(3Q)^c \cap B(0, A)^c} |T(g_1, b_{2, Q})(x)| \, dx 
\\ \nonumber 
&\lesssim \varepsilon \mu + \varepsilon^2 \|g_1\|_{L^{\infty}} \sum_Q \|b_{2, Q}\|_{L^1} 
\lesssim \varepsilon \mu. 
\end{align}
Symmetrically, 
\begin{align}\label{eq:Ibg}
\mathcal{I}_{b, g} 
\lesssim \varepsilon \mu. 
\end{align}
Moreover, Lemma \ref{lem:Ibb} implies that for any $A \ge A_0$ large enough, 
\begin{align}\label{eq:Ibb}
\mathcal{I}_{b, b}  
&\le |\Omega| + |\{x \in \Omega^c \cap B(0, A)^c: |T(b_1, b_2)(x)| > 1/4\}| 
\\ \nonumber 
&\lesssim |\Omega| + \int_{\Omega^c \cap B(0, A)^c} |T(b_1, b_2)(x)|^{\frac12} \, dx 
\lesssim \varepsilon \mu. 
\end{align}
Consequently, \eqref{Tgb-1} follows from \eqref{Tgb-2} and \eqref{eq:Igg}--\eqref{eq:Ibb}. 

It remains to verify  \eqref{L1Li-3}. By homogeneity, it is enough to show that given $\varepsilon>0$, there exists $\delta_0 = \delta_0(\varepsilon) > 0$ such that  
\begin{align}\label{JJ-1}
\mathcal{J} 
:= |\{x \in \Rn: |(\tau_h T - T)(f_1, f_2)(x)| > 1\}|
\lesssim \varepsilon \mu, 
\end{align}
for all $0 < |h| \le \delta_0$ and $f_1, f_2 \in L^1(\Rn)$ with $\|f_1\|_{L^1} = \|f_2\|_{L^1} = \mu$. By the Calder\'{o}n--Zygmund decomposition above, we have 
\begin{align}\label{JJ-2}
\mathcal{J} 
&\le |\{x \in \Rn: |(\tau_h T - T)(g_1, g_2)(x)| > 1/4\}| 
\\ \nonumber 
&\quad+ |\{x \in \Rn: |(\tau_h T - T)(g_1, b_2)(x)| > 1/4\}| 
\\ \nonumber 
&\quad+ |\{x \in \Rn: |(\tau_h T - T)(b_1, g_2)(x)| > 1/4\}| 
\\ \nonumber 
&\quad+ |\{x \in \Rn: |(\tau_h T - T)(b_1, b_2)(x)| > 1/4\}| 
\\ \nonumber 
&=: \mathcal{J}_{g, g} + \mathcal{J}_{g, b} + \mathcal{J}_{b, g} + \mathcal{J}_{b, b}.   
\end{align}
Let $0<|h| \le \delta_0$. Then \eqref{CZ-2} and \eqref{TAD-2} give  
\begin{align}\label{eq:Jgg}
\mathcal{J}_{g, g}  
\lesssim \|(\tau_h T - T)(g_1, g_2)\|_{L^p}^p 
\lesssim \varepsilon^{2p} \|g_1\|_{L^{p_1}}^p \|g_2\|_{L^{p_2}}^p 
\lesssim \varepsilon \mu.   
\end{align}
To bound $\mathcal{J}_{g, b}$, we split 
\begin{align*}
\mathcal{J}_{g, b} 
&\le \Big|\Big\{x \in \Rn: \sum_Q 
|(\tau_h T -T) (g_1, b_{2, Q})(x)| > 1/4\Big\}\Big| 
\\ 
&\le \Big|\Big\{x \in \Rn: \sum_{\ell(Q) \le 2|h|} 
|\tau_h T (g_1, b_{2, Q})(x)| > 1/16\Big\}\Big|  
\\  
&\quad+ \Big|\Big\{x \in \Rn: \sum_{\ell(Q) \le 2|h|} 
|T(g_1, b_{2, Q})(x)| > 1/16\Big\}\Big|
\\  
&\quad+ \Big|\Big\{x \in \Rn: \sum_{\ell(Q) > 2|h|} 
|(\tau_h T -T) (g_1, b_{2, Q})(x)| > 1/8\Big\}\Big| 
\\  
&= 2 \Big|\Big\{x \in \Rn: \sum_{\ell(Q) \le 2|h|} 
|T(g_1, b_{2, Q})(x)| > 1/16\Big\}\Big| 
\\  
&\quad+ \Big|\Big\{x \in \Rn: \sum_{\ell(Q) > 2|h|} 
|(\tau_h T -T) (g_1, b_{2, Q})(x)| > 1/8\Big\}\Big| 
\\  
&\lesssim |\Omega| + \sum_{\ell(Q) \le 2|h|} \int_{\Omega^c} |T(g_1, b_{2, Q})| 
+ \sum_{\ell(Q) > 2|h|} \int_{\Omega^c} |(\tau_h T -T) (g_1, b_{2, Q})|.  
\end{align*}
Thus, by Lemma \ref{lem:Jgb} applied to $\eta=\varepsilon^2$ and \eqref{CZ-2}--\eqref{CZ-4}, 
\begin{align}\label{eq:Jgb}
\mathcal{J}_{g, b} 
\lesssim |\Omega| + \varepsilon^2 \|g_1\|_{L^{\infty}} \sum_Q \|b_{2, Q}\|_{L^1} 
\lesssim  \varepsilon \mu. 
\end{align}
Analogously, 
\begin{align}\label{eq:Jbg}
\mathcal{J}_{b, g} 
\lesssim  \varepsilon \mu. 
\end{align}
As above, to analyze $\mathcal{J}_{b, b}$, we write  
\begin{align*}
\mathcal{J}_{b, b} 
&\le \Big|\Big\{x \in \Rn: \sum_{Q_1, Q_2} 
|(\tau_h T -T) (b_{1, Q_1}, b_{2, Q_2})(x)| > 1/4\Big\}\Big| 
\\ 
&\le \Big|\Big\{x \in \Rn: \sum_{\substack{\ell(Q_1) \le 2|h| \\ \ell(Q_2) \le 2|h|}} 
|(\tau_h T -T) (b_{1, Q_1}, b_{2, Q_2})(x)| > 1/8\Big\}\Big|  
\\  
&\quad+ \Big|\Big\{x \in \Rn: \sum_{\substack{\ell(Q_1) > 2|h| \\ Q_2 \in \Lambda_2}}
|(\tau_h T - T)(b_{1, Q_1}, b_{2, Q_2})(x)| > 1/16\Big\}\Big|
\\  
&\quad+ \Big|\Big\{x \in \Rn: \sum_{\substack{Q_1 \in \Lambda_1 \\ \ell(Q_2) > 2|h|}} 
|(\tau_h T -T) (b_{1, Q_1}, b_{2, Q_2})(x)| > 1/16\Big\}\Big| 
\\  
&=: \mathcal{J}_{b, b}^1 + \mathcal{J}_{b, b}^2 + \mathcal{J}_{b, b}^3. 
\end{align*}
Invoking \eqref{Jbb-1} and the Cauchy--Schwarz inequality, we obtain 
\begin{align*}
\mathcal{J}_{b, b}^1 
&\le 2 \Big|\Big\{x \in \Rn: \sum_{\substack{\ell(Q_1) \le 2|h| \\ \ell(Q_2) \le 2|h|}} 
|T(b_{1, Q_1}, b_{2, Q_2})(x)| > 1/16\Big\}\Big|   
\\
&\lesssim |\Omega| + \int_{\Omega^c} 
\bigg[ \sum_{\substack{\ell(Q_1) \le 2|h| \\ \ell(Q_2) \le 2|h|}} 
|T(b_{1, Q_1}, b_{2, Q_2})(x)| \bigg]^{\frac12} \, dx 
\\
&\lesssim |\Omega| + F_1(|h|)^{\frac12} \int_{\Omega^c} 
\prod_{j=1}^2 \bigg[ \sum_{Q_j} \|b_{j, Q_j}\|_{L^1} 
\frac{\ell(Q_j)^{\frac{\delta'}{2}}}{|x-c_{Q_j}|^{n+\frac{\delta'}{2}}} \bigg]^{\frac12} \, dx
\\
&\le |\Omega| + F_1(|h|)^{\frac12} \prod_{j=1}^2 \bigg[ \sum_{Q_j} \|b_{j, Q_j}\|_{L^1} 
\int_{(3Q_j)^c} \frac{\ell(Q_j)^{\frac{\delta'}{2}}}{|x-c_{Q_j}|^{n+\frac{\delta'}{2}}} \, dx \bigg]^{\frac12}
\\
&\lesssim |\Omega| + F_1(|h|)^{\frac12} \mu
\lesssim \varepsilon \mu, 
\end{align*}
for any $0<|h| \le \delta_0$ sufficiently small, provided \eqref{CZ-3}--\eqref{CZ-4} and that $\lim_{t \to 0} F_1(t)=0$. Much as in the same way, the inequality \eqref{Jbb-2} implies  
\begin{align*}
\mathcal{J}_{b, b}^2 
&\lesssim |\Omega| + \int_{\Omega^c} 
\bigg[ \sum_{\substack{\ell(Q_1) > 2|h| \\ Q_2 \in \Lambda_2}} 
|(\tau_h T - T)(b_{1, Q_1}, b_{2, Q_2})(x)| \bigg]^{\frac12} \, dx 
\lesssim |\Omega| + F_1(|h|)^{\frac12} \mu
\lesssim \varepsilon \mu. 
\end{align*}
Symmetrically, one has  
\begin{align*}
\mathcal{J}_{b, b}^3  
\lesssim \varepsilon \mu. 
\end{align*} 
Gathering the estimates above, we achieve   
\begin{align}\label{eq:Jbb}
\mathcal{J}_{b, b}  
\le \mathcal{J}_{b, b}^1 + \mathcal{J}_{b, b}^2 + \mathcal{J}_{b, b}^3
\lesssim \varepsilon \mu. 
\end{align}
Therefore, \eqref{JJ-1} follows from \eqref{JJ-2}--\eqref{eq:Jbb}. 
\end{proof}

\subsection{Hypotheses \eqref{H1}--\eqref{H3} imply weighted $L^1 \times L^1 \to L^{\frac12, \infty}$ compactness}

\begin{lemma}\label{lem:osct}
Let $T$ be a bilinear operator associated with a standard bilinear Calder\'{o}n--Zygmund kernel. Assume that $T$ satisfies the hypotheses \eqref{H1}--\eqref{H3}. Then for all $f_1, f_2 \in L_c^{\infty}(\Rn)$ and cubes $Q \subset \Rn$,
\begin{align*}
\inf_{c \in \R} \bigg(\fint_Q |T(f_1, f_2)(x) - c|^{\frac14} \, dx \bigg)^4 
\lesssim \lambda(Q) \inf_Q \M(f_1, f_2),
\end{align*}
where $\lim_{N \to \infty} \sup_{Q \notin \Q(N)} \lambda(Q) = 0$.
\end{lemma}

\begin{proof}
Given a cube $Q \subset \Rn$, define
\begin{align*}
\alpha_Q 
& := \iint_{\R^{2n} \setminus (3Q)^2} 
K(c_Q, y, z) f_1(y) \, f_2(z) \, dy \, dz, 
\\
\beta_Q 
& := |Q|^2 \sup_{f \in \K} \bigg(\fint_Q |f(x)|^{\frac14} \, dx \bigg)^4,
\end{align*}
where $\K := \{T(f_1, f_2): \|f_i\|_{L^1} \le 1, i=1, 2\}$. Write $f_i^0 := f_i \mathbf{1}_{3Q}$ and $f_i^{\infty} := f_i \mathbf{1}_{(3Q)^c}$, $i=1, 2$. Note that 
\begin{align}\label{eq:tos-1}
\bigg(\fint_Q |T(f_1^0, f_2^0)(x)|^{\frac14} \, dx\bigg)^4
\le \beta_Q |Q|^{-2} \|f_1^0\|_{L^1} \|f_2^0\|_{L^1}
\lesssim \beta_Q \inf_Q \M(f_1, f_2).
\end{align}
Denote 
\begin{align*}
F(Q)
:= F_1(\ell(Q)) F_2(\ell(Q)) \sum_{k \ge 0} 2^{-k \delta'} F_3(\rd(2^k Q, \I)). 
\end{align*}
We split 
\begin{align}\label{eq:tos-2}
T(f_1, f_2)(x) - \alpha_Q 
= T(f_1^0, f_2^0)(x) + G_Q(x), 
\end{align}
where much as in \eqref{XRK}, for each $x, \xi \in Q$,
\begin{align}\label{eq:tos-3}
|G_Q(x)|
& := \bigg|\iint_{\R^{2n} \setminus (3Q)^2} 
[K(x, y, z) - K(c_Q, y, z)] f_1(y) f_2(z) \, dy \, dz \bigg|
\\ \nonumber 
&\lesssim \sum_{k \ge 0} \iint_{R_k} 
\frac{\ell(Q)^{\delta'} F(x, y, z)}{(|x-y| + |x-z|)^{2n+\delta'}} |f_1(y)| |f_2(z)| \, dy \, dz 
\\ \nonumber
&\lesssim F_1(\ell(Q)) F_2(\ell(Q)) \sum_{k \ge 0} 2^{-k \delta'} F_3(\rd(2^k Q, \I)) 
\prod_{i=1}^2 \langle |f_i| \rangle_{B(x, 2^{k+2}\ell(Q))}  
\\ \nonumber
& \lesssim F(Q) \, \inf_{\xi \in Q} \mathcal{M}(f_1, f_2)(\xi), 
\end{align}
where $R_k :=\{(y, z) \in \R^{2n}: 2^k \ell(Q) \le |x-y| + |x-z|< 2^{k+1} \ell(Q)\}$, and 
\begin{align*}
F(x, y, z) 
:= F_1(|x-c_Q|) F_2(|x-y|+|x-z|) F_3 \bigg(1 + \frac{|x+y| + |x+z|}{1+|x-y|+|x-z|} \bigg). 
\end{align*}
In view of \eqref{eq:tos-1}--\eqref{eq:tos-3}, it suffices to take $\lambda(Q) = \beta_Q + F(Q)$ and show 
\begin{align}\label{eq:tos-4}
\lim_{N \to \infty} \sup_{Q \notin \Q(N)} \beta_Q = 0.
\end{align}

Let $\varepsilon \in(0, 1)$. By the definition of $\K$ and \eqref{L1Li-1}--\eqref{L1Li-3}, there exists $r = r(\varepsilon) \ge 1$ such that 
\begin{align}
\label{fk121}
\sup_{f \in \K} & \|f\|_{L^{\frac12, \infty}} \le M < \infty,
\\
\label{fk122}
\sup_{f \in \K} & \|f \mathbf{1}_{B(0, r)^c} \|_{L^{\frac12, \infty}} \le \varepsilon,
\\
\label{fk123}
\sup_{f \in \K} & \|\tau_h f - f \|_{L^{\frac12, \infty}} \le \varepsilon, \quad 0 < |h| < 1/r.
\end{align}
Choose a nonnegative function $\varphi \in \mathscr{C}_c^{\infty}(\Rn)$ with $\supp \varphi \subset B(0, 1)$ and $\int_{\Rn} \varphi \, dx =1$. Let $t = 1/r$ and $\varphi_t(x) := t^{-n} \varphi(x/t)$. Let $N \ge 1$ be determined below and $f \in \K$. In the case $\ell(Q) < 2^{-N}$, by Kolmogorov's inequality \eqref{eq:Kolm} and \eqref{fk123},
\begin{align*}
\mathscr{J}_1 
& := |Q|^{-\frac12} \int_Q |\varphi_t* |f|^{\frac14}|
\le |Q|^{\frac12} \|\varphi_t* |f|^{\frac14}|\|_{L^{\infty}}
\\
& \lesssim |Q|^{\frac12} \sup_{x \in \Rn} \fint_{B(x, t)} |f|^{\frac14}
\lesssim (2^N/r)^{-\frac{n}{2}} \||f|^{\frac14}\|_{L^{2, \infty}},
\end{align*}
and
\begin{align*}
\mathscr{J}_2
& := |Q|^{-\frac12} \int_Q \big|\varphi_t* |f|^{\frac14} - |f|^{\frac14} \big|
\\
& \le |Q|^{-\frac12} \int_{B(0, t)} \varphi_t(y) 
\int_Q \big| |f(x-y)|^{\frac14} - |f(x)|^{\frac14} \big| \, dx \, dy
\\
& \lesssim \int_{B(0, t)} \varphi_t(y) \| |\tau_y f|^{\frac14} - |f|^{\frac14} \|_{L^{2, \infty}} \, dy
\\
& \le \int_{B(0, t)} \varphi_t(y) \| \tau_y f - f \|_{L^{1/2, \infty}}^{\frac14} \, dy
\le \varepsilon^{\frac14},
\end{align*}
which gives 
\begin{align}\label{fk124}
|Q|^{-\frac12} \int_Q |f|^{\frac14} 
\le \mathscr{J}_1 + \mathscr{J}_2
\lesssim (2^N/r)^{-\frac{n}{2}} + \varepsilon^{\frac14}.
\end{align}
In the case $\ell(Q) > 2^N$, \eqref{eq:Kolm} and \eqref{fk122} imply
\begin{align}\label{fk125}
|Q|^{-\frac12} \int_Q |f|^{\frac14}
& = |Q|^{-\frac12} \int_{Q \cap B(0, r)} |f|^{\frac14}
+ |Q|^{-\frac12} \int_{Q \cap B(0, r)^c} |f|^{\frac14}
\\ \nonumber
& \lesssim [r/\ell(Q)]^{\frac{n}{2}} \||f|^{\frac14} \|_{L^{2, \infty}}
+ \||f|^{\frac14} \mathbf{1}_{B(0, r)^c}\|_{L^{2, \infty}}
\\ \nonumber
& = [r/\ell(Q)]^{\frac{n}{2}} \|f\|_{L^{1/2, \infty}}^{\frac14} 
+ \|f \mathbf{1}_{B(0, r)^c}\|_{L^{1/2, \infty}}^{\frac14} 
\\ \nonumber
& \le (2^N/r)^{-\frac{n}{2}} + \varepsilon^{\frac14}.
\end{align}
In the case $2^{-N} \le \ell(Q) \le 2^N$ and $\rd(Q, 2^N \I) > N$, when $2^N > 2r$, there holds $B(0, r) \subset 2^N \I$,  $Q \cap B(0, r) = \varnothing$, and 
\begin{align}\label{fk126}
|Q|^{-\frac12} \int_Q |f|^{\frac14}
= |Q|^{-\frac12} \int_{Q \cap B(0, r)^c} |f|^{\frac14}
\lesssim \|f \mathbf{1}_{B(0, r)^c}\|_{L^{1/2, \infty}}^{\frac14} 
\le \varepsilon^{\frac14}.
\end{align}
Take a large integer $N_0 = N_0(\varepsilon) > 2r$ so that $(2^{N_0}/r)^{-\frac{n}{2}} \le \varepsilon^{\frac14}$. Hence, by \eqref{fk124}--\eqref{fk126}, for all $N \ge N_0$, $\sup_{Q \notin \Q(N)} \beta_Q \lesssim \varepsilon$. This shows \eqref{eq:tos-4}.
\end{proof}

Now we are ready to show $\text{(a)} \Longrightarrow \text{(d)}'$ in Theorem \ref{thm:cpt}. Let $(w_1, w_2) \in A_{(1, 1)}$ and $w=w_1 w_2$. Then $w^{\frac12} \in A_1$ by Lemma \ref{lem:weight}. 
It follows from \eqref{Trr}, \cite[Corollary 3.9]{LOPTT}, and \cite[Theorem 3.3]{LOPTT} that 
\begin{align}\label{eq:tmb}
T, \M: L^1(w_1) \times L^1(w_2) \to L^{\frac12, \infty}(w^{\frac12}) \quad \text{is bounded}.
\end{align}
Let $f_i \in L^1(w_i) \cap L^{\infty}_c(\Rn)$ with $\|f_i\|_{L^1(w_i)} \le 1$, $i=1, 2$. If we let $G := \sgn(T(f_1, f_2)) |T(f_1, f_2)|^{\frac14}$ and $H := \M(f_1, f_2)^{\frac14}$, then \eqref{eq:tmb} implies
\begin{align}\label{eq:ghb}
\|G\|_{L^{2, \infty}(w^{\frac12})} + \|H\|_{L^{2, \infty}(w^{\frac12})} \lesssim 1.
\end{align}
In view of Lemma \ref{lem:osct}, 
\begin{align}\label{eq:csg}
\fint_Q |G - \langle G \rangle_Q|
\le 2 \inf_{c \in \R} \fint_Q \big| G - \sgn(c) |c|^{\frac14} \big| 
\lesssim \lambda(Q)^{\frac14} \inf_Q H. 
\end{align}
Let $\D$ be a dyadic grid. We claim that
\begin{align}\label{eq:GPG}
M_{\D}^{\#} (G - P_{\D(N)} G) 
\lesssim \varepsilon_N \, M_{\D} H
\quad \text{ and } \quad
\varepsilon_N := \sup_{Q \in \D \setminus \D(N-1)} \lambda(Q)^{\frac14}.
\end{align}
Let $\mathcal{R}_N := \{Q \in \D(N): \ell(Q) = 2^N\}$. For every $Q \in \mathcal{R}_N$, set 
\begin{align*}
\S_N(Q) 
:= \big\{P \in \D \setminus \D(N): P \in \ch(Q') \text{ for some } Q' \in \D(N), Q' \subset Q  \big\}.
\end{align*}
Then for every $Q \in \mathcal{R}_N$, 
\begin{align*}
Q = \bigcup_{P \in \S_N(Q)} P
\quad \text{ and } \quad 
\S_N(Q) \text{ is pairwise disjoint}.
\end{align*}
Note that
\begin{align}\label{eq:pdg}
P_{\D(N)} f
= \sum_{Q \in \mathcal{R}_N} \sum_{P \in \S_N(Q)} 
\big(\langle f \rangle_P - \langle f \rangle_Q\big) \mathbf{1}_P.
\end{align}
Fix $Q' \in \D$ and $x \in Q'$. If $Q' \in \D(N)$, then there exists a unique $Q \in \mathcal{R}_N$ so that $Q' \subset Q$. For every $P \subset Q'$,
\begin{align*}
G(x) - P_{\D(N)} G(x)
= G(x) - \big(\langle G \rangle_P - \langle G \rangle_Q\big), \quad x \in P,
\end{align*}
which gives $\langle G - P_{\D(N)} G \rangle_{Q'} = \langle G \rangle_Q$ and 
\begin{align*}
& \fint_{Q'} |(G - P_{\D(N)} G) - \langle G - P_{\D(N)} G \rangle_{Q'}|
= \frac{1}{|Q'|} \sum_{\substack{P \in \S_N(Q) \\ P \subset Q'}} 
\int_P |G - \langle G \rangle_P|
\\
& \lesssim \frac{1}{|Q'|} \sum_{\substack{P \in \S_N(Q) \\ P \subset Q'}} 
\lambda(P)^{\frac14} \int_P H
\le \varepsilon_N \fint_{Q'} H
\le \varepsilon_N M_{\D} H(x). 
\end{align*}
If $Q' \notin \D(N)$ and $\ell(Q') \le 2^N$, then $Q'' \notin \D(N)$ for all $Q'' \subset Q'$. Thus, $P_{\D(N)} G$ is a constant on $Q'$, which along with \eqref{eq:csg} yields
\begin{align*}
& \fint_{Q'} |(G - P_{\D(N)} G) - \langle G - P_{\D(N)} G \rangle_{Q'}|
= \fint_{Q'} |G - \langle G \rangle_{Q'}|
\lesssim \varepsilon_N \fint_{Q'} H
\le \varepsilon_N \, M_{\D} H(x). 
\end{align*}
If $\ell(Q') > 2^N$, then $Q \subset Q'$ for all $Q \in \mathcal{R}_N$ with $Q \cap Q' \ne \varnothing$. Thus, $\langle P_{\D(N)} G \rangle_{Q'} = 0$. In addition, by \eqref{eq:pdg}, for every $Q \in \mathcal{R}_N$,
\begin{align*}
& \int_Q |P_{\D(N)} G|
= \sum_{P \in \S_N(Q)} |P| |\langle G \rangle_P - \langle G \rangle_Q|
\le \int_Q | G - \langle G \rangle_Q|,
\end{align*}
which together with \eqref{eq:csg} implies
\begin{align*}
\fint_{Q'} |P_{\D(N)} G|
& \le \frac{1}{|Q'|} \sum_{\substack{Q \in \mathcal{R}_N \\ Q \subset Q'}}
\int_Q |P_{\D(N)} G|
\le \frac{1}{|Q'|} \sum_{\substack{Q \in \mathcal{R}_N \\ Q \subset Q'}}
\int_Q | G - \langle G \rangle_Q|
\\
& \lesssim \frac{1}{|Q'|} \sum_{\substack{Q \in \mathcal{R}_N \\ Q \subset Q'}}
\lambda(Q)^{\frac14} \int_Q H
\le \varepsilon_N \fint_{Q'} H
\le \varepsilon_N \, M_{\D} H(x),
\end{align*}
and
\begin{align*}
& \fint_{Q'} |(G - P_{\D(N)} G) - \langle G - P_{\D(N)} G \rangle_{Q'}|
\\
& \le \fint_{Q'} |G - \langle G \rangle_{Q'}|
+ \fint_{Q'} |P_{\D(N)} G|
\\
& \lesssim \lambda(Q')^{\frac14} \inf_{Q'} H + \varepsilon_N \fint_{Q'} H
\lesssim \varepsilon_N \, M_{\D} H(x). 
\end{align*}
This verifies \eqref{eq:GPG}.

Since $w^{\frac12} \in A_1$ and
\begin{align*}
\|M_{\D}(G - P_{\D(N)} G)\|_{L^{2, \infty}(w^{\frac12})}
\lesssim \|M_{\D}(M_{\D} G)\|_{L^{2, \infty}(w^{\frac12})}
\lesssim \|G\|_{L^{2, \infty}(w^{\frac12})} \lesssim 1,
\end{align*}
we use \eqref{eq:GPG} and the dyadic Fefferman--Stein inequality (cf. \cite[Lemma 3.2]{DLZ}) to deduce 
\begin{align*}
& \|G - P_{\D(N)} G\|_{L^{2, \infty}(w^{\frac12})}
\lesssim \|M_{\D}^{\#} (G - P_{\D(N)} G) \|_{L^{2, \infty}(w^{\frac12})}
\\
& \lesssim \varepsilon_N \|M_{\D} H\|_{L^{2, \infty}(w^{\frac12})}
\lesssim \varepsilon_N \|H\|_{L^{2, \infty}(w^{\frac12})}
\lesssim \varepsilon_N \to 0.
\end{align*}
The operator $P_{\D(N)} G$ is compact because $P_{\D(N)}$ is a finite-dimensional operator and $G$ is a bounded operator. Thus, $G$ is compact.  

If we set $\Psi(f) := \sgn(f) |f|^4$, then the fact $|\Psi(f) - \Psi(g)| \lesssim |f - g| (|f|^3 + |g|^3)$ and H\"{o}lder's inequality imply 
\begin{align*}
\|\Psi(f) - \Psi(g)\|_{L^{\frac12, \infty}(w^{\frac12})}
\lesssim \|f - g\|_{L^{2, \infty}(w^{\frac12})}
\Big[\|f\|_{L^{2, \infty}(w^{\frac12})}^3 + \|g\|_{L^{2, \infty}(w^{\frac12})}^3 \Big]. 
\end{align*}
This means $\Psi$ is continuous. Hence, $T(f_1, f_2) = \Psi(G)$ is compact from $(L^1(w_1) \cap L_c^{\infty}(\Rn)) \times (L^1(w_2) \cap L_c^{\infty}(\Rn))$ to $L^{\frac12, \infty}(w^{\frac12})$. By the density of $L^1(w_i) \cap L_c^{\infty}(\Rn)$ in $L^1(w_i)$, we conclude that $T: L^1(w_1) \times L^1(w_2) \to L^{\frac12, \infty}(w^{\frac12})$ is compact. 
\qed

\subsection{Estimates involving bad functions}
This section contains some useful estimates, which have been used to prove Theorem \ref{thm:T110}.

\begin{lemma}\label{lem:Igb}
For any $\eta>0$, there exists $A_0=A_0(\eta)>0$ such that for all dyadic cubes $Q$, 
\begin{align*}
\int_{(3Q)^c \cap B(0, A)^c} |T(g_1, b_{2, Q})(x)| \, dx  
\le \eta \|g_1\|_{L^{\infty}} \|b_{2, Q}\|_{L^1}, 
\quad\text{ for all } A \ge A_0.  
\end{align*}
\end{lemma}

\begin{proof}
Let $\eta>0$. Let $A \ge 2^{10}$ and $N=[\frac12 \log_2 A] \ge 5$. Consider the case $Q \in \D(N)$. Then 
\begin{align*}
Q \subset B(0, (N+2) 2^N) 
\quad\text{ and}\quad 
3Q \subset B(0, (N+3)2^N) \subset B(0, 2^{2N-1}) \subset B(0, A/2), 
\end{align*}
which implies $(3Q)^c \cap B(0, A)^c = B(0, A)^c$, and for all $x \in B(0, A)^c$ and $z \in Q$, 
\begin{align*}
\min\{|x-z|, |x+z|\} 
\ge |x| - |z|
\ge A - A/2 = A/2 
\ge 16 \ell(Q) 
\ge 32 |z-c_Q|. 
\end{align*}
This, along with the cancellation of $b_{2, Q}$ and the H\"{o}lder condition of $K$, gives 
\begin{align}\label{BAT-1}
& \int_{(3Q)^c \cap B(0, A)^c} |T(g_1, b_{2, Q})(x)| \, dx  
\\ \nonumber 
&= \int_{B(0, A)^c} \bigg|\int_Q \int_{\Rn} \big(K(x, y, z) - K(x, y, c_Q) \big) 
g_1(y) b_{2, Q}(z) \, dy \, dz \bigg| \, dx 
\\ \nonumber 
&\lesssim \int_Q |b_{2, Q}(z)| \bigg[\int_{\Rn} \int_{B(0, A)^c}  
\frac{\ell(Q)^{\delta'} \, |g_1(y)|}{(|x-y| + |x-z|)^{2n+\delta'}} F(x, y, z) \, dx \, dy\bigg] \, dz
\\ \nonumber 
&\lesssim \int_Q |b_{2, Q}(z)| \bigg[\iint_{|x-y| +|x-z| \ge A/2}  
\frac{A^{\delta'} \, \|g_1\|_{L^{\infty}}}{(|x-y|+|x-z|)^{2n+\delta'}} F_3(A) \, dx \, dy\bigg] \, dz
\\ \nonumber 
&\lesssim F_3(A) \|g_1\|_{L^{\infty}} \|b_{2, Q}\|_{L^1}, 
\end{align}
where 
\begin{align*}
F(x, y) := F_1(|x-y|+|x-z|) F_2(|x-y|+|x-z|) F_3 (|x+y|+|x+z|) 
\lesssim F_3(A),  
\end{align*}
provided the monotonicity of $F_3$ and the boundedness of $F_1$ and $F_2$. 

Let us treat the case $Q \notin \D(N)$. For any $x \in (3Q)^c$ and $z \in Q$, we have $|x-z| \ge \ell(Q) \ge 2 |z-c_Q|$. Then by the cancellation of $b_{2, Q}$, the H\"{o}lder condition, and Lemma \ref{lem:improve} part \eqref{Holder-imp}, 
\begin{align*}
|T (g_1, b_{2, Q})(x)| 
&= \bigg|\iint_{\R^{2n}} \big(K(x, y, z) - K(x, y, c_Q) \big) g_1(y) b_{2, Q}(z) \, dy \, dz \bigg|
\\
&\lesssim \iint_{\R^{2n}} \frac{\ell(Q)^{\delta'}}{(|x-y| + |x-z|)^{2n+\delta'}} 
F(x, y, z) |g_1(y)| |b_{2, Q}(z)| \, dy \, dz,
\end{align*}
where 
\begin{align*}
F(x, y, z) 
:= F_1(|z-c_Q|) F_2(|x-y|+|x-z|) F_3 \bigg(1 + \frac{|x+y| + |x+z|}{1+|x-y|+|x-z|} \bigg). 
\end{align*}
For any $k \ge 0$ and $z \in Q$, set $R_k(z) :=\{(x, y) \in \R^{2n}: 2^k \ell(Q) \le |x-y| + |x-z|< 2^{k+1} \ell(Q)\}$. Then for any $(x, y) \in R_k(z)$, 
\begin{align}\label{XRK}
&1 + \frac{|x+y| + |x+z|}{1+|x-y|+|x-z|} 
\ge \frac{2|z|}{1+|x-y|+|x-z|} 
\\ \nonumber 
&\ge \frac{2 \d(2^k Q, \I)}{1+2^{k+1} \ell(Q)} 
\ge \frac12 \frac{\d(2^k Q, \I)}{\max\{1, 2^k \ell(Q)\}} 
= \frac12 \rd(2^k Q, \I). 
\end{align}
This in turn yields for any $z \in Q$, 
\begin{align}\label{GZ}
G(z) 
&:= \int_{\Rn} \int_{(3Q)^c}  \frac{\ell(Q)^{\delta'}}{(|x-y| + |x-z|)^{2n+\delta'}} F(x, y, z) |g_1(y)|\, dx \, dy
\\ \nonumber 
&\lesssim \|g_1\|_{L^{\infty}} F_1(\ell(Q)) F_2(\ell(Q)) \sum_{k \ge 0} \iint_{R_k(z)} 
\frac{\ell(Q)^{\delta'} \, dx \, dy}{(2^k \ell(Q))^{2n+\delta'}} F_3(\rd(2^k Q, \I)) 
\\ \nonumber
&\lesssim \|g_1\|_{L^{\infty}} F_1(\ell(Q)) F_2(\ell(Q)) \sum_{k \ge 0} 2^{-k \delta'/2} F_3(\rd(2^k Q, \I))
=: \|g_1\|_{L^{\infty}} F(Q), 
\end{align}
and hence, 
\begin{align}\label{BAT-2}
\int_{(3Q)^c \cap B(0, A)^c} |T(g_1, b_{2, Q})| \, dx  
\lesssim \int_Q |b_{2, Q}| G \, dz
\lesssim F(Q) \|g_1\|_{L^{\infty}} \|b_{2, Q}\|_{L^1}.
\end{align}
As a consequence of \eqref{BAT-1} and \eqref{BAT-2}, we deduce 
\begin{align*}
\int_{(3Q)^c \cap B(0, A)^c} |T(g_1, b_{2, Q})| \, dx  
\lesssim \big[F_3(A) + \sup_{Q \notin \D(N)} F(Q) \big] 
\|g_1\|_{L^{\infty}} \|b_Q\|_{L^1} 
\le \eta \|g_1\|_{L^{\infty}} \|b_Q\|_{L^1} , 
\end{align*}
provided $A \ge 2^{10}$ sufficiently large and the fact that $\lim_{N \to \infty} \sup_{Q \notin \D(N)} F(Q) = 0$ (cf. \cite[Lemma 4.5 (d)]{CYY}). This completes the proof. 
\end{proof}

\begin{lemma}\label{lem:Ibb}
For any $\eta>0$, there exists $A_0=A_0(\eta)>0$ such that 
\begin{align*}
\int_{\Omega^c \cap B(0, A)^c} |T(b_1, b_2)(x)|^{\frac12} \, dx 
\lesssim \eta \mu, 
\quad\text{ for all } A \ge A_0.  
\end{align*}
\end{lemma}

\begin{proof}
Let $\eta>0$, $A \ge 2^{10}$, and $N=[\frac12 \log_2 A] \ge 5$. It suffices to prove that for all $x \in \Omega^c \cap B(0, A)^c$,  
\begin{align}\label{tbb-1}
|T(b_{1, Q_1}, b_{2, Q_2})(x)| 
\lesssim F_3(A) \|b_{1, Q_1}\|_{L^1} 
\frac{\ell(Q_1)^{\frac{\delta'}{2}}}{|x-c_{Q_1}|^{n+\frac{\delta'}{2}}} 
\|b_{2, Q_2}\|_{L^1} 
\frac{\ell(Q_2)^{\frac{\delta'}{2}}}{|x-c_{Q_2}|^{n+\frac{\delta'}{2}}} 
\end{align}
whenever $Q_1, Q_2 \in \D(N)$, and 
\begin{align}\label{tbb-2}
|T(b_{1, Q_1}, b_{2, Q_2})(x)| 
\lesssim F_1(\min\{\ell(Q_1), \ell(Q_2)\}) F_2(\max\{\ell(Q_1), \ell(Q_2)\}) \|b_{1, Q_1}\|_{L^1}  
\\ \nonumber 
\times 
\frac{\ell(Q_1)^{\frac{\delta'}{2}}}{|x-c_{Q_1}|^{n+\frac{\delta'}{2}}} 
\int_{Q_2} \frac{|b_{2, Q_2}(z)| \, \ell(Q_2)^{\frac{\delta'}{2}}}{|x-c_{Q_{1, 2}}|^{n+\frac{\delta'}{2}}} 
F_3 \bigg( \frac{2|z|}{1+ |x-c_{Q_{1,2}}|} \bigg) \, dz
\end{align}
whenever $Q_2 \notin \D(N)$, where $|x-c_{Q_{1, 2}}| := |x-c_{Q_1}| + |x-c_{Q_2}|$. An analogy holds for $Q_1 \notin \D(N)$. Assuming these estimates momentarily, let us conclude the proof as follows. By \eqref{CZ-1}, we write 
\begin{align*}
\int_{\Omega^c \cap B(0, A)^c} |T(b_1, b_2)|^{\frac12} \, dx 
\le \mathscr{I}_1 + \mathscr{I}_2 + \mathscr{I}_3,
\end{align*}
where
\begin{align*}
\mathscr{I}_1
& := \int_{\Omega^c \cap B(0, A)^c} 
\bigg[ \sum_{Q_1, Q_2 \in \D(N)} |T(b_{1, Q_1}, b_{2, Q_2})| \bigg]^{\frac12} \, dx,
\\
\mathscr{I}_2
& := \int_{\Omega^c \cap B(0, A)^c} 
\bigg[ \sum_{\substack{Q_1 \in \Lambda_1 \\ Q_2 \notin \D(N)}}  
|T(b_{1, Q_1}, b_{2, Q_2})| \bigg]^{\frac12} \, dx,
\\
\mathscr{I}_3
& := \int_{\Omega^c \cap B(0, A)^c} 
\bigg[ \sum_{\substack{Q_1 \not\in \D(N) \\ Q_2 \in \Lambda_2}}  
|T(b_{1, Q_1}, b_{2, Q_2})| \bigg]^{\frac12} \, dx. 
\end{align*}
We only bound $\mathscr{I}_2$ since the terms $\mathscr{I}_1$ and $\mathscr{I}_3$ can be estimated in a similar way.  
If we denote 
\begin{align*}
H_Q(x) 
:= F_1(\ell(Q)) F_2(\ell(Q)) 
\sum_{k=0}^{\infty} \frac{2^{-k\delta'/2}}{|2^k Q|} F_3(\rd(2^k Q, \I)) \, \mathbf{1}_{2^{k+2} Q}(x),
\end{align*}
then by \eqref{tbb-2},
\begin{align*}
|T(b_{1, Q_1}, b_{2, Q_2})(x)| 
\lesssim \frac{\ell(Q_1)^{\frac{\delta'}{2}}}{|x-c_{Q_1}|^{n+\frac{\delta'}{2}}} 
\prod_{j=1}^2 \|b_{j, Q_j}\|_{L^1} \, H_{Q_2}(x), \quad x \in \Omega^c, \, Q_2 \notin \D(N),
\end{align*}
which together with the Cauchy--Schwarz inequality yields  
\begin{align*}
\mathscr{I}_2 
\lesssim \int_{\Omega^c} \bigg[ \sum_{Q_1 \in \Lambda_1} 
\frac{\ell(Q_1)^{\frac{\delta'}{2}} \|b_{1, Q_1}\|_{L^1}}{|x-c_{Q_1}|^{n+\frac{\delta'}{2}}} \bigg]^{\frac12} 
\bigg[\sum_{Q_2 \notin \D(N)} \|b_{2, Q_2}\|_{L^1} H_{Q_2}(x) \bigg]^{\frac12} \, dx
\le (\mathscr{I}_{2, 1} \times \mathscr{I}_{2, 2} )^{\frac12},   
\end{align*}
where 
\begin{align*}
\mathscr{I}_{2, 1} 
:= \sum_{Q_1} \|b_{1, Q_1}\|_{L^1}  
\int_{(3Q_1)^c} \frac{\ell(Q_1)^{\frac{\delta'}{2}}}{|x-c_{Q_1}|^{n+\frac{\delta'}{2}}} \, dx 
\lesssim \sum_{Q_1} \|b_{1, Q_1}\|_{L^1}  
\lesssim \mu, 
\end{align*}
and 
\begin{align*}
\mathscr{I}_{2, 2} 
:= \sum_{Q_2 \notin \D(N)} \|b_{2, Q_2}\|_{L^1} \|H_{Q_2}\|_{L^1} 
\lesssim \sum_{Q_2 \notin \D(N)} \|b_{2, Q_2}\|_{L^1} F(Q_2)
\lesssim \mu \sup_{Q_2 \notin \D(N)} F(Q_2)
\le \eta^2 \mu, 
\end{align*}
provided \eqref{CZ-3} and $A \ge A_0$ sufficiently large. Hence, 
\begin{align*}
\mathscr{I}_2 
\lesssim (\mathscr{I}_{2, 1} \times \mathscr{I}_{2, 2} )^{\frac12}
\lesssim \eta \mu. 
\end{align*}

Next, let us turn to the proof of \eqref{tbb-1} and \eqref{tbb-2}. Let $Q_1, Q_2 \in \D(N)$. By symmetry, we may assume $\ell(Q_1) \le \ell(Q_2)$. Then 
\begin{align*}
Q_1 \cup Q_2 \subset B(0, (N+2) 2^N) 
\quad\text{ and}\quad 
3Q_1 \cup 3Q_2 \subset B(0, 2^{2N-1}) \subset B(0, A/2), 
\end{align*}
which implies for all $x \in \Omega^c \cap B(0, A)^c \subset (3Q_1 \cup 3Q_2)^c \cap B(0, A)^c = B(0, A)^c$, $y \in Q_1$, and $z \in Q_2$, 
\begin{align*}
\min\{|x-y|, |x+y|\} 
\ge |x| - |y|
\ge A - A/2 = A/2 
\ge 16 \ell(Q_1) 
\ge 32 |y-c_{Q_1}|,  
\end{align*}
and 
\begin{align*}
\min\{|x-z|, |x+z|\} 
\ge |x| - |z|
\ge A - A/2 = A/2 
\ge 16 \ell(Q_2) 
\ge 32 |z-c_{Q_2}|.  
\end{align*}
These further give 
\begin{align*}
|x-y| \simeq |x-c_{Q_1}|, \qquad |x-z| \simeq |x-c_{Q_2}|,
\end{align*}
and 
\begin{align*}
F(x, y, z) := F_1(|x-y|+|x-z|) F_2(|x-y|+|x-z|) F_3 (|x+y|+|x+z|) 
\lesssim F_3(A). 
\end{align*}
Using these estimates, the cancellation of $b_{1, Q_1}$, and the H\"{o}lder condition of $K$, we deduce  
\begin{align*}
|T(b_{1, Q_1}, b_{2, Q_2})(x)| 
&= \bigg|\iint_{\R^{2n}} \big(K(x, y, z) - K(x, c_{Q_1}, z) \big) 
b_{1, Q_1}(y) b_{2, Q_2}(z) \, dy \, dz\bigg|
\\  
&\lesssim \iint_{\R^{2n}} \frac{\ell(Q_1)^{\delta'} F(x, y, z)}{(|x-y| + |x-z|)^{2n+\delta'}} 
|b_{1, Q_1}(y)|  |b_{2, Q_2}(z)| \, dy \, dz
\\  
&\lesssim F_3(A) \|b_{1, Q_1}\|_{L^1} 
\frac{\ell(Q_1)^{\frac{\delta'}{2}}}{|x-c_{Q_1}|^{n+\frac{\delta'}{2}}} 
\|b_{2, Q_2}\|_{L^1} 
\frac{\ell(Q_2)^{\frac{\delta'}{2}}}{|x-c_{Q_2}|^{n+\frac{\delta'}{2}}}, 
\end{align*}
which coincides with \eqref{tbb-1}. 

To justify \eqref{tbb-2}, firstly let $Q_2 \notin \D(N)$ with $\ell(Q_1) \le \ell(Q_2)$. Let $x \in \Omega^c$, $y \in Q_1$, and $z \in Q_2$. Then, we have 
\begin{align*}
& \min\{|x-y|, |x-c_{Q_1}|\} \ge \ell(Q_1) \ge 2 |y-c_{Q_1}|, 
\\
&\min\{|x-z|, |x-c_{Q_2}|\} \ge \ell(Q_2) \ge 2 |z-c_{Q_2}|, 
\end{align*} 
which implies  
\begin{align*}
& |x-c_{Q_1}|/2 \le |x-y| \le 3|x- c_{Q_1}|/2, 
\\
&|x-c_{Q_2}|/2 \le |x-z| \le 3|x - c_{Q_2}|/2, 
\end{align*} 
and 
\begin{align*}
&1+|x-y|+|x-z| 
\le 1 + 2|x-c_{Q_1}| + 2|x-c_{Q_2}| 
\\ 
&\le 2(1 + |x-c_{Q_1}| + |x-c_{Q_2}|)
=: 2(1+|x-c_{Q_{1, 2}}|). 
\end{align*}
By the monotonicity of $F_1, F_2, F_3$, 
\begin{align*}
F(x, y, z) 
&:= F_1(|y-c_{Q_1}|) F_2(|x-y|+|x-z|) F_3 \bigg(1 + \frac{|x+y| + |x+z|}{1+|x-y|+|x-z|} \bigg) 
\\
&\le F_1(\ell(Q_1)) F_2(\ell(Q_2)) F_3 \bigg(\frac{2|z|}{1+|x-c_{Q_{1, 2}}|}\bigg). 
\end{align*}
Now it follows from the cancellation of $b_{1, Q_1}$, the H\"{o}lder condition of $K$, and Lemma \ref{lem:improve} part \eqref{Holder-imp} that 
\begin{align*}
|T (b_{1, Q_1}, b_{2, Q_2})(x)| 
&= \bigg|\iint_{\R^{2n}} \big(K(x, y, z) - K(x, c_{Q_1}, z) \big) 
b_{1, Q_1}(y) b_{2, Q_2}(z) \, dy \, dz \bigg|
\\
&\lesssim \iint_{\R^{2n}} \frac{\ell(Q_1)^{\delta'} F(x, y, z)}{(|x-y| + |x-z|)^{2n+\delta'}} 
|b_{1, Q_1}(y)| |b_{2, Q_2}(z)| \, dy \, dz,
\end{align*}
which along with the estimates above shows \eqref{tbb-2}. Analogously, in the case $\ell(Q_1) > \ell(Q_2)$, we use the cancellation of $b_{2, Q_2}$ to deduce the desired estimate. 
\end{proof}

\begin{lemma}\label{lem:Jgb}
For any $\eta>0$, there exists $\delta_0=\delta_0(\eta)>0$ such that for all $0< |h| \le \delta_0$,  
\begin{align*}
\int_{(3Q)^c} |T(g_1, b_{2, Q})| \, dx 
\le \eta \|g_1\|_{L^{\infty}} \|b_{2, Q}\|_{L^1}, 
\quad\text{if } \ell(Q) \le 2|h|;   
\end{align*}
and 
\begin{align*}
\int_{(3Q)^c} |(\tau_h T - T)(g_1, b_{2, Q})| \, dx  
\le \eta \|g_1\|_{L^{\infty}} \|b_{2, Q}\|_{L^1}, 
\quad\text{if } \ell(Q) > 2|h|.  
\end{align*}
\end{lemma}

\begin{proof}
First, let us handle the case $\ell(Q) \le 2|h|$. For all $x \in (3Q)^c$ and $z \in Q$, there holds $|x-z| \ge \ell(Q) \ge 2|z-c_Q|$. Then using the cancellation of $b_{2, Q}$, the H\"{o}lder condition of $K$, and Lemma \ref{lem:improve} part \eqref{Holder-imp}, we obtain 
\begin{align*}
\int_{(3Q)^c} |T(g_1, b_{2, Q}&)| \, dx 
= \int_{(3Q)^c} \bigg|\iint_Q \big[K(x, y, z) - K(x, y, c_Q) \big] g_1(y) b_{2, Q}(z) \, dy \, dz \bigg| \, dx
\\ 
&\lesssim \int_Q |b_{2, Q}(z)| \bigg[\iint_{|x-y|+|x-z| \ge \ell(Q)} 
\frac{\ell(Q)^{\delta'} F(x, y, z) \|g_1\|_{L^{\infty}}}{(|x-y|+|x-z|)^{2n+\delta'}} \, dx \, dy \bigg] \, dz 
\\
&\lesssim F_1(|h|) \|g_1\|_{L^{\infty}} \|b_{2, Q}\|_{L^1}, 
\end{align*}
where 
\begin{align*}
F(x, y, z) 
:= F_1(|z-c_Q|) F_2(|x-y|+|x-z|) F_3 \bigg(1 + \frac{|x+y|+|x+z|}{1+|x-y|+|x-z|} \bigg) 
\lesssim F_1(|h|). 
\end{align*}

Next, we deal with the case $\ell(Q) > 2|h|$. Then for all $x \in (3Q)^c$ and $z \in Q$, we have $|x-z| \ge \ell(Q) \ge 2|h|$. Then by the H\"{o}lder condition of $K$ and Lemma \ref{lem:improve} part \eqref{Holder-imp}, 
\begin{align*}
&\int_{(3Q)^c} |(\tau_h T - T)(g_1, b_{2, Q})| \, dx  
\\ 
&= \int_{(3Q)^c} \bigg|\iint_{\R^{2n}} 
\big(K(x-h, y, z) - K(x, y, z) \big) g_1(y) b_{2, Q}(z) \, dy \, dz \bigg| \, dx 
\\ 
&\lesssim \int_Q |b_{2, Q}(z)| \bigg[\iint_{|x-y|+|x-z| \ge 2|h|} 
\frac{|h|^{\delta'} F(x, y, z) \|g_1\|_{L^{\infty}}}{(|x-y|+|x-z|)^{2n+\delta'}} \, dx \, dy \bigg] \, dz 
\\
&\lesssim F_1(|h|) \|g_1\|_{L^{\infty}} \|b_{2, Q}\|_{L^1}, 
\end{align*}
where we have used that 
\begin{align*}
F(x, y, z) 
= F_1(|h|) F_2(|x-y|+|x-z|) F_3 \bigg(1 + \frac{|x+y|+|x+z|}{1+|x-y|+|x-z|} \bigg) 
\lesssim F_1(|h|). 
\end{align*}
Noting that $\lim_{|h| \to 0} F_1(|h|) =0$, we conclude the proof. 
\end{proof}

\begin{lemma}\label{lem:Jbb}
Let $h \in \Rn$ with $|h|>0$. For all $x \in (3Q_1 \cup 3Q_2)^c$, there holds 
\begin{align}\label{Jbb-1}
|T(b_{1, Q_1}, b_{2, Q_2})(x)| 
\lesssim F_1(|h|) \prod_{j=1}^2 \|b_{j, Q_j}\|_{L^1} 
\frac{\ell(Q_j)^{\frac{\delta'}{2}}}{|x-c_{Q_j}|^{n+\frac{\delta'}{2}}}
\end{align}
whenever $\ell(Q_1) \le 2|h|$ and $\ell(Q_2) \le 2|h|$; and 
\begin{align}\label{Jbb-2}
|(\tau_h T - T)(b_{1, Q_1}, b_{2, Q_2})(x)|
\lesssim F_1(|h|) \prod_{j=1}^2 \|b_{j, Q_j}\|_{L^1} 
\frac{\ell(Q_j)^{\frac{\delta'}{2}}}{|x-c_{Q_j}|^{n+\frac{\delta'}{2}}}
\end{align}
whenever $\max\{\ell(Q_1),  \ell(Q_2)\} > 2|h|$.  
\end{lemma}

\begin{proof}
Let $\ell(Q_1) \le 2|h|$ and $\ell(Q_2) \le 2|h|$. By symmetry, we may assume that $\ell(Q_1) \le \ell(Q_2)$. Let $x \in (3Q_1 \cup 3Q_2)^c$, $y \in Q_1$, and $z \in Q_2$. Then it is easy to see that  
\begin{align*}
|x-y| \ge \ell(Q_1) \ge 2|y-c_{Q_1}| 
\quad\text{ and }\quad  
|x-z| \ge \ell(Q_2) \ge 2|z-c_{Q_2}|, 
\end{align*}
which implies $|x-y| \simeq |x-c_{Q_1}|$ and $|x-z| \simeq |x-c_{Q_2}|$. By the cancellation of $b_{1, Q_1}$, the H\"{o}lder condition of $K$, and Lemma \ref{lem:improve} part \eqref{Holder-imp}, 
\begin{align*}
|T(b_{1, Q_1}, b_{2, Q_2})(x)|
&= \bigg|\iint_{\R^{2n}} \big(K(x, y, z) - K(x, c_{Q_1}, z) \big) 
b_{1, Q_1}(y) b_{2, Q_2}(z) \, dy \, dz \bigg| 
\\ 
&\lesssim \iint_{\R^{2n}} \frac{\ell(Q_1)^{\delta'} F(x, y, z)}{(|x-y|+|x-z|)^{2n+\delta'}} 
|b_{1, Q_1}(y)| |b_{2, Q_2}(z)| \, dy \, dz  
\\
&\lesssim F_1(|h|) \prod_{j=1}^2 \|b_{j, Q_j}\|_{L^1} 
\frac{\ell(Q_j)^{\frac{\delta'}{2}}}{|x-c_{Q_j}|^{n+\frac{\delta'}{2}}}, 
\end{align*}
where 
\begin{align*}
F(x, y, z) 
&:= F_1(|y-c_{Q_1}|) F_2(|x-y|+|x-z|) F_3 \bigg(1 + \frac{|x+y|+|x+z|}{1+|x-y|+|x-z|} \bigg) 
\lesssim F_1(|h|). 
\end{align*}
This gives \eqref{Jbb-1}.

To show \eqref{Jbb-2}, note that for all $x \in (3Q_1 \cup 3Q_2)^c$, $y \in Q_1$, and $z \in Q_2$, 
\begin{align}\label{eq:xcq}
|x - y| + |x - z| 
\simeq |x - c_{Q_1}| + |x - c_{Q_2}|
\simeq |x - h - y| + |x - h - z|.
\end{align} 
In the case $\min\{\ell(Q_1), \ell(Q_2)\} \ge |h|$, we use Lemma \ref{lem:improve} part \eqref{Holder-imp} to obtain 
\begin{align*}
&|(\tau_h T - T)(b_{1, Q_1}, b_{2, Q_2})(x)|  
\\ 
&= \bigg|\iint_{\R^{2n}} 
\big(K(x-h, y, z) - K(x, y, z) \big) b_{1, Q_1}(y) b_{2, Q_2}(z) \, dy \, dz \bigg| 
\\ 
&\lesssim \iint_{\R^{2n}} \frac{|h|^{\delta'} F(x, y, z)}{(|x-y|+|x-z|)^{2n+\delta'}} 
|b_{1, Q_1}(y)| |b_{2, Q_2}(z)| \, dy \, dz
\\
&\lesssim F_1(|h|) \prod_{j=1}^2 \|b_{j, Q_j}\|_{L^1} 
\frac{\ell(Q_j)^{\frac{\delta'}{2}}}{|x-c_{Q_j}|^{n+\frac{\delta'}{2}}},
\end{align*}
where 
\begin{align*}
F(x, y, z) 
= F_1(|h|) F_2(|x-y|+|x-z|) F_3 \bigg(1 + \frac{|x+y|+|x+z|}{1+|x-y|+|x-z|} \bigg) 
\lesssim F_1(|h|). 
\end{align*}
In the case $\min\{\ell(Q_1), \ell(Q_2)\} < |h|$, we may assume $\ell(Q_2) \le \ell(Q_1)$. Then $\ell(Q_2) < |h| < 2|h| < \ell(Q_1)$, which along with \eqref{eq:xcq} and Lemma \ref{lem:improve} part \eqref{Holder-imp} gives
\begin{align*}
&|(\tau_h T - T)(b_{1, Q_1}, b_{2, Q_2})(x)|  
\\ 
&= \bigg|\iint_{\R^{2n}} 
\big[K(x-h, y, z) - K(x-h, y, c_{Q_2}) - K(x, y, z) + K(x, y, c_{Q_2}) \big] 
\\
&\qquad \times 
b_{1, Q_1}(y) b_{2, Q_2}(z) \, dy \, dz \bigg| 
\\ 
&\lesssim \iint_{\R^{2n}} \frac{|z-c_{Q_2}|^{\delta'} 
F_1(|z-c_{Q_2}|)}{(|x-c_{Q_1}| + |x - c_{Q_2}|)^{2n+\delta'}} 
|b_{1, Q_1}(y)| |b_{2, Q_2}(z)| \, dy \, dz
\\
&\lesssim F_1(|h|) \prod_{j=1}^2 \|b_{j, Q_j}\|_{L^1} 
\frac{\ell(Q_j)^{\frac{\delta'}{2}}}{|x-c_{Q_j}|^{n+\frac{\delta'}{2}}}.
\end{align*}
The proof is complete.  
\end{proof}

\section{$L^{\infty} \times L^{\infty} \to \CMO$ compactness}\label{sec:Linfty}
In this section we aim to show ${\rm (a) \Longrightarrow (e)'}$ and ${\rm (e) \Longrightarrow (c)'}$ in Theorem \ref{thm:cpt}.

\subsection{$L^{\infty} \times L^{\infty} \to \CMO$ compactness implies weighted $L^{p_1} \times L^{p_2} \to L^p$ compactness}

\begin{proposition}\label{pro:LLi} 
Let $0<\lambda < 1/2$. Let $T$ be a bilinear operator associated with a standard bilinear Calder\'{o}n--Zygmund kernel. Then the following are equivalent: 
\begin{list}{\textup{(\theenumi)}}{\usecounter{enumi}\leftmargin=1cm \labelwidth=1cm \itemsep=0.2cm 
			\topsep=.2cm \renewcommand{\theenumi}{\roman{enumi}}}

\item\label{LLi-1} $T$ is bounded from $L^{\infty}(\Rn) \times L^{\infty}(\Rn)$ to $\BMO(\Rn)$.

\item\label{LLi-2} $T$ is bounded from $L^{\infty}(w_1^{\infty}) \times L^{\infty}(w_2^{\infty})$ to $\BMO_{\lambda}(w^{\infty})$ for all $(w_1, w_2) \in A_{(\infty, \infty)}$, where $w=w_1 w_2$.
			
\item\label{LLi-3} $T$ is bounded from $L^{p_1}(\Rn) \times L^{p_2}(\Rn)$ to $L^p(\Rn)$ for all (or for some) $p_1, p_2 \in (1, \infty]$, where $\frac1p = \frac{1}{p_1} + \frac{1}{p_2}>0$. 
			
\item\label{LLi-4} $T$ is bounded from $L^{p_1}(w_1^{p_1}) \times L^{p_2}(w_2^{p_2})$ to $L^p(w^p)$ for all (or for some) $p_1, p_2 \in (1, \infty]$ and for all $(w_1, w_2) \in A_{(p_1, p_2)}$, where $\frac1p = \frac{1}{p_1} + \frac{1}{p_2}>0$ and $w=w_1 w_2$.  

\end{list} 
\end{proposition}

\begin{proof}
We follow the scheme $\eqref{LLi-1} \Longrightarrow \eqref{LLi-3} \Longrightarrow \eqref{LLi-4} \Longrightarrow \eqref{LLi-3} \Longrightarrow \eqref{LLi-2} \Longrightarrow \eqref{LLi-1}$. Both $\eqref{LLi-4} \Longrightarrow \eqref{LLi-3}$ and $\eqref{LLi-2} \Longrightarrow \eqref{LLi-1}$ are trivial. Thus, it suffices to prove $\eqref{LLi-1} \Longrightarrow \eqref{LLi-3}$, $\eqref{LLi-3} \Longrightarrow \eqref{LLi-4}$, and $\eqref{LLi-3} \Longrightarrow \eqref{LLi-2}$. 
 
First, let us demonstrate $\eqref{LLi-3} \Longrightarrow \eqref{LLi-4}$. Assume \eqref{LLi-3} holds. By \cite[Proposition 7.4.7]{Gra2}, \eqref{LLi-3} holds for all $p_1, p_2 \in (1, \infty)$, which along with \cite[Corollary 3.9]{LOPTT} gives that \eqref{LLi-4} holds for all $p_1, p_2 \in (1, \infty)$. The latter and Theorem \ref{thm:RdF} imply \eqref{LLi-4} holds for all exponents $p_1, p_2 \in (1, \infty]$. 

Next, we show $\eqref{LLi-1} \Longrightarrow \eqref{LLi-3}$. Assume \eqref{LLi-1} holds. Let $f_2$ be an arbitrary function in $L^{\infty}(\Rn)$. Define a linear operator 
\begin{align*}
T_{f_2} (f_1) := T(f_1, f_2). 
\end{align*}
Then it is not hard to verify that the kernel $K_{f_2}$ of $T_{f_2}$ is a standard Calder\'{o}n--Zygmund kernel with a constant $\|K_{f_2}\|_{\mathrm{CZ}(\delta)}$ satisfying 
\begin{align}\label{KCZ-1}
\|K_{f_2}\|_{\mathrm{CZ}(\delta)} 
\lesssim \|K\|_{\mathrm{CZ}(\delta)} \|f_2\|_{L^{\infty}}.
\end{align} 
By definition, 
\begin{align}\label{KCZ-2}
\|T_{f_2}(f_1)\|_{\BMO} 
= \|T(f_1, f_2)\|_{\BMO} 
\lesssim \|f_2\|_{L^{\infty}} \|f_1\|_{L^{\infty}},
\end{align} 
which says that $T_{f_2}$ is bounded from $L^{\infty}(\Rn)$ to $\BMO(\Rn)$. In view of \eqref{KCZ-1} and \eqref{KCZ-2}, a careful examination of the proof of \cite[p. 50]{J83} gives 
\begin{align}\label{KCZ-3}
\|T_{f_2}(f_1)\|_{L^1} 
\lesssim \big(\|K_{f_2}\|_{\mathrm{CZ}(\delta)} 
+ \|T_{f_2}\|_{L^{\infty} \to \BMO} \big) \|f_1\|_{\mathrm{H}^1} 
\lesssim \|f_2\|_{L^{\infty}} \|f_1\|_{\mathrm{H}^1}. 
\end{align} 
Then by the interpolation theorem \cite[p. 43]{J83}, the estimates \eqref{KCZ-2} and \eqref{KCZ-3} imply 
\begin{align}\label{KCZ-4}
\|T(f_1, f_2)\|_{L^{r_1}} 
= \|T_{f_2}(f_1)\|_{L^{r_1}} 
\lesssim \|f_1\|_{L^{r_1}} \|f_2\|_{L^{\infty}}, \quad\, 1<r_1<\infty,  
\end{align} 
where the implicit constant is independent of $f_1$ and $f_2$. Similarly, given $f_1 \in L^{\infty}(\Rn)$, defining $T_{f_1} (f_2)(x) := T(f_1, f_2)(x)$ for $x \in \Rn$, we deduce 
\begin{align}\label{KCZ-5}
\|T(f_1, f_2)\|_{L^{r_2}} 
= \|T_{f_1}(f_2)\|_{L^{r_2}} 
\lesssim \|f_1\|_{L^{\infty}} \|f_2\|_{L^{r_2}}, \quad\, 1<r_2<\infty. 
\end{align} 
Interpolating between \eqref{KCZ-4} and \eqref{KCZ-5} yields  
\begin{align*}
\|T(f_1, f_2)\|_{L^p} 
\lesssim \|f_1\|_{L^{p_1}} \|f_2\|_{L^{p_2}}, 
\end{align*} 
for all $\frac1p = \frac{1}{p_1} + \frac{1}{p_2}$ with $1<p, p_1, p_2 <\infty$. This along with \cite[Theorem 3]{GT1} implies \eqref{LLi-3} holds.

Finally, to justify $\eqref{LLi-3} \Longrightarrow \eqref{LLi-2}$, in view of Theorem \ref{thm:Mbdd}, it suffices to prove 
\begin{align}\label{MR-1}
M^{\#}_{\lambda} \big(T(f_1, f_2)\big)(x) 
\lesssim \mathcal{M}(f_1, f_2)(x), \quad \, x \in \Rn.  
\end{align}
Assume \eqref{LLi-3} holds. Let $x \in \Rn$ and $Q$ be a cube containing $x$. Let $f_1 w_1, f_2 w_2 \in L^{\infty}(\Rn)$. Write $f_i^0 = f_i \mathbf{1}_{3Q}$, $i=1, 2$. Define 
\begin{align*}
\mathscr{G}_Q(x) 
& := \iint_{\R^{2n} \setminus (3Q)^2} 
[K(x, y, z) - K(c_Q, y, z)] f_1(y) f_2(z) \, dy \, dz, \quad x \in Q. 
\\
\beta_Q & := T(f_1, f_2)(x) - T(f_1^0, f_2^0)(x) - \mathscr{G}_Q(x), \quad \text{a.e. } x \in Q. 
\end{align*}
Note that $\beta_Q$ does not depend on $x$. Then
\begin{align}\label{MR-2}
\bigg(\fint_Q |T(f_1, f_2) - \beta_Q|^{\lambda} \, dx \bigg)^{\frac{1}{\lambda}}
\lesssim \bigg(\fint_Q |T(f_1^0, f_2^0)|^{\lambda} \, dx \bigg)^{\frac{1}{\lambda}} 
+ \sup_Q |\mathscr{G}_Q|.
\end{align}
By Proposition \ref{pro:WL} and the assumption \eqref{LLi-3}, $T$ is bounded from $L^1(\Rn) \times L^1(\Rn)$ to $L^{\frac12, \infty}(\Rn)$, which along with \eqref{eq:Kolm} yields 
\begin{align}\label{MR-21}
\bigg(\fint_Q |T(f_1^0, f_2^0)|^{\lambda} \, dx \bigg)^{\frac{1}{\lambda}}
\lesssim |Q|^{-2} \|T(f_1^0, f_2^0)\|_{L^{\frac12, \infty}}
\lesssim \prod_{i=1}^2 \langle |f_i| \rangle_{3Q} 
\le \mathcal{M}(f_1, f_2)(x).
\end{align}
Additionally, for every $\xi \in Q$, 
\begin{align}\label{MR-22}
|\mathscr{G}_Q(\xi)|
\lesssim \iint_{\max\{|\xi-y|, |\xi-z| \} \ge \ell(Q)} 
\frac{\ell(Q)^{\delta} \, |f_1(y)| |f_2(z)|}{(|\xi-y| + |\xi-z|)^{2n+\delta}} \, dy \, dz
\lesssim \mathcal{M}(f_1, f_2)(x). 
\end{align}
Consequently, \eqref{MR-1} follows from \eqref{MR-2}--\eqref{MR-22}. 
\end{proof}

We are ready to prove ${\rm (e) \Longrightarrow (c)'}$ in Theorem \ref{thm:cpt}. Assume that 
\begin{align}\label{LLp-1}
\text{$T$ is compact from $L^{\infty}(\Rn) \times L^{\infty}(\Rn)$ to $\CMO(\Rn)$}. 
\end{align}
Then $T$ is bounded from $L^{\infty}(\Rn) \times L^{\infty}(\Rn)$ to $\BMO(\Rn)$, which along with Proposition \ref{pro:LLi} implies 
\begin{align}\label{LLp-2}
\text{$T$ is bounded from $L^{q_1}(v_1^{q_1}) \times L^{q_2}(v_2^{q_2})$ to $L^q(v^q)$}
\end{align} 
for all $q_1, q_2 \in (1, \infty]$ and for all $(v_1, v_2) \in A_{(q_1, q_2)}$, where $\frac1q = \frac{1}{q_1} + \frac{1}{q_2}>0$ and $v=v_1 v_2$. Thus, we apply Theorem \ref{thm:EP-Linfty}, \eqref{LLp-1}, and \eqref{LLp-2} to conclude that 
\begin{align}\label{LLp-3}
\text{$T$ is compact from $L^{p_1}(w_1^{p_1}) \times L^{p_2}(w_2^{p_2})$ to $L^p(w^p)$}
\end{align}
for all $p_1, p_2 \in (1, \infty]$ and for all $(w_1, w_2) \in A_{(p_1, p_2)}$, where $\frac1p = \frac{1}{p_1} + \frac{1}{p_2}>0$ and $w=w_1 w_2$. This corresponds to ${\rm (c)'}$. 
\qed

\subsection{Hypotheses \eqref{H1}--\eqref{H3} imply weighted $L^{\infty} \times L^{\infty} \to \CMO$ compactness} 
In oder to verify ${\rm (a) \Longrightarrow (e)'}$ in Theorem \ref{thm:cpt}, let us first present a compactness criterion in $\CMO_{\lambda}(w^{\infty})$. 

\begin{theorem}\label{thm:cmob}
Let $0 < \lambda < 1$ and $w$ be a weight on $\Rn$ such that $w \in L_{\loc}^{\infty}(\Rn)$ and $w^{-\lambda} \in L_{\loc}^1(\Rn)$. Then a subset $\K \subset \CMO_{\lambda}(w^{\infty})$ is precompact if and only if 
\begin{align}\label{eq:nass}
\K \subset \BMO_{\lambda}(w^{\infty})
\quad \text{ and } \quad
\lim_{N \to \infty} \sup_{f \in \K} \sup_{Q \notin \Q(N)} 
\big(\esssup_Q w \big) \O_{\lambda}(f; Q) = 0. 
\end{align}
\end{theorem}

\begin{proof}
Since the necessity is trivial, we only focus on the sufficiency. Assume that \eqref{eq:nass} holds. Then there exists some $N_0 \ge 1$ such that 
\begin{align}\label{eq:naa}
\sup_{f \in \K} \sup_{Q \notin \Q(N_0)} 
\big(\esssup_Q w \big) \O_{\lambda}(f; Q) \le 1. 
\end{align}
Take a large cube $Q_0$ with $\ell(Q_0) > 2^{N_0}$ such that $\bigcup_{Q \in \Q(N_0)} Q \subset Q_0$. Then for all $f \in \K$ and $Q \in \Q(N_0)$, \eqref{eq:naa} gives
\begin{align*}
\big(\esssup_Q w \big) \O_{\lambda}(f; Q)
\le (|Q_0|/|Q|)^{\frac{1}{\lambda}} \big(\esssup_{Q_0} w \big) \O_{\lambda}(f; Q_0)
\le (2^{nN_0} |Q_0|)^{\frac{1}{\lambda}},
\end{align*}
which along with \eqref{eq:naa} implies  
\begin{align*}
\sup_{f \in \K} \|f\|_{\BMO_{\lambda}(w^{\infty})} \le A_0 < \infty. 
\end{align*}
Let $m_Q(f)$ be a median of $f$ on $Q$, namely, a real number satisfying
\begin{align*}
\big|\{x \in Q : f(x) > m_Q(f) \}\big| \le |Q|/2
\quad \text{ and } \quad
\big|\{x \in Q : f(x) < m_Q(f) \}\big| \le |Q|/2.
\end{align*}
Let $f \in \K$. Fix a cube $Q_0$. We may assume that $m_{Q_0}(f) = 0$ because $\O_{\lambda}(f-c; Q) = \O_{\lambda}(f; Q)$. Fix a cube $Q \supset Q_0$. We would like to prove
\begin{align}\label{eq:spqa}
\sup_{f \in \K} \int_Q |f|^{\lambda} 
= A_Q < \infty.
\end{align}
Let $a_Q = a_Q(f)$ be a constant such that $\O_{\lambda}(f; Q) = \big(\fint_Q |f - a_Q|^{\lambda} \big)^{\frac{1}{\lambda}}$. Then
\begin{align}\label{eq:spqa-0}
\bigg(\fint_Q |f - a_Q|^{\lambda} \bigg)^{\frac{1}{\lambda}}
= \O_{\lambda}(f; Q)
\le A_0 \big(\esssup_Q w \big)^{-1}.
\end{align}
In the case $a_Q \ge 0$, by the fact $m_{Q_0}(f) = 0$, 
\begin{align}\label{eq:spqa-1}
\int_Q |f - a_Q|^{\lambda}
\ge \int_{\{x \in Q_0: f(x) \le 0 \}} |a_Q|^{\lambda}
\ge |a_Q|^{\lambda} |Q_0|/2.
\end{align}
Similarly, in the case $a_Q < 0$,
\begin{align}\label{eq:spqa-2}
\int_Q |f - a_Q|^{\lambda}
\ge \int_{\{x \in Q_0: f(x) \ge 0 \}} |a_Q|^{\lambda}
\ge |a_Q|^{\lambda} |Q_0|/2.
\end{align}
Collecting \eqref{eq:spqa-0}--\eqref{eq:spqa-2}, we obtain
\begin{align*}
\fint_Q |f|^{\lambda}
\le \fint_Q |f - a_Q|^{\lambda} + |a_Q|^{\lambda}
\le \bigg[1 + \frac{2|Q|}{|Q_0|}\bigg] \fint_Q |f - a_Q|^{\lambda}
\le \bigg[1 + \frac{2|Q|}{|Q_0|}\bigg] 
A_0^{\lambda} \big(\esssup_Q w \big)^{-\lambda},
\end{align*}
which shows \eqref{eq:spqa}.

Let $\varepsilon > 0$. By our assumption \eqref{eq:nass}, there exists some $N = N(\varepsilon) \ge 1$ so that 
\begin{align*}
\sup_{f \in \K} \O_{\lambda}(f; Q) 
\le \varepsilon \big(\esssup_Q w \big)^{-1}, \quad \forall Q \text{ with } \ell(Q) < 2^{-N}.
\end{align*}
Given a cube $Q$, we split $Q = \bigcup_{j=1}^{j_0} Q_j$, where $\{Q_j\}_{j=1}^{j_0}$ is pairwise disjoint and $\ell(Q_j) < 2^{-N}$. Letting $\Phi_f := \sum_{j=1}^{j_0} a_{Q_j}(f) \mathbf{1}_{Q_j}$, we have
\begin{align}\label{eq:phif}
\int_Q |f - \Phi_f|^{\lambda}
& = \sum_{j=1}^{j_0} \int_{Q_j} |f - a_{Q_j}(f)|^{\lambda}
= \sum_{j=1}^{j_0} |Q_j| \O_{\lambda}(f; Q_j)|^{\lambda}
\\ \nonumber
& \le \varepsilon^{\lambda} \sum_{j=1}^{j_0} |Q_j| 
\big(\esssup_{Q_j} w \big)^{-\lambda}
\le \varepsilon^{\lambda} \sum_{j=1}^{j_0} \int_{Q_j} w^{-\lambda}
\le \varepsilon^{\lambda} w^{-\lambda}(Q),
\end{align}
which together with \eqref{eq:spqa} gives
\begin{align*}
\sum_{j=1}^{j_0} |Q_j| |a_{Q_j}(f)|^{\lambda}
= \int_Q |\Phi_f|^{\lambda}
\le \int_Q |f|^{\lambda} + \int_Q |f - \Phi_f|^{\lambda}
\le A_Q + \varepsilon^{\lambda} w^{-\lambda}(Q).
\end{align*}
This indicates $\sup_{f \in \K} |a_{Q_j}(f)| \le B_0 = B_0(Q) < \infty$. Let $t = t(Q) := (\varepsilon |Q|^{-1})^{\frac{1}{\lambda}}$ and $k_0 := [B_0/t]$. Then there exists $b_j \in \{kt: k=-k_0, -k_0+1, \ldots, k_0\}$ such that $|a_{Q_j}(f) - b_j| \le t$, for each $j=1, \ldots, j_0$. Define $\Phi := \sum_{j=1}^{j_0} b_j \mathbf{1}_{Q_j}$. The estimate \eqref{eq:phif} implies
\begin{align*}
\int_Q |f - \Phi|^{\lambda}
\le \int_Q |f - \Phi_f|^{\lambda} + \int_Q |\Phi_f - \Phi|^{\lambda}
\le \varepsilon^{\lambda} w^{-\lambda}(Q) + \varepsilon. 
\end{align*}
Since the number of $\Phi$ is finite, 
\begin{align}\label{eq:ktbd}
\text{$\K$ is totally bounded in $L^{\lambda}(Q)$. }
\end{align}

Let $\{f_j\} \subset \K$ be an arbitrary sequence. Take a sequence of cubes $\{Q_k\}_{k=1}^{\infty}$ such that $Q_0 \subset Q_1 \subset Q_2 \subset \cdots$ and $\bigcup_{k=1}^{\infty} Q_k = \Rn$. In view of \eqref{eq:ktbd}, we first choose a subsequence (which we relabel) $\{f_j\}$ such that it converges in $L^{\lambda}(Q_1)$. Then pick a subsequence (which we relabel) $\{f_j\}$ such that it converges in $L^{\lambda}(Q_2)$. Repeating this process, we eventually obtain a diagonal subsequence (which we relabel) $\{f_j\}$ such that
\begin{align}\label{eq:fjjc}
\{f_j\} \to f \, \text{ in } L_{\loc}^{\lambda}(\Rn).
\end{align} 
Given $\varepsilon > 0$, by \eqref{eq:nass}, there exists some $N_0 = N_0(\varepsilon) \ge 2$ so that for all $N \ge N_0$,
\begin{align*}
\sup_{f \in \K} \sup_{Q \notin \Q(N)} 
\big(\esssup_Q w \big) \O_{\lambda}(f; Q) \le \varepsilon,
\end{align*}
which leads to 
\begin{align}\label{eq:cau-1}
\sup_{Q \notin \Q(N)} 
\big(\esssup_Q w \big) \O_{\lambda}(f_k - f_j; Q) 
\lesssim \varepsilon.
\end{align}
On the other hand, for all $Q \in \Q(N_0)$, there holds $\ell(Q) \ge 2^{-N_0}$, $Q \subset B(0, 4^{N_0})$, and 
\begin{align}\label{eq:cau-2}
\sup_{Q \in \Q(N_0)} 
\big(\esssup_Q w \big) \O_{\lambda}(f_k - f_j; Q) 
\le \big(\esssup_{B(0, 4^{N_0})} w \big) 
\bigg(2^{N_0 n} \int_{B(0, 4^{N_0})} |f_k - f_j|^{\lambda} \bigg)^{\frac{1}{\lambda}}
\to 0.
\end{align}
Thus, by \eqref{eq:cau-1}--\eqref{eq:cau-2}, there exists some $N_1 \ge N_0$ such that for all $k, j \ge N_1$,
\begin{align}\label{eq:cau-3}
\|f_k - f_j\|_{\BMO_{\lambda}(w^{\infty})} 
= \sup_{Q \in \Q} \big(\esssup_Q w \big) \O_{\lambda}(f_k - f_j; Q) 
\lesssim \varepsilon.
\end{align}
Using \eqref{eq:fjjc} and the fact 
\begin{align*}
|\O_{\lambda}(f_k - f_j; Q)^{\lambda} - \O_{\lambda}(f_k - f; Q)^{\lambda}| 
\le \O_{\lambda}(f_j - f; Q)^{\lambda}
\le \fint_Q |f_j - f|^{\lambda},
\end{align*}
we achieve 
\begin{align*}
\lim_{j \to \infty} \big(\esssup_Q w \big) \O_{\lambda}(f_k - f_j; Q) 
= \big(\esssup_Q w \big) \O_{\lambda}(f_k - f; Q)|, \quad \forall Q \in \Q.
\end{align*}
which along with \eqref{eq:cau-3} yields  
\begin{align*}
\big(\esssup_Q w \big) \O_{\lambda}(f_k - f; Q) 
\lesssim \varepsilon, \quad \forall Q \in \Q, \, k \ge N_1.
\end{align*}
Hence, 
\begin{align}\label{eq:cau-4}
\lim_{k \to \infty} \|f_k - f\|_{\BMO_{\lambda}(w^{\infty})} = 0.
\end{align}
Observing 
\begin{align*}
& \|(M_{\lambda, N}^{\#} f) w\|_{L^{\infty}}
= \sup_{Q \notin \Q(N)} \big(\esssup_Q w \big) \O_{\lambda}(f; Q) 
\\
& \lesssim \|f_k - f\|_{\BMO_{\lambda}(w^{\infty})}
+ \sup_{Q \notin \Q(N)} \big(\esssup_Q w \big) \O_{\lambda}(f_k; Q),
\end{align*}
we use first \eqref{eq:nass} and then \eqref{eq:cau-4} to deduce $f \in \CMO_{\lambda}(w^{\infty})$. This shows the precompactness of $\K$ in $\CMO_{\lambda}(w^{\infty})$. 
\end{proof}

Now let us give the proof of ${\rm (a) \Longrightarrow (e)'}$. Let $0 < \lambda < 1/2$ and $(w_1, w_2) \in A_{(\infty, \infty)}$. Then $w := w_1 w_2 \in L_{\loc}^{\infty}(\Rn)$. It follows from Lemma \ref{lem:weight} that $w^{-\frac12} \in A_1$, and $w_1^{-1}, w_2^{-1} \in A_2$, which together with \eqref{eq:RH} gives 
\begin{align}\label{eq:wir}
\langle w_i^{-r} \rangle_Q^{1/r} \lesssim \langle w_i^{-1} \rangle_Q
\quad \text{and} \quad
w^{-\lambda} \in L_{\loc}^1(\Rn), \quad Q \in \Q,
\end{align}
where $r \in (1, 2)$.

Let $Q \in \Q$ and $\|f_i w_i\|_{L^{\infty}} \le 1$, $i=1, 2$. Write $f_i^0 = f_i \mathbf{1}_{3Q}$, $i=1, 2$. Then by \eqref{eq:wir},
\begin{align}\label{eq:wer}
|f_i| \le w_i^{-1} \, \text{ a.e.}, \quad f_i \in L_{\loc}^r(\Rn), 
\quad \text{and} \quad 
\|f_i^0\|_{L^r} \lesssim |Q|^{\frac1r} \langle w_i^{-1} \rangle_{3Q}.
\end{align}
Denote
\begin{align*}
\K & := \{T(f_1, f_2): \|f_i w_i\|_{L^{\infty}} \le 1, i=1, 2\} ,
\\
\G & := \{T(g_1, g_2): \|g_i\|_{L^r} \le 1, i=1, 2\}, 
\\ 
\eta_Q & := \sup_{g \in \G} \eta_Q(g)
= \sup_{g \in \G} |Q|^{\frac{2}{r}} \langle |g|^{\lambda} \rangle_Q^{\frac{1}{\lambda}},
\\
\beta_Q & := T(f_1, f_2)(x) - T(f_1^0, f_2^0)(x) - \mathscr{H}_Q(x), \quad \text{a.e. } x \in Q, 
\end{align*}
where
\begin{align*}
\mathscr{H}_Q(x) 
:= \iint_{\R^{2n} \setminus (3Q)^2} 
[K(x, y, z) - K(c_Q, y, z)] f_1(y) f_2(z) \, dy \, dz.
\end{align*}
Note that $\beta_Q$ is independent of $x$. For $f = T(f_1, f_2)$, we split 
\begin{align}\label{eq:bhq}
\O_{\lambda}(f; Q)
\le \bigg(\fint_Q |f - \beta_Q|^{\lambda} \bigg)^{\frac{1}{\lambda}}
\lesssim \bigg(\fint_Q |T(f_1^0, f_2^0)|^{\lambda} \bigg)^{\frac{1}{\lambda}}
+ \esssup_Q |\mathscr{H}_Q|.
\end{align}
By the definition of $\G$ and \eqref{eq:wer}, we have
\begin{align}\label{eq:bhq-1}
\big(\esssup_Q w\big) 
\bigg(\fint_Q |T(f_1^0, f_2^0)|^{\lambda} \bigg)^{\frac{1}{\lambda}}
\le \big(\esssup_Q w\big) \eta_Q |Q|^{-\frac{2}{r}} \|f_1^0\|_{L^r} \|f_2^0\|_{L^r}
\lesssim [\vec{w}]_{A_{(\infty, \infty)}}\eta_Q.
\end{align}
To bund $\mathscr{H}_Q$, let $x \in Q$ and denote 
\begin{align*}
R_k(x) := \{(y, z) \in \R^{2n}: 2^k \ell(Q) \le |x-y| + |x-z| < 2^{k+1} \ell(Q)\}.
\end{align*}
Then for all $x \in Q$ and $y \in R_k(x)$, 
\begin{align*}
&1 + \frac{|x+y| + |x+z|}{1+|x-y|+|x-z|}
\ge \frac{4|x|}{1+|x-y|+|x-z|} 
\\
&\ge \frac{4\d(2^k Q, \I)}{1+2^{k+1} \ell(Q)}
\ge \frac{\d(2^k Q, \I)}{\max\{2^k \ell(Q), 1\}}
= \rd(2^k Q, \I), 
\end{align*}
and hence, 
\begin{align*}
F(x, y, z) 
&:= F_1(|x-c_Q|) F_2(|x-y|+|x-z|) F_3 \bigg(1+\frac{|x+y|+|x+z|}{1+|x-y|+|x-z|} \bigg)
\\
&\le F_1(\ell(Q)) F_2(\ell(Q)) F_3(\rd(2^k Q, \I)). 
\end{align*}
This, along with the H\"{o}lder condition of $K$ and \eqref{eq:wer}, implies  
\begin{align}\label{eq:bhq-2}
|\mathscr{H}_Q(x)|
&\lesssim \sum_{k \ge 0} \iint_{R_k(x)} 
\frac{\ell(Q)^{\delta'} F(x, y, z)}{(|x-y| + |x-z|)^{2n+\delta'}} |f_1(y)| |f_2(z)| \, dy \, dz
\\ \nonumber
&\lesssim F_1(\ell(Q)) F_2(\ell(Q)) 
\sum_{k \ge 0} 2^{-k \delta'} F_3(\rd(2^k Q, \I)) 
\prod_{i=1}^2 \langle |f_i| \rangle_{2^{k+3} Q}
\\ \nonumber
&\lesssim [\vec{w}]_{A_{(\infty, \infty)}} \big( \esssup_Q w \big)^{-1} 
F_1(\ell(Q)) F_2(\ell(Q)) 
\sum_{k \ge 0} 2^{-k \delta'} F_3(\rd(2^k Q, \I)) 
\\ \nonumber
& =: [\vec{w}]_{A_{(\infty, \infty)}} \big( \esssup_Q w \big)^{-1} F(Q).
\end{align}
Thus, it follows from \eqref{eq:bhq}--\eqref{eq:bhq-2} that
\begin{align*}
\sup_{f \in \K} \|f\|_{\BMO_{\lambda}(w^{\infty})}
= \sup_{f \in \K} \sup_{Q \in \Q} \big( \esssup_Q w \big) \O_{\lambda}(f; Q)
\lesssim \sup_{Q \in \Q} [\eta_Q + F(Q)],
\end{align*}
and 
\begin{align*}
\sup_{f \in \K} \|(M_{\lambda, N}^{\#} f) w\|_{L^{\infty}}
= \sup_{f \in \K} \sup_{Q \notin \Q(N)} \big( \esssup_Q w \big) \O_{\lambda}(f; Q)
\lesssim \sup_{Q \notin \Q(N)} [\eta_Q + F(Q)],
\end{align*}
Note that $\sup_{Q \in \Q} F(Q) \lesssim 1$ and $\lim_{N \to \infty} \sup_{Q \notin \Q(N)} F(Q) = 0$. In view Theorem \ref{thm:cmob}, to conclude the compactness of $T: L^{\infty}(w_1^{\infty}) \times L^{\infty}(w_2^{\infty}) \to \CMO_{\lambda}(w^{\infty})$, it suffices to justify 
\begin{align}\label{eq:etq}
\sup_{Q \in \Q} \eta_Q \lesssim 1
\quad \text{ and } \quad 
\lim_{N \to \infty} \sup_{Q \notin \Q(N)} \eta_Q = 0.
\end{align}

Recall that we have proved that ${\rm (a)} \Longrightarrow {\rm (c)}$, which gives 
\begin{align}\label{eq:rrcp}
T: L^r(\Rn) \times L^r(\Rn) \to L^{r/2}(\Rn) \quad \text{is compact}.
\end{align}
By the fact $0 < \lambda < r/2$, Jensen's inequality, and \eqref{eq:rrcp}, we have
\begin{align}\label{eq:gge-0}
\sup_{Q \in \Q} \eta_Q 
\le \sup_{g \in \G} \|g\|_{L^{\frac{r}{2}}}
= \sup_{\substack{\|g_1\|_{L^r} \le 1 \\ \|g_2\|_{L^r} \le 1}}
\|T(g_1, g_2)\|_{L^{\frac{r}{2}}}
\lesssim 1.
\end{align}
Let $\varepsilon > 0$. In view of \eqref{eq:rrcp}, $\G$ is precompact in $L^{\frac{r}{2}}(\Rn)$. Then one can find a finite number of functions $\{g_k\}_{k=1}^{k_0} \subset \G$ so that 
\begin{align}\label{eq:gge-1}
\sup_{g \in \G} \min_{1 \le k \le k_0} \|g - g_k\|_{L^{\frac{r}{2}}} 
\le \varepsilon.
\end{align}
Since $\mathscr{C}_c^{\infty}(\Rn)$ is dense in $L^{\frac{r}{2}}(\Rn)$, for each $k \in \{1, \ldots, k_0\}$, there exists some $\phi_k \in \mathscr{C}_c^{\infty}(\Rn)$ such that $\|g_k - \phi_k\|_{L^{\frac{r}{2}}} \le \varepsilon$. This together with \eqref{eq:gge-1} implies
\begin{align}\label{eq:gge-2}
\sup_{g \in \G} \min_{1 \le k \le k_0} \|g - \phi_k\|_{L^{\frac{r}{2}}} 
\lesssim \varepsilon.
\end{align}
Then using $\eta_Q(g) \lesssim \eta_Q(\phi_k) + \eta_Q(g-\phi_k)$, we have for all $Q \in \Q$,
\begin{align}\label{eq:gge-3}
\eta_Q 
\lesssim \max_{1 \le k \le k_0} \eta_Q(\phi_k) + \varepsilon.
\end{align}
Note that $\bigcup_{1 \le k \le k_0} \supp (\phi_k) \subset B(0, 2^{N_0})$ for some $N_0 \ge 1$. Choose $N_1 = N_1(\varepsilon) \ge 1$ so that 
\begin{align*}
2^{-nN_1 \frac{2}{r}} \max_{1 \le k \le k_0} \|\phi_k\|_{L^{\infty}} 
+ 2^{-nN_1 (\frac{1}{\lambda} - \frac{2}{r})} \|\phi_k\|_{L^{\lambda}}
\le \varepsilon
\end{align*}
Then for all $N > \max\{N_0 + 1, N_1\}$, if $\ell(Q) < 2^{-N}$, 
\begin{align}\label{eq:gge-4}
\eta_Q(\phi_k) 
\lesssim 2^{-nN \frac{2}{r}} \|\phi_k\|_{L^{\infty}}
\lesssim \varepsilon;
\end{align}
if $\ell(Q) > 2^N$, 
\begin{align}\label{eq:gge-5}
\eta_Q(\phi_k) 
\lesssim 2^{-nN (\frac{1}{\lambda} - \frac{2}{r})} \|\phi_k\|_{L^{\lambda}}
\lesssim \varepsilon;
\end{align}
if $2^{-N} \le \ell(Q) \le 2^N$ and $\rd(Q, 2^N \I) > N$, there holds 
\begin{align}\label{eq:gge-6}
Q \cap \supp \phi_k = \varnothing
\quad \text{ and } \quad
\eta_Q(\phi_k) =0.
\end{align}
Therefore, \eqref{eq:etq} follows from \eqref{eq:gge-0} and \eqref{eq:gge-3}--\eqref{eq:gge-6}.
\qed

\section{Applications}\label{sec:app}
This section is devoted to presenting some examples of bilinear operators, which satisfy the hypotheses \eqref{H1}, \eqref{H2}, and \eqref{H3}. Thus, they satisfy the unweighted and weighted compactness as stated in Theorem \ref{thm:cpt}.

\subsection{Bilinear continuous paraproducts} 
Let us prove Theorem \ref{thm:Pibc}. Recall that $b \in \CMO(\Rn)$. By \cite[Lemma 5.1]{Hart14}, we see that 
\begin{align}\label{PBCZ-1}
\text{$\pi_b$ is a bilinear Calder\'{o}n--Zygmund operator} 
\end{align}
satisfying 
\begin{align*}
\pi_b (1, 1) = b 
\quad\text{ and }\quad 
\pi_b^{*1} (1, 1) = \pi_b^{*2} (1, 1) =0, 
\end{align*}
which yields \eqref{H3}. Additionally, it was shown in \cite[Proposition 3.1]{BLOT1} that 
\begin{align}\label{PBCZ-2}
\text{$\pi_b$ is compact from $L^{p_1}(\Rn) \times L^{p_2}(\Rn)$ to $L^p(\Rn)$}
\end{align}
for all $\frac1p = \frac{1}{p_1} + \frac{1}{p_2}$ with $p_1, p_2 \in (1, \infty)$. Then \eqref{H1} and \eqref{H2} follow from Theorems \ref{thm:CCZK} and \ref{thm:WCP}, \eqref{PBCZ-1}, and \eqref{PBCZ-2}. 
\qed

\subsection{Bilinear dyadic paraproducts} 
In \cite[Chapter 3]{M92}, Meyer constructed the smooth 1-dimensional wavelets with compact support. Let $\psi^0 := \phi$ be the father wavelet, and $\psi^1 := \psi$ be the mother wavelet. Then the $n$-dimensional wavelet is given by $\psi^{\eta}(x) := \prod_{i=1}^n \psi^{\eta_i}(x_i)$, where $\eta \in \{0, 1\}^n \setminus \{0\}$.

\begin{definition}\label{def:wavelet}
We say that $\big\{\psi_I^{\eta}: I \in \D, \eta \in \{0, 1\}^n \setminus \{0\} \big\}$ is a \emph{system of wavelets} if $\psi_I^{\eta}(x) : = 2^{nk/2} \psi^{\eta}(2^k x - l)$, where $\psi^{\eta}$ is an $n$-dimensional wavelet and $I = 2^{-k} ([0, 1)^n + l)$, and this collection has the following fundamental properties of a wavelet basis:

\begin{enumerate}
\item It is an orthonormal basis of $L^2(\Rn)$; 

\item Localization: $\supp(\psi_I^{\eta}) \subset C_0 I$; 

\item Regularity: $|\partial^{\alpha} \psi_I^{\eta}| \lesssim \ell(I)^{-|\alpha|-n/2}$ for any multi-index $|\alpha| \le N_0$; 

\item Cancellation: $\int_{\Rn} x^{\alpha} \psi_I^{\eta}(x) \, dx = 0$ for any multi-index $|\alpha| \le N_1$.
\end{enumerate} 
\end{definition}

Since the different $\eta$'s play a minor role, $ \psi_I^{\eta}$ will be often abbreviated as $ \psi_I$. We are going to show Theorem \ref{thm:Pib}. Similarly to \eqref{def:PN}, we define 
\begin{align*}
\mathcal{P}_N f  := \sum_{I \in \D(N)} \langle f, \psi_I \rangle \, \psi_I 
\quad\text{ and }\quad 
\mathcal{P}_N^{\perp} f := f - \mathcal{P}_N f.  
\end{align*} 
Let us first prove 
\begin{align}\label{eq:penb-0}
\mathcal{E}_N(b)
:= \sup_{Q \in \D} \bigg(\frac{1}{|Q|} \sum_{\substack{I \subset Q: I \notin \D(N)}} 
|\langle b, \psi_I \rangle|^2 \bigg)^{\frac12} 
\longrightarrow 0, \quad \text{as } N \to \infty.
\end{align} 
For every $Q \in \D$, we use Bessel's and John--Nirenberg inequalities to deduce
\begin{align}\label{eq:pbbe-0}
\frac{1}{|Q|} \sum_{I \subset Q} |\langle b, \psi_I \rangle|^2
& = \frac{1}{|Q|} \sum_{I \subset Q} |\langle (b - \langle b \rangle_{C_0 Q}) \mathbf{1}_{C_0 Q}, \psi_I \rangle|^2
\\ \nonumber
& \le \frac{1}{|Q|} \int_{C_0 Q} |b - \langle b \rangle_{C_0 Q}|^2
\lesssim \|b\|_{\BMO}^2.
\end{align} 
Let $b \in \CMO(\Rn)$. Then for every $\varepsilon > 0$, there exists some $f \in \mathscr{C}_c^{\infty}(\Rn)$ such that $\|b-f\|_{\BMO} \le \varepsilon$. This and \eqref{eq:pbbe-0} yield 
\begin{align*}
\mathcal{E}_N(b)
\le \mathcal{E}_N(b-f) + \mathcal{E}_N(f)
\lesssim \varepsilon + \mathcal{E}_N(f).
\end{align*}
Thus, \eqref{eq:penb-0} is reduced to the following
\begin{align}\label{eq:pbbe-1}
\lim_{N \to \infty} \mathcal{E}_N(f) = 0, \quad \forall f \in \mathscr{C}_c^{\infty}(\Rn).
\end{align}
Let $f \in \mathscr{C}_c^{\infty}(\Rn)$ with $\supp f \subset B(0, 2^{N_0})$, where $N_0 \ge 1$. Then for all $N \ge 2N_0 + C_0$ and $I \in \D$ with $2^{-N} \le \ell(I) \le 2^N$ and $\rd(I, 2^N \I) > N$, we see that $C_0 I \cap \supp f = \varnothing$. Note also that 
\begin{align*}
|\langle f, \psi_I \rangle| \lesssim \ell(I) |I|^{\frac12} \|\nabla f\|_{L^{\infty}}
\quad \text{ and } \quad
|\langle f, \psi_I \rangle| \lesssim |I|^{-\frac12} \|f\|_{L^1},
\end{align*}
which gives for all $N \ge 2N_0 + C_0$,
\begin{align*}
&  \frac{1}{|Q|} \sum_{\substack{I \notin \D(N) \\ I \subset Q}} 
|\langle f, \psi_I \rangle|^2
= \frac{1}{|Q|} \sum_{\substack{I \subset Q: \ell(I) < 2^{-N}}} 
|\langle f, \psi_I \rangle|^2 
+ \frac{1}{|Q|} \sum_{\substack{I \subset Q: \ell(I) > 2^N}} 
|\langle f, \psi_I \rangle|^2 
\\
& \lesssim \frac{\|\nabla f\|_{L^{\infty}}^2}{|Q|} 
\sum_{k > N} 2^{-2k} \sum_{I \subset Q: \ell(I) = 2^{-k}} |I|
+ \|f\|_{L^1}^2 \sum_{k > N} \frac{1}{|Q|} \sum_{I \subset Q: \ell(I) = 2^k} |I|^{-1}
\\
& \lesssim \sum_{k > N} 2^{-2k}
+ \sum_{k > N} 2^{-2nk} \sum_{I \subset Q: \ell(I) = 2^k} \frac{|I|}{|Q|}
\lesssim 2^{-2N}.
\end{align*}
This verifies \eqref{eq:pbbe-1}.

Using the orthogonality of $\{\psi_I\}$, the Cauchy--Schwarz inequality, and Lemma \ref{lem:Car}, we obtain
\begin{align*}
\|\mathcal{P}_N^{\perp} \Pi_b(f_1, f_2)\|_{L^2}^2
& = \sum_{I \notin \D(N)} |\langle b, \psi_I \rangle|^2
|\langle f_1, \phi_I \rangle|^2 |\langle f_2, \phi_I \rangle|^2
\\
& \le \prod_{j=1}^2 \bigg(\sum_{I \notin \D(N)} |\langle b, \psi_I \rangle|^2
|\langle f_j, \phi_I \rangle|^4 \bigg)^{\frac12}
\\
& \lesssim \mathcal{E}_N(b)^2 \|f_1\|_{L^4}^2 \|f_2\|_{L^4}^2,
\end{align*} 
which along with \eqref{eq:penb-0} gives
\begin{align}\label{eq:penb-1}
\|\mathcal{P}_N^{\perp} \Pi_b\|_{L^4 \times L^4 \to L^2}
\lesssim \mathcal{E}_N(b) \to 0.
\end{align} 
Analogously,
\begin{align}\label{eq:penb-2}
\|\Pi_b\|_{L^4 \times L^4 \to L^2}
\lesssim \|b\|_{\BMO}.
\end{align} 
Since $\mathcal{P}_N \Pi_b$ is a finite-dimensional operator, the estimates \eqref{eq:penb-1}--\eqref{eq:penb-2} imply 
\begin{align}\label{eq:penb-3}
\Pi_b: L^4(\Rn) \times L^4(\Rn) \to L^2(\Rn) \quad \text{ is compact}
\end{align} 

The kernel of $\Pi_b$ is given by 
\begin{align*}
K(x, y_1, y_2) 
:= \sum_{I \in \D} \langle b, \psi_I \rangle \psi_I(x) \phi_I(y_1) \phi_I(y_2).
\end{align*}
Let $A := \max\{|x - y_1|, |x - y_2|\} > 0$. We may assume that $K \not \equiv 0$. If $K(x, y_1, y_2) \neq 0$, then there exists a dyadic cube $I \in \D$ such that $x \in C_0 I$ and $y_1, y_2 \in I$. Thus, $2 A \le (C_0 + 1) \ell(I)$ and 
\begin{align*}
|K(x, y_1, y_2)| 
&\lesssim \|b\|_{\BMO} \sum_{k \in \Z} 2^{-2nk} 
\sum_{I \in \D: \ell(I)=2^k} \mathbf{1}_{C_0 I}(x) \mathbf{1}_I(y_1) \mathbf{1}_I(y_2)
\\
&\lesssim \|b\|_{\BMO} \sum_{k \in \Z: 2^k \gtrsim A} 2^{-2nk}
\lesssim \frac{\|b\|_{\BMO}}{(|x-y_1| + |x-y_2|)^{2n}}.
\end{align*}
Likewise,
\begin{align*}
& |\nabla_x K(x, y_1, y_2)| + |\nabla_{y_1} K(x, y_1, y_2)| + |\nabla_{y_2} K(x, y_1, y_2)| 
\\
& \lesssim \sum_{k \in \Z: 2^k \gtrsim A} 2^{-(2n+1) k}
\lesssim \frac{\|b\|_{\BMO}}{(|x-y_1| + |x-y_2|)^{2n+1}},  
\end{align*}
which immediately yields that $K$ is a standard bilinear Calder\'{o}n--Zymund kernel with $\delta=1$. Together with \eqref{eq:penb-2}--\eqref{eq:penb-3} and Theorems \ref{thm:CCZK}--\ref{thm:T1CMO}, this implies \eqref{H1}--\eqref{H3} hold. 
\qed

\subsection{Bilinear pseudo-differential operators}
In the Fourier analytic computations below, let $|\cdot|$ denote the Euclidean norm; equivalence of norms allows us to retain the stated kernel estimates. For any $\sigma \in \mathcal{S}_{\rho,\delta}^m$, we define the seminorm 
\begin{align*}
P_{\alpha, \beta, \gamma}^{m, \rho, \delta}(\sigma) 
:= \sup_{x, \, \xi, \, \eta \in \Rn}  
\frac{\big|\partial_{x}^{\alpha} \partial_{\xi}^{\beta} \partial_{\eta}^{\gamma} \sigma(x, \xi, \eta) \big|}{(1+ |\xi| + |\eta|)^{m + \delta|\alpha| - \rho (|\beta| + |\gamma|)}} 
\end{align*}
for all multi-indices $\alpha$, $\beta$ and $\gamma$.

Combining \cite[Theorem 2.1]{BMNT} with \cite[Corollary 1]{GT1}, one has that for any $\sigma \in \mathcal{S}_{1,\delta}^0$ with $\delta \in [0, 1)$, 
\begin{align}\label{TsCZO}
\text{$T_{\sigma}$ is a bilinear Calder\'{o}n--Zygmund operator.}
\end{align} 
In what follows, let $K_{\sigma}$ denote the kernel of $T_{\sigma}$.

\begin{proof}[\bf Proof of Theorem \ref{thm:Tsig}]
Note that $\mathcal{K}_{1, \delta}^0 \subset \mathcal{S}_{1, \delta}^0$. It follows from \eqref{TsCZO} that  
\begin{align}\label{TCZO-1}
\text{$T_{\sigma}$ is a bilinear Calder\'{o}n--Zygmund operator.}
\end{align} 
Considering Theorems \ref{thm:CCZK}, \ref{thm:WCP}, and \ref{thm:T1CMO}, we are reduced to proving that   
\begin{align}\label{TT-1}
\text{$T_{\sigma}$ is compact from $L^4(\Rn) \times L^4(\Rn)$ to $L^2(\Rn)$}. 
\end{align}
Our proof is a bilinear extension of the proof of \cite[Theorem 3.2]{CST}, which itself builds on the idea in \cite[Theorem E]{Cor}.  

Let $\psi$ be a smooth cut-off function such that $\mathbf{1}_{B(0, 1)} \le \psi \le \mathbf{1}_{B(0, 2)}$. For each $j \in \N$, let 
\begin{align}\label{def:psi}
\phi_j(x, \xi, \eta) = \psi_j(x) \psi_j(\xi)  \psi_j(\eta)  
:= \psi(2^{-j} x) \psi(2^{-j} \xi) \psi(2^{-j} \eta) .
\end{align} 
Then we see that 
\begin{align}\label{def:sig}
\sigma_j := \sigma \phi_j \in \mathcal{S}_{1, \delta}^0 
\quad\text{ uniformly in $j \in \N$},  
\end{align}
which along with \eqref{TsCZO} gives 
\begin{align}\label{TCZO-2}
\text{$T_{\sigma_j}$ is a bilinear Calder\'{o}n--Zygmund operator.}
\end{align}
By Lemma \ref{lem:limit}, to obtain \eqref{TT-1}, it suffices to show  
\begin{align}\label{TT-2}
\lim_{j \to \infty} \big\|T_{\sigma_j} - T_{\sigma}\big\|_{L^4 \times L^4 \to L^2} = 0
\end{align}
and 
\begin{align}\label{TT-3}
\text{$T_{\sigma_j}$ is compact from $L^4(\Rn) \times L^4(\Rn)$ to $L^2(\Rn)$, 
\quad $\forall j \in \N$.} 
\end{align}

Let us bound $\|T_{\sigma_j} - T_{\sigma}\|_{L^4 \times L^4 \to L^2}$. By Lemma \ref{lem:KP}, 
\begin{align}\label{KK21}
\|K_{\sigma_j} - K_{\sigma}\|_{\mathrm{CZ}(1)} 
= \|K_{\sigma_j - \sigma}\|_{\mathrm{CZ}(1)}
\lesssim \sum_{\substack{|\alpha| \le 1 \\ |\beta| + |\gamma| \le 4n+4}} 
P_{\alpha, \beta, \gamma}^{0, 1, \delta}(\sigma_j - \sigma).
\end{align}
From \cite[Theorem 2.1]{BMNT} and its proof, it follows that both symbols of $T_{\sigma}^{*1}$ and $T_{\sigma}^{*2}$ belong to $\mathcal{S}_{1, \delta}^0$ and 
\begin{equation}\label{KK22}
\begin{aligned}
|(\sigma_j - \sigma)(x, \xi, \eta)| 
&\lesssim P_{0, 0, 0}^{0, 1, \delta}(\sigma_j - \sigma), 
\\ 
|(\sigma_j^{* k} - \sigma^{* k})(x, \xi, \eta)| 
&\lesssim \sum_{|\alpha|, |\beta|, |\gamma| \le 2N_0}  
P_{\alpha, \beta, \gamma}^{0, 1, \delta}(\sigma_j - \sigma), 
\end{aligned}
\end{equation}
where $N_0>n/2$ sufficiently large. Pick $N := \max\{2n+2, N_0\}$. Then together with \eqref{KK21} and \eqref{KK22}, \cite[Corollary 1]{GT1} implies  
\begin{align}\label{KK23}
\|T_{\sigma_j} - T_{\sigma}\|_{L^4 \times L^4 \to L^2} 
\lesssim \sum_{|\alpha|, |\beta|, |\gamma| \le 2N} 
P_{\alpha, \beta, \gamma}^{0, 1, \delta}(\sigma_j - \sigma).  
\end{align}
Thus, \eqref{TT-2} is a consequence of \eqref{KK23} and Lemma \ref{lem:Pab}.

It remains to demonstrate \eqref{TT-3}. Throughout the proof of \eqref{TT-3}, $j$ is fixed, and the implicit constants may depend on $j$. By \eqref{TCZO-2} and \cite[Theorem 3]{GT1}, we have  
\begin{align}\label{TFK-1}
\sup_{\substack{\|f_1\|_{L^4} \le 1 \\ \|f_2\|_{L^4} \le 1}} 
\big\|T_{\sigma_j}(f_1, f_2) \big\|_{L^2} 
\lesssim 1.
\end{align}
Note that $\supp \sigma_j \subset \{(x, \xi, \eta) \in \R^{3n}: |x| \le 2^{j+1}, |\xi| \le 2^{j+1}, |\eta| \le 2^{j+1}\}$. Then, 
\begin{align*}
K_{\sigma_j}(x, y, z) = 0
\quad\text{ and }\quad 
T_{\sigma_j}(f_1, f_2)(x) = 0, 
\quad\text{ for all } |x|>2^{j+1}, 
\end{align*}
which yields   
\begin{align}\label{TFK-2}
\lim_{A \to \infty} \sup_{\substack{\|f_1\|_{L^4} \le 1 \\ \|f_2\|_{L^4} \le 1}} 
\big\|T_{\sigma_j}(f_1, f_2) \mathbf{1}_{B(0, A)^c}\big\|_{L^2} = 0.
\end{align}
We are going to show 
\begin{align}\label{TFK-3}
\lim_{|h| \to 0} \sup_{\substack{\|f_1\|_{L^4} \le 1 \\ \|f_2\|_{L^4} \le 1}} 
\big\|\tau_h T_{\sigma_j}(f_1, f_2) - T_{\sigma_j}(f_1, f_2)\big\|_{L^2} = 0.
\end{align}
Once this is established, \eqref{TT-3} follows from Theorem \ref{thm:KRLpq} and \eqref{TFK-1}--\eqref{TFK-3}. 

The kernel of $T_{\sigma_j}$ is given by 
\begin{align}\label{def:Ksi}
K_{\sigma_j}(x, y, z) 
= \iint_{\R^{2n}} \sigma_j(x, \xi, \eta) e^{2\pi i [(x-y) \cdot \xi + (x-z) \cdot \eta]} d\xi \, d\eta. 
\end{align}
Fix an even number $N>n$. By definition, we see that $\sigma_j \in \mathcal{S}_{1, 0}^{-2N}$ and 
\begin{align}\label{KKSS}
\|K_{\sigma_j}\|_{L^{\infty}} 
+ \|K_{\Delta_{\xi}^N \sigma_j}\|_{L^{\infty}} 
+ \|K_{\Delta_{\eta}^N \sigma_j}\|_{L^{\infty}} 
\lesssim \iint_{\R^{2n}} \frac{d\xi \, d\eta}{(1+|\xi|+|\eta|)^{2N}} 
\lesssim 1. 
\end{align}
Using integration by parts and the fact that $e^{i \xi \cdot x} = (-1)^k |x|^{-2k} \Delta_{\xi}^k e^{i \xi \cdot x}$ for all $k \in \N$ and $x \neq 0$, we deduce  
\begin{align}\label{KD-1}
K_{\Delta_{\xi}^N \sigma_j}(x, y, z) 
&= \iint_{\R^{2n}} \Delta_{\xi}^N \sigma_j(x, \xi, \eta) 
e^{2\pi i \xi \cdot (x-y)} e^{2\pi i \eta \cdot (x-z)} \, d\xi \, d\eta 
\\ \nonumber 
&= \iint_{\R^{2n}} \Delta_{\xi}^N \big(e^{2\pi i \xi \cdot (x-y)} \big) 
e^{2\pi i \eta \cdot (x-z)} \sigma_j(x, \xi, \eta) \, d\xi \, d\eta 
\\ \nonumber
&= (2\pi |x-y|)^{2N} K_{\sigma_j}(x, y, z).  
\end{align}
Similarly, 
\begin{align}\label{KD-2}
K_{\Delta_{\eta}^N \sigma_j}(x, y, z) 
= (2\pi |x-z|)^{2N} K_{\sigma_j}(x, y, z).  
\end{align}
Thus, it follows from \eqref{KD-1} and \eqref{KD-2} that 
\begin{align}\label{KSIJ}
K_{\sigma_j}(x, y, z) 
= \frac{K_{\sigma_j}(x, y, z) + K_{\Delta_{\xi}^N \sigma_j}(x, y, z) + K_{\Delta_{\eta}^N \sigma_j}(x, y, z)}{1 + (2\pi |x-y|)^{2N} + (2\pi |x-z|)^{2N}}, 
\end{align}
which along with \eqref{KKSS} implies that for any $h \in \Rn$, 
\begin{align}\label{KSJ-1}
|K_{\sigma_j}(x-h, y, z) - K_{\sigma_j}(x, y, z)|
\lesssim \sum_{i=1}^4 \mathbf{K}_j^i(x, y, z), 
\end{align}
where 
\begin{align*}
\mathbf{K}_j^1(x, y, z) 
&:= \bigg|\frac{1}{1 + (2\pi |x-h-y|)^{2N} + (2\pi |x-h-z|)^{2N}} 
\\
&\qquad\qquad- \frac{1}{1 + (2\pi |x-y|)^{2N} + (2\pi |x-z|)^{2N}}\bigg|, 
\\
\mathbf{K}_j^2(x, y, z) 
&:= \frac{|K_{\sigma_j}(x-h, y, z) - K_{\sigma_j}(x, y, z)|}{1 + (2\pi |x-y|)^{2N} + (2\pi |x-z|)^{2N}}, 
\\
\mathbf{K}_j^3(x, y, z) 
&:= \frac{|K_{\Delta_{\xi}^N \sigma_j}(x-h, y, z) 
- K_{\Delta_{\xi}^N \sigma_j}(x, y, z)|}{1 + (2\pi |x-y|)^{2N} + (2\pi |x-z|)^{2N}}, 
\\
\mathbf{K}_j^4(x, y, z) 
&:= \frac{|K_{\Delta_{\eta}^N \sigma_j}(x-h, y, z) 
- K_{\Delta_{\eta}^N \sigma_j}(x, y, z)|}{1 + (2\pi |x-y|)^{2N} + (2\pi |x-z|)^{2N}}. 
\end{align*}
Denote 
\begin{align}\label{KSJ-2}
\mathbf{G}_j^i(x) 
:= \iint_{\R^{2n}} \mathbf{K}_j^i(x, y, z) |f_1(y)| \, |f_2(z)| \, dy \, dz, \quad i=1, 2, 3, 4.
\end{align}
Note that 
\begin{align*}
\mathbf{K}_j^1(x, y, z) 
\lesssim |h| \sum_{k=0}^{2N-1} 
\frac{(1+|x-h-y|+|x-h-z|)^{-k}}{(1+|x-y|+|x-z|)^{2N-k-1}}
=: |h| \sum_{k=0}^{2N-1} \mathbf{K}_{j, k}^1(x, y, z).  
\end{align*}
If we set $\phi_k(x) := \frac{(1+|x-h|)^{-k/2}}{(1+|x|)^{(2N-k-1)/2}}$, then 
\begin{align*}
\mathbf{G}_{j, k}^1(x) 
:= \iint_{\R^{2n}} \mathbf{K}_{j, k}^1(x, y, z) |f_1(y)| \, |f_2(z)| \, dy \, dz 
\le \phi_k*|f_1|(x) \, \phi_k*|f_2|(x), 
\end{align*}
and hence, 
\begin{align}\label{GJJ-1}
\|\mathbf{G}_j^1\|_{L^2}
&\lesssim |h| \sup_k \|\mathbf{G}_{j, k}^1\|_{L^2}
\le |h| \sup_k \|\phi_k*|f_1|\|_{L^4} \|\phi_k*|f_2|\|_{L^4}
\\ \nonumber 
&\le |h| \|f_1\|_{L^4} \|f_2\|_{L^4} \sup_k \|\phi_k\|_{L^1}^2
\lesssim |h| \|f_1\|_{L^4} \|f_2\|_{L^4}. 
\end{align}
Moreover, to control $\mathbf{G}_j^2$, we use \eqref{def:Ksi} and the mean value theorem to arrive at 
\begin{align*}
&|K_{\sigma_j}(x-h, y, z) - K_{\sigma_j}(x, y, z)|
\\ 
&\le \iint_{\R^{2n}} |\sigma_j(x-h, \xi, \eta) - \sigma_j(x, \xi, \eta)| \, d\xi \, d\eta 
+ \iint_{\R^{2n}} |\sigma_j(x, \xi, \eta)| \big|e^{2\pi i h \cdot (\xi+\eta)} -1 \big| \, d\xi \, d\eta
\\
&\lesssim |h| \iint_{\R^{2n}} \int_0^1 |\nabla_x \sigma_j(x - th, \xi, \eta)| \, dt \, d\xi \, d\eta 
+ |h| \iint_{\R^{2n}} |\sigma_j(x, \xi, \eta)| |\xi + \eta| \, d\xi \, d\eta 
\\
&\lesssim |h| \iint_{\R^{2n}} \frac{\psi_{2j}(\xi) \psi_{2j}(\eta)}{(1+|\xi|+|\eta|)^{2N}} \, d\xi \, d\eta 
+ |h| \iint_{\R^{2n}} \frac{|\xi| + |\eta|}{(1+|\xi|+|\eta|)^{2N}} \, d\xi \, d\eta 
\lesssim |h|, 
\end{align*}
which gives  
\begin{align}\label{def:GJ2}
\mathbf{G}_j^2(x) 
\lesssim \iint_{\R^{2n}} \frac{|h| |f_1(y)| \, |f_2(z)|}{(1+|x-y|+|x-z|)^{2N}} \, dy \, dz
\lesssim |h| \mathcal{M}(f_1, f_2)(x), 
\end{align}
and 
\begin{align}\label{GJJ-2}
\|\mathbf{G}_j^2\|_{L^2}
\lesssim |h| \|\mathcal{M}(f_1, f_2)\|_{L^2}
\lesssim |h| \|f_1\|_{L^4} \|f_2|\|_{L^4}. 
\end{align}
The same argument holds for $\mathbf{G}_j^3$ and $\mathbf{G}_j^4$: 
\begin{align}\label{GJJ-34}
\|\mathbf{G}_j^i\|_{L^2}
\lesssim |h| \|f_1\|_{L^4} \|f_2\|_{L^4}, \quad i=3, 4. 
\end{align}
Therefore, \eqref{KSJ-1}--\eqref{GJJ-34} lead to 
\begin{align*}
\|\tau_h T_{\sigma_j}(f_1, f_2) - T_{\sigma_j}(f_1, f_2)\|_{L^2}  
\lesssim \sum_{i=1}^4 \|\mathbf{G}_j^i\|_{L^2} 
\lesssim |h| \|f_1\|_{L^4} \|f_2\|_{L^4}, 
\end{align*}
which yields \eqref{TFK-3}. 
\end{proof}

\begin{lemma}\label{lem:Pab}
Let $0<\rho \le 1$, $0 \le \delta \le 1$, and $m \in \R$. Let $\sigma \in \mathcal{K}_{\rho, \delta}^m$ and define $\sigma_j$ as in \eqref{def:sig}. Then for all multi-indices $\alpha$, $\beta$, and $\gamma$, 
\begin{align*}
\lim_{j \to \infty} P_{\alpha, \beta, \gamma}^{m, \rho, \delta}(\sigma-\sigma_j) = 0.  
\end{align*}
\end{lemma}

\begin{proof}
Fix multi-indices $\alpha$, $\beta$, and $\gamma$. By Leibniz's rule, we have   
\begin{align*}
&|\partial_x^{\alpha} \partial_{\xi}^{\beta} \partial_{\eta}^{\gamma} (\sigma - \sigma_j)(x, \xi, \eta)|
\lesssim \sum_{\substack{\alpha_0 + \alpha_1 = \alpha \\ \beta_0 + \beta_1 = \beta \\ 
\gamma_0 + \gamma_1 = \gamma}} 
S_{\alpha_0, \beta_0, \gamma_0}^{\alpha_1, \beta_1, \gamma_1}
\\ 
&=: \sum_{\substack{\alpha_0 + \alpha_1 = \alpha \\ \beta_0 + \beta_1 = \beta \\ 
\gamma_0 + \gamma_1 = \gamma}} 
|\partial_x^{\alpha_0} \partial_{\xi}^{\beta_0} \partial_{\eta}^{\gamma_0} \sigma(x, \xi, \eta)| 
|\partial_x^{\alpha_1} \partial_{\xi}^{\beta_1} \partial_{\eta}^{\gamma_1} (1-\phi_j)(x, \xi, \eta)|.   
\end{align*}
If $|\alpha_1| = |\beta_1| = |\gamma_1| = 0$, then by definition and the fact that $\supp(1-\phi_j) \subset \{|x| + |\xi| + |\eta| \ge 2^j\}$, 
\begin{align*}
S_{\alpha_0, \beta_0, \gamma_0}^{\alpha_1, \beta_1, \gamma_1}
\lesssim C_{\alpha, \beta, \gamma}(x, \xi, \eta) (1+ |\xi| + |\eta|)^{m + \delta|\alpha| - \rho (|\beta| + |\gamma|)} 
\mathbf{1}_{\{|x| + |\xi| + |\eta| \ge 2^j\}}. 
\end{align*}
If $|\alpha_1| \neq 0$ and $|\beta_1| = |\gamma_1| = 0$, then 
\begin{align*}
|\partial_x^{\alpha_1} (1-\phi_j)(x, \xi, \eta)| 
\le |\partial_x^{\alpha_1} \psi_j(x)| |\psi_j(\xi)| |\psi_j(\eta)| 
\lesssim 2^{-j |\alpha_1|} \mathbf{1}_{\{2^j \le |x| \le 2^{j+1}\}}, 
\end{align*}
which yields  
\begin{align*}
S_{\alpha_0, \beta_0, \gamma_0}^{\alpha_1, \beta_1, \gamma_1}
& \lesssim C_{\alpha_0, \beta, \gamma}(x, \xi, \eta) (1+ |\xi| + |\eta|)^{m + \delta|\alpha_0| - \rho (|\beta| + |\gamma|)} 
\mathbf{1}_{\{2^j \le |x| \le 2^{j+1}\}}
\\
&\le C_{\alpha_0, \beta, \gamma}(x, \xi, \eta) (1+ |\xi| + |\eta|)^{m + \delta|\alpha| - \rho (|\beta| + |\gamma|)} 
\mathbf{1}_{\{|x| + |\xi| + |\eta| \ge 2^j\}}. 
\end{align*}
If $|\beta_1| \neq 0$ or $|\gamma_1| \neq 0$, then we may assume that $|\beta_1| \neq 0$ and 
\begin{align*}
&|\partial_x^{\alpha_1} \partial_{\xi}^{\beta_1} \partial_{\eta}^{\gamma_1} (1-\phi_j)(x, \xi, \eta)| 
\\
&\le |\partial_x^{\alpha_1} \psi_j(x)| |\partial_{\xi}^{\beta_1} \psi_j(\xi)| |\partial_{\eta}^{\gamma_1} \psi_j(\eta)| 
\lesssim 2^{-j (|\alpha_1| + |\beta_1| + |\gamma_1|)} 
\mathbf{1}_{\{2^j \le |\xi| \le 2^{j+1}\}} 
\mathbf{1}_{\{|\eta| \le 2^{j+1}\}} 
\\
&\lesssim (1+|\xi| + |\eta|)^{-(|\beta_1| + |\gamma_1|)} \mathbf{1}_{\{|\xi| \ge 2^j\}}
\le (1+|\xi| + |\eta|)^{-\rho (|\beta_1| + |\gamma_1|)} \mathbf{1}_{\{|x| + |\xi| + |\eta| \ge 2^j\}},
\end{align*}
provided that $2^j \le |\xi| \le 1 + |\xi| + |\eta| \le 1+2^{j+1} + 2^{j+1} \le 5 \cdot 2^j$. This in turn implies 
\begin{align*}
S_{\alpha_0, \beta_0, \gamma_0}^{\alpha_1, \beta_1, \gamma_1}
& \lesssim C_{\alpha_0, \beta_0, \gamma_0}(x, \xi, \eta) 
(1+ |\xi| + |\eta|)^{m + \delta |\alpha_0| - \rho (|\beta_0| + |\gamma_0|)} 
\\
&\quad\times (1+|\xi| + |\eta|)^{-\rho (|\beta_1| + |\gamma_1|)} \mathbf{1}_{\{|x| + |\xi| + |\eta| \ge 2^j\}} 
\\
&\le C_{\alpha_0, \beta_0, \gamma_0}(x, \xi, \eta) (1+ |\xi| + |\eta|)^{m + \delta|\alpha| - \rho (|\beta| + |\gamma|)} 
\mathbf{1}_{\{|x| + |\xi| + |\eta| \ge 2^j\}}. 
\end{align*}
Since the fact $\sigma \in \mathcal{K}_{\rho, \delta}^m$ gives $\lim_{|x| + |\xi| + |\eta| \to \infty} C_{\alpha', \beta', \gamma'}(x, \xi, \eta) = 0$ for all multi-indices $\alpha'$, $\beta'$, and $\gamma'$, we utilize the preceding estimates to conclude that $\lim_{j \to \infty} P_{\alpha, \beta, \gamma}^{m, \rho, \delta}(\sigma-\sigma_j) = 0$. 
\end{proof}

\begin{lemma}\label{lem:KP}
Let $0<\rho \le 1$, $0 \le \delta \le 1$, and $m >1-2n$. Let $\sigma \in \mathcal{S}_{\rho, \delta}^m$ and $K_{\sigma}$ be the kernel of $T_{\sigma}$. Then for every $N \in \N_+$ and $\alpha, \beta, \gamma \in \N^n$ with $2N \rho > 2n + m + |\alpha| + |\beta| + |\gamma|$,
\begin{align*}
|\partial_x^{\alpha} \partial_y^{\beta} \partial_z^{\gamma} K_{\sigma}(x, y, z)|
\lesssim \frac{{\displaystyle\sum}_{\substack{\alpha' \le \alpha \\ |\beta'| + |\gamma'| \le 2N_0}} 
P_{\alpha', \beta', \gamma'}^{m, \rho, \delta}(\sigma)}{(|x-y| + |x-z|)^{2n+m + 2N(1-\rho) + |\alpha| + |\beta| + |\gamma|}}, 
\end{align*}
whenever $x \neq y$ or $x \neq z$, where $N_0$ is an integer such that $N_0 \rho \ge n + |m| + N + |\alpha| + |\beta| + |\gamma|$. 
\end{lemma}

\begin{proof}
Let $\psi$ and $\sigma_j$ be defined in \eqref{def:psi} and \eqref{def:sig}, respectively. Let $K_{\sigma_j}$ be the kernel of $T_{\sigma_j}$. Since $\sigma_j$ converges pointwise to $\sigma$, it follows that $K_{\sigma_j}$ converges to $K_{\sigma}$ in the sense of distributions. Hence, it is enough to show the same bound holds for $K_{\sigma_j}$ uniformly in $j \in \N$. By definition and Leibniz's rule, there holds  
\begin{align*}
&|\partial_x^{\alpha_2} \partial_{\xi}^{\beta_2} \partial_{\eta}^{\gamma_2} \sigma_j(x, \xi, \eta)|
\\ 
&\le \sum_{\substack{\alpha_0 + \alpha_1 = \alpha_2 \\ \beta_0 + \beta_1 = \beta_2 \\ 
\gamma_0 + \gamma_1 = \gamma_2}} 
C_{\alpha_2, \beta_2, \gamma_2}^{\alpha_1, \beta_1, \gamma_1} 
|\partial_x^{\alpha_1} \partial_{\xi}^{\beta_1} \partial_{\eta}^{\gamma_1} \sigma(x, \xi, \eta)| 
|\partial_x^{\alpha_0} \psi_j(x)| |\partial_{\xi}^{\beta_0} \psi_j(\xi)| 
|\partial_{\eta}^{\gamma_0} \psi_j(\eta)|
\\
&\le \sum_{\substack{\alpha_0 + \alpha_1 = \alpha_2 \\ \beta_0 + \beta_1 = \beta_2 \\ 
\gamma_0 + \gamma_1 = \gamma_2}} 
C_{\alpha_2, \beta_2, \gamma_2}^{\alpha_1, \beta_1, \gamma_1} 
P_{\alpha_1, \beta_1, \gamma_1}^{m, \rho, \delta}(\sigma) 
(1+|\xi| + |\eta|)^{m+\delta |\alpha_1| - \rho(|\beta_1| + |\gamma_1|)-|\beta_0| - |\gamma_0|},  
\end{align*}
where the constant $C_{\alpha_2, \beta_2, \gamma_2}^{\alpha_1, \beta_1, \gamma_1}$ is independent of $j$ and we have used that for $\beta_0 \neq 0$ (or $\gamma_0 \neq 0$), 
\begin{align*}
|\partial_{\xi}^{\beta_0} \psi_j(\xi)| |\partial_{\eta}^{\gamma_0} \psi_j(\eta)|
\lesssim 2^{-j (|\beta_0| + |\gamma_0|)} \mathbf{1}_{\{|\xi|, |\eta| \le 2^{j+1}\}} 
\lesssim (1+|\xi| + |\eta|)^{-|\beta_0| - |\gamma_0|}, 
\end{align*}
provided the construction of $\psi_j$. Hence, 
\begin{align}\label{PPm-1}
\sup_{j \in \N} P_{\alpha_2, \beta_2, \gamma_2}^{m, \rho, \delta}(\sigma_j)
\lesssim \sum_{\substack{\alpha_1 \le \alpha_2 \\ \beta_1 \le \beta_2 \\ \gamma_1 \le \gamma_2}}  
P_{\alpha_1, \beta_1, \gamma_1}^{m, \rho, \delta}(\sigma). 
\end{align}
Additionally, observe that 
\begin{align}\label{PPm-2}
\partial_x^{\alpha} \partial_y^{\beta} \partial_z^{\gamma} K_{\sigma_j}(x, y, z)
= \iint_{\R^{2n}} e^{2\pi i \xi \cdot (x-y)} e^{2\pi i \eta \cdot (x-z)} 
\widetilde{\sigma}_j(x, \xi, \eta) \, d\xi \, d\eta, 
\end{align}
where 
\begin{align*}
\widetilde{\sigma}_j(x, \xi, \eta)
= \sum_{\alpha_1 + \alpha_2 + \alpha_3 = \alpha} C_{\alpha, \beta, \gamma}^{\alpha_1, \alpha_2}
\, \xi^{\alpha_2 + \beta} \, \eta^{\alpha_3 + \gamma} 
\, \partial_x^{\alpha_1} \sigma_j(x, \xi, \eta), 
\end{align*}
with  
\begin{align}\label{PPm-3}
P_{\alpha', \beta', \gamma'}^{m+|\alpha| + |\beta| + |\gamma|, \rho, \delta} 
\big(\xi^{\alpha_2 + \beta} \, \eta^{\alpha_3 + \gamma} 
\, \partial_x^{\alpha_1} \sigma_j(x, \xi, \eta) \big)
\lesssim \sum_{\substack{\beta'' \le \beta' \\ \gamma'' \le \gamma'}} 
P_{\alpha_1+\alpha', \beta'', \gamma''}^{m, \rho, \delta}(\sigma_j), 
\end{align}
where the implicit constant is independent of $j$. By \eqref{PPm-1}--\eqref{PPm-3}, it suffices to show 
\begin{align}\label{Kj}
|K_{\sigma_j}(x, y, z)|
\lesssim \frac{\sum_{|\beta'| + |\gamma'| \le 2N_0} 
P_{0, \beta', \gamma'}^{m, \rho, \delta}(\sigma_j)}{(|x-y| + |x-z|)^{2n+m+2N(1-\rho)}}, 
\quad\text{ uniformly in } j \in \N,  
\end{align}
whenever $x \neq y$ or $x \neq z$. 

By symmetry under the simultaneous interchange of $(y, \xi)$ and $(z, \eta)$, we may assume that $|x-y| \ge |x-z|$. Thus, $|x-y| \simeq |x-y| + |x-z|$. Consider first the case $|x-y| \ge 1$.  Let $\Delta_{\xi}$ denote  the Laplace operator in the variable $\xi$. Using the fact that $e^{i \xi \cdot x} = (-1)^k |x|^{-2k} \Delta_{\xi}^k e^{i \xi \cdot x}$ for all $k \in \N$ and $x \neq 0$, we obtain 
\begin{align*}
|K_{\sigma_j}(x, y, z)| 
&= \frac{1}{(2\pi |x-y|)^{2N_0}} \bigg|
\iint_{\R^{2n}} \Delta_{\xi}^{N_0} (e^{2\pi i \xi \cdot (x-y)}) e^{2\pi i \eta \cdot (x-z)} 
\sigma_j(x, \xi, \eta) \, d\xi \, d\eta \bigg|
\\
&= \frac{1}{(2\pi |x-y|)^{2N_0}} \bigg|
\iint_{\R^{2n}} e^{2\pi i \xi \cdot (x-y)} e^{2\pi i \eta \cdot (x-z)} 
\Delta_{\xi}^{N_0} \sigma_j(x, \xi, \eta) \, d\xi \, d\eta \bigg|
\\
&\lesssim \frac{1}{|x-y|^{2N_0}} 
\sum_{|\beta| =2N_0} P_{0, \beta, 0}^{m, \rho, \delta}(\sigma_j) 
\iint_{\R^{2n}} \frac{d\xi \, d\eta}{(1+|\xi| + |\eta|)^{2N_0 \rho -m}}
\\
&\le \frac{1}{|x-y|^{2N_0}} 
\sum_{|\beta| =2N_0} P_{0, \beta, 0}^{m, \rho, \delta}(\sigma_j) 
\bigg(\int_{\Rn} \frac{d\xi}{(1+|\xi|)^{N_0 \rho - \frac{m}{2}}} \bigg)^2 
\\
&\lesssim \frac{1}{|x-y|^{2n+m+2N(1-\rho)}} 
\sum_{|\beta| =2N_0} P_{0, \beta, 0}^{m, \rho, \delta}(\sigma_j), 
\end{align*}
where we have used that $2N_0 \rho-m > 2n$ and $2N_0 \ge 2n+m+2N(1-\rho)$.

It remains to deal with the case $0< r:= |x-y| < 1$. Letting $u := \frac{x-y}{|x-y|}$ and $\widetilde{\psi} := 1-\psi$, we rewrite 
\begin{align*}
K_{\sigma_j}(x, y, z) 
&= r^{-2n} \iint_{\R^{2n}} e^{2\pi i \xi \cdot u} e^{2\pi i r^{-1} \eta \cdot (x-z)} 
\sigma_j(x, r^{-1}\xi, r^{-1}\eta) \, d\xi \, d\eta 
\\
&= r^{-2n} \iint_{\R^{2n}} e^{2\pi i \xi \cdot u} e^{2\pi i r^{-1} \eta \cdot (x-z)} 
\sigma_j(x, r^{-1}\xi, r^{-1}\eta) \psi(\xi) \psi(\eta)\, d\xi \, d\eta 
\\
&\quad+ r^{-2n} \iint_{\R^{2n}} e^{2\pi i \xi \cdot u} e^{2\pi i r^{-1} \eta \cdot (x-z)} 
\sigma_j(x, r^{-1}\xi, r^{-1}\eta) \widetilde{\psi}(\xi) \psi(\eta) \, d\xi \, d\eta 
\\
&\quad+ r^{-2n} \iint_{\R^{2n}} e^{2\pi i \xi \cdot u} e^{2\pi i r^{-1} \eta \cdot (x-z)} 
\sigma_j(x, r^{-1}\xi, r^{-1}\eta) \widetilde{\psi}(\eta) \, d\xi \, d\eta 
\\
&=: K_{\sigma_j}^1(x, y, z) + K_{\sigma_j}^2(x, y, z) + K_{\sigma_j}^3(x, y, z).
\end{align*}
The condition $m+2n-1>0$ implies 
\begin{align*}
|K_{\sigma_j}^1(x, y, z)| 
&\le r^{-2n} P_{0, 0, 0}^{m, \rho, \delta}(\sigma_j) 
\int_{|\xi| \le 2} \int_{|\eta| \le 2} (1+r^{-1}|\xi| + r^{-1} |\eta|)^m \, d\xi \, d\eta 
\\
&\simeq P_{0, 0, 0}^{m, \rho, \delta}(\sigma_j) 
\int_{0}^{\frac{2}{r}} \int_{0}^{\frac{2}{r}} 
(1+s+t)^m s^{n-1} t^{n-1} \, ds \, dt 
\\
&\le P_{0, 0, 0}^{m, \rho, \delta}(\sigma_j) 
\int_{0}^{\frac{2}{r}} \int_{0}^{\frac{2}{r}} 
(1+s+t)^{m+2n-2} \, ds \, dt 
\\
&\le P_{0, 0, 0}^{m, \rho, \delta}(\sigma_j) 
(1+r^{-1})^{m+2n} 
\\
&\lesssim P_{0, 0, 0}^{m, \rho, \delta}(\sigma_j) 
\, r^{-2n-m-2N(1-\rho)}. 
\end{align*}
To proceed, note that for any $|\xi| \ge 1$ or $|\eta| \ge 1$, 
\begin{align*}
\big|\partial_{\xi}^{\beta} \big(\sigma_j(x, r^{-1}\xi, r^{-1}\eta) \big)\big|
&\le P_{0, \beta, 0}^{m, \rho, \delta}(\sigma_j) \, r^{-|\beta|} 
(1+ r^{-1}|\xi| + r^{-1} |\eta|)^{m-\rho |\beta|}
\\ 
& \simeq P_{0, \beta, 0}^{m, \rho, \delta}(\sigma_j) \, r^{-m-|\beta|(1-\rho)} 
(1+ |\xi| + |\eta|)^{m-\rho |\beta|}, 
\end{align*}
provided that $1+|\xi| + |\eta| \ge r +|\xi| + |\eta| \ge (1 + |\xi| + |\eta|)/2$. Note also that $\supp \widetilde{\psi} \subset \{|\xi| \ge 1\}$, and for any $|\beta|>0$, $\partial_{\xi}^{\beta} \widetilde{\psi}(\xi) \neq 0$ implies $1 \le |\xi| \le 2$. Thus, we use integration by parts to arrive at 
\begin{align*}
|K_{\sigma_j}^2(x, y, z)|
& = \frac{r^{-2n}}{(2\pi)^{2N}} 
\bigg|\iint_{\R^{2n}} \Delta_{\xi}^N \big(e^{2\pi i \xi \cdot u} \big) 
e^{2\pi i r^{-1} \eta \cdot (x-z)} 
\sigma_j(x, r^{-1}\xi, r^{-1}\eta) \widetilde{\psi}(\xi) \psi(\eta) \, d\xi \, d\eta \bigg|
\\ 
&\le r^{-2n} \iint_{\R^{2n}}  
\big| \Delta_{\xi}^N \big(\sigma_j(x, r^{-1}\xi, r^{-1}\eta) \widetilde{\psi}(\xi) \big) \big| 
|\psi(\eta)| \, d\xi \, d\eta 
\\
&\lesssim r^{-2n} \iint_{\substack{\{|\xi| \ge 1\} \\ \{|\eta| \le 2\}}}  
\big| \Delta_{\xi}^N \big(\sigma_j(x, r^{-1}\xi, r^{-1}\eta) \big) \big| 
|\widetilde{\psi}(\xi)| |\psi(\eta)| \, d\xi \, d\eta 
\\
&\quad+ r^{-2n} \iint_{\substack{\{1\le |\xi| \le 2\} \\ \{|\eta| \le 2\}}} 
\sum_{0<|\beta| \le 2N} 
\big| \partial_{\xi}^{2N-\beta} \big(\sigma_j(x, r^{-1}\xi, r^{-1}\eta) \big) \big| 
|\partial_{\xi}^{\beta} \widetilde{\psi}(\xi) | 
|\psi(\eta)| \, d\xi \, d\eta 
\\
&\lesssim \sum_{|\beta| \le 2N} P_{0, \beta, 0}^{m, \rho, \delta}(\sigma_j) 
\iint_{\substack{\{|\xi| \ge 1\} \\ \{|\eta| \le 2\}}}  
\frac{r^{-2n-m- 2N(1-\rho)}}{(1+|\xi| + |\eta|)^{2N \rho -m}}
\, d\xi \, d\eta 
\\
&\lesssim \sum_{|\beta| \le 2N} P_{0, \beta, 0}^{m, \rho, \delta}(\sigma_j)  
\, r^{-2n-m-2N(1-\rho)}, 
\end{align*}
where we used that for any $|\beta| \le 2N$, 
\begin{align*}
\int_{\{|\xi| \ge 1\}}  
\frac{|\xi|^{-|\beta|} \, d\xi}{(1+|\xi| + |\eta|)^{(2N-|\beta|)\rho -m}} 
\le \int_{\{|\xi| \ge 1\}} \frac{d\xi}{|\xi|^{2N\rho-m}}
\lesssim 1. 
\end{align*}
Similarly, 
\begin{align*}
|K_{\sigma_j}^3(x, y, z)| 
&\le r^{-2n} \iint_{\R^{2n}} 
\big|\Delta_{\xi}^N \big(\sigma_j(x, r^{-1}\xi, r^{-1}\eta) \big)\big| 
|\widetilde{\psi}(\eta)| \, d\xi \, d\eta 
\\
&\le \sum_{|\beta| = 2N} P_{0, \beta, 0}^{m, \rho, \delta}(\sigma_j) 
\iint_{\R^{2n}} \frac{r^{-2n-m-2N(1-\rho)}}{(1+|\xi| + |\eta|)^{2N\rho-m}} \, d\xi \, d\eta 
\\
&\lesssim \sum_{|\beta| = 2N} P_{0, \beta, 0}^{m, \rho, \delta}(\sigma_j) \, 
r^{-2n-m-2N(1-\rho)}, 
\end{align*}
provided $2N\rho-m>2n$. Thus, we conclude \eqref{Kj} from these estimates above. 
\end{proof}

\subsection{Bilinear commutators} 
Let us present the proof of Theorem \ref{thm:bT}. We will apply the strategy in the proof of Theorem \ref{thm:Tsig}. 
It was proved in \cite[Theorem 1]{BO} that for any $\sigma \in \mathcal{S}_{1,0}^1$ and Lipschitz function $b$ with $\nabla b \in L^{\infty}(\Rn)$, 
\begin{align}\label{bTCZO}
\text{$[b, T_{\sigma}]_k$ is a bilinear Calder\'{o}n--Zygmund operator, \quad $k=1, 2$.}
\end{align} 
By the fact that $\mathcal{K}_{1,0}^1 \subset \mathcal{S}_{1,0}^1$ and \eqref{bTCZO}, we see that  
\begin{align}\label{bTCZO-1}
\text{$[b, T_{\sigma}]_k$ is a bilinear Calder\'{o}n--Zygmund operator, \quad $k=1, 2$.}
\end{align} 
In view of Theorems \ref{thm:CCZK}, \ref{thm:WCP}, and \ref{thm:T1CMO}, it suffices to prove  
\begin{align}\label{bT-1}
\text{$[b, T_{\sigma}]_k$ is compact from $L^4(\Rn) \times L^4(\Rn)$ to $L^2(\Rn)$, \quad $k=1, 2$}. 
\end{align}

For each $j \in \N$, let $\phi_j$ be defined in \eqref{def:psi}. Then it is not hard to check that 
\begin{align*}
\sigma_j := \sigma \phi_j \in \mathcal{S}_{1, 0}^1 
\quad\text{ uniformly in $j \in \N$},  
\end{align*}
which together with \eqref{bTCZO} gives 
\begin{align}\label{bTCZO-2}
\text{$[b, T_{\sigma_j}]_k$ is a bilinear Calder\'{o}n--Zygmund operator,}
\end{align}
for all $k=1, 2$ and $j \in \N$. By Lemma \ref{lem:limit}, to obtain \eqref{bT-1}, it is enough to justify  
\begin{align}\label{bT-2}
\lim_{j \to \infty} \big\|[b, T_{\sigma_j}]_k - [b, T_{\sigma}]_k \big\|_{L^4 \times L^4 \to L^2} = 0
\end{align}
and 
\begin{align}\label{bT-3}
\text{$[b, T_{\sigma_j}]_k$ is compact from $L^4(\Rn) \times L^4(\Rn)$ to $L^2(\Rn)$, \quad $\forall j \in \N$.} 
\end{align}

We only focus on the case $k=1$. The kernel of $[b, T_{\sigma}]_1$ is given by 
\begin{align}\label{Ksb}
K_{\sigma}^b(x, y, z) := (b(x) - b(y)) K_{\sigma}(x, y, z).
\end{align} 
Then it follows from Lemma \ref{lem:KP} that 
\begin{align}\label{nu-0}
\|K_{\sigma_j}^b - K_{\sigma}^b\|_{\mathrm{CZ}(1)} 
= \|K_{\sigma_j - \sigma}^b\|_{\mathrm{CZ}(1)}
\lesssim \sum_{|\alpha|, |\beta|, |\gamma| \le N_0} 
P_{\alpha, \beta, \gamma}^{1, 1, 0}(\sigma_j - \sigma), 
\end{align}
where $N_0 > 4n+8$. To avoid the confusion of notation, we use $\nu_j$ and $\widetilde{\nu}_j$ instead of $\sigma_j$ and $\widetilde{\sigma}_j$ respectively in \cite[p. 294]{BO}. As calculated there, one has 
\begin{align}\label{nu-1}
P_{\alpha, \beta, \gamma}^{0,1, 0}(\nu_j) 
\lesssim P_{\alpha, \widetilde{\beta}, \gamma}^{1,1, 0}(\sigma)
\quad \text{ and } \quad
P_{\alpha, \beta, \gamma}^{0,1, 0}(\widetilde{\nu}_j) 
\lesssim P_{\alpha, \beta, \widetilde{\gamma}}^{1,1, 0}(\sigma), 
\end{align}
where $\widetilde{\beta} := \beta + (0, \ldots, 1, \ldots, 0)$ and $\widetilde{\gamma} := \gamma + (0, \ldots, 1, \ldots, 0)$. Invoking \eqref{nu-1} and checking the proof of \cite[Lemmas 7--8]{BO}, we obtain 
\begin{align}\label{nu-2}
\|[b, T_{\sigma}]_1 (1, 1)\|_{\BMO} 
+ \|[b, T_{\sigma}]_1^{*i} (1, 1)\|_{\BMO} 
&\lesssim \sum_{|\alpha|, |\beta|, |\gamma| \le N_0} 
P_{\alpha, \beta, \gamma}^{1,1, 0}(\sigma), 
\end{align}
for each $i=1, 2$, where $N_0>2n$ large enough. Moreover, from the proof of  \cite[Lemma 10]{BO}, we see that the constant of the weak boundedness property of $[b, T_{\sigma}]_1$ is dominated by 
\begin{align}\label{nu-3}
\|K_{\sigma}^b\|_{\mathrm{CZ}(1)}
+ \|\nabla b\|_{L^{\infty}} \sum_{j=1}^n 
\big(\|T_{\nu_j}\|_{L^4 \times L^4 \to L^2} 
+ \|T_{\widetilde{\nu}_j}\|_{L^4 \times L^4 \to L^2}\big)
\lesssim \sum_{|\alpha|, |\beta|, |\gamma| \le N_0} 
P_{\alpha, \beta, \gamma}^{1,1, 0}(\sigma), 
\end{align}
where the above inequality can be shown as \eqref{KK23} by means of \eqref{nu-1}. Accordingly, by \eqref{nu-0}, \eqref{nu-2}, \eqref{nu-3}, and \cite[Theorem 1.1]{Hart14} (or \cite{LMOV}), we conclude that 
\begin{align*}
\big\|[b, T_{\sigma_j}]_1 - [b, T_{\sigma}]_1 \big\|_{L^4 \times L^4 \to L^2} 
\lesssim \sum_{|\alpha|, |\beta|, |\gamma| \le N_0} 
P_{\alpha, \beta, \gamma}^{1,1, 0}(\sigma_j - \sigma), 
\end{align*}
which along with Lemma \ref{lem:Pab} gives \eqref{bT-2}.

Next, we prove \eqref{bT-3}. Throughout the proof of \eqref{bT-3}, $j$ is fixed, and the implicit constants may depend on $j$. Combining \eqref{bTCZO-2} with \cite[Theorem 3]{GT1}, we achieve  
\begin{align}\label{bTFK-1}
\sup_{\substack{\|f_1\|_{L^4} \le 1 \\ \|f_2\|_{L^4} \le 1}} 
\big\|[b, T_{\sigma_j}]_1(f_1, f_2) \big\|_{L^2} 
\lesssim 1.
\end{align}
The support of $\sigma_j$ implies  
\begin{align*}
K_{\sigma_j}(x, y, z)=0
\quad\text{ and }\quad 
[b, T_{\sigma_j}]_1(f_1, f_2)(x) = 0, 
\quad\text{ for all } |x|>2^{j+1}, 
\end{align*}
which gives   
\begin{align}\label{bTFK-2}
\lim_{A \to \infty} \sup_{\substack{\|f_1\|_{L^4} \le 1 \\ \|f_2\|_{L^4} \le 1}} 
\big\|[b, T_{\sigma_j}]_1(f_1, f_2) \mathbf{1}_{B(0, A)^c}\big\|_{L^2} = 0.
\end{align}
In view of Theorem \ref{thm:KRLpq} and \eqref{bTFK-1}--\eqref{bTFK-2}, the estimate \eqref{bT-3} is reduced to showing  
\begin{align}\label{bTFK-3}
\lim_{|h| \to 0} \sup_{\substack{\|f_1\|_{L^4} \le 1 \\ \|f_2\|_{L^4} \le 1}} 
\big\|\tau_h [b, T_{\sigma_j}]_1(f_1, f_2) - [b, T_{\sigma_j}]_1(f_1, f_2)\big\|_{L^2} = 0.
\end{align}

Let $N>n$ be an even number. It follows from \eqref{KKSS}, \eqref{KSIJ}, and \eqref{Ksb} that for any $h \in \Rn$, 
\begin{align}\label{KSB}
|K_{\sigma_j}^b(x-h, y, z) - K_{\sigma_j}^b(x, y, z)|
\lesssim \sum_{i=0}^4 \mathbf{K}_j^{b, i}(x, y, z), 
\end{align}
where 
\begin{align*}
\mathbf{K}_j^{b, 0}(x, y, z) 
&:= |b(x-h) - b(x)| |K_{\sigma_j}(x, y, z)|, 
\\
\mathbf{K}_j^{b, 1}(x, y, z) 
&:= |b(x-h) - b(y)|
\bigg|\frac{1}{1 + (2\pi |x-h-y|)^{2N} + (2\pi |x-h-z|)^{2N}} 
\\
&\qquad\qquad\qquad\qquad\qquad 
- \frac{1}{1 + (2\pi |x-y|)^{2N} + (2\pi |x-z|)^{2N}}\bigg|, 
\\
\mathbf{K}_j^{b, 2}(x, y, z) 
&:= |b(x-h) - b(y)|
\frac{|K_{\sigma_j}(x-h, y, z) - K_{\sigma_j}(x, y, z)|}{1 + (2\pi |x-h-y|)^{2N} + (2\pi |x-h-z|)^{2N}}, 
\\
\mathbf{K}_j^{b, 3}(x, y, z) 
&:= |b(x-h) - b(y)| \frac{|K_{\Delta_{\xi}^N \sigma_j}(x-h, y, z) 
- K_{\Delta_{\xi}^N \sigma_j}(x, y, z)|}{1 + (2\pi |x-h-y|)^{2N} + (2\pi |x-h-z|)^{2N}}, 
\\
\mathbf{K}_j^{b, 4}(x, y, z) 
&:= |b(x-h) - b(y)| \frac{|K_{\Delta_{\eta}^N \sigma_j}(x-h, y, z) 
- K_{\Delta_{\eta}^N \sigma_j}(x, y, z)|}{1 + (2\pi |x-h-y|)^{2N} + (2\pi |x-h-z|)^{2N}}. 
\end{align*}
Set  
\begin{align*}
\mathbf{G}_j^{b, i}(x) 
:= \iint_{\R^{2n}} \mathbf{K}_j^{b, i}(x, y, z) |f_1(y)| \, |f_2(z)| \, dy \, dz, \quad i=0, 1, 2, 3, 4.
\end{align*}
The fact $\sigma_j \in \mathcal{S}_{1, 0}^{-2N}$ gives 
\begin{align*}
\mathbf{G}_j^{b, 0}(x)
\lesssim \iint_{\R^{2n}} \frac{\|\nabla b\|_{L^{\infty}} |h| |f_1(y)| \, |f_2(z)|}{(1+|x-y|+|x-z|)^{2N}} \, dy \, dz
\lesssim |h| \mathcal{M}(f_1, f_2)(x), 
\end{align*}
and hence, 
\begin{align}\label{GJB-0}
\|\mathbf{G}_j^{b, 0}\|_{L^2}
\lesssim |h| \|\mathcal{M}(f_1, f_2)\|_{L^2} 
\lesssim |h| \|f_1\|_{L^4} \|f_2\|_{L^4}. 
\end{align}
Since it is not hard to check that 
\begin{align*}
\mathbf{K}_j^{b, 1}(x, y, z) 
\lesssim \|\nabla b\|_{L^{\infty}} |h| \sum_{k=0}^{2N-1} 
\frac{(1+|x-h-y|+|x-h-z|)^{-k}}{(1+|x-y|+|x-z|)^{2N-k-1}},   
\end{align*}
we use the same argument as in \eqref{GJJ-1} to deduce 
\begin{align}\label{GJB-1}
\|\mathbf{G}_j^{b, 1}\|_{L^2}
\lesssim |h| \|f_1\|_{L^4} \|f_2\|_{L^4}. 
\end{align}
Moreover, analogously to \eqref{def:GJ2}, 
\begin{align*}
\mathbf{G}_j^{b, 2}(x) 
\lesssim \iint_{\R^{2n}} \frac{\|\nabla b\|_{L^{\infty}} |h| |f_1(y)| \, |f_2(z)| \, dy \, dz}{(1+|x-h-y|+|x-h-z|)^{2N-1}}
\lesssim |h| \mathcal{M}(f_1, f_2)(x-h), 
\end{align*}
which yields 
\begin{align}\label{GJB-2}
\|\mathbf{G}_j^{b, 2}\|_{L^2}
\lesssim |h| \|\mathcal{M}(f_1, f_2)\|_{L^2}
\lesssim |h| \|f_1\|_{L^4} \|f_2\|_{L^4}. 
\end{align}
Much as above, there holds 
\begin{align}\label{GJB-34}
\|\mathbf{G}_j^{b, i}\|_{L^2}
\lesssim |h| \|f_1\|_{L^4} \|f_2\|_{L^4}, \quad i=3, 4. 
\end{align}
Consequently, from \eqref{KSB}--\eqref{GJB-34}, we conclude 
\begin{align*}
\|\tau_h [b, T_{\sigma_j}]_1(f_1, f_2) - [b, T_{\sigma_j}]_1(f_1, f_2)\|_{L^2}  
\lesssim \sum_{i=0}^4 \|\mathbf{G}_j^{b, i}\|_{L^2} 
\lesssim |h| \|f_1\|_{L^4} \|f_2\|_{L^4}, 
\end{align*}
which justifies \eqref{bTFK-3} as desired. 
\qed

\appendix

\section{A dyadic characterization of $\CMO$ space}\label{sec:A-CMO} 
We give a characterization of $\CMO$ space using the projection $P_{\D(N)}^{\perp}$ and dyadic $\BMO$.

\begin{theorem}\label{thm:PDN}
$b \in \CMO(\Rn)$ if and only if 
\begin{align*}
b \in \BMO(\Rn) 
\quad \text{and} \quad 
\lim_{N \to \infty} \sup_{\D} \|P_{\D(N)}^{\perp} b\|_{\BMO_{\D}} = 0.
\end{align*}
\end{theorem}

\begin{proof}
Given a dyadic grid $\D$, we simply write $P_N^{\perp} = P_{\D(N)}^{\perp}$. Observe that for any $J \in \D$ and $N \in \N$, 
\begin{align}\label{bJ-1}
\frac{1}{|J|} \sum_{\substack{I \subset J \\ I \notin \D(N)}} |\langle b, h_I \rangle|^2 
= \frac{1}{|J|} \sum_{\substack{I \in \D \\ I \subset J}} 
\big| \big\langle (P_N^{\perp} b - \langle P_N^{\perp} b \rangle_J) \mathbf{1}_J, h_I \big\rangle \big|^2 
= \fint_J |P_N^{\perp} b - \langle P_N^{\perp} b \rangle_J|^2 
\end{align}
and 
\begin{align}\label{bJ-2}
\frac{1}{|J|} \sum_{\substack{I \subset J \\ I \notin \D(N)}} |\langle b, h_I \rangle|^2 
\le \frac{1}{|J|} \sum_{\substack{I \in \D \\ I \subset J}} 
\big|\big\langle (b - \langle b \rangle_J) \mathbf{1}_J, h_I \big\rangle \big|^2 
= \fint_J |b - \langle b \rangle_J|^2.  
\end{align}
Then by \eqref{bJ-1}, \eqref{bJ-2}, and the John--Nirenberg inequality, 
\begin{align}\label{pnbb}
\sup_{N \ge 1} \|P_N^{\perp} b\|_{\BMO_{\D}} 
\lesssim \|b\|_{\BMO_{\D}}. 
\end{align}

Let $b \in \CMO(\Rn)$ and $\varepsilon>0$. By definition and \eqref{pnbb}, there exists $f \in \mathscr{C}_c^{\infty}(\Rn)$ such that 
\begin{align}\label{bfnn-1}
\sup_{N \ge 1} \sup_{\D} \|P_N^{\perp}(b-f)\|_{\BMO_{\D}}
\lesssim \sup_{\D} \|b-f\|_{\BMO_{\D}} 
\le \|b-f\|_{\BMO} 
\le \varepsilon. 
\end{align}
In addition, in light of Lemma \ref{lem:CC}, there exists $N_0 = N_0(\varepsilon) \ge 1$ so that 
\begin{align}\label{bfnn-2}
\sup_{N \ge N_0} \sup_{\D} \|P_N^{\perp} f\|_{\BMO_{\D}}
\le \varepsilon. 
\end{align}
Thus, it follows from \eqref{bfnn-1} and \eqref{bfnn-2} that for all $N \ge N_0$, 
\begin{align*}
\sup_{\D} \|P_N^{\perp} b\|_{\BMO_{\D}}
\le \sup_{\D} \|P_N^{\perp} (b-f)\|_{\BMO_{\D}}
+ \sup_{\D} \|P_N^{\perp} f\|_{\BMO_{\D}}
\le \varepsilon, 
\end{align*}
which shows the sufficiency. 

To prove the converse, let $b \in \BMO(\Rn)$ and $\lim\limits_{N \to \infty} \sup\limits_{\D} \|P_N^{\perp} b\|_{\BMO_{\D}} = 0$. In view of \eqref{ffbb-1}, it suffices to verify 
\begin{align}\label{QQN}
\lim_{N \to \infty} \sup_{J \notin \Q(N)} 
\bigg(\fint_J |b - \langle b \rangle_J|^2 \bigg)^{\frac12} = 0. 
\end{align}
By \cite[p. 792]{HP}, there exist $2^n$ dyadic grids $\mathscr{D}_k$, $k=1, \ldots, 2^n$, such that each cube $J \in \Q$ is contained in some dyadic cube $J_k \in \mathscr{D}_k$ with $\ell(J) \le \ell(J_k) \le 6 \ell(J)$. This leads to 
\begin{align*}
\bigg(\fint_J |b - \langle b \rangle_J|^2 \bigg)^{\frac12} 
\le 2 \bigg(\fint_J |b - \langle b \rangle_{J_k}|^2 \bigg)^{\frac12} 
\le 2 \times 6^{\frac{n}{2}} \bigg(\fint_{J_k} |b - \langle b \rangle_{J_k}|^2 \bigg)^{\frac12}. 
\end{align*}
Moreover, for any $N \ge 4$ and $J \notin \Q(8N)$, the case $\ell(J) < 2^{-8N}$ implies $\ell(J_k) \le 6 \ell(J) < 6 \cdot 2^{-8N} < 2^{-N}$; the case $\ell(J) > 2^{8N}$ gives $\ell(J_k) \ge \ell(J) > 2^{8N}$; and in the case $2^{-8N} \le \ell(J) \le 2^{8N}$ and $\rd(J, 2^{8N} \I) > 8N$, we have $\ell(J_k) \le 6 \ell(J) \le 6 \cdot 2^{8N}$, $\d(J, 2^{8N} \I) > 8N \cdot 2^{8N}$, and 
\begin{align*}
\rd(J_k, 2^N \I) 
\ge \frac{\d(J, 2^N \I) - \ell(J_k)}{\max\{\ell(J_k), 2^N\}} 
\ge \frac{\d(J, 2^{8N} \I)}{6 \cdot 2^{8N}} - 1 
\ge \frac43 N -1 > N. 
\end{align*}
Hence, \eqref{QQN} is deduced to showing that given a dyadic grid $\D$, 
\begin{align}\label{NQ}
\lim_{N \to \infty} \sup_{J \in \Xi_i(N)} 
\bigg(\fint_J |b - \langle b \rangle_J|^2 \bigg)^{\frac12} = 0, \quad i=1, 2, 3, 
\end{align}
where  
\begin{align*}
\Xi_1(N) & :=\{J \in \D: \ell(J) < 2^{-N}\},
\\
\Xi_2(N) & :=\{J \in \D: \ell(J) > 2^{8N}\},
\\
\Xi_3(N) & :=\{J \in \D: \rd(J, 2^N \I) > N\}. 
\end{align*}
Note that for all $J \in \Xi_1(N) \cup \Xi_3(N)$, there is no $I \in \D(N)$ satisfying $I \subset J$. This implies that for all $J \in \Xi_1(N) \cup \Xi_3(N)$, the inequality in \eqref{bJ-2} should be an identity, which together with \eqref{bJ-1} yields  
\begin{align}\label{NQ-1}
\fint_J |b - \langle b \rangle_J|^2 
= \fint_J |P_N^{\perp} b - \langle P_N^{\perp} b \rangle_J|^2
\lesssim \|P_N^{\perp} b\|_{\BMO_{\D}}^2, 
\quad\forall J \in \Xi_1(N) \cup \Xi_3(N).  
\end{align}
Let us turn to the case $J \in \Xi_2(N)$. By \eqref{bJ-2}, we have 
\begin{align}\label{Jbd-1}
\fint_J |b - \langle b \rangle_J|^2 
\le \frac{1}{|J|} \sum_{\substack{I \subset J \\ I \notin \D(N)}} 
\big|\big\langle b - \langle b \rangle_J, h_I \big\rangle \big|^2 
+ \frac{1}{|J|} \sum_{\substack{I \subset J \\ I \in \D(N)}} 
\big|\big\langle b - \langle b \rangle_J, h_I \big\rangle \big|^2. 
\end{align}
The equation \eqref{bJ-1} gives 
\begin{align}\label{Jbd-2}
\frac{1}{|J|} \sum_{\substack{I \subset J \\ I \notin \D(N)}} 
\big|\big\langle b - \langle b \rangle_J, h_I \big\rangle \big|^2 
= \fint_J |P_N^{\perp} b - \langle P_N^{\perp} b \rangle_J|^2
\lesssim \|P_N^{\perp} b\|_{\BMO_{\D}}^2,  
\end{align}
and \cite[eq. (5.27)]{CYY} implies 
\begin{align}\label{Jbd-3}
&\frac{1}{|J|} \sum_{\substack{I \subset J \\ I \in \D(N)}} 
\big|\big\langle b - \langle b \rangle_I, h_I \big\rangle \big|^2 
\le \sum_{\substack{I \subset J \\ I \in \D(N)}} \frac{|I|}{|J|}
\bigg(\fint_I |b - \langle b \rangle_I| \bigg)^2 
\\ \nonumber 
&\le \sum_{k=-N}^N \sum_{\substack{I \in \D(N) \\ \ell(I) = 2^k}} 2^{nk} |J|^{-1} \|b\|_{\BMO}^2
\lesssim \sum_{k=-N}^N 2^{nk} N^n 2^{(N-k)n} |J|^{-1} \|b\|_{\BMO}^2
\\ \nonumber
&\lesssim 2^{nN} N^{n+1} |J|^{-1} \|b\|_{\BMO}^2
\le 2^{nN} 2^{N(n+1)} 2^{-8nN} \|b\|_{\BMO}^2
\le 2^{-5nN} \|b\|_{\BMO}^2,  
\end{align}
for all $N \ge 1$. Gathering \eqref{Jbd-1}--\eqref{Jbd-3}, we obtain 
\begin{align}\label{NQ-2}
\fint_J |b - \langle b \rangle_J|^2 
\lesssim \|P_N^{\perp} b\|_{\BMO_{\D}}^2 + 2^{-5nN} \|b\|_{\BMO}^2, 
\quad \forall J \in \Xi_2(N). 
\end{align}
Consequently, \eqref{NQ} follows from \eqref{NQ-1}, \eqref{NQ-2}, and that $\lim\limits_{N \to \infty} \sup\limits_{\D} \|P_N^{\perp} b\|_{\BMO_{\D}} = 0$. 
\end{proof}

\begin{lemma}\label{lem:CC}
For any $f \in \mathscr{C}_c^{\infty}(\Rn)$, we have 
\begin{align*}
\lim_{N \to \infty} \sup_{\D} \|P_{\D(N)}^{\perp} f\|_{\BMO_{\D}} = 0.
\end{align*}
\end{lemma}

\begin{proof}
Let $\varepsilon>0$ be an arbitrary number and $f \in \mathscr{C}_c^{\infty}(\Rn)$. We may assume that $\supp(f) \subset 2^{N_1} \I$ for some $N_1 \in \N_+$. Since $f$ is uniformly continuous, there exists $N_2 \in \N_+$ such that $|f(x) - f(y)| \le \varepsilon$ whenever $|x-y| < 2^{-N_2}$. Define $f_0(x) = f(c_I)$ if $x \in I \in \D$ with $\ell(I) = 2^{-N_2}$. Then $\|f-f_0\|_{L^{\infty}} \le \varepsilon$, and \eqref{pnbb} gives 
\begin{align}\label{Pff-1}
\|P_{\D(N)}^{\perp} (f-f_0)\|_{\BMO_{\D}} 
\lesssim \|f-f_0\|_{\BMO_{\D}} 
\le 2 \|f-f_0\|_{L^{\infty}} 
\lesssim \varepsilon, 
\end{align} 
for all $N \ge 1$. 

Let $N \ge N_2$ and $I \in \D$ with $\ell(I) \le 2^{-N}$. Then there exists $I^* \in \D$ with $\ell(I^*) = 2^{-N}$ so that $I \subset I^*$. Thus, $f_0$ is constant on $I$, which leads to   
\begin{align}\label{Pff-2}
\langle f_0, h_I \rangle = 0. 
\end{align}

Let $N \ge N_1+1$ and $I \in \D$ with $2^{-N} \le \ell(I) \le 2^N$ and $\rd(I, 2^N \I) > N$. Then $\supp(f_0) \subset 2^{N_1+1} \I \subset 2^N \I$ and $\d(I, 2^N \I) > N 2^N$. Hence, $I \cap \supp(f_0) = \varnothing$ and 
\begin{align}\label{Pff-3}
\langle f_0, h_I \rangle = 0. 
\end{align}

Choose $N_0 := \max\{N_1 + 1, N_2\}$. Let $N \ge N_0$. Note that for any $k \ge N$, 
\begin{align*}
\#\big\{I \in \D: \ell(I) = 2^k, \langle f_0, h_I \rangle \neq 0 \big\} 
\le 2^n,  
\end{align*}
which implies 
\begin{align}\label{Pff-4}
\bigg\|\sum_{I \in \D: \ell(I) \ge 2^N} 
\langle f_0, h_I \rangle h_I \bigg\|_{L^{\infty}} 
\lesssim \sum_{k=N}^{\infty} 2^{n-k} \|f_0\|_{L^1} 
\lesssim 2^{-N} \|f\|_{L^{\infty}}. 
\end{align}
Now collecting \eqref{Pff-2}--\eqref{Pff-4}, we conclude that for any $N \ge N_0$ large enough, 
\begin{align*}
\bigg\|\sum_{I \notin \D(N)} 
\langle f_0, h_I \rangle h_I \bigg\|_{L^{\infty}} 
\lesssim 2^{-N} \|f\|_{L^{\infty}}
\le \varepsilon, 
\end{align*}
which along with \eqref{Pff-1} gives 
\begin{align*}
\|P_{\D(N)}^{\perp} f\|_{\BMO_{\D}} 
\lesssim \|P_{\D(N)}^{\perp} (f-f_0)\|_{\BMO_{\D}} 
+ \bigg\|\sum_{I \notin \D(N)} 
\langle f_0, h_I \rangle h_I \bigg\|_{L^{\infty}} 
\lesssim \varepsilon, 
\end{align*} 
where the implicit constants are independent of $\D$. This completes the proof. 
\end{proof}

\section{Weak boundedness property}\label{sec:A-WBP}
We prove the equivalence of the weak boundedness property between our dyadic and Hart's smooth versions \cite{Hart14} when the size condition is satisfied. For this purpose, given an integer $M \ge 0$, let 
\begin{align*}
\B_M 
:= \big\{\phi \in \mathscr{C}_c^{\infty}(\Rn): \supp \phi \subset B(0, 1), \,
\|\partial^{\alpha} \phi\|_{L^{\infty}} \le 1, \, |\alpha| \le M\big\},
\end{align*}
and $\phi^{a, t}(x) = \phi((x-a)/t)$. Suppose that $T, T^{*1}, T^{*2}: L_c^{\infty}(\Rn) \times L_c^{\infty}(\Rn) \to L_{\loc}^1(\Rn)$ satisfy 
\begin{align*}
\int_{\Rn} T(f_1, f_2)(x) f_3(x) \, dx 
=: \langle T(f_1, f_2), f_3 \rangle
= \langle T^{*1}(f_3, f_2), f_1 \rangle
= \langle T^{*2}(f_1, f_3), f_2 \rangle,
\end{align*}
for all $f_1, f_2, f_3 \in L_c^{\infty}(\Rn)$. 

Let $C_K$ be the smallest constant such that
\begin{align*} 
|K(x, y, z)| \le \frac{C_K}{(|x-y| + |x-z|)^{2n}}, \quad x \ne y \text{ or } x \ne z. 
\end{align*}
Let WBP denote the weak boundedness property in Definition \ref{def:WCP}, and $C_{\mathrm{WBP}}$ be the smallest constant. We say that $T$ satisfies H-WBP if 
\begin{align*}
|\langle T(\phi_1^{a_1, t}, \phi_2^{a_2, t}), \phi_3^{a_3, t} \rangle|
\le C \, t^n , \quad\forall \phi_i \in \B_M, a_i \in \Rn, t>0.
\end{align*}
Let $C_{\text{H-WBP}}^{M}$ be the smallest constant above.

\begin{theorem}\label{thm:WBP}
There holds
\begin{align*}
C_K + C_{\rm{WBP}} 
\simeq C_K + C_{\rm H \text{-} WBP}^1 
\simeq C_K + C_{\rm H \text{-} WBP}^M, \quad \forall M \ge 1.
\end{align*} 
In particular, if $T: L_c^{\infty}(\Rn) \times L_c^{\infty}(\Rn) \to L_{\loc}^1(\Rn)$ is a bilinear operator associated with a standard bilinear Calder\'{o}n--Zygmund kernel, then $T$ can be extended to a bounded operator from $L^{p_1}(\Rn) \times L^{p_2}(\Rn)$ to $L^p(\Rn)$ for all $\frac1p = \frac{1}{p_1} + \frac{1}{p_2}$ with $p_1, p_2 \in (1, \infty)$ if and only if $T$ satisfies the WBP and $T(1, 1), T^{*1}(1, 1), T^{*2}(1, 1) \in \BMO(\Rn)$.
\end{theorem}

\begin{proof}
It suffices to show the equivalence between constants since the boundedness of $T$ follows from \cite[Theorem 1.1]{Hart14}. Before starting the proof, we present two useful estimates:
\begin{align}\label{eq:jjf}
\lim_{j \to \infty} \langle T(f_{1, j}, f_{2, j}), f_{3, j} \rangle 
= \langle T(f_1, f_2), f_3 \rangle,
\end{align}
where $f_{i, j} \to f_i$ a.e., $\sup_{i, j} \|f_{i, j}\|_{L^{\infty}} \le C_0 < \infty$, $\bigcup_{i, j} (\supp f_i \cup \supp f_{i, j}) \subset E$, and $E$ is a compact set; and
\begin{align}\label{eq:uuu}
|\langle T(h_1, h_2), h_3 \rangle| 
\lesssim (C_K + C_{\rm WBP})  (k+1) |Q| \prod_{i=1}^3 \|h_i\|_{L^{\infty}},
\end{align}
where $h_i = \sum_{Q' \in \D_k(Q)} \alpha_{i, Q'} \mathbf{1}_{Q'}$ and $\alpha_{i, Q'} \in \R$. 

To show \eqref{eq:jjf}, note that for all $u_j \rightharpoonup u$ in $L^1(E)$, $v_j \to v$ a.e., and $\sup_j \|v_j\|_{L^{\infty}} \le C_0$, letting $\varepsilon > 0$ and $E_j^{\varepsilon} := \{x \in E: |v_j(x) - v(x)| > \varepsilon\}$, we have
\begin{align*}
\bigg|\int_E u_j v_j - \int_E u v \bigg|
& \le \bigg|\int_E (u_j - u) v \bigg| 
+ \int_{\Rn \setminus E_j^{\varepsilon}} |u_j| |v_j - v|
+ \int_{E_j^{\varepsilon}} |u_j| |v_j - v| 
\\
& \le \bigg|\int_E (u_j - u) v \bigg| 
+ \varepsilon \|u_j\|_{L^1} 
+ 2C_0 \|u_j\|_{L^1(E_j^{\varepsilon})}.
\end{align*}
Since $\sup_j \|u_j\|_{L^1(E)} < \infty$ and $\lim_{\delta \to 0} \sup_j \sup_{E' \subset E: |E'| \le \delta} \|u_j\|_{L^1(E')} = 0$ (see \cite[Chapter IV]{DS}), the above gives
\begin{align*}
\limsup_{j \to \infty} \bigg|\int_E u_j v_j - \int_E u v \bigg|
\le \varepsilon \sup_j \|u_j\|_{L^1}.
\end{align*}
By the arbitrariness of $\varepsilon$, we obtain 
\begin{align}\label{eq:euu}
\lim_{j \to \infty} \int_E u_j v_j = \int_E u v.
\end{align}
For $f_1 \in L_c^{\infty}(\Rn)$ with $\supp f_1 \subset E$, and $h \in L^{\infty}(E)$, by Lebesgue dominated convergence theorem, 
\begin{align*}
\int_E T^{*2}(f_1, f_{3, j}) \, h 
= \int_E T(f_1, h) \, f_{3, j}
\to \int_E T(f_1, h) \, f_3 
= \int_E T^{*2}(f_1, f_3) \, h, 
\end{align*}
which means $T^{*2}(f_1, f_{3, j})|_E \rightharpoonup T^{*2}(f_1, f_3)|_E$ in $L^1(E)$. Thus, by \eqref{eq:euu},
\begin{align*}
\langle T(f_1, f_{2, j}), f_{3, j} \rangle 
= \int_E T^{*2}(f_1, f_{3, j}) \, f_{2, j} 
\to \int_E T^{*2}(f_1, f_3) \, f_2
= \langle T(f_1, f_2), f_3 \rangle,
\end{align*}
which gives for every $h \in L^{\infty}(E)$,
\begin{align*}
\int_E T^{*1}(f_{3, j}, f_{2, j}) \, h 
= \langle T(h, f_{2, j}), f_{3, j} \rangle 
\to \langle T(h, f_2), f_3 \rangle
= \int_E T^{*1}(f_3, f_2) \, h. 
\end{align*}
This means $T^{*1}(f_{3, j}, f_{2, j})|_E \rightharpoonup T^{*1}(f_3, f_2)|_E$ in $L^1(E)$. Applying \eqref{eq:euu} again yields
\begin{align*}
\lim_{j \to \infty} \langle T(f_{1, j}, f_{2, j}), f_{3, j} \rangle 
= \lim_{j \to \infty} \int_E T^{*1}(f_{3, j}, f_{2, j}) \, f_{1, j} 
= \int_E T^{*1}(f_3, f_2) \, f_1
= \langle T(f_1, f_2), f_3 \rangle,
\end{align*}
which justifies \eqref{eq:jjf}. 

In order to prove \eqref{eq:uuu}, observe that for every $0 < d < \ell(Q)$,
\begin{align}\label{eq:logd}
\iint_{\substack{y, z \in Q \\ \max\{|x-y|, |x-z| \ge d\}}}
\frac{dy \, dz}{(|x-y| + |x-z|)^{2n}}
\lesssim 1 + \log(\ell(Q)/d),
\end{align}
which implies 
\begin{align*}
& \iiint_{Q^3 \setminus \bigcup_{P \in \D_k(Q)} P^3} 
\frac{dx \, dy \, dz}{(|x-y| + |x-z|)^{2n}}
\le \sum_{P \in \D_k(Q)} \int_P \iint_{Q^2 \setminus P^2} 
\frac{dy \, dz \, dx}{(|x-y| + |x-z|)^{2n}}
\\
& \le \sum_{P \in \D_k(Q)} \int_P 
\iint_{\substack{y, z \in Q \\ \max\{|x-y|, |x-z| \ge \dist(x, \partial P)\}}}
\frac{dy \, dz \, dx}{(|x-y| + |x-z|)^{2n}}
\\
& \lesssim \sum_{P \in \D_k(Q)} \int_P 
\bigg[1 + \log \frac{\ell(Q)}{\ell(P)} \frac{\ell(P)}{\dist(x, \partial P)} \bigg] \, dx
\lesssim (k+1) |Q|.
\end{align*}
In addition, for $P_1 \ne P_2$ or $P_1 \ne P_3$, one can take $g_{i, j} \in \mathscr{C}_c^{\infty}(\Rn)$ such that $\supp g_{i, j} \subset P_i$, $0 \le g_{i, j} \le 1$, and $g_{i, j} \to \mathbf{1}_{P_i}$ a.e. Thus, by \eqref{eq:jjf},
\begin{align*}
\langle T(\mathbf{1}_{P_1}, \mathbf{1}_{P_2}), \mathbf{1}_{P_3} \rangle
& = \lim_{j \to \infty} \langle T(g_{1, j}, g_{2, j}), g_{3, j} \rangle
\\
& = \lim_{j \to \infty} \iiint T(g_{1, j}, g_{2, j}) g_{3, j}
= \int_{P_3} \int_{P_2} \int_{P_1} K(x, y, z) \, dy \, dz \, dx,
\end{align*}
which together with WBP and the size condition of $K$ gives
\begin{align*}
|\langle T(h_1, h_2), h_3 \rangle| 
& \le \sum_{P \in \D_k(Q)} |\alpha_{1, P}| |\alpha_{2, P}| |\alpha_{3, P}|
|\langle T(\mathbf{1}_P, \mathbf{1}_P), \mathbf{1}_P \rangle| 
\\
&\quad + \sum_{\substack{P_1, P_2, P_3 \in \D_k(Q) \\ P_1 \ne P_2 \text{ or } P_1 \ne P_3}} 
|\alpha_{1, P_1}| |\alpha_{2, P_2}| |\alpha_{3, P_3}|
|\langle T(\mathbf{1}_{P_1}, \mathbf{1}_{P_2}), \mathbf{1}_{P_3} \rangle| 
\\
& \lesssim (C_K + C_{\rm WBP})  (k+1) |Q| \prod_{i=1}^3 \|h_i\|_{L^{\infty}},
\end{align*}
This verifies \eqref{eq:uuu}.

Let $M \ge 1$. Note that $\B_M \subset \B_1$, and hence, $C_{\rm H \text{-} WBP}^M \le C_{\rm H \text{-} WBP}^1$. Thus, the desired equivalence will be proved once the following two steps are completed. 

\noindent$\bullet$ {\bf Step I: $C_{\rm H \text{-} WBP}^1 \lesssim C_K + C_{\rm WBP}$.} 
Let $Q \in \D$ and $\varphi_i \in \mathscr{C}_c^{\infty}(\Rn)$ with $\supp \varphi_i \subset Q$, $\|\varphi_i\|_{L^{\infty}} \le 1$, and $\|\nabla \varphi_i\|_{L^{\infty}} \le \ell(Q)^{-1}$. Then the operator 
\begin{align*}
E_k^Q \varphi_i := \sum_{Q' \in \D_k(Q)} \langle \varphi_i \rangle_{Q'} \, \mathbf{1}_{Q'}
\end{align*}
satisfies 
\begin{align*}
\|E_k^Q \varphi_i\|_{L^{\infty}} \lesssim 1, \quad
\|E_k^Q \varphi_i - \varphi_i\|_{L^{\infty}} \lesssim 2^{-k}, \quad
\| (E_k^Q - E_{k-1}^Q) \varphi_i\|_{L^{\infty}} \lesssim 2^{-k}, \quad \forall k \ge 1.
\end{align*}
For every $N \ge 1$, we use these estimates, WBP, and \eqref{eq:uuu} to obtain
\begin{align*}
& |\langle T(E_N^Q \varphi_1, E_N^Q \varphi_2), E_N^Q \varphi_3 \rangle|
\\
& \le |\langle T(\mathbf{1}_Q, \mathbf{1}_Q), \mathbf{1}_Q \rangle| 
|\langle \varphi_1 \rangle_Q| |\langle \varphi_2 \rangle_Q| |\langle \varphi_3 \rangle_Q|
\\
& \quad + \sum_{k=1}^N \big[ 
\big| \big\langle T \big( (E_k^Q - E_{k-1}^Q)  \varphi_1, E_k^Q \varphi_2 \big), E_k^Q \varphi_3 \big\rangle \big|
\\
&\qquad \qquad + \big| \big\langle T \big(E_{k-1}^Q  \varphi_1, (E_k^Q - E_{k-1}^Q) \phi_2 \big), E_k^Q \varphi_3 \big\rangle \big|
\\
&\qquad \qquad + \big| \big\langle T \big( E_{k-1}^Q  \varphi_1, E_{k-1}^Q \varphi_2 \big), (E_k^Q - E_{k-1}^Q) \varphi_3 \big\rangle \big| \big]
\\
& \lesssim C_{\rm WBP} |Q| 
+ (C_K + C_{\rm WBP}) |Q| \sum_{k=1}^N (k+1) 2^{-k}
\\
& \lesssim (C_K + C_{\rm WBP}) |Q|.
\end{align*}
This along with \eqref{eq:jjf} implies 
\begin{align}\label{eq:vvp}
|\langle T(\varphi_1, \varphi_2), \varphi_3 \rangle|
\lesssim (C_K + C_{\rm WBP}) |Q|.
\end{align}

In order to estimate $C_{\rm H \text{-} WBP}^1$, let $\phi_i \in \B_1$, $a_i \in \Rn$, and $t>0$. If $\max\{|a_i - a_j|: 1 \le i, j \le 3\} \le 4 t$, we choose some dyadic cube $Q$ such that $\bigcup_{i=1}^3 \supp \phi_i^{a_i, t} \subset Q$ and $\ell(Q) \simeq t$. This allows us to invoke \eqref{eq:vvp} and obtain
\begin{align*}
|\langle T(\phi_1^{a_1, t}, \phi_2^{a_2, t}), \phi_3^{a_3, t} \rangle|
\lesssim (C_K + C_{\rm WBP}) |Q|
\simeq (C_K + C_{\rm WBP}) t^n.
\end{align*}
If $\max_{i, j} |a_i - a_j| > 4 t$, we have for all $(y, z, x) \in \supp \phi_1^{a_1, t} \times \supp \phi_2^{a_2, t} \times \supp \phi_3^{a_3, t}$,
\begin{align*}
4t < \max_{i, j} |a_i - a_j|
\le 2t + \max\{|x-y|, |x-z|, |y-z|\}
\le 2t + |x-y| + |x-z|,
\end{align*}
which along with the size condition yields 
\begin{align*}
|\langle T(\phi_1^{a_1, t}, \phi_2^{a_2, t}), \phi_3^{a_3, t} \rangle|
\lesssim C_K t^{-2n} \prod_{i=1}^3 \|\phi_i^{a_i, t}\|_{L^1}
\lesssim C_K t^n.
\end{align*}

\noindent$\bullet$ {\bf Step II: $C_{\rm WBP} \lesssim C_K + C_{\rm H \text{-} WBP}^M$.} 
Fix a dyadic grid $\D$ and a cube $Q \in \D$. Let $\W(Q)$ be the collection of maximal dyadic subcubes $Q' \subset Q$ such that $10 Q' \subset Q$. Then $\W(Q)$ is pairwise disjoint, 
\begin{align*}
& \sum_{Q' \in \W(Q)} |Q'| = |Q|, \quad 
4 \ell(Q') \le \dist(x, \partial Q) \le 12 \ell(Q'), \quad x \in 2Q'
\\
& \sum_{Q' \in \W(Q)} \mathbf{1}_{2Q'} \lesssim 1, \quad \text{ and } \,
\ell(Q_1) \simeq \ell(Q_2) \text{ if } 2Q_1 \cap 2Q_2 \ne \varnothing.
\end{align*}
Given $P \in \W(Q)$, choose $\widetilde{\psi}_P \in \mathscr{C}_c^{\infty}(\Rn)$ such that $\mathbf{1}_P \le \widetilde{\psi}_P \le \mathbf{1}_{2P}$ and $|\partial^{\alpha} \widetilde{\psi}_P| \lesssim \ell(P)^{-|\alpha|}$. Then define
\begin{align*}
\psi_P := \frac{\widetilde{\psi}_P}{\sum_{Q' \in \W(Q)} \widetilde{\psi}_{Q'}} \mathbf{1}_Q,
\end{align*}
which satisfies $\psi_P \in \mathscr{C}_c^{\infty}(\Rn)$, $\supp \psi_P \subset 2P$, 
\begin{align}\label{eq:psp}
0 \le \psi_P \le 1, \quad 
|\partial^{\alpha} \psi_P| \lesssim \ell(P)^{-|\alpha|}, \quad \text{and} \quad
\sum_{P \in \W(Q)} \psi_P = \mathbf{1}_Q \, \text{ a.e.}
\end{align}
If we set $\mathcal{N}(P) := \{Q' \in \W(Q): \dist(2Q', 2P) < \ell(P)\}$, then $\ell(Q') \simeq \ell(P)$ for all $Q' \in \mathcal{N}(P)$, and $\# \mathcal{N}(P) \lesssim 1$. Pick an increasing sequence of finite subcollections $\W_N(Q) \subset \W(Q)$ such that $\bigcup_{N \ge 1} \W_N(Q) = \W(Q)$. Denote $\Psi_N := \sum_{P \in \W_N(Q)} \psi_P$, which verifies $\Psi_N \in \mathscr{C}_c^{\infty}(\Rn)$, $\supp \Psi_N \subset Q$, $0 \le \Psi_N \le 1$, and $\lim_{N \to \infty} \Psi_N = \mathbf{1}_Q$ a.e. In view of \eqref{eq:psp}, we use H-WBP to obtain 
\begin{align*}
\mathscr{J}_1 
:= \sum_{\substack{P_1, P_2, P_3 \in \W_N(Q) \\ P_1, P_2 \in \mathcal{N}(P_3)}} 
|\langle T(\psi_{P_1}, \psi_{P_2}), \psi_{P_3} \rangle|
\lesssim \sum_{\substack{P_1, P_2, P_3 \in \W_N(Q) \\ P_1, P_2 \in \mathcal{N}(P_3)}} 
C_{\rm H \text{-} WBP}^M \, |P_3|
\lesssim C_{\rm H \text{-} WBP}^M |Q|.
\end{align*}
If $P_1 \notin \mathcal{N}(P_3)$ or $P_2 \notin \mathcal{N}(P_3)$, then $\max\{|x-y|, |x-z|\} \ge \ell(P_3) \ge \dist(x, \partial Q)/12$ for all $(y, z, x) \in (2P_1, 2P_2, 2P_3)$, which together with the size condition and \eqref{eq:logd} implies
\begin{align*}
\mathscr{J}_2
& := \sum_{\substack{P_1, P_2, P_3 \in \W_N(Q) \\ P_1 \notin \mathcal{N}(P_3) \text{ or } P_2 \notin \mathcal{N}(P_3)}} 
|\langle T(\psi_{P_1}, \psi_{P_2}), \psi_{P_3} \rangle|
\\
& \lesssim C_K \int_Q \iint_{\substack{y, z \in Q \\ \max\{|x-y|, |x-z|\} \ge \dist(x, \partial Q)/12}}
\frac{dy \, dz \, dx}{(|x-y| + |x-z|)^{2n}}
\\
& \lesssim C_K \int_Q \bigg(1 + \log\frac{\ell(Q)}{\dist(x, \partial Q)} \bigg) \, dx
\lesssim C_K |Q|.
\end{align*}
Consequently, 
\begin{align*}
|\langle T(\Psi_N, \Psi_N), \Psi_N \rangle|
\le \mathscr{J}_1 + \mathscr{J}_2
\lesssim (C_K + C_{\rm H \text{-} WBP}^M) |Q|, \quad \forall N \ge 1, 
\end{align*}
which along with \eqref{eq:jjf} gives
\begin{align*}
|\langle T(\mathbf{1}_Q, \mathbf{1}_Q), \mathbf{1}_Q \rangle|
= \lim_{N \to \infty} |\langle T(\Psi_N, \Psi_N), \Psi_N \rangle|
\lesssim (C_K + C_{\rm H \text{-} WBP}^M) |Q|.
\end{align*}
This shows $C_{\rm WBP} \lesssim C_K + C_{\rm H \text{-} WBP}^M$.
\end{proof}

\section{Some properties of Lorentz spaces}\label{sec:A-Lorentz}
We collect a number of useful properties in Lorentz spaces. Let $(\Sigma, \mu)$ be a $\sigma$-finite measure space, and   $\mathcal{M}(\mu)$ denote the family of all measurable functions in $(\Sigma, \mu)$. Then $\mathcal{M}(\mu)$ is  equipped with the topology of local convergence in measure.

\begin{definition}
Given a quasi-Banach space $\mathbb{X}$ over $(\Sigma, \mu)$, a function $f \in \mathbb{X}$  is said to \emph{have absolutely continuous quasi-norm} if $\|f \mathbf{1}_{E_k}\|_{\mathbb{X}} \to 0$ for every sequence $\{E_k\}_{k=1}^{\infty}$ such that $E_k \to \varnothing$ $\mu$-a.e. 

Let $\mathbb{X}_a$ denote the collection of all functions in $\mathbb{X}$ which have absolutely continuous quasi-norm. If $\mathbb{X}_a = \mathbb{X}$, then the space $\mathbb{X}$ is said to \emph{have absolutely continuous quasi-norm}.

The set $K \subset \mathbb{X}_a$ is said to \emph{have uniformly absolutely continuous quasi-norm} in $\mathbb{X}$, denoted by $K \subset \mathrm{UAC}(\mathbb{X})$, if $\sup_{f \in K} \|f \mathbf{1}_{E_k}\|_{\mathbb{X}} \to 0$ for every sequence $\{E_k\}_{k=1}^{\infty}$ such that $E_k \to \varnothing$ $\mu$-a.e. 
\end{definition}

The following result gives a criterion of precompactness quasi-Banach function spaces (cf. \cite[Theorem 3.17]{CGO}), which improves \cite[p. 31, Exercise 8]{BS} in the context of Banach function spaces.

\begin{lemma}[{\cite{BS, CGO}}]\label{lem:UAC}
Let $\mathbb{X}$ be a quasi-Banach function space and $K \subset \mathbb{X}_a$. Then $K$ is precompact in $\mathbb{X}$ if and only if it is precompact in $\mathcal{M}(\mu)$ and $K \subset \mathrm{UAC}(\mathbb{X})$. 
\end{lemma}

\begin{lemma}\label{lem:ACN}
For any $0<p, q < \infty$, $L^{p, q}(\Rn)$ has absolutely continuous quasi-norm, but $L^{p, \infty}(\Rn)$ does not have absolutely continuous quasi-norm. 
\end{lemma}

\begin{proof}
Let $0<p, q < \infty$ and $\{E_k\}_{k=1}^{\infty}$ be a sequence of measurable sets satisfying $E_k \to \varnothing$ almost everywhere. Let $\varepsilon>0$ and $f \in L^{p, q}(\Rn)$. For any $t>0$, define 
\begin{align*}
F(t) := t^{1-\frac1q} \, d_f(t)^{\frac1p} 
\quad\text{ and }\quad 
F_k(t) := t^{1-\frac1q} \, d_{f \mathbf{1}_{E_k}}(t)^{\frac1p},   
\end{align*}
where $d_f(t) := |\{x \in \Rn: |f(x)|>t\}|$. Note that $d_{f \mathbf{1}_{E_k}}(t) \le \int_{\{|f| > t\}} \mathbf{1}_{E_k}(x) \, dx \to 0$, then 
\begin{align*}
\lim_{k \to \infty} F_k(t) = 0, \quad 
0 \le F_k(t) \le F(t), \quad\text{ and }\quad 
F \in L^q(\R_+). 
\end{align*}
Hence, it follows from Lebesgue dominated convergence theorem that 
\begin{align*}
\lim_{k \to \infty} \|f \mathbf{1}_{E_k}\|_{L^{p, q}} 
= \lim_{k \to \infty} p^{\frac1p} \|F_k\|_{L^q} 
= 0. 
\end{align*}
This shows every function in $L^{p, q}(\Rn)$ has absolutely continuous quasi-norm. 

To continue, let $f(x) = |x|^{- \frac{n}{p}}$ and $E_k=B(0, k^{-1})$ for each $k \ge 1$. Then for any $s>0$, 
\begin{align*}
d_f(s) 
= |\{x \in \Rn: |x|^{-\frac{n}{p}} > s\}|
= |B(0, s^{-\frac{p}{n}})| 
= s^{-p} \nu_n, 
\end{align*}
and 
\begin{align*}
d_{f \mathbf{1}_{E_k}}(s) 
= |\{x \in B(0, k^{-1}): |x|^{-\frac{n}{p}} > s\}|  
= k^{-n} \nu_n \mathbf{1}_{\{0<s<k^{\frac{n}{p}}\}} 
+ s^{-p} \nu_n \mathbf{1}_{\{s \ge k^{\frac{n}{p}}\}}, 
\end{align*}
where $\nu_n$ is the volume of the unit ball in $\Rn$. Thus, 
\begin{align*}
\|f \mathbf{1}_{E_k}\|_{L^{p, \infty}} 
= \sup_{s>0} s \, d_{f \mathbf{1}_{E_k}}(s)^{\frac1p} 
= \nu_n^{\frac1p} 
= \|f\|_{L^{p, \infty}}. 
\end{align*}
Since $\lim_{k \to \infty} |E_k| = 0$, this asserts that $L^{p, \infty}(\Rn)$ does not have absolutely continuous quasi-norm. 
\end{proof}

Let $\mu$ be a Borel measure on $\Rn$. A \emph{$\mu$-simple function} is a finite linear combination of characteristic functions of measurable subsets in $\Rn$, where  these subsets may have infinite $\mu$-measure. A \emph{finitely $\mu$-simple function} has the form $f = \sum_{j=1}^N a_j \mathbf{1}_{B_j}$, where $1 \le N < \infty$, $a_j \in \R$, and $B_j$ are pairwise disjoint measurable sets with $\mu(B_j) < \infty$. When $\mu$ is the Lebesgue measure, we shall suppress the presence of $\mu$.

\begin{lemma}\label{lem:dense-Lpq}
Let $w$ be a weight on $\Rn$ such that $w \in L^1_{\loc}(\Rn)$. Then $\mathscr{C}_c^{\infty}(\Rn)$ is dense in $L^{p, q}(w)$ for all $0<p, q<\infty$, but, in general, neither the class of continuous functions nor the class of compactly supported functions is dense in $L^{p, \infty}(w)$ for all $p \in (0, \infty)$. 
\end{lemma}

\begin{proof}
Let $0<p, q<\infty$. The density of $\mathscr{C}_c^{\infty}(\Rn)$ in $L^{p, q}(w)$ is a consequence of the following three facts: 
\begin{align}\label{den-1} 
\text{finitely $w$-simple functions are dense in $L^{p, q}(w)$}, 
\end{align}
\begin{equation}\label{den-2}
\begin{aligned} 
&\text{every finitely $w$-simple function can be approximated}
\\ 
&\text{in $L^{p, q}(w)$ by simple functions with compact supports}, 
\end{aligned}
\end{equation}
\begin{equation}\label{den-3}
\begin{aligned} 
&\text{every characteristic function $\mathbf{1}_E$ of a bounded measurable}
\\ 
&\text{set can be approximated in $L^{p, q}(w)$ by $\mathscr{C}_c^{\infty}(\Rn)$ functions}. 
\end{aligned}
\end{equation}
The statement \eqref{den-1} was given by \cite[Theorem 1.4.13]{Gra1}. To show \eqref{den-2}, it suffices to consider  $f=\mathbf{1}_E$, where $E \subset \Rn$ is a measurable set with $w(E)<\infty$. Let $f_N := \mathbf{1}_{E \cap B(0, N)}$ for each $N \ge 1$. Then $\supp (f_N) \subset B(0, N)$, $|E \cap B(0, N)| \le |B(0, N)| < \infty$, and for any $s>0$, 
\begin{align*}
d_{f_N - f}(s) 
&:= w(\{x \in \Rn: |f_N(x) - f(x)| > s\})
\\
&= w(\{x \in E \backslash B(0, N): \mathbf{1}_E(x) > s\})
\\
&= w(E \backslash B(0, N)) \, \mathbf{1}_{\{0<s<1\}}, 
\end{align*}
which in turn gives 
\begin{align*}
\|f_N - f\|_{L^{p, q}(w)}
&= p^{\frac1q} \bigg(\int_0^{\infty} 
\big[s \, d_{f_N-f}(s)^{\frac1p} \big]^q \, \frac{ds}{s} \bigg)^{\frac1q} 
\\
&= p^{\frac1q} \bigg(\int_0^1 
w(E \backslash B(0, N))^{\frac{q}{p}} s^{q-1} \, ds \bigg)^{\frac1q}
\\
&= p^{\frac1q} q^{-\frac1q} w(E \backslash B(0, N))^{\frac1p}
\to 0, \quad\text{ as } N \to \infty. 
\end{align*}
Thus, \eqref{den-2} holds. 

To prove \eqref{den-3}, let $E$ be a bounded measurable set such that $E \subset B(0, A)$ for some $A \ge 1$. Write  $f=\mathbf{1}_E$. Choose a positive function $\varphi \in \mathscr{C}_c^{\infty}(\Rn)$ so that $\supp \varphi \subset B(0, 1)$ and $\int_{\Rn} \varphi \, dx=1$. For any $t \in (0, 1)$, denote $\varphi_t(x) := t^{-n} \varphi(t^{-1} x)$ and $f_t := \varphi_t*f$.  It is easy to check that 
\begin{equation}\label{fft-1}
\begin{aligned}
& f_t \in \mathscr{C}_c^{\infty}(\Rn), \quad 
\supp f_t \subset B(0, 2A), 
\\ 
&\lim_{t \to 0} |f_t(x) - f(x)| =0, \, \, \text{a.e. } x \in \Rn,  
\\
& \sup_{0 < t < 1} |f_t(x)| \le \mathbf{1}_{B(0, 2A)}(x), \, \, x \in \Rn, 
\\
&\text{and}\quad  
\mathbf{1}_{B(0, 2A)} \in L^1(w),  
\end{aligned}
\end{equation}  
provided that $w \in L^1_{\loc}(\Rn)$. Then by Lebesgue dominated convergence theorem, 
\begin{align}\label{fft-12}
\lim_{t \to 0} \|f_t-f\|_{L^1(w)} = 0.
\end{align}
Since convergence in measure is weaker than convergence in either $L^r(w)$ of $L^{r, \infty}(w)$ for any $r \in (0, \infty]$, the estimate \eqref{fft-12} enables us to choose a decreasing sequence $\{t_k\}_{k=1}^{\infty} \subset (0, 1)$ such that 
\begin{align}\label{fft-2}
\lim_{k \to \infty} t_k = 0
\quad\text{ and }\quad 
d_{f_{t_k} -f}(k^{-1}) 
\le 2^{-k}. 
\end{align}
Moreover, by \eqref{fft-1} and \eqref{fft-2}, 
\begin{align*}
d_{f_{t_k} -f}(s) 
&\le w(B(0, 2A)), \quad \,  \, 0<s<1, 
\\
d_{f_{t_k} -f}(s) 
&\le 2^{-k}, \qquad\qquad 1/k \le s \le 1, 
\end{align*}
which implies 
\begin{align*}
\|f_{t_k} - f\|_{L^{p, q}(w)}^q  
&= p \int_0^{\infty} \big[s \, d_{f_{t_k} - f}(s)^{\frac1p} \big]^q \, \frac{ds}{s} 
\\
&\lesssim \int_0^{\frac1k} w(B(0, 2A))^{\frac{q}{p}} s^{q-1} \, ds 
+ \int_{\frac1k}^1 2^{-k \frac{q}{p}} s^{q-1} \, ds
\\
&\lesssim w(B(0, 2A))^{\frac{q}{p}} k^{-q} + 2^{-k \frac{q}{p}}
\to 0, \quad\text{ as } k \to \infty. 
\end{align*}
This completes the proof of \eqref{den-3}. 

Finally, to demonstrate the nondensity,  choose $w \equiv 1$ and $f := |x|^{-\frac{n}{p}} \in L^{p, \infty}(\Rn)$, let $\psi \in \mathscr{C}(\Rn)$ and $\phi$ be a function with $\supp \phi \subset B(0, A)$ for some $A>0$. By the continuity of $\psi$, there exists $\delta>0$ such that 
\begin{align*}
\big| |\psi(x)| - |\psi(0)| \big|
\le |\psi(x) - \psi(0)| 
\le 1, 
\quad\text{ for all } x \in B(0, \delta). 
\end{align*}
Set $s_0 := \max\{\delta^{-\frac{n}{p}}, |\psi(0)| + 1\}$. Then for any $s>s_0$ and $x \in B(0, s^{-\frac{p}{n}})$, 
\begin{align*}
f(x) \ge s > s_0 
\ge |\psi(0)| + 1 
\ge |\psi(x)|. 
\end{align*}
Thus, $f(x) - |\psi(x)| > s-s_0 =: t$ and 
\begin{align*}
t \, d_{f - \psi}(t)^{\frac1p} 
&:= t |\{x \in \Rn: |f(x) - \psi(x)| > t\}|^{\frac1p}
\\
&\, \, \ge t |\{x \in B(0, s^{-\frac{p}{n}}): \big|f(x) - |\psi(x)|\big| > t\}|^{\frac1p}
\\
&= t |B(0, s^{-\frac{p}{n}})|^{\frac1p} 
= t s^{-1} \nu_n^{\frac1p}, 
\end{align*}
which implies 
\begin{align}\label{nonden-1}
\|f- \psi\|_{L^{p, \infty}}
\ge \sup_{t>0} \frac{t}{t+s_0} \nu_n^{\frac1p} 
= \nu_n^{\frac1p}. 
\end{align}
On the other hand, for any $s>0$, 
\begin{align*}
d_f(s) = \nu_n s^{-p}
\quad\text{ and }\quad 
d_{\phi}(s) \le |B(0, A)| = \nu_n A^n, 
\end{align*}
which implies 
\begin{align*}
d_{f- \phi}(s) 
\ge d_f(2s) - d_{\phi}(s)
\ge s^{-p} \nu_n (2^{-p} - s^p A^n), 
\end{align*}
and hence, 
\begin{align}\label{nonden-2}
\|f - \phi\|_{L^{p, \infty}} 
\ge \sup_{0<s<A^{-\frac{n}{p}}/2} \nu_n^{\frac1p} (2^{-p} - s^p A^n)^{\frac1p} 
= \nu_n^{\frac1p}/2. 
\end{align}
As a consequence of \eqref{nonden-1} and \eqref{nonden-2}, we have already proved that there exist $\varepsilon_0>0$ and $f \in L^{p, \infty}(\Rn)$ so that $\|f - \psi\|_{L^{p, \infty}} \ge \varepsilon_0$ for all $\psi \in \mathscr{C}(\Rn)$, and $\|f - \phi\|_{L^{p, \infty}} \ge \varepsilon_0$ for all functions $\phi$ with compact support. This means that the collection of continuous functions and functions with compact supports is not dense in $L^{p, \infty}(\Rn)$. 
\end{proof}

Recall that $\tau_h f(x) := f(x-h)$ is the translation of a function $f$ on $\Rn$ by $h \in \Rn$. A bilinear operator $T$ from $\mathcal{S}(\Rn) \times \mathcal{S}(\Rn) \to \mathcal{S}'(\Rn)$ is called \emph{translation invariant}, if 
\begin{align*}
\tau_h \big(T(f_1, f_2)\big)
= T(\tau_h f_1, \tau_h f_2) 
\end{align*}
for all $f_1, f_2 \in \mathcal{S}(\Rn)$ and all $h \in \Rn$. 

The following is a bilinear extension of \cite[Proposition A.1]{BLOT2}. 

\begin{lemma}\label{lem:trans}
Let $0<p, q, p_1, q_1, p_2, q_2 < \infty$. A nonzero translation invariant bilinear operator $T$ is not compact from $L^{p_1, q_1}(\Rn) \times L^{p_2, q_2}(\Rn)$ to $L^{p, q}(\Rn)$.
\end{lemma}

\begin{proof}
Since the compactness is stronger than the boundedness, we may assume that $T$ is bounded from $L^{p_1, q_1}(\Rn) \times L^{p_2, q_2}(\Rn)$ to $L^{p, q}(\Rn)$. Pick $f, g \in L_c^{\infty}(\Rn)$ such that $\|f\|_{L^{p_1, q_1}} = \|g\|_{L^{p_2, q_2}} = 1$ and $0<\|T(f, g)\|_{L^{p, q}} <\infty$. Let $\{h_j\}_{j=1}^{\infty}$ be a sequence in $\Rn$ satisfying $\lim_{j \to \infty} |h_j| = \infty$. Then denote $f_j := \tau_{h_j} f$ and $g_j := \tau_{h_j} g$ for every $j \ge 1$. 

Assume that $T$ is compact from $L^{p_1, q_1}(\Rn) \times L^{p_2, q_2}(\Rn)$ to $L^{p, q}(\Rn)$. Then by Theorem \ref{thm:KRLpq}, given $\varepsilon>0$, there exists $A=A(\varepsilon)>0$ so that 
\begin{align*}
\|T(f_j, g_j) \mathbf{1}_{B(0, A)^c}\|_{L^{p, q}}
\le \varepsilon \|f_j\|_{L^{p_1, q_1}} \|g_j\|_{L^{p_2, q_2}}
= \varepsilon \|f\|_{L^{p_1, q_1}} \|g\|_{L^{p_2, q_2}}
= \varepsilon,  
\end{align*}
for all $j \ge 1$. Choose $0<\varepsilon<\frac12 \|T(f, g)\|_{L^{p, q}}$. Note that for any measurable function $\phi$, 
\begin{align*}
d_{\phi} (t) \ge d_{\phi_j}(t) 
\ge |\{|x|<|h_j| - A: |\phi(x)|>t\}|, 
\quad \forall j: |h_j| > A, 
\end{align*}
where $\phi_j(x) := \phi(x) \mathbf{1}_{\{|x+h_j|>A\}}(x)$. From the estimates above and the translation invariance of $T$, we conclude 
\begin{align*}
\varepsilon
&\ge \lim_{j \to \infty} \|T(f_j, g_j) \mathbf{1}_{B(0, A)^c}\|_{L^{p, q}}
= \lim_{j \to \infty} \|T(f, g)(\cdot-h_j) \mathbf{1}_{B(0, A)^c}\|_{L^{p, q}}
\\
&= \lim_{j \to \infty} \|T(f, g) \mathbf{1}_{\{|\cdot+h_j|>A\}}\|_{L^{p, q}} 
= \|T(f, g)\|_{L^{p, q}}
> 2 \varepsilon,  
\end{align*}
where Lebesgue dominated convergence theorem was used in the last equality. This is a contradiction. Consequently, $T$ is not compact from $L^{p_1, q_1}(\Rn) \times L^{p_2, q_2}(\Rn)$ to $L^{p, q}(\Rn)$.
\end{proof}

\end{document}